\documentclass[12pt]{amsart}
\usepackage{amsfonts}
\usepackage{amsmath}
\usepackage{amsxtra}
\usepackage{amssymb,latexsym}
\usepackage[mathcal]{eucal}
\usepackage{amscd}
\usepackage{tikz-cd}

\makeindex

\newtheorem{theo}{{\bfseries Theorem}}[section]
\newtheorem{prop}[theo]{{\bfseries Proposition}}
\newtheorem{lem}[theo]{{\bfseries Lemma}}
\newtheorem{cor}[theo]{{\bfseries Corollary}}
\newtheorem{df}[theo]{{\bfseries Definition}}

\newtheorem{ex}[theo]{{\bfseries Example}}

\newtheorem{ques}[theo]{{\bfseries Question}}

\newtheorem{add}[theo]{{\bfseries Addendum}}

\def \N {\mathbb N}

\def \Z {\mathbb Z}

\def \Q {\mathbb Q}

\def \O {\mathcal O}

\def \a {\alpha }
\def \b {\beta}

\def \ep {\epsilon}
\def \om {\omega}

\def \g {\gamma}

\def \r {\rho}
\def \s {\sigma}
\def \t {\tau}

\def \tto {\longrightarrow}

\usepackage{amssymb,latexsym}
\usepackage[mathcal]{eucal}
\usepackage{amscd}
\usepackage{amsmath}
\usepackage{graphicx}
\usepackage{amsfonts}
\usepackage{amscd}

\numberwithin{equation}{section}

\begin{document}

\title[Rock, Paper, Scissors, Etc]{\bfseries  Rock, Paper, Scissors, Etc - \\ Topics in the Theory of Regular Tournaments}
\vspace{1cm}
\author{Ethan Akin}
% and W. Patrick Hooper}
\address{Mathematics Department \\
    The City College \\ 137 Street and Convent Avenue \\
       New York City, NY 10031, USA     }
\email{ethanakin@earthlink.net}

\date{June 2018, revised August, 2020}

{\footnotesize \begin{abstract} The classic Rock-Paper-Scissors game of size $3$ and its extension, Rock-Paper-Scissors-Lizard-Spock,
 are modeled by directed graphs called tournaments. They can be further extended to any odd size.
The extended games  are regular tournaments where each strategy beats and is beaten by exactly half of the
 alternatives.   We survey the properties of regular tournaments, which we will call games. In the process we describe a number of constructions for such games.
 These include games on groups of odd order
 and the associated games on coset spaces. We obtain a new lower bound for the number of games of size $2n+1$.

\end{abstract}}



 \maketitle

\tableofcontents

\section*{Introduction}\label{secintro}

The classic Rock-Paper-Scissors game of size $3$ was extended on the TV show \emph{The
Big Bang Theory} to Rock-Paper-Scissors-Lizard-Spock with size $5$. The variation was  originated by Sam Kass with Karen Bryla.
Here we consider  games of larger size. We have a set $I$ of alternative plays. For any pair of distinct
distinct plays, $i, j \in I$, one consistently beats the other. We write $i \to j$ if $i$ beats $j$. Such an arrangement is modeled by a type of
directed graph, or digraph\index{digraph}, called
 a \emph{tournament}\index{tournament}.

 The number of
$j \in I$ which are beaten by $i$ is called the \emph{score} of $i$, denoted $s_i$. The tournaments we will focus upon are those which are balanced
in that every $i \in I$ has the same score.  This requires that the number $|I|$ of elements of $I$ be odd and the score $s_i$ be equal to
$\frac{1}{2}(|I| - 1)$ for all $i$.  Thus, if $|I| = 2n + 1$, each $i$ beats $n$ alternatives and is beaten by the $n$ others. Such a tournament
is called \emph{regular}\index{tournament!regular}.  We will use the term \emph{game}\index{game} to refer to a regular tournament.

The theory of digraphs is described in \cite{CZ} and \cite{HNC}. An elaborate exposition of recent work is given in  \cite{BG}.
A lovely survey of the theory of tournaments is
given in \cite{HM}.  Many computational problems for tournaments are considered in \cite{M}. We will  consider here some of the special properties and
constructions for games, i.e. for regular tournaments.  I would like to express my thanks to my colleague W. Patrick Hooper for his helpful insights and
enjoyable conversations as these matters developed.

Here is an outline of the work.
\vspace{.5cm}

\textbf{Section \ref{secrelations}:} We begin with the elementary definitions and results about relations and digraphs.
For a finite set $I$ with cardinality
$|I|$,  a \emph{relation }\index{relation}$\Pi$ on $I$ is a subset of $I \times I$ with
$\Pi^{-1} = \{ (j,i) : (i,j) \in I \}$ the reverse relation\index{reverse}. A \emph{path}
 $[i_1,\dots,i_n] \in I$ is a sequence\index{path} with $(i_1,i_2),\dots,(i_{n-1},i_n) \in \Pi$.
 A \emph{cycle}\index{cycle} $\langle i_1,\dots, i_n \rangle$ is a
sequence of distinct elements such that $[i_1, \dots, i_n, i_1]$ is a -closed- path.

For $J \subset I$, the \emph{restriction}\index{restriction} of the relation $\Pi$ to $J$ is the relation $\Pi|J = \Pi \cap (J \times J)$.

A relation $\Pi$ is a \emph{digraph}\index{digraph} when $\Pi \cap \Pi^{-1} = \emptyset$. For a digraph we write $i \to j$ when $(i,j) \in \Pi$.
 We let  $\Pi(i)$ denote the set of outputs of $i$, i.e. $\{ j : (i,j) \in \Pi \}$ so that  $\Pi^{-1}(i)$ denotes the set of inputs of $i$.
A digraph $\Pi$ on $I$ is called \emph{Eulerian}\index{digraph!Eulerian} when for each $i \in I$ the number of inputs equals the number of outputs.
When this number is the same
for all $i \in I$ then $\Pi$ is called regular\index{digraph!regular}. Observe that $\emptyset \subset I \times I$ is an Eulerian digraph.

If $\Pi$ is an Eulerian subgraph of $\Gamma$, then $\Gamma$ is Eulerian if and only if $\Gamma \setminus \Pi$ is Eulerian.
Hence, the disjoint union of Eulerian subgraphs is Eulerian. Since a cycle is an Eulerian subgraph, it
follows, as was observed by Euler, that a digraph is Eulerian if and only if it can be written
as a disjoint union of cycles. Notice that disjoint subgraphs can have vertices in common.

If $\Pi$ and $\Gamma$ are digraphs on $I$ and $J$, then a
\emph{morphism}\index{morphism}\index{digraph!morphism} $\r : \Pi \tto \Gamma$ is a map $\r : I \tto J$ such that
for $i, j \in I$ with $\r(i) \not= \r(j)$, $i \to j$ in $\Pi$ if and only if $\r(i) \to \r(j)$ in $\Gamma$.  If $\r$ is a bijective morphism,
then $\r^{-1} : \Gamma \tto \Pi$ is a morphism as well and we call $\r$ an isomorphism.  If, in addition, $\Gamma = \Pi$, then $\r$ is an
automorphism and we let $Aut(\Pi)$ denote the group of automorphisms of $\Pi$. We write $\bar \r$\index{$\bar \r$}
for the product $\r \times \r : I \times I \tto J \times J$. If $\r$ is a bijection, then it is an isomorphism when $\bar \r(\Pi) = \Gamma$.

Since $i \to j$ implies $j \nrightarrow i$  it follows that every
automorphism of a tournament has odd order. Hence, the order $|Aut(\Pi)|$
is odd.
\vspace{.5cm}

\textbf{Section \ref{secgames}:} We begin
our study of games. For a  tournament $\Pi$ on $I$ of size $2n + 1$, the following are equivalent:
\begin{itemize}
\item $\Pi$  is Eulerian.
\item $\Pi$  is  regular.
\item For every $i \in I$ the score $s_i = n$.
\end{itemize}
We will call a regular tournament a \emph{game}\index{game}.  Up to isomorphism, the games of size $3$ and of size $5$ are unique.

If $J \subset I$ with $|J| = 2k + 1$ and $\Pi$ is a game on $I$ then the restriction $\Pi|J$ is a \emph{subgame}\index{subgame} if the tournament
$\Pi|J$ is Eulerian. Equivalently, it is a subgame when $|\Pi(i) \cap J| = k$ for all $i \in J$.

Given a game $\Pi$ on $J$ with $|J| = 2n - 1$ and
 $K \subset J$ with $|K| = n$ we can choose two new vertices $u, v$ and build the \emph{extension}\index{extension} of $\Pi$ via $K$ and $u \to v$
 to obtain a game $\Gamma$ on $I = J \cup \{ u,v \}$ with $u \to v$ and $K = \Gamma(v)$.  On the other hand, if $\Gamma$ is a game on $I$ of
 size $2n + 1$ and $J \subset I$ of size $2n - 1$ is such that the restriction $\Gamma|J$ is a subgame, then with $\{ u, v \} = I \setminus J$
 and $u \to v$ the game $\Gamma$ is the extension of $\Gamma|J$ via $\Gamma(v)$. When a game $\Gamma$ of size $2n + 1$ has a subgame of
 size $2n - 1$ then we call $\Gamma$ a \emph{reducible game}\index{game!reducible}\index{reducible}. A game is reducible via $u \to v$ if and
 only if there does not exist $i \in I$ such that $i \to u$ and $i \to v$ or an $i \in I$ such that $u \to i$ and $v \to i$. The game is
 \emph{completely reducible}\index{completely reducible}\index{reducible!completely}\index{game!completely reducible} when there
 is an increasing sequence $I_1 \subset I_2 \subset \dots \subset I_n = I$ with $|I_k| = 2k + 1$ and $\Gamma|I_k$ a subgame for $k = 1, \dots, n$.

It is obvious that if $\Pi$ is a digraph on $I$, then it is a subgraph of some tournament on $I$.
 If $\Pi$ is Eulerian,
then it is a subgraph of some game on $I$.
\vspace{.5cm}

\textbf{Section \ref{secgroupgames}:} A game $\Gamma$ on $I$ is \emph{point-symmetric} \index{point-symmetric} when the group $Aut(\Gamma)$ acts
transitively on the set of vertices $I$.  As the definition suggests, such games are constructed using groups.
Let $G$ be a finite group with odd order, $|G| = 2n + 1$. Let $e$ denote the
identity element.

A \emph{game subset}\index{game subset}
$A$ is a subset of $G \setminus \{ e \} $ of size $n$ with $A$ disjoint from $A^{-1} = \{ a^{-1} : a \in A \}$.  Equivalently, $A \subset G$ such that
$G \setminus \{ e \} $ is the disjoint union of $A$ and $A^{-1}$. Since no element of $G$ has
order $2$ the $n$ pairs $\{ \{ a,a^{-1}\} : a \in G \setminus \{ e \} \}$ partition the set $G \setminus \{ e \} $. A game subset is obtained
by choosing one element from each pair. Thus, there are $2^n$ game subsets.

Given a game subset $A$, we define on $G$ the game $\Gamma[A] = \{ (i,j) \in G \times G : i^{-1}j \in A \}$. We call such a game
a \emph{group game}\index{group game}. Because the digraph is the Cayley graph of the group $G$ associated with the set of generators $A$, such a
game is also called a \emph{Cayley tournament}\index{Cayley tournament}.

 If $\ell_k$ is the left translation
by $k \in G$, i.e. $\ell_k(i) = ki$, then it is clear that $\ell_k$ is an automorphism of $\Gamma[A]$. Thus, $G$, identified  with the group of left
translations, is a subgroup of $Aut(\Gamma[A])$. Conversely, if $\Gamma$ is a game on $G$ such that all the left translations are contained in
$Aut(\Gamma)$, then  $\Gamma = \Gamma[A]$ with $A = \Gamma(e)$. It follows that the group games are point-symmetric.

If $A \subset G$ is a game subset, then $A^{-1}$ is a game subset with $\Gamma[A^{-1}]$ equal to the reverse game $\Gamma[A]^{-1}$.
If $G$ is abelian, then the map $i \to i^{-1}$ is an isomorphism from a group game onto its reverse game.

If $\r$ is a permutation of $G$ with $\r(e) = e$ and $A$  a game subset of $G$, then $B = \r(A)$ is a game subset with
$\r : \Gamma[A] \tto \Gamma[B]$  an isomorphism if and only if $i^{-1}j \in A$ implies $\r(i)^{-1}\r(j) \in B$. In particular, this holds if
$\r$ is a group automorphism of $G$ since in that case $\r(i)^{-1}\r(j) = \r(i^{-1}j)$.
We let $G^*$ denote the automorphism group of the group $G$.

When  $Aut(\Gamma[A])$ consists exactly of the group $G$ of left translations, $\Gamma[A]$ is called a \emph{tournament regular representation}
\index{tournament regular representation} of $G$, a \emph{TRR}\index{TRR}.  If $\Gamma[A]$ is a TRR and $B$ is a game subset
with $\Gamma[B]$ isomorphic to $\Gamma[A]$, then there is a unique $\xi \in G^*$
such that $B = \xi(A)$ and $\rho = \xi$ is the unique
isomorphism $\rho : \Gamma[A] \to \Gamma[B]$ such that $\rho(e) = e$. In particular,
the  set $\{ \xi(A) : \xi \in G^* \}$ is the set of
game subsets $B$ such that $\Gamma[B]$ is isomorphic to $\Gamma[A]$. Thus, there are exactly $|G^*|$
such game subsets.

In the case when $G$ is cyclic it can be taken to be $\Z_{2n + 1}$, the additive group of integers mod $2n + 1$. The group game $\Gamma[A]$ on
$\Z_{2n + 1}$ is also called a \emph{rotational tournament}\index{rotational tournament} with symbol $A$.

The group automorphisms of $\Z_{2n + 1}$
are multiplications by the units $\Z_{2n + 1}^*$ of the ring $\Z_{2n + 1}$ and so the automorphism group has order $\phi(2n + 1)$ which counts the
numbers between $1$ and $2n + 1$ which are relatively prime to $2n + 1$.

The set
$[1,n] = \{ 1, 2, \dots, n \}$ is a game subset with $\Gamma[[1,n]]$ a TRR for $\Z_{2n +1}$. For $n > 2$
the number of game subsets ( $ = 2^n$) is greater than $\phi(2n + 1)$ and so there exist game subsets $B$ for $\Z_{2n + 1}$ with
$\Gamma[B]$ not isomorphic to $\Gamma[[1,n]]$.

A Fermat prime $p$ is a prime of the form $2^k + 1$ of which only five are known. Only when $2n + 1$ is a square-free product of
Fermat primes is $\phi(2n + 1)$ a power of $2$.  Otherwise, $\Z_{2n+1}^*$ contains a multiplicative subgroup of odd order.

For an odd order group $G$, if $H \subset G^*$ is a subgroup of odd order, then there exists a game subset $A$ of $G$ such that
$\r(A) = A$ for all $\r \in H$. In that case, $\Gamma[A]$ is a group game such that $Aut(\Gamma[A])$ contains, in addition to the translations
$\ell_k$,  the automorphisms $\r \in H$.

We close the section by showing that the only group games which are reducible are those isomorphic to  $\Gamma[[1,n]]$ on $\Z_{2n + 1}$ for some
$n$.
\vspace{.5cm}

\textbf{Section \ref{secinvert}:} Suppose that $\Pi$ and $\Gamma$ are tournaments on a set $I$ and that $\r$ is a permutation of $I$.
We define  $\Delta(\r,\Pi,\Gamma) = \{ (i,j) \in \Pi : (\r(j),\r(i)) \in \Gamma \} $ so that $\Delta(\r,\Pi,\Gamma)$ is the subgraph of $\Pi$ on
which $\bar \r$ reverses direction. We write $\Delta(\Pi,\Gamma)$ for $\Delta(1_I,\Pi,\Gamma)$.

We say that $\r$ \emph{preserves scores}\index{preserves scores} when for all $i \in I, \ |\Pi(i)| = |\Gamma(\r(i))|$. The permutation $\r$
preserves scores if and only if $\Delta(\r,\Pi,\Gamma)$ is Eulerian. If $\Pi$ and $\Gamma$ are
games, then any permutation preserves scores. In particular, if $\Pi$ is a game, then $\Gamma$ is a game if and only if
$\Delta(\r,\Pi,\Gamma)$ is Eulerian.

If $\Delta$ is any subgraph of $\Pi$, we define $\Pi/\Delta$ to be $\Pi$ with $\Delta$ reversed, so that
$\Pi/\Delta = (\Pi \setminus \Delta) \cup \Delta^{-1}$. If $\Delta =  \Delta(\r,\Pi,\Gamma)$, then $\r$ is an isomorphism from
$\Pi/\Delta$ to $\Gamma$. Thus, if $\Pi$ and $\Gamma$ are games on $I$, then $\Gamma$ can be obtained from $\Pi$ by reversing the
Eulerian subgraph $\Delta(\Pi,\Gamma)$. Since an Eulerian graph is a disjoint union of cycles, it follows that $\Gamma$ can be obtained by
successively reversing a sequence of cycles. Furthermore, reversing a cycle can be accomplished by reversing a sequence of 3-cycles.
Thus, we can obtain the game $\Gamma$ from the game $\Pi$ by reversing a sequence of 3-cycles.

A \emph{decomposition}\index{decomposition} for an Eulerian digraph $\Delta$ is a collection of disjoint cycles which covers $\Delta$. It is a
\emph{maximum decomposition}\index{maximum decomposition}
when it is a decomposition of maximum cardinality. We call this maximum cardinality the \emph{span}\index{span} of $\Delta$, denoting it
by $\s(\Delta)$\index{$\s(\Pi)$}. We call $\b(\Delta) = |\Delta| - 2 \s(\Pi)$\index{$\b(\Pi)$} the
\emph{balance invariant}\index{balance invariant} of $\Delta$.

Assume that $\Pi$ and $\Gamma$ are distinct games on $I$. If $\Pi'$ is obtained from $\Pi$ by reversing a 3-cycle, then
$|\b(\Delta(\Pi',\Gamma)) - \b(\Delta(\Pi,\Gamma))| = 1.$ Furthermore, there exists a 3-cycle in $\Pi$ such that
$\b(\Delta(\Pi',\Gamma)) = \b(\Delta(\Pi,\Gamma)) - 1.$. It follows that $\b(\Delta(\Pi,\Gamma))$
is the minimum number of 3-cycles
which must be reversed in order to obtain $\Gamma$ from $\Pi$.

If $\Pi$ is a game which admits a decomposition by 3-cycles, then such a decomposition is clearly a maximum decomposition. We call such a
game a \emph{Steiner game}\index{Steiner game}.  It is a classical result that a set $I$ with $|I| = 2n + 1$ carries some Steiner game if and
only if $n$ is not congruent to $-1$ mod $3$.
\vspace{.5cm}

\textbf{Section \ref{secinterchange}:} If $\Pi$ is a game on $I$ with $|I| = 2n + 1$, then $\Delta \mapsto \Pi/\Delta$ is a one-to-one
correspondence between the Eulerian subgraphs of $\Pi$ -including the empty subgraph- and the set of games on $I$. Thus, the number
of Eulerian subgraphs is the same for all games of size $2n + 1$. The number of 3-cycles contained in $\Pi$ is also the same for all
games of size $2n + 1$.

Define the \emph{interchange graph}\index{interchange graph} to be the -undirected- graph with vertices the games on $I$  and with $\Pi$ and
$\Gamma$  connected by an edge when each is obtained from the other by reversing a 3-cycle. Thus, the interchange graph is a regular, connected
graph.

The distance between two games, $\Pi$ and $\Gamma$, is the length of a path with shortest distance between them.  Such a shortest length
path is called a \emph{geodesic}\index{geodesic}. The distance $d(\Pi,\Gamma)$ is $\b(\Delta(\Pi,\Gamma))$. If $d(\Pi,\Gamma) = k$ then
there are at least $k!$ distinct geodesics between $\Pi$ and $\Gamma$. There exist examples where there are
more than $k!$ geodesics between such games.

In particular, $d(\Pi,\Pi^{-1}) =  \b(\Pi)$.  If $\Pi$ is a game of size $2n - 1$ and $\Gamma$ is an extension of $\Pi$, then
$\b(\Gamma) \leq \b(\Pi) + 2n - 1$. By induction it follows that if $\Gamma$ is a completely reducible game of size $2n + 1$, then
$\b(\Gamma) \leq n^2$. If $\Gamma$ is the group game $\Gamma[[1,n]]$ on $\Z_{2n + 1}$, then $\b(\Gamma) = n^2$.
I conjecture that for any game $\Gamma$ of size $2n + 1$ $\b(\Gamma) \leq n^2$, and, more generally, that the diameter of the interchange graph is $n^2$ for
$I$ with $|I| = 2n + 1$.

If $\Pi$ is a Steiner game, then there is a decomposition by 3-cycles and so the span $\s(\Pi) = n(2n+1)/3$. Thus,

$$ d(\Pi,\Pi^{-1}) \ = \ \b(\Pi) \ = \ |\Pi| - 2 \s(\Pi) \ = \ n(2n+1)/3.$$

In particular, it follows that for $n > 1$ then game $\Gamma[[1,n]]$ is never Steiner.
\vspace{.5cm}

\textbf{Section \ref{secdoublelex}:} If $\Pi$ is a tournament of size $n$ on $I$ we define the \emph{double}\index{double} $2\Pi$ to
be the game of size $2n + 1$ on $\{ 0 \} \cup I \times \{ -1, +1 \}$ with $2\Pi(0) = I \times \{ -1 \}, (2\Pi)^{-1}(0) =  I \times \{ +1 \}$. Let
$i\pm$ denote $(i,\pm 1)$. For $2\Pi$  $i- \to i+$ for all $i \in I$ and if $i \to j$ in $\Pi$ then
$$i- \to j-,\ i+ \to j+, \ j- \to i+, \  j+ \to i-.  $$
We let $\Pi_{\pm}$ denote the restriction of $2\Pi$ to $I \times \{ \pm 1 \}$. Each is clearly isomorphic to $\Pi$.

Any double is completely reducible and so the only group games which could be isomorphic to a double are the isomorphs of $\Gamma[[1,n]]$
on $\Z_{2n + 1}$. The game $\Gamma[[1,n]]$ is indeed isomorphic to the double on its restriction to $[1,n]$.

If every $i \in I$ has both inputs and outputs in $I$, then any automorphism of $2\Pi$ fixes $0$ and leaves $\Pi_-$ and $\Pi_+$ invariant.
This induces an isomorphism between $Aut(\Pi)$ and $Aut(2\Pi)$.
Using this one can construct examples of games which are \emph{rigid}\index{rigid}, i.e. which have trivial
automorphism groups and also games which are not isomorphic to their reverse games.

If $\Pi$ is itself a game, we can construct other examples by reversing subgames.  For example, if $\Pi$ is a Steiner game
then $2\Pi/\Pi_+$ is a Steiner game.

Another construction is the \emph{lexicographic product}\index{lexicographic product} of two digraphs.
Let $\Gamma$ be a digraph on a set $I$
and $\Pi$ be a digraph on a
set $J$.
Define $\Gamma \ltimes \Pi$\index{$\Gamma \ltimes \Pi$} on the set $I \times J$ so that for $p, q \in I \times J$
\begin{displaymath}
p \to q \quad \Longleftrightarrow \quad \begin{cases} p_1 \to q_1 \ \text{in} \ \Gamma,
\quad \text{or}\\ p_1 = q_1 \ \text{and} \ p_2 \to q_2 \ \text{in} \ \Pi. \end{cases}
\end{displaymath}
The map $p \to p_1$ is a surjective morphism from $\Gamma \ltimes \Pi$ to $\Gamma$.

If $\Pi$ and $\Gamma$ are games, then $\Gamma \ltimes \Pi$ is a game. The automorphism group
$Aut(\Gamma \ltimes \Pi)$ is the semi-direct product $ Aut(\Gamma)\ltimes [Aut(\Pi)]^I$ where
$Aut(\Gamma)$ acts on the right by composition on the set of maps from $I$ to $Aut(\Pi)$ regarded as
the product group $[Aut(\Pi)]^I$.

 With $\Gamma_1$ the game of size $3$, we let $\Gamma_{k} =
\Gamma_{k-1} \ltimes \Gamma_1$ so that $\Gamma_k$ is a game on a set of size $3^k$.
From the above computation of the automorphism group it follows that $|Aut(\Gamma_k)| \ = \ (3)^{(3^k - 1)/2}$.
It is known that the order of the automorphism group of a tournament of size $p$ is at most $3^{(p - 1)/2}$, which is $3^n$ when $p = 2n + 1$.
So $\Gamma_k$ is a game with the automorphism group as large as possible.

Finally, if  $\Gamma$ and $\Pi$ are Steiner games, then $\Gamma \ltimes \Pi$ is Steiner.
\vspace{.5cm}

\textbf{Section \ref{secpointed}:} We call a game $\Pi$ on a set $I$ a \emph{pointed game}\index{pointed game}\index{game!pointed} when a particular vertex,
labeled $0$ is singled out. We let $I_+ = \Pi^{-1}(0), I_- = \Pi(0)$ and let $\Pi_{\pm}$ be the tournament which is the restriction of $\Pi$ to
$I_{\pm}$. If $|I| = 2n + 1$ and $\Gamma_+, \Gamma_-$ are arbitrary tournaments on sets of size $n$, then there exists a pointed game $\Pi$ with
$\Pi_+$ isomorphic to $\Gamma_+$ and $\Pi_-$ isomorphic to $\Gamma_-$.
\vspace{.5cm}

\textbf{Section \ref{secinterchangeagain}:}  Given $I$ with $|I| = 2n + 1$, we fix $0 \in I$ and let $I_0 = I \setminus \{ 0 \}$.
The map $\Pi \mapsto J = \Pi(0)$ associates to every game a size $n$ subset of $I_0$. The games which map to $J$ are all the
pointed games with $I_+ = I_0 \setminus J, I_- = J$. The set of such games forms a convex subset, in the suitable sense, of the interchange
graph.  From this we obtain
the lower bound ${2n \choose n} \cdot 2^{n(n - 1)}$ for
the number of games on a set $I$ with $|I| = 2n + 1$. Dividing by $(2n + 1)!$ we obtain a lower bound for the number of isomorphism classes of
games of size $2n + 1$.
\vspace{.5cm}

\textbf{Section \ref{sechomogeneous}:} Assume that $H$ is a subgroup of a group $G$ of odd order. A subset $A$ of $G$ is a \emph{game subset
for the pair $(G,H)$}\index{game subset for the pair $(G,H)$} if $A$ is a game subset for $G$ such that $i \in A \cap G \setminus H$ implies that
the double coset $HiH$ is contained in $A$. Furthermore, $A_0 = A \cap H$ is a game subset for $H$. The double cosets partition $G$ and
$i \not\in H$ implies that $HiH$ is disjoint from $H(i^{-1})H$. If we choose one double coset
from each such pair and choose a game subset $A_0$ for $H$, then
the union is a game subset for the pair.

Let $A$ be a
game subset for $(G,H)$. For the coset space\index{coset space} consisting of the left cosets, $G/H = \{ iH : i \in G \}$ define
$A/H = \{ iH : iH \subset A \}$. The set $\Gamma[A/H] = \{ (iH, jH) : i^{-1}jH \in A/H \}$\index{$\Gamma[A/H]$} is a game on $G/H$
which we will call a \emph{coset space game}\index{coset space game}\index{game!coset space}.
For each $k \in G$ the bijection $\ell_k$ on $G/H$ given by $iH \mapsto kiH$ is an automorphism of $\Gamma[A/H]$ and so
there is a group homomorphism from $G$ to $Aut(\Gamma[A/H])$. The projection $\pi: G \tto G/H$ is a morphism
from $\Gamma[A]$ onto $\Gamma[A/H]$.
Furthermore, there is an isomorphism from the game $\Gamma[A]$  to the lexicographic product $\Gamma[A/H] \ltimes \Gamma[A_0]$.
In particular, $Aut(\Gamma[A])$ is isomorphic to the semi-direct product $Aut(\Gamma[A/H]) \ltimes Aut(\Gamma[A_0])^{G/H}$.

If $H$ is a normal subgroup of $G$, so that $\pi : G \tto G/H$ is
a group homomorphism onto the quotient group, then a subset $A$ of $G$ is a game subset for $(G,H)$ if and only if there exist $B$ a game subset for
$G/H$ and $A_0$ a game subset of $H$ so that $A = A_0 \cup \pi^{-1}(B)$. In that case, the games $\Gamma[A/H]$ and $\Gamma[B]$ are equal.

On the other hand, assume that $G$ is a group of odd order acting on a game $\Pi$ on $I$.
For $a \in I$ the evaluation map $\Phi_a : G \to I$ is defined by $\Phi_a(g) = g \cdot a$.
$Iso_a = \{ g : g \cdot a = a \} = \Phi_a^{-1}( \{ a \})$\index{$Iso_i$} is a subgroup of $G$
called the \emph{isotropy subgroup}\index{isotropy subgroup} of $a$.
Let $Ga = \Phi_a(G) \subset I$\index{$Gi$} denote the $G$ orbit of $a$ and let
$\Pi_a = \Pi \cap (Ga \times Ga)$ be the restriction of $\Pi$ to $Ga$. Of course,
$G$ acts transitively on $I$ exactly when $Ga = I$ in which case $\Pi_a = \Pi$.

Let $H = Iso_a = \Phi_a^{-1}( \{ a \})$. Choose $A_0$ a game subset for $H$ and let $A = A_0 \cup \Phi_a^{-1}(\Pi(a))$.
The set $A \subset G$ is a game subset for $(G,H)$.
The restriction $\Pi_a = \Pi|Ga$ of $\Pi$ to $Ga$ is a subgame of $\Pi$. The map $\Phi_a$ is a
morphism from $\Gamma[A]$ to $\Pi$ and it factors through the canonical projection $\pi$ to define the bijection $\theta_a : G/H \tto Ga$ which is
an isomorphism from $\Gamma[A/H] \tto \Pi_a$.

Applied with $G = Aut(\Gamma)$ we see that the restriction of $\Gamma$ to an $Aut(\Gamma)$ orbit is a subgame which is isomorphic to a coset space game.
In particular,  a tournament is point-symmetric if and only if it is isomorphic to a coset space game.

 The lexicographic product of two group games is isomorphic to a group game and the
 lexicographic product of two coset space games is isomorphic to a coset space game.
\vspace{.5cm}

\textbf{Section \ref{secseven}:} Every game of size $7$ is isomorphic to one of the following three examples.

\textbf{Type I-} The group game $\Gamma_I = \Gamma[[1,2,3]]$ on $\Z_7$ has $Aut(\Gamma[[1,2,3]]) = \Z_7$
acting via translation and is reducible via each pair $i, i+3$.
The collection $\{ m_a( [1,2,3]) : a \in \Z_7^* \}$ are the $6 = \phi(7)$ game
subsets of $\Z_7$ whose games are isomorphic to
$\Gamma[[1,2,3]]$.

The group game $\Gamma_I$ is isomorphic to the
double $2\Pi$ with $\Pi$ the restriction to $[1,2,3]$.

{\bfseries Type II}- For the group game  $\Gamma_{II} = \Gamma[[1,2,4]]$ on  $\Z_7$ is invariant with respect to the multiplicative action of
$H = \{ 1, 2, 4 \}$, the non-trivial, odd order subgroup of $\Z_7^*$. The two game subsets not of Type I are
 $[1,2,4]$ and $[6,5,3] = m_6([1,2,4]) = -[1,2,4]$

The game $\Gamma_{II}$ is not reducible. With $\Pi$ a 3-cycle, $\Gamma_{II}$ is isomorphic to the double $2 \Pi$ with $\Pi_+$ reversed.

{\bfseries Type III}- $\Gamma_{III}$ is the double of $2 \Pi$ with $\Pi$ a 3-cycle. So the automorphism group is isomorphic to that of $\Pi$
and so is cyclic of order $3$. Every automorphism fixes $0$.

The game $\Gamma_{III}$ is  reducible but not via any pair which contains $0$.

Since $\Pi$ is isomorphic to its reversed game, it follows that $\Gamma_{III}$ is isomorphic to its reversed game as well.

The games of Type II and III are Steiner games.
\vspace{.5cm}

\textbf{Section \ref{seciso}:} We consider various isomorphism examples.

There exist non-isomorphic tournaments with
isomorphic doubles.

Every game of size greater than $3$ admits non-isomorphic extensions.

There exist reducible games which can be reduced
in different ways to get non-isomorphic games.
\vspace{.5cm}

\textbf{Section \ref{secnine}:} With $9 = 2 \cdot 4 + 1$ there are $2^4 = 16$ game subsets. We look at the group games.

Consider the cyclic group $G = \Z_9$.

The Type I games come from the $6 = \phi(9)$ subsets $\{ m_a(A) : a \in \Z_9^* \}$ with $A = [1,4]$ or, equivalently,
$A = \{ 1, 3, 5, 7 \}$.

The Type II games are the $6$ subsets $\{ m_a(A) : a \in \Z_9^* \}$ with $A = \{ 1, 5, 6, 7 \}$. In this case, as for Type I, the
automorphism group consists only of the translations by elements of $\Z_9$. These group games are not reducible.

Thus, the Type I and Type II games are non-isomorphic tournament regular representations of the group $\Z_9$.

The Type III games account for the $4$ remaining game subsets. With $H = \{0, 3, 6 \}$ the subgroup of $G$, there are
four game subsets for the pair $(G,H)$. Each game is isomorphic to $\Gamma_3 \ltimes \Gamma_3$ with automorphism group
$\Z_3 \ltimes (\Z_3)^{\Z_3}$.

If, instead, the group is the product group $G = \Z_3 \times \Z_3$ then it is a $2$ dimensional vector space over the field $\Z_3$. The four
one-dimensional subspaces are four subgroups $H$ of order $3$.  For each such $H$ there are four game
subsets for the pair $(G,H)$. This accounts for the $16$ game subsets. Each game is isomorphic to $\Gamma_3 \ltimes \Gamma_3$ with automorphism group
$\Z_3 \ltimes (\Z_3)^{\Z_3}$. In particular, the group $\Z_3 \times \Z_3$ does not admit a tournament regular representation.
\vspace{.5cm}

\textbf{Section \ref{unitour}:}  We here consider infinite tournaments. A countable tournament $\Gamma$ is called \emph{universal} if for every
countable tournament $\Pi$ on $S$, $S_0$ a finite subset of $S$ and $\r : \Pi|S_0 \tto \Gamma$ an embedding (=  an injective morphism),
there exists an embedding $\t : \Pi \tto \Gamma$ which extends $\r$. In the language of \cite{C} these are the
tournaments which are generic for a family of all finite tournaments.

Countable universal tournaments exist and are unique up to isomorphism. They are symmetric tournaments and so by using
the automorphism group, we construct a countable group which acts transitively on the universal tournament.  This yields a countable group
game which contains the universal tournament as a subgame and so contains copies of every finite tournament.
 \vspace{1cm}

\section{Relations and Digraphs}\label{secrelations}

Until Section \ref{unitour} we restrict ourselves to finite sets. For a finite set $I$ we will let $|I|$  denote
the cardinality of $I$. The symmetric group on $I$, that is,
the group of permutations on $I$, is denoted $S(I)$.

Following \cite{A93} we call a subset $\Pi$ of $I \times I$ a \emph{relation}\index{relation} on a  $I$ with $\Pi^{-1} = \{ (i,j) : (j,i) \in \Pi \}$
the \emph{reverse relation}\index{reverse relation}\index{relation!reverse}.
A pair $(i,j) \in \Pi$ is an \emph{edge}\index{edge} in $\Pi$.
For $i \in I$, $\Pi(i) = \{ j : (i,j) \in \Pi \}$ is the set of \emph{outputs}\index{outputs} of $i$,
so that $\Pi^{-1}(i) = \{ j : (j,i) \in \Pi \}$ is the set of \emph{inputs}\index{inputs} of $i$. Thus, a function on $I$ is a relation $\Pi$ such
that each $\Pi(i)$ is a singleton set, e.g. the identity map\index{identity map} $1_I$ is the diagonal $ \{ (i,i) : i \in I \}$. We call $i$ a \emph{vertex}\index{vertex}
 of $\Pi$ when
$\Pi(i) \cup \Pi^{-1}(i)$ is nonempty. Thus, $i \in I$ is a vertex of $\Pi$ when it has at least one input or output.
%We call $\Pi$ a \emph{surjective relation} on $I$ when $\Pi(i)$ and $\Pi^{-1}(i)$ are nonempty for every $i \in I$.

For $J \subset I$, the \emph{restriction}\index{restriction} of $\Pi$ to $J$ is the relation $\Pi|J = \Pi \cap (J \times J)$ on $J$.

 Given relations $\Pi, \Gamma$ on $I$ the
\emph{composition}\index{composition} $\Pi \circ \Gamma = \{ (i,j) :$ there exists $k $ such that
$(i,k) \in \Gamma, (k,j) \in \Pi \}$. Composition is associative and
we inductively define $\Pi^{n + 1} = \Pi^n \circ \Pi = \Pi \circ \Pi^n$, for $n \geq 0$, with $\Pi^1 = \Pi$ and $\Pi^0 = 1_I$ and let
$\Pi^{-n} = (\Pi^{-1})^n$.  We define
$\O \Pi = \bigcup_{n=1}^{\infty}  \Pi^n$\index{$\O \Pi$}. Observe that $ \O (\Pi^{-1}) = (\O \Pi)^{-1}$ and so we may omit the parentheses.

 A $\Pi$ \emph{path}\index{path} from $i_0$ to $i_n$ (or simply a path when $\Pi$ is understood) $[i_0,\dots,i_n] $ is a
 sequence of elements of $I$ with $n \geq 1$ such that
$(i_k,i_{k+1}) \in \Pi$ for $k = 0,\dots, n-1$. The length of the path is $n$. It is a \emph{closed path}\index{path!closed} when $i_n = i_0$.
A path is \emph{simple}\index{path!simple} when the vertices
$i_0,\dots,i_n$ are distinct and \emph{edge-simple}\index{path!edge-simple} when the edges $(i_0,i_1), \dots, (i_{n-1},i_{n})$  are distinct. Clearly,
a simple path is edge-simple, but an edge-simple path may cross itself. An\emph{ $n$ cycle}\index{cycle}, denoted $\langle i_1, \dots, i_n \rangle$, is
a closed path $[i_n, i_1, \dots, i_n]$ such that the vertices $i_1, \dots, i_n$ are distinct, i.e.  $[i_1, \dots, i_n]$ is a simple path, and so
the closed path $[i_n, i_1, \dots, i_n]$ is edge-simple.  A path \emph{spans}\index{path!spanning} $I$ when every $i \in I$ occurs on the path.

Depending on context we will regard a path or a cycle as a sequence of vertices or as a subgraph, i.e. use $[i_0,\dots,i_n]$ for
$\{ (i_k,i_{k+1}) : k = 0,\dots, n-1 \} \subset  \Pi$, and similarly, $\langle i_1, \dots, i_n \rangle = [i_n, i_1, \dots, i_n] \subset \Pi$.

Notice that $(i,j) \in \Pi^n$ exactly when there is a path from $i$ to $j$ of length $n$.

A relation $\Pi$ on $I$ is \emph{reflexive}\index{reflexive relation}\index{relation!reflexive} when $1_I \subset \Pi$.
It is is \emph{symmetric}\index{symmetric relation}\index{relation!symmetric} when $\Pi^{-1} = \Pi$. It is
 \emph{transitive}\index{transitive relation}\index{relation!transitive} when $\Pi \circ \Pi \subset \Pi$ or, equivalently, when
 $\Pi = \O \Pi$. For any relation $\Pi$, $\O \Pi$ is the smallest transitive relation which contains $\Pi$. Observe that,
 \begin{equation}\label{eq00iden}
 \O (1_I \cup \Pi) \ = \ 1_I \cup \O \Pi.
 \end{equation}
 If $\Pi$ is symmetric or transitive, then the reflexive relation $1_I \cup \Pi$ satisfies the corresponding property.

We call $i \in I$ a \emph{recurrent vertex}\index{vertex!recurrent} when $(i,i) \in \O \Pi$. The
relation $\O \Pi \cap \O \Pi^{-1}$ is an equivalence relation on the set
of recurrent vertices. Of course, every $i \in I$ is recurrent for $1_I \cup \Pi$. Since
$\O (1_I \cup \Pi) \cap \O (1_I \cup \Pi)^{-1} = 1_I \cup (\O \Pi \cap \O \Pi^{-1})$,
 the recurrent point equivalence classes for $\O \Pi \cap \O \Pi^{-1}$ are equivalence classes for
$\O (1_I \cup \Pi) \cap \O (1_I \cup \Pi)^{-1}$ and the remaining $\O (1_I \cup \Pi) \cap \O (1_I \cup \Pi)^{-1}$
equivalence classes are singleton sets.

A non-empty subset $J \subset I$ is \emph{strongly connected in $I$}\index{strongly connected}
when $J \times J \subset \O \Pi $ and so
$J \times J \subset \O \Pi \cap \O \Pi^{-1}$. That is, a subset is strongly connected if and only if
it is contained in an $\O \Pi \cap \O \Pi^{-1}$ equivalence class.
Thus, the $\O \Pi \cap \O \Pi^{-1}$ equivalence classes are the maximal strongly connected subsets.
We call $\Pi$ \emph{strong}\index{strong}\index{relation!strong} when the entire
set of vertices of $\Pi$ is a strongly connected set and so the set of vertices comprises a single  $\O \Pi \cap \O \Pi^{-1}$ equivalence class.

\begin{lem}\label{lem00aa} Assume that $\Pi$ is a relation on $I$ and that $J$ is a nonempty subset of $I$.
If the restriction  $\Pi|J$ is strong, then $J$ is strongly connected and so is contained in an $\O \Pi \cap \O \Pi^{-1}$ equivalence class. Conversely, if
$J$ is  an $\O \Pi \cap \O \Pi^{-1}$ equivalence class, then the restriction $\Pi|J$ is strong. \end{lem}

\begin{proof} The first result is obvious.
 % If $\Pi|J$ is strong, then $J$ is strongly connected in $I$ and so is contained in an $\O \Pi \cap \O \Pi^{-1}$ equivalence class.

Now assume that $J$ is  an $\O \Pi \cap \O \Pi^{-1}$ equivalence class.
If two points $i, j \in J$ then there is a closed path from $i$ to $i$ which passes through $j$. All of the
points on the closed path are $\O \Pi \cap \O \Pi^{-1}$ equivalent to $i$ and $j$ and so are contained in $J$. Hence, the closed path for $\Pi$ is
a closed path for $\Pi|J$.

\end{proof} \vspace{.5cm}

{\bfseries Remark:} Notice that if $i, j$ are distinct elements of an $\O \Pi \cap \O \Pi^{-1}$ equivalence class, the set $J = \{ i, j \}$ is strongly
connected. On the other hand, the restriction $\Pi|J$ is strong only if $(i,j)$ and $(j,i)$ are both elements of $\Pi$. \vspace{.5cm}

A subset $J \subset I$ is an \emph{invariant set}\index{invariant set} for a relation $\Pi$ on $I$ when $\Pi(J) \subset J$ and so $\O \Pi(J) \subset J$.
A subset $J$ is invariant for $\Pi$ if and only if the complement $I \setminus J$ is invariant for $\Pi^{-1}$. For any subset $J$, the set
$J \cup \O \Pi(J)$ is the smallest invariant set which contains $J$. It is clear that $\Pi$ is strong if and only if $I$
contains no proper invariant subset. In particular, if $i,j$ are vertices of $\Pi$ and there is no path from $j$ to $i$ then $\{ j \} \cup \O \Pi(j)$
is an invariant set which contains $j$ but not $i$.

We will call a relation $\Pi$ on $I$ a \emph{digraph}\index{digraph}\index{graph!directed} when  $\Pi \cap \Pi^{-1} = \emptyset$.  In particular,
we have $\Pi \cap 1_I = \emptyset$. That is, we are interpreting $\Pi$ as a
 directed graph with every pair of distinct elements of $I$ connected by at most one oriented edge and
 no element of $I$ is connected by an edge to itself. We write $i \to j$ when $(i,j)$ is an edge of the digraph.

 We will call a relation $\Pi$ on $I$ an \emph{undirected graph}\index{undirected graph}\index{graph!undirected} when $\Pi \cap 1_I = \emptyset$ and
 $\Pi = \Pi^{-1}$. That is, we are interpreting $\Pi$ as a
 graph with every pair of distinct elements of $I$ connected by at most one unoriented edge and
 no element of $I$ is connected by an edge to itself.

 A digraph or undirected graph $\Pi$ on $I$ is called \emph{bipartite}\index{bipartite}\index{graph!bipartite} when $I$ is the union of
 disjoint sets $J,K$ and $\Pi \subset (J \times K) \cup (K \times J)$. That is, elements of $J$ are connected by an edge only to elements of $K$ and
 vice-versa.

  For a graph on $I$ the number of vertices is the
 \emph{size} of the graph. \index{size}\index{graph!size}

% We will write $i \to j$ for $(i,j) \in \Pi$
% when the digraph is understood. We will call $\Pi(i)$ the set of \emph{outputs} from $i$ and $\Pi^{-1}(i)$ the set of \emph{inputs }to $i$.
%Thus, $i \in I$ is a vertex of $\Pi$ when it has at least one input or output.
%If $I \subset J$ then a digraph on $I$ is a digraph on
% $J$ since $I \times I \subset J \times J$.

 A subset of a digraph $\Pi$ is a digraph and we will call it a \emph{subgraph}\index{subgraph}
  of $\Pi$. Observe that disjoint subgraphs may have
 vertices in common. We will call two subgraphs \emph{separated}\index{separated subgraphs} when they have no vertices in common. Of course,
 separated subgraphs are disjoint.  We call $\Pi$ \emph{connected}\index{digraph!connected}
 if it cannot be written as the disjoint union of two separated proper subgraphs.

% If $\Pi$ is a digraph on $I$ and $J \subset I$ then the \emph{restriction} of $\Pi$ to $J$ is the subgraph
% $\Pi \cap (J \times J)$.

 A digraph $\Pi$ on $I$ is called a round robin tournament, or simply a \emph{tournament}\index{tournament}, when it is complete on $I$, i.e.
 when every pair of distinct elements of $I$ is connected by an edge.
Thus, a digraph $\Pi$ is a tournament when
 $\Pi \cup \Pi^{-1} \cup 1_I = I \times I$. Clearly, if $\Pi$ is a tournament on $I$
 and $J$ is a subset of $I$, then the restriction $\Pi|J$ is  a tournament on $J$.

 For a tournament on $I$ the cardinality $|I|$ is the
size of the tournament. By convention, $\Pi = \emptyset$ on a singleton set $I$ is the trivial tournament of size $1$.

 For a tournament $\Pi$ we think of $i \to j$ to mean $i$ beats $j$. Hence, the \emph{score}\index{score} for $i$ is the
 cardinality of the output set, $|\Pi(i)|$.  The \emph{score vector}\index{score vector}
  $(s_1, \dots, s_p)$ for a tournament of size $p$, consists of
 the scores of the elements, listed in non-decreasing order. The score vector for a tournament was introduced and characterized by Landau
 \cite{L3} as a tool for his study of animal behavior. It is clear that the sum of the scores is the number of edges $p(p - 1)/2$.

 Clearly, any digraph on $I$ can be extended to occur as a subgraph of some tournament on $I$.

 The following is a sharpening by Moon, \cite{M} Theorem 3, of a result of Harary and Moser, see \cite{HM} Theorem 7.

\begin{prop}\label{prop06} If $\Pi$ is a strong tournament on $I$ with $|I| = p > 1$ and $i \in I$,
then for every $\ell$ with $3 \leq \ell \leq p$
there exists a $\ell$-cycle  in $\Pi$ passing through $i$.
\end{prop}

\begin{proof} If $\Pi(i)$ or $\Pi^{-1}(i)$ is empty then there is no path from $i$ to $i$ and so the tournament is not strong.
If for no $j_1 \in \Pi(i)$ and $j_2 \in \Pi^{-1}(i)$ is it true that $j_1 \to j_2$ then  $\Pi(i)$ is a proper
invariant subset and so the
tournament is not strong.  Hence, there is
a 3-cycle through $i$.

Now suppose that $\langle i_1,\dots,i_r \rangle$ is a cycle through $i$ with $r < p$.  We show that we can enlarge the cycle to one of length $r + 1$.
\vspace{.5cm}

\textbf{Case 1} (There exists $j$ not on the cycle but such that both $\Pi(j)$ and $\Pi^{-1}(j)$ meet the cycle):
By relabeling we may assume that $i_1 \in \Pi^{-1}(j)$.
Let $s$ be the largest integer such that $i_1,\dots,i_s \in \Pi^{-1}(j)$. By hypothesis, $s < r$ and $i_{s+1} \in \Pi(j)$.
Hence, $\langle i_1,\dots,i_s,j,i_{s+1},\dots,i_r \rangle$ is a $r+1$ cycle.
\vspace{.5cm}

\textbf{Case 2 }(For every $j$ not on the cycle either  $\Pi(j)$ or $\Pi^{-1}(j)$ does not meet the cycle): Observe that if $\Pi(j)$ does not
meet the cycle, then $i_1,\dots,i_r \in \Pi^{-1}(j)$. Let $A = \{ j : i_1,\dots,i_r \in \Pi^{-1}(j) \}$ and
$B  = \{ j : i_1,\dots,i_r \in \Pi(j) \}$.
By assumption, $r < p$ and $I \setminus \{i_1,\dots,i_r \} = A \cup B$. There must exist $u \in A$ and $v \in B$ with
$u \to v$. Otherwise, $A$ and $\{i_1, \dots, i_r \} \cup A$ are invariant subsets and at least one is a proper subset. Thus, the tournament
would not be strong.  Hence, $\langle i_1, u ,v, i_3, \dots,i_r \rangle$
(omitting $i_2$) is a $r + 1$ cycle which contains $i$.

\end{proof} \vspace{.5cm}

 A \emph{Hamiltonian cycle} \index{Hamiltonian cycle}\index{cycle!Hamiltonian} in a digraph is
a cycle which passes through every vertex.  Proposition \ref{prop06} implies that a tournament admits a Hamiltonian cycle if it is strong.
The converse is obviously true.

The opposite extreme of a strong tournament is an \emph{order}\index{order}. A relation $\Pi$ on $I$ is an order (to be precise, a strict, total order)
when $\Pi$ is a transitive tournament. For example, for $[1,p] = \{ 1, 2, \dots, p \}$ we let $i \to j$ when $i < j$ to define the \emph{standard order}
\index{standard order}\index{order!standard} on $[1,p] $.

\begin{prop}\label{prop06order} For $\Pi$ is a tournament on $I$, with $|I| = p$, the following conditions are equivalent.
\begin{itemize}
\item[(i)] $\Pi$ is an order.
\item[(ii)] $\Pi$ contains no cycles.
\item[(iii)] $\Pi$ contains no $3$-cycles.
\item[(iv)] No vertex of $\Pi$ is recurrent.
\item[(v)] Every equivalence class of $\O (1_I \cup \Pi) \cap \O (1_I \cup \Pi)^{-1}$ is a singleton, i.e. $1_I = \O (1_I \cup \Pi) \cap \O (1_I \cup \Pi)^{-1}$.
\item[(vi)] The score vector of $\Pi$ is $(0,1,\dots,p-1)$.
\item[(vii)] There is a bijection $k \mapsto i_k$  from $[1,p]$ to $I$ which induces an isomorphism from the standard order onto $\Pi$, that is,
$i_k \to i_{\ell}$ if and only if $k < \ell$.
\end{itemize}
\end{prop}

\begin{proof} Observe that $i \to j \to k$ and $i \not\to k$ implies $\langle i, j, k \rangle$ is a $3$-cycle. The equivalences of (i)-(v) are then easy to check.
It is obvious that (vii) implies (i) and (vi).  For the converse directions we use induction on $p$.

If $\Pi$ is an order, and $i \in I$ then
$\Pi|(I \setminus \{ i \})$ is an order. Choose $i$ the maximum element in the ordering.
By induction hypothesis there is a numbering $k \mapsto i_k$ of $I \setminus \{ i \}$ as in (vii)
inducing an isomorphism from the standard order on $[1,p-1]$ onto $\Pi|(I \setminus \{ i \})$.
Extend by letting $i_p = i$.
% Let $k^*$
%be the maximum $k$ such that $j_k \to i$. If there is none such then let $k^* = 0$. By transitivity $j_{\ell} \to i$ for $\ell < k^*$. Define
%$i_k = j_k$ for $k \leq k^*$, $i_{k^* + 1} = i$ and $i_{k} = j_{k-1}$ for $ k^* + 1 < k \leq p$.

If $\Pi$ has score vector $(0,1,\dots, p-1)$, then let
$i \in I$ with score $0$. Every $k \to i$ for $k \not= i$ and so $\Pi|(I \setminus \{ i \})$ has score vector $(0,\dots, p-2)$. By induction hypothesis
we have $k \mapsto i_k$ as in (vii) for $\Pi|(I \setminus \{ i \})$. As before, extend by letting $i_p = i$.

\end{proof}

{\bfseries Remark:} For any tournament $\Pi$ it is clear that $\O \Pi$ induces an order on the set of $\O (1_I \cup \Pi) \cap \O (1_I \cup \Pi)^{-1}$
equivalence classes, with $[i] \to [j]$ for distinct classes $[i],  [j]$ when  $(i,j) \in \O \Pi$. Furthermore, the map $j \mapsto [j]$ is a
surjective morphism from $\Pi$ onto this order.
\vspace{.5cm}

A digraph $\Pi$ on $I$ is \emph{Eulerian}\index{digraph!Eulerian}\index{Eulerian}
 when for all $i \in I$, $|\Pi(i)| = |\Pi^{-1}(i)|$. That is, each element of $I$ has the same number of
inputs and outputs.
In \cite{HNC} such a digraph is called an \emph{isograph}. In general,
for an Eulerian digraph the cardinality $|\Pi(i)| = |\Pi^{-1}(i)|$ may vary with $i$. When it does not, when $k = |\Pi(i)| = |\Pi^{-1}(i)|$ is the
same for all $i$, the digraph is called \emph{regular}\index{digraph!regular}\index{regular} or $k$-regular.

By convention, we will regard the empty digraph as Eulerian.  Otherwise, for an Eulerian digraph $\Pi$ on $I$, we will assume that every $i \in I$ is a
vertex, i.e. $|\Pi(i)| = |\Pi^{-1}(i)| > 0$ for all $i \in I$.
 % Obviously, if $i$ is a vertex of
%a Eulerian graph then neither $\Pi(i)$ not $\Pi^{-1}(i)$ is empty.
% Notice that if $|\Pi(i)| = k_+$ and $|\Pi^{-1}(i)| = k_-$ for all $i$ then
%$|I|\cdot k_+$ and $|I|\cdot k_-$ are each equal to the number of edges. So $k_+ = k_-$ and $\Pi$ is $k = k_+ = k_-$ regular.

\begin{prop}\label{prop06euler} Let $\Pi$ be a digraph on $I$.
\begin{enumerate}
\item[(a)] If there exist $k_+, k_-$ such that $|\Pi(i)| = k_+$ and $|\Pi^{-1}(i)| = k_-$ for all $i \in I$, then
 $k_+ = k_-$ and  $\Pi$ is $k = k_+ = k_-$ regular.

\item[(b)]If $\Pi$ is a tournament, then the following are equivalent
\begin{itemize}
\item[(i)] $\Pi$ is Eulerian.
\item[(ii)] $\Pi$ is regular.
\item[(iii)] $|\Pi(i_1)| = |\Pi(i_2)|$ for all $i_1, i_2 \in I$.
\end{itemize}
When these conditions hold, $|I|$ is odd and $\Pi$ is $k$-regular with $k = \frac{|I|-1}{2}$.

\end{enumerate}\end{prop}

\begin{proof} (a) Because $|I|\cdot k_+$ and $|I|\cdot k_-$ are each equal to the number of edges, it follows that $k_+ = k_-$.

(b) For a tournament $\Pi$, $|I| = |\Pi(i)| + |\Pi^{-1}(i)| + 1$ for all $i \in I$. So for an Eulerian tournament,
$|I|$ is odd and if $|I| = 2n + 1$ then $|\Pi(i)| = |\Pi^{-1}(i)| = n$. So (i) $\Leftrightarrow$ (ii). Obviously, (ii) $\Rightarrow$ (iii).
If (iii) holds so that $|\Pi(i)| = k_+$ for all $i$ then $|\Pi^{-1}(i)| = |I| - k_+ - 1$ for all $i$ and
so (a) implies  $\Pi$ is regular, proving (iii) $\Rightarrow$ (ii).

\end{proof} \vspace{.5cm}

Henceforth, we will call a regular tournament a \emph{game}\index{game}, for short,
 as these are the tournaments which generalize the Rock-Paper-Scissors game. Clearly, a game of size $2n + 1$ is a tournament
 with score vector given by $s_r = n$ for $r = 1, \dots, 2n+1$.

  The \emph{trivial game}\index{game!trivial}  has size $1$.
  That is, $I$ is a singleton and the unique digraph is the empty subset of $I \times I$. %We will call an Eulerian digraph $\Pi$  on $I$
%a \emph{proper Eulerian digraph} \index{proper Eulerian digraph} \index{digraph!proper Eulerian} when every $i \in I$ is a vertex, i.e.
%$|\Pi(i)|= |\Pi^{-1}(i)| > 0$ for all $i \in I$.

A cycle is obviously Eulerian since $|\Pi(i)| = 1 = |\Pi^{-1}(i)|$ for every $i$ on the cycle.

Notice that if a digraph $\Pi$ is a tournament or is Eulerian, then $\Pi^{-1}$ is a digraph satisfying the corresponding property. Thus, the reverse of
a game is a game.

\begin{lem}\label{lem01} If $\Pi_1$ is an Eulerian subgraph of a digraph $\Pi$, then $\Pi$ is Eulerian
if and only if $\Pi \setminus \Pi_1$ is Eulerian. In particular, the
union of disjoint Eulerian  graphs on $I$ is Eulerian. \end{lem}

\begin{proof} If $\Gamma = \Pi \setminus \Pi_1$ then $\Pi(i)$ is the disjoint union of $\Pi_1(i)$ and $\Gamma(i)$. Hence,
$|\Pi(i)| = |\Pi_1(i)| + |\Gamma(i)|$. Similarly, $|\Pi^{-1}(i)| = |\Pi^{-1}_1(i)| + |\Gamma^{-1}(i)|$. By assumption,
$|\Pi_1(i)| = |\Pi_1^{-1}(i)|$. So $|\Pi(i)| = |\Pi^{-1}(i)|$ if and only if $|\Gamma(i)| = |\Gamma^{-1}(i)|$.

\end{proof} \vspace{.5cm}

The following observation is essentially due to Euler, see, e.g. \cite{HNC} Theorem 12.5.

 \begin{theo}\label{theo02}  Any  nonempty Eulerian digraph can be written as a disjoint union of cycles.
 \end{theo}

 \begin{proof} Let $\Pi$ be a nonempty  Eulerian digraph on $I$. By assumption
  $\Pi(i)$ and $\Pi^{-1}(i)$ are nonempty for every $i \in I$. Beginning with any vertex we can build a
 simple path $[i_1, \dots, i_k ]$ and continue until $i_p \in \Pi(i_k)$ for some $p < k$ and so, necessarily, $p < k - 1$. Then
 $\langle i_p,\dots,i_k\rangle$   is a cycle in $\Pi$.

By Lemma \ref{lem01} $\Pi \setminus \langle i_p,\dots,i_k\rangle$   is an Eulerian digraph and so, if it is nonempty, it contains a
 cycle disjoint from $\langle i_p, \dots, i_k\rangle$.

 Continue inductively to exhaust $\Pi$.

 \end{proof} \vspace{.5cm}

 {\bfseries Remark:} The decomposition of an Eulerian digraph into disjoint cycles is usually not unique.
 \vspace{.5cm}

The following is essentially Theorem 7.4 of \cite{CZ}.

 \begin{theo}\label{theo02a} Assume that  $[i_0, \dots, i_k] $ with $k > 1$ is a closed edge-simple path for a digraph $\Pi$ (and so $i_k = i_0$).
Regarded as a subgraph of $\Pi$, $[i_0, \dots, i_k] $ is a strong, Eulerian digraph
 and so is a disjoint union of cycles. Conversely, if  $\Pi$ is a connected, Eulerian digraph, then it admits
 a spanning, closed, edge-simple path. In particular, a connected, Eulerian digraph is strong.  \end{theo}

  \begin{proof} Of course, if the vertices $i_1, \dots, i_{k-1}$ are distinct, then $[i_0, \dots, i_k] $ is a single $k$ cycle (and conversely).
  However, while we are assuming the edges are distinct, the vertices need not be. Nonetheless, the input edge $(i_{r-1},i_{r})$ for $i_r$
  is balanced by the output edge $(i_r,i_{r+1})$. Since the edges are distinct, each vertex has the same number of inputs and outputs.

  Since $[i_0, \dots, i_k] $  is Eulerian, it is a disjoint union of cycles by Theorem \ref{theo02}.

  Conversely, assume that $\Pi$ is a connected Eulerian digraph and so is a disjoint union of cycles $C_1, \dots, C_{n}$.
  We prove the existence of the required spanning path by induction on $n$.  If $\Pi$ consists of a single cycle, the
  result is obvious.

  Define the reflexive, symmetric relation
  $R$ on $[1, n] = \{ 1, \dots, n\}$ by $(p,q) \in R$ when $C_p$ and $C_q$ have a vertex in common. Thus, $\O R$ is an equivalence relation on
  $[1, n]$. If $\Pi_1$ is the union of cycles in an $\O R$ equivalence class, then $\Pi_1$ and $\Pi \setminus \Pi_1$ have no vertices
  in common.  Since $\Pi$ is assumed to be connected, it must be the union of a single $\O R$ equivalence class. Let $k$ be the smallest
  positive integer such that $[1, n] \times [1, n] \subset R^k$ and so for some $p, q \in [1, n]$ $(p,q) \in R^k \setminus R^{k-1}$. For any
  $q_1 \not= q$ there is an $R$ path from $p$ to $q_1$ with length at most $k$ and  $q$ does not lie on such a path. It follows that
  the Eulerian digraph $\Gamma = \bigcup_{r \not= q} C_r$ is connected and is the union of $n-1$ disjoint cycles. By induction hypothesis
  there exists $[i_0, \dots, i_{\ell}] $,  a closed edge-simple path which spans $\Gamma$. The cycle $C_q = \langle j_1,\dots,j_r \rangle$
  has a vertex in common with  $\Gamma$. By relabeling we may assume $i_{\ell} = i_0 = j_r$. Then $[i_0, \dots, i_{\ell}, j_1,\dots,j_r] $
  is a closed, edge-simple path which spans $\Pi$.

  \end{proof} \vspace{.5cm}

Given a map $\r : I \tto J$ we let $\bar \r$ \index{$\bar \r$} denote the product map $\r \times \r : I \times I \tto J \times J$.

\begin{df}\label{def03} Let $\Pi$ and $\Gamma$ be digraphs on $I$ and $J$, respectively. A
\emph{morphism}\index{morphism}\index{digraph!morphism} $\r : \Pi \tto \Gamma$ is a
map $\r : I \tto J$ such that $(\bar \r)^{-1}(\Gamma) = \Pi \setminus (\bar \r)^{-1}(1_J)$.  That is, for $i_1, i_2 \in I$ with
$\r(i_1) \not= \r(i_2)$ $\r(i_1) \to \r(i_2)$ if and only if $i_1 \to i_2$. In particular, if $\r$ is injective, then it is a morphism
if and only if $(\bar \r)^{-1}(\Gamma) = \Pi$ and if it is bijective then it is a morphism if and only if $\bar \r(\Pi) = \Gamma$.\end{df}

Clearly, if $\r$ is a bijective morphism then $\r^{-1}$ is a morphism and so $\r$ is an isomorphism\index{isomorphism}.
Two digraphs are isomorphic when each can be obtained from the
other by relabeling the vertices.

An \emph{automorphism}\index{automorphism} of $\Pi$ is an
isomorphism with $\Pi = \Gamma$. We let $Aut(\Pi)$\index{$Aut(\Pi)$} denote the automorphism group of $\Pi$.

We call $\r$ an \emph{embedding} \index{embedding} when it is an injective morphism.

If $J \subset I$, then the inclusion map from $J$ to $I$ induces an embedding from the restriction $\Pi|J$ to $\Pi$.
On the other hand, if $\r : \Pi \tto \Gamma$ is an embedding and $I_1 = \r(I) \subset J$, then $\r : I \tto I_1$ induces
an isomorphism from $\Pi$ onto the restriction $\Gamma|I_1$.

An automorphism $\r$ of a digraph $\Pi$ is a permutation of the vertices of $\Pi$ and so is a product of disjoint permutation cycles.
Observe that if $i \to \r(i)$ then $\rho(i) \to \rho^2(i)$. So if $\rho$
includes the permutation cycle $(i_1,\dots,i_k)$ and $i_1 \to i_2$ then
$i_2 \to i_3$, ... ,$i_k \to i_1$. Thus,  $\langle i_1,\dots,i_k \rangle$ is a
$k$-cycle in the digraph. Otherwise, $i_2 \to i_1$ and so
$i_1 \to i_k$,...,$i_3 \to i_2$. In that case  $\langle i_k,i_{k-1},\dots,i_1 \rangle$ is a
$k$-cycle in $\Pi$.

Since $i \to j$ implies that $j \not\to i$, an automorphism of a tournament can contain no
transposition.  In fact, it contains no even permutation cycle.

\begin{prop}\label{prop04} If $\rho$ is an automorphism of a
tournament $\Pi$, then $\rho$ is a permutation of odd order.
If a pair $\{ i, j \}$ is $\rho$ invariant then $\rho$ fixes each element of the pair. \end{prop}

\begin{proof} If $\rho^{2k}$ is the identity but $\rho^k$ is not, then
for some vertex $i_1,$ $ \rho^k(i_1) = i_2  \not= i_1$.
Since $\rho^k(i_2) = i_1$, the pair $(i_1, i_2)$ is a transposition for the
automorphism $\rho^k$ which we have seen cannot happen.

If $\{ i, j \}$ is invariant for a permutation then the restriction to $\{ i, j \}$ is
either the identity or a transposition and the latter is impossible for
an automorphism.

\end{proof}

%The theory of digraphs is described in \cite{CZ} and \cite{HNC}. A lovely survey of the theory of tournaments is
%given in \cite{HM}.  Many computational problems for tournaments are considered in \cite{M}.

 \vspace{1cm}

\section{Games}\label{secgames}

Recall that when $\Gamma$ is a regular tournament on $I$, we call $\Gamma$  a \emph{game}\index{game} of size $|I| = 2n + 1$.

A subset $J$ of cardinality $|J| = 2k+1$ forms a \emph{subgame}\index{subgame} when the restriction $\Gamma|J$ is Eulerian, i.e.
for each $i \in J$, $\Gamma(i) \cap J$ has cardinality $k$.

A $3$-cycle is a subgame of size $3$. If three vertices do not form a $3$-cycle, then they form a $3$-order, i.e. an order on the three vertices with score
vector $(0,1,2)$.

Of course, any $3$-cycle, i.e. any game of size $3$, is isomorphic to the original
Rock-Paper-Scissors game.  The same uniqueness holds for $5$.

\begin{theo}\label{theo05} Up to isomorphism there is one game $\Gamma$ of size $5$. \end{theo}

\begin{proof} Choose any vertex and label it $0$. The pair of outpoints
$\Gamma(0)$ form an edge which we label $1$ and $2$ with $1 \to 2$.
The input pair $\Gamma^{-1}(0)$ we label $3$ and $4$ with $3 \to 4$. The remaining directions are now determined.
We began with $3, 4 \to 0 \to 1,2 $
\begin{itemize}
\item $ 0,1  \to 2 \quad \Rightarrow \quad 2 \to 3, 4$.
\item $3 \to 0,4  \quad \Rightarrow \quad 1, 2 \to 3$.
\item $2,3 \to 4 \quad \Rightarrow \quad 4 \to 0,1$.
\item $1 \to 2,3\quad \Rightarrow \quad 0,4 \to 1$.
\end{itemize}
We can diagram the result.
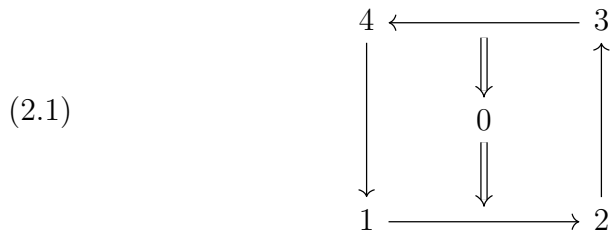
\begin{equation}\label{cd1}
\begin{tikzcd}
4  \arrow[dd] && 3 \arrow[ll,""{name=U,below}] \\
& |[alias=Z]|0\arrow[Rightarrow,from=U,to=Z]  \\
1 \arrow[rr,""{name=D,above}] \arrow[Rightarrow,from=Z,to=D]&& 2 \arrow[uu]
%\arrow[Rightarrow, from=U, to=D,"0"]
\end{tikzcd}
\end{equation}

\end{proof} \vspace{.5cm}

There is a general construction which builds a game of size $2n + 1$ from one of size $2n - 1$.

Let $\Pi$ be a game of size $2n - 1$ on the set of vertices $J$ and $K \subset J$
with $|K| = n$. Let $u,v$ be two additional vertices and let
$I = J \cup \{ u,v \}$.  Define $\Gamma$ with $\Gamma|J = \Pi$ and so that $u \to v$ and $i \to u$ for all
$i \in K$.

Thus, $|\Gamma^{-1}(u)| = n$ and  $|\Gamma(i)| = n$ for each $i \in K$. In order that $\Gamma$ be a game it is necessary that
$u \to j$ for all $j \in J \setminus K$ and
$v \to i$ for all $i \in K$. So $|\Gamma(v)| = n$ requires $j \to v$ for all $j \in J \setminus K$. Conversely,
these  conditions imply that $\Gamma$ is a game.

 We then call
$\Gamma$ an \emph{extension}\index{extension}\index{extension!via $u \to v$ and $K$} of $\Pi$ via $u \to v$ and $K$.

Now assume that $\Gamma$ is a game of size $2n + 1$ on the set of vertices $I$ and $u,v \in I$
such that the restriction $\Pi = \Gamma|J$ with $J = I \setminus \{ u,v \}$ is a subgame of $\Gamma$.   Assume that $u \to v$,
and let
$K = \Gamma^{-1}(u)$ so that $|K| = n$.
% Observe that if $j \in K$, Then $\Pi^{-1}(j) = \Gamma^{-1}(j) \setminus \{ v \}$ has $n - 1$ elements, while
%$|\Gamma^{-1}(j)| = n$. Hence, $v \to j$ for all $j \in K$. Since $n = |K| = |\Gamma(v)|$ it follows that $\Gamma(v) = K$.
%Thus,
We see that $\Gamma$ is the extension of $\Pi$ via  $u \to v$ and
$K $.

We say that $\Gamma$ is \emph{reducible}\index{game!reducible}\index{reducible}  via $\{ u,v \}$ when $\Gamma$ restricts to a
subgame on $I \setminus \{ u, v \}$.

%If $\Gamma$ is a game of size $2n + 1$ on the set of vertices $I$ and $u,v \in I$
%such that the restrict to $J = I \setminus \{ u,v \}$ is a subgame of $\Gamma$
%then we say that $\Gamma$ is \emph{reducible}\index{game!reducible}\index{reducible}  via $\{ u,v \}$.    If $u \to v$,
%then $\Gamma$ is the extension of $\Pi = \Gamma|J$ via  $u \to v$ and
%$K = \Gamma^{-1}(u)$.

Notice that the game of size $3$ is reducible to the trivial game. The game $\Gamma$ of size $5$ is reducible via any pair $\{ u,v \}$
such that $\Gamma|(I \setminus \{ u, v \})$ is a $3$-cycle, e.g. via $4 \to 1$ in (\ref{cd1}).

\begin{prop}\label{prop07} Let $\Gamma$ be a game of size $2n + 1$ on the set of vertices $I$.
\begin{enumerate}
\item[(a)] For all $u, v \in I$, $\Gamma(u) = \Gamma(v)$ or $\Gamma^{-1}(u) = \Gamma^{-1}(v)$ implies $u = v$.
\item[(b)] Every edge is contained in at least one $3$-cycle.
\item[(c)] Every edge is contained in at most $n$ $3$-cycles.
\item[(d)] For $u \to v \in I$ the following conditions are equivalent:
\begin{itemize}
\item[(i)] $\Gamma$ is reducible via $\{ u, v \}$.
\item[(ii)] The restriction $\Gamma|(I \setminus \{ u, v \})$ is a subgame of $\Gamma$.
\item[(iii)] The edge between $u$ and $v$ is contained in $n$ $3$-cycles.
\item[(iv)] $|\Gamma^{-1}(u) \cap \Gamma(v)| = n$.
\item[(v)] $\Gamma^{-1}(u) = \Gamma(v)$.
\item[(vi)] $\Gamma(u) \setminus \{ v \} = \Gamma^{-1}(v) \setminus \{ u \}$.
\item[(vii)] There does not exist $i \in I \setminus \{ u,v \}$ such that $i \to u$ and $i \to v$.
 \item[(viii)] There does not exist $i \in I \setminus \{ u,v \}$ such that $u \to i$ and $v \to i$.
 \item[(ix)] $\Gamma^{-1}(u) \cap \Gamma^{-1}(v) = \emptyset$.
 \item[(x)] $\Gamma(u) \cap \Gamma(v) = \emptyset$.
 \item[(xi)] $\Gamma(u) \cup \Gamma(v) \cup \{ u,v \} = I$.
 \item[(xii)] $\Gamma^{-1}(u) \cup \Gamma^{-1}(v) \cup \{ u,v \} = I$.
 \end{itemize}

\item[(e)] If $\Gamma$ is the extension of $\Pi$ via $u \to v$ and $K$, then the
reversed game $\Gamma^{-1}$ is the extension
of $\Pi^{-1}$ via $v \to u$ and $K$.
\item[(f)] For all $u \in I$ there is at most one $v$ such that $u \to v$ and
$\Gamma$ is reducible via $\{ u, v \}$, and
 there is at most one $v$ such that $v \to u$ and  $\Gamma$ is reducible via $\{ u, v \}$.
\end{enumerate}
\end{prop}

\begin{proof} Assume $u \to v$ in $\Gamma$.

(a)  $v \in \Gamma(u) \setminus \Gamma(v)$ and $u \in \Gamma^{-1}(v) \setminus \Gamma^{-1}(u)$.

(b),(c)  $\langle u, v, w \rangle$ is a $3$-cycle if an only if $w \in \Gamma(v) \cap \Gamma^{-1}(u)$.
Because $|\Gamma(v)| = |\Gamma^{-1}(u)| = n$ and neither of these sets meets $\{ u, v \}$ it follows that
$1 \leq |\Gamma(v) \cap \Gamma^{-1}(u)| \leq n$.

(d) (iii)$\Leftrightarrow$ (iv) $\Leftrightarrow$ (v): The edge $u \to v$ is contained in $n$ $3$-cycles if and only if
 $|\Gamma(v) \cap \Gamma^{-1}(u)| = n$ and so if and only if
$\Gamma(v) = \Gamma^{-1}(u)$.

 (i) $\Leftrightarrow$ (ii): This is the definition of reducibility.

(i) $\Rightarrow$ (iii),(iv),(v): If $\Gamma$ is reducible via $u \to v$  and $K$, then $ K = \Gamma(v) \cap \Gamma^{-1}(u)$.
 It follows that $|\Gamma(v) \cap \Gamma^{-1}(u)| = n$ and so
(iii) and (v) hold as well.

(v) $\Leftrightarrow$ (vi), (v) $\Leftrightarrow$ (vii), (v) $\Leftrightarrow$ (viii): These are obvious.

(v) $\Rightarrow$ (ii): Together condition (v) and its equivalent (vi)
imply that $|\Gamma(i) \setminus \{ u,v \}| = n-1$ for all $i \in I \setminus \{ u,v \}$ and so $\Gamma|(I \setminus \{ u,v \})$ is a subgame.

%(vi),(vii) $\Rightarrow$ (i): If $i \in J = I \setminus \{u, v \}$, then $i \to u, v$
%if and only if $\Gamma(i) \cap J$ contains only $n - 2$ elements.
%Similarly, $u,v \to i$ if and only if $\Gamma^{-1}(i) \cap J$ contains only $n - 2$ elements. In either of these cases,
%$\Gamma|J$ is not a subgame.

(vii) $\Leftrightarrow$ (ix) and (viii) $\Leftrightarrow$ (x): Obvious.

(x) $\Leftrightarrow$ (xi), (ix) $\Leftrightarrow$ (xii): Since $|\Gamma(u) \setminus \{ v \}| = n-1$ and $|\Gamma(v)| = n$, the union of
$\Gamma(u) \setminus \{ v \}$ and $\Gamma(v)$ is $I \setminus \{ u, v\}$ if and only if they are disjoint, proving (x) $\Leftrightarrow$ (xi).
The equivalence (ix) $\Leftrightarrow$ (xii) is proved similarly.

(e): In the extension
of $\Pi^{-1}$ via $v \to u$ and $K$ all of the arrows of $\Gamma$ have been reversed.

(f): If $\Gamma$ is reducible via $ u \to v_1 $ and via $u \to v_2$, then by (d)(v)
$\Gamma(v_1) = \Gamma^{-1}(u) = \Gamma(v_2)$ and so by (a) $v_1 = v_2$. Similarly, for
$v_1 \to u$ and $v_2 \to u$.

\end{proof} \vspace{.5cm}

\begin{cor}\label{cor07a} Any non-trivial game $\Gamma$ is a strong digraph. \end{cor}

\begin{proof} Let $u, v \in I$. If $u \not\to v$, then $v \to u$ and the edge is contained
in a $3$-cycle $\langle u, w, v \rangle$. So there is a path of length $1$ or $2$ from $u$ to $v$.

\end{proof} \vspace{.5cm}

\begin{cor}\label{cor07b} For $\Pi$  a game  on  $J$ and $K \subset J$
with $|J| = 2n - 1, |K| = n$, assume that $\Gamma$ is an extension of $\Pi$ via $u \to v$ and $K$. Let $i, j \in J$ with $i \to j$.
\begin{itemize}
\item[(a)] $\Gamma$ is reducible via $i \to u$ if and only if $i \in K$ and $J \setminus K = \Pi^{-1}(i)$, so that
$K = \{ i \} \cup \Pi(i)$.
\item[(b)] $\Gamma$ is reducible via $v \to i$ if and only if $i \in K$ and $J \setminus K = \Pi(i)$, so that
$K = \{ i \} \cup \Pi^{-1}(i)$.
\item[(c)] $\Gamma$ is not reducible via $u \to i$ or via $i \to v$.
\item[(d)] $\Gamma$ is reducible via $i \to j$ if and only if $\Pi$ is reducible via $i \to j$ and
$K \cap \{ i, j \}$ is a singleton set.
\end{itemize}
\end{cor}

\begin{proof} The conditions of (a) say that no $j \to i$ and $j \to u$ and so Proposition \ref{prop07} (d)(vii)
implies (a). Similarly, (viii) implies  (b). Either (vii) or (viii) implies (d).
Finally, (c) follows from Proposition \ref{prop07} (f).

\end{proof} \vspace{.5cm}

For a tournament $\Pi$ on $I$,  Fisher et al in \cite{FLMR} call a pair $u,v \in I$  a \emph{domination pair}\index{domination pair}
if for every $i \in I \setminus \{ u,v \}$
either $u \to i$ or $v \to i$.  That is, $\Pi(u) \cup \Pi(v) \cup \{ u,v \} = I$. The
authors define the \emph{domination graph}\index{domination graph}
$dom(\Pi)$\index{$dom(\Pi)$} to be the subgraph consisting of those $(u,v) \in \Pi$ such that $\{ u,v \} $ is a domination pair and they provide a
partial characterization of the possible domination graphs for a general tournament.

For a tournament $\Pi$ on $I$, a pair $u,v \in I$ is called a \emph{mixed pair}\index{mixed pair}
if for every $i \in I \setminus \{ u,v \}$
either $u \to i \to v$ or $v \to i \to u$. Clearly, a mixed pair is a domination pair. When $\Pi$ is a game, i.e. a regular tournament, Proposition
\ref{prop07}(d) implies that a domination pair is a mixed pair and that such pairs $\{ u,v \}$ are exactly the pairs with respect to which
$\Gamma$ is reducible. Using the description of the mixed pair graphs given in \cite{BCL}, Cho et al \cite{CKL}
 give the characterization of the domination graph for a regular
tournament which we sketch below.

Thus, for a game $\Pi$ the domination graph is given by
\begin{equation}\label{eqred01}
dom(\Pi) \ = \ \{ \ (u,v) \in \Pi : \  \Pi \ \text{is reducible via} \ u \to v \ \}.
\end{equation}
Of course, if $\Pi$ is not reducible then $dom(\Pi)$ is empty.

\begin{prop}\label{prop07c} Let $\Pi$ be a game on $I$ with $|I| = 2n + 1$.
\begin{enumerate}

\item[(a)] If $i \in I$ then $(i,j) \in dom(\Pi)$ for at most one $j \in I$ and $(j,i) \in dom(\Pi)$ for at most one $j \in I$.

\item[(b)] Let $[i_0,\dots,i_m]$ be a path in $\Pi$ with $i_0 \not= i_m$ and let $J = I \setminus \{i_0, \dots, i_m \}$. The path $[i_0,\dots,i_m]$
is contained in $dom(\Pi)$ if and only if the following conditions hold:
\begin{itemize}
\item[(i)] If $p,q \in [0,m]$, then  $i_p \to i_q$ in $\Pi$ if and only if $q - p$ is odd and positive or is even and negative.
\item[(ii)] If $j \in J$, and $j \to i_p$ then $j \to i_q$ if and only if $q - p$ is even.
\end{itemize}
 When these conditions hold, the path is simple. In addition, $(i_m,i_0) \in dom(\Pi)$ -so that $\langle i_0,\dots,i_m \rangle$ is a cycle in
 $dom(\Pi)$- if and only if $J = \emptyset$, or, equivalently, $m = 2n$. In that case, $dom(\Pi) = \langle i_0,\dots,i_m \rangle$
  and the cycle is a Hamiltonian cycle in $\Pi$.

  \item[(c)] If $dom(\Pi)$ is not empty and does not consist of a single Hamiltonian cycle on $\Pi$, then it is a disjoint union of
  separated, simple, non-closed
  paths.
  \end{enumerate}
  \end{prop}

\begin{proof} (a) is a restatement of Proposition \ref{prop07} (f).

(b) Assume that $[i_0,\dots,i_m]$ is a path in $dom(\Pi)$. For $q > p$ we prove by induction on $q - p$
that $i_p \to i_q$ if and only if $q - p$ is odd. If $i_p \to i_{q-1}$ then since
$\Pi$ is reducible via $i_{q-1} \to i_q$ if follows from Proposition \ref{prop07}(d)(vii) and (viii) that $i_q \to i_p$.  Similarly,  $i_{q-1} \to i_p$
implies $i_p \to i_q$. Condition (i) follows and condition (ii) similarly follows from Proposition \ref{prop07}(d)(vii) and (viii).

On the other hand,
these conditions imply that $i_p \to i_{p+1}$ for $p = 0,\dots,m-1$ and that $\Pi$ is reducible via each $i_p \to i_{p+1}$ by
Proposition \ref{prop07}(d) again. That is they imply that $[i_0,\dots,i_m]$ is a path in $dom(\Pi)$.

Now assume that $[i_0,\dots,i_m]$ is a path in $dom(\Pi)$. The path is simple by (a).

If $m$ is odd, then $i_0 \to i_m$ and so $(i_m,i_0) \not\in \Pi$. If $m$ is even and $j \in J$, then either $i_0, i_m \to j$ or $j \to i_0, i_m$ and
so $\Pi$ is not reducible via $i_m \to i_0$, On the other hand if   $J$ is empty, then $m$ is even and condition (i) implies $\Pi$ is reducible
via $i_m \to i_0$. So if $J$ is empty, $\langle i_0,\dots,i_m \rangle \subset dom(\Pi)$
 is a Hamiltonian cycle in $\Pi$ and then
 (a) implies that every edge of $dom(\Pi)$ lies on the cycle.

(c)  If $dom(\Pi)$ is not empty and does not consist of a single Hamiltonian cycle, then by (b) it contains no cycle.
By (a) any edge in $dom(\Pi)$ can be extended uniquely
to a maximal path in $dom(\Pi)$ and none of the remaining edges of $dom(\Pi)$ has a vertex on the path. Proceed by
exhaustion to obtain $dom(\Pi)$ as a
union of these separate maximal paths.

\end{proof} \vspace{.5cm}

The following result is Theorem 5.2 of \cite{CH}.

\begin{theo}\label{theo08} Let $\Gamma$ be a game of size $2n + 1$ on the set of vertices $I$. Each vertex $i \in I$ is
contained in exactly $n(n+1)/2$ $3$-cycles. The entire game contains $(2n+1)n(n+1)/6$ $3$-cycles. \end{theo}

\begin{proof} There are $n$ vertices in the output set $\Gamma(i)$. Each of these has $n$ output edges and these are all distinct for
a total of $n^2$ outputs from these vertices.  Between these vertices there are $n(n-1)/2$  edges each of which is one of the $n^2$
output edges from a vertex of $\Gamma(i)$. The remaining $n(n+1)/2 = n^2 -[n(n-1)/2]$ edges terminate at a vertex of $\Gamma^{-1}(i)$ and these
are the $3$-cycles which contain $i$.  Multiplying by the number $2n +1$ of vertices $i$ we obtain the total number of $3$-cycles in $\Gamma$ after
we divide by 3 to correct for the triple counting.

\end{proof} \vspace{.5cm}

In general,
 for a tournament of size $p$ with score vector $s = (s_1,\dots, s_p)$ the total number $N_s$ of $3$-cycles is
given by the formula
\begin{equation}\label{eqnewx}
N_s \ = \  p(p-1)(2p-1)/12 \ \ - \  \ \frac{1}{2}\sum_{i=1}^p \ s_i^2.
 \end{equation}
See \cite{HNC} Corollary 11.10b or \cite{HM} Corollary 6b.

It follows from Theorem \ref{theo08} that in a game with $n > 1$, at least one edge is contained in more than one cycle.
For if every edge were contained in exactly one cycle then the number of cycles would be the number of edges divided by 3,
i.e. $n(2n + 1)/3$, but if $n > 1$, then $(n+1)/2 > 1$.

In \cite{A1} Alpach shows that every edge in a game of size $2n + 1$ is contained in some $k$-cycle for all $k = 3,\dots, 2n+1$, complementing Moon's
vertex result Proposition \ref{prop06}.

\begin{theo}\label{theo08a} Let $\Pi$ be an Eulerian digraph on a set $I$ with $|I|$ odd.  There exists a game $\Gamma$ on $I$ which
contains $\Pi$ as a subgraph. \end{theo}

\begin{proof}  Let $2n + 1 = |I|$ and let $k_i = |\Pi(i)| = |\Pi^{-1}(i)|$ for $i \in I$. We can clearly extend $\Pi$ to some tournament on $I$.

For a tournament $\Gamma$ on $I$ let $s_i = |\Gamma(i)|$, the score of $i$. Let
\begin{equation}\label{eqt08a1}
I_+ = \{ i \in I : s_i > n \}, \ I_0 = \{ i \in I : s_i = n \}, \ I_- = \{ i \in I : s_i < n \}.
\end{equation}
Call $\sum \{ s_i - n : i \in I_+ \cup I_0 \}$ the \emph{deviation} of $\Gamma$. Clearly the deviation is non-negative
and since the total sum of the scores is
$n(2n+1)$ it follows that $\Gamma$ is a game, i.e. $s_i = n$ for all $i$, if and only if the deviation is zero.

If $\Pi \subset \Gamma$ then for all $i \in I$,
then
\begin{equation}\label{eqt08a2}
|(\Gamma \setminus \Pi)(i)| = s_i - k_i, \qquad |(\Gamma \setminus \Pi)^{-1}(i)| = 2n - s_i - k_i.
\end{equation}

Thus, $i \in I$ lies in $I_+$ exactly when the number of its $\Gamma \setminus \Pi$ outputs is greater than the number of its
$\Gamma \setminus \Pi$ inputs.

Now assume  $\Pi \subset \Gamma$ and the deviation of $\Gamma$ is positive. We will show that
there exists a tournament $\Gamma'$ containing $\Pi$ and with smaller
deviation. Hence, the tournaments which contain $\Pi$ with minimum deviation are the games which contain $\Pi$.

I claim there exists a $\Gamma \setminus \Pi$ path from a vertex in $I_+$ to a vertex in $I_-$. If not, then
$\hat I = I_+ \cup \O (\Gamma \setminus \Pi)(I_+)$
is disjoint from $I_-$ and so is contained in $I_+ \cup I_0$. The restriction $(\Gamma \setminus \Pi)|\hat I$
contains every edge of $\Gamma \setminus \Pi$ from a vertex in $\hat I \supset I_+$. Since $I_+$
is assumed to be nonempty it follows that the total number of outputs for $(\Gamma \setminus \Pi)|\hat I$ is strictly greater than
than the total number of inputs.
But these sums are both equal to the number of edges in $(\Gamma \setminus \Pi)|\hat I$. The contradiction establishes the existence of the required path.

Truncate the path to begin at the last occurrence of a vertex in $I_+$ and then terminate at the first occurrence of a vertex in $I_-$. Then eliminate any
intermediate repeated vertices, by removing the piece between the repeats. Thus we obtain a simple $\Gamma \setminus \Pi$ path
$[i_1,i_2,\dots,i_k]$ with $i_1 \in I_+, i_k \in I_-$ and $i_p \in I_0$ for $1 < p < k$.

From $\Gamma$ define $\Gamma'$ by reversing the edges of $[i_1,i_2,\dots,i_k] \subset \Gamma \setminus \Pi$.
This decreases the score of $i_1$ by $1$, increases the score of
$i_k$ by $1$ and leaves every other score unchanged. Thus, the deviation of $\Gamma'$ is one less than that of $\Gamma$.
Since the path lies in $\Gamma \setminus \Pi$ it follows that $\Pi \subset \Gamma'$.

\end{proof}

\vspace{1cm}

\section{Group Games}\label{secgroupgames}

A digraph $\Pi$ on $I$ is called \emph{point-symmetric} \index{point-symmetric} \index{symmetric!point} when its automorphism group acts transitively on
the set, $I$, of vertices.  That is, for $i, j \in I$ there exists $\r \in Aut(\Pi)$ such that $\r(i) = j$. From Proposition \ref{prop06euler}(a),
it is clear that a point-symmetric digraph is regular and so a point-symmetric tournament is a game. These are constructed using digraphs on groups.

Let $\Z_{2n + 1}$ denote the additive group of integers mod $2n + 1$. We label the congruence
classes as $0, 1, \dots, 2n$. So $-k = 2n + 1 - k$ for $k \in [1,2n] = \{ 1, \dots, 2n \}$. We let $\Z_{2n + 1}^*$ denote
the multiplicative group of units in the ring of integers mod $2n + 1$, i.e. the congruence classes of the integers relatively prime to $2n + 1$.

In general, let $G$ be a group of order $2n + 1$. Notice that if $G$ is abelian and $2n + 1$ is square-free, then $G$ is cyclic and
so is isomorphic to $\Z_{2n + 1}$. The smallest non-abelian group of odd order is the semi-direct product $\Z_3 \ltimes \Z_7$ with
$\Z_3$ regarded as a subgroup of $\Z_7^*$ acting on $\Z_7$ by multiplication. With order less than 20 the groups of odd order are cyclic
except for $\Z_3 \times \Z_3$ of order 9.

For $G$  a group of order $2n + 1$  there is no element of order 2.
With $e$ the identity element of $G$,  the set $G \setminus \{ e \}$ is
 partitioned by the set of $n$ pairs $\{ i, i^{-1} \}$. We will call $A$ a \emph{graph subset}\index{graph subset} of $G$ when
it is a nonempty subset of $G$ with $A$ disjoint from $A^{-1} = \{ a^{-1} : a \in A \}$. A \emph{game subset}\index{game subset}
 $A$ is a graph subset of cardinality $n$  so that
$G$ is the disjoint union  $A \cup A^{-1} \cup \{ e \}$. Thus,  a game subset $A$ is obtained by choosing one element from each of the
pairs $\{ i, i^{-1} \}$.  It follows that there $2^n$ game subsets.

If $A$ is a graph subset, then the associated digraph on $G$ is $\Gamma[A] = \{ (i,j) : i^{-1}j  \in A \}$\index{$\Gamma[A]$}. For $k \in G$ define the left
translation map $\ell_k$ on $G$ by $\ell_k(i) = ki$. Each $\ell_k$\index{$\ell_k$} is an automorphism of $\Gamma[A]$. That is, $\Gamma[A]$ left-invariant,
i.e. $\bar{\ell_k}(\Gamma[A]) = \Gamma[A]$.  Thus,
$\Gamma[A](i) = \ell_i(A) = iA$ for $i \in G$, while
$\Gamma[A]^{-1}(i) = \ell_i(A^{-1}) = iA^{-1}$. It follows that $\Gamma[A]$ is point-symmetric and so is regular. If $A$ is a game subset, then
$\Gamma[A]$ is a tournament and so is a game. We call a game of this sort a \emph{group game}\index{group game}\index{game!group}.

It is easy to see that $\Gamma[A]$ is connected when $A$ generates the group and that,
otherwise, the left cosets of the subgroup generated by $A$ decompose the graph into separate pieces.  When $A$ generates $G$
the digraph $\Gamma[A]$ is the Cayley graph\index{Cayley graph}  of the group $G$ with respect to the set $A$ of generators. For this reason,
what we are calling a group game is called a \emph{Cayley tournament} \index{Cayley tournament} in \cite{BI} and \cite{G}. In the special case
$G = \Z_{2n+1}$ with group subset $A$, the game $\Gamma[A]$ is called a \emph{rotational tournament}\index{rotational tournament} with symbol $A$, see
\cite{A2} and \cite{A3}.

Notice that if $A$ is a graph subset (or a game subset), then $A^{-1}$ is a graph subset (resp. a game subset)
and $\Gamma[A^{-1}]$ is the reversed digraph $\Gamma[A]^{-1}$.

If $H$ is a subgroup of $G$ and $A$ is a graph subset (or a game subset) for $G$, then $A \cap H$ is a graph subset (resp. a game subset)
for $H$ and $\Gamma[A \cap H]$ is the restriction $\Gamma[A]|H$.

We could define a right invariant graph associated with $A$ by $\{ (i,j) : ji^{-1}  \in A \}$.  We do not bother, because it is clear that the
map $i \to i^{-1}$ is an isomorphism from $\Gamma[A]$ onto the right invariant game associated with $A^{-1}$. In the abelian case the right invariant
game for $A$ is the same as $\Gamma[A]$ and so when $G$ is abelian, the map $i \to i^{-1}$ is an isomorphism from $\Gamma[A]$ to the reversed
game $\Gamma[A^{-1}]$.

\begin{lem}\label{lem09} Let $A$ and $B$ be graph subsets of $G$ and let  $\r \in S(G)$, i.e. $\r$ is  a permutation of $G$.
The following are equivalent
\begin{itemize}
\item[(i)] $\r$ is an isomorphism from $\Gamma[A]$ to $\Gamma[B]$.

\item[(ii)]For all  $ i, j \in G \qquad i^{-1}j \in A \quad \Longleftrightarrow \quad \r(i)^{-1}\r(j) \in B.$

\item[(iii)] For all $i \in G, \r(iA) \ = \ \r(i)B.$
\end{itemize}

In particular, if $\r$ is an isomorphism from $\Gamma[A]$ to $\Gamma[B]$, then  $\r(e) = e$ implies $\r(A) = B$.
\end{lem}

\begin{proof} It is obvious that (i) $\Leftrightarrow$ (ii) and (ii) $\Leftrightarrow$ (iii). From (iii),
$\rho(e) = e$ implies $\r(A) = B$.

\end{proof}

{\bfseries Remark:} If $A$ is a game subset with $\r$  a permutation such that $\r(e) = e$ and $B = \r(A)$, then
$$\r(A^{-1}) = \r(G \setminus (A \cup \{e \})) = G \setminus (B \cup \{e \}).$$
Hence, $i^{-1} \in A \ \Leftrightarrow \ \r(i)^{-1} \in B$, and so $B^{-1} = \r(A^{-1})$, if and only if $B$ is a game subset.
\vspace{.5cm}

For a map $\Phi : G \times I \tto I$ we write
\begin{equation}\label{eqaction1}
\begin{split}
g i \ = \  \Phi(g,i) \ = \  \Phi^g(i) \ = \ \Phi_i(g), \hspace{2cm}\\
\Phi^{\#}  \quad \text{is defined by} \ g \mapsto \Phi^g. \hspace{2.5cm} \end{split}
\end{equation}
$\Phi$ is an action of the group $G$ on the set $I$ when $\Phi^{\#}$ is a homomorphism from $G$ to the permutation group $S(I)$.
\index{$\Phi^g$}\index{$\Phi_i$}\index{$\Phi^{\#}$} The \emph{orbit}\index{orbit} of $i$ is $Gi = \Phi_i(G)$\index{$Gi$}.
The orbit relation $\{ (i,j) : j \in Gi \}$ is an equivalence
relation and so the orbits partition the set $I$.

The action is called \emph{free}\index{free action}\index{action!free}
 when $gi = i$ for some $i \in I$ only when $g = e$, the identity element of $G$.
 The action is called \emph{effective}\index{effective action}\index{action!effective}
 when $gi = i$ for all $i \in I$ only when $g = e$. That is, the homomorphism $\Phi^{\#} : G \tto S(I)$ is injective.
 Of course, a free action is effective.
 The action is called \emph{transitive}\index{transitive action}\index{action!transitive} if for some $i \in I$,  $I = Gi$.
 In that case, there is only one orbit and so $I = Gi$ for all $i \in I$.

 Here are some useful facts about such actions.

\begin{prop}\label{prop10} Let  $\Phi : G \times I \tto I$ be an action of a finite group $G$ on a finite set $I$.
\begin{enumerate}
\item[(a)] The group $G$ acts freely on $I$ if and only if for all $i \in I$ the map $\Phi_i : G \tto I$ is injective.
In particular, if  $G$ acts freely on $I$ then $|G|$ divides $|I|$.

\item[(b)] Any two of the following three conditions implies the third.
\begin{itemize}
\item The group $G$ acts freely on $I$.
\item The group $G$ acts  transitively on $I$.
\item The cardinalities $|G|$ and $|I|$ are equal.
\end{itemize}

\item[(c)] If $G$ is abelian and acts effectively and transitively on $I$, then it acts freely on $I$ and $|G| = |I|$.

\item[(d)] Assume that $|I|= p$ is prime. If $G$ acts transitively on $I$, then $G$ contains a cyclic subgroup $H$  which
acts transitively on $I$ and which has order a power of $p$. If, in addition, the action of $G$ is effective, then $|H| = |I|$.

\item[(e)] Assume that $|G| = |I|= k$ and that the group $G$ is cyclic with generator $g$. The group acts freely on $I$
if and only if, the permutation $\Phi^g$ on $I$  consists of a single permutation k-cycle. If $k$ is prime
and the action is non-trivial then it is free.

\item[(f)] Assume that $|G| = 2n+1$ and that $G$ acts freely on $I$. If $A \subset I$
with $|A| = n$, then $g(A) = A$ if and only if $g = e$.

\item[(g)] If $|G|$ and $|I|$ are odd then the number of $G$ orbits in $I$ is odd.
\end{enumerate}
 \end{prop}

\begin{proof} (a): If $gi = i$ for some $g$ not the identity then the map $\Phi_i$ is not injective.
 If $g_1i = g_2i$ with $g_1 \not= g_2$ then $g_1^{-1}g_2i = i$. In any case, $I$ is the disjoint union of the orbits.
 When the action is free, each of these has cardinality $|G|$. So $|I|$ is a multiple of $|G|$.

 (b): Assume first that $|G| = |I|$. Since the cardinalities are finite and equal, the map $\Phi_i : G \to I$
  is injective if and only if it is surjective.
 On the other hand, if the map is bijective, then $|G| = |I|$.

 (c): Assume that $gi = i$. If $j \in I$, then because the action is transitive, there exists $h \in G$ such that $hi = j$.
 Because the group is abelian,
 $gj = ghi = hgi = hi = j$. That is, $gj = j$ for all $j$ and so $g = e$  because the action is effective. Then (b) implies that
 $|G|$ = $|I|$.

 (d): Assume that $p = |I|$, which is assumed to be prime. Fix $i \in I$ and let $Iso_i = \Phi_i^{-1}(\{ i \}) = \{ g : gi = i \}$.  This is a subgroup of
 $G$ called the \emph{isotropy subgroup}\index{isotropy subgroup} of $i$. Since $gi = hi$ if and only if $g^{-1}h \in Iso_i$, it follows that
 $\Phi_i$ factors to define an injection from $G/Iso_i$, the set of left cosets $\{ gIso_i : g \in G \}$, into $I$. Since the action is
 transitive, the induced map is a bijection and so the subgroup $Iso_i$ has index $p$.   We need the following bit of group theory.\vspace{.5cm}

 \begin{lem}\label{lemgroup} If $G$ is a finite group and $J$ is a subgroup with prime index $p$, then there exists a cyclic subgroup $H$ of $G$
 with order a power of $p$
 such that the restriction to $H$ of the quotient map $G \to G/J$ is surjective.\end{lem}

\begin{proof} Assume that  $|J| = p^e a$ with $e \geq 0$ and $a$ relatively prime to $p$. Since $|G/J| = p$,  $|G| = p^{e+1} a$.
By the First Sylow Theorem \cite{H} Theorem 4.2.1, there exists a subgroup $P$ of $T$ with $|P| = p^{e+1}$.  It follows that $P$ is not
contained in $J$. Let $t \in P \setminus J$. Since $P$ is a $p$-group, $t$ has order a power of $p$. Let $H$ be the cyclic group generated by
$t$. Hence, $|H/(H \cap J)|$ is a positive power of $p$ and the restriction of the quotient map factors to an injection from $H/(H \cap J)$ into
$G/J$. Since $|G/J| = p$, this map is a bijection and so the quotient map takes $H$ onto $G/J$.
\end{proof} \vspace{.5cm}

Now apply the Lemma with $J = Iso_i$. If $j \in I$, there exists $g \in G$ such that $gi = j$. There exist $h \in H$ and $s \in Iso_i$ such that
$g = hs$. Hence, $hi = hsi = gi = j$.  That is, $H$ acts transitively.

If, in addition, the action of $G$ is effective, then the abelian group $H$ acts transitively and effectively and so by (c) the action is free with
$|H| = |I|$.

 (e): If the action is free and $i \in I$, then $(i \ gi \dots \ g^{k-1}i)$ is a permutation k-cycle.
 Conversely, if this is a k-cycle, then
 the action is transitive and so is free by (b). In any case, the generator $g$ acts
 as a permutation whose order divides the
 order of $g$ which is $k$. The order of the permutation is the least common multiple
 of the orders of the permutation cycles contained therein.
 If $k$ is a prime, then either all cycles have order 1 and the action is trivial or
 there is a k-cycle and the action is free.

 (f):  Let $H$ be the subgroup generated  by $g$. The order $|H|$ divides $2n + 1$.
 On the other hand if  $A$ is invariant and the
 action is free then $H$ acts freely on $A$. So by (a), $|H|$ divides $|A| = n$.
 Since $n$ and $2n + 1$ are relatively prime,  $|H| = 1$ and so $g = e$.

(g): For $i \in I$, the map $\Phi_i$ factors to a bijection of the quotient $G/Iso_i$ onto the orbit $Gi$.
Hence, $|Gi|$ is odd for every $i$. Thus, the orbits partition the odd cardinality set $I$ into
sets of odd cardinality and so there must be an odd number of them.

\end{proof} \vspace{.5cm}

In \cite{S1a} Sabidussi observes the following.

\begin{theo}\label{theo11} (a) Let $\Pi$ be a digraph on the set of elements of a group $G$. The group $G$ acting freely on itself by left
translation\index{left translation action}\index{action!left translation}  is a subgroup of $Aut(\Pi)$ if and only
if there exists a -necessarily unique- graph subset $A$ of $G$ such that $\Pi$ equals $\Gamma[A]$
the associated digraph.

(b) The left translation action of $G$ induces a free action on the collection of
translates of game subsets, i.e. for a game subset $A$, $iA = jA$ implies $i = j$.

\end{theo}

\begin{proof} (a): It is clear from the definition of $\Gamma[A]$ that each left translation is an automorphism. On the other hand,
if the translations are automorphisms, then let $A = \Pi(e)$. It is clear that $i \to j$
if and only if $e = \ell_{i^{-1}}(i) \to \ell_{i^{-1}}(j) = i^{-1}j$.
That is, $i \to j$ if and only if $i^{-1}j \in A$.  Since $\Pi$ is a digraph, $e \to i$
implies that $i \nrightarrow e$ and so $e \nrightarrow i^{-1}$. Thus,
$A$ is a graph subset and $\Pi = \Gamma[A]$.

 (b): This follows from Proposition \ref{prop10} (f) or from Proposition \ref{prop07} (a)  applied to $\Gamma[A]$.

\end{proof} \vspace{.5cm}

From (a) we obtain the following result of Turner \cite{T}.

\begin{cor}\label{cor11aa} Let $\Pi$ be a digraph on $I$ with $|I| = 2n + 1$ prime. If $\Pi$ is point-symmetric, i.e. $Aut(\Pi)$ acts transitively on $I$, then
there exists a graph subset $A$ of $\Z_{2n + 1}$ such that $\Pi$ is isomorphic to $\Gamma[A]$ on $\Z_{2n + 1}$. \end{cor}

\begin{proof} By definition $Aut(\Pi)$ is a subgroup of $S(I)$ and so acts effectively on $I$. By Proposition \ref{prop10} (d) and (c) there exists
a cyclic subgroup $H$ of $Aut(\Pi)$ which acts freely and transitively on $I$ and with $|H| = |I|$. We can identify $H$ with $\Z_{2n +1}$ since it
is cyclic and also with $I$ via $\Phi_i(h) = hi$ for any fixed $i \in I$. Thus, we can regard $\Pi$ as a digraph on $\Z_{2n +1}$ and the translation action
is identified with a subgroup of $Aut(\Pi)$. With these identifications, $\Pi$ becomes a digraph of the form $\Gamma[A]$ by Theorem \ref{theo11} (a).

\end{proof} \vspace{.5cm}

For any group $G$ we will regard $G$ as a subgroup of $S(G)$ by identifying $i \in G$ with
the left translation $\ell_i$\index{$\ell_i$}.
For any graph subset $A$ of a group $G$ we will use this identification to regard $G$ as a subgroup of $Aut(\Gamma[A])$. In particular, a group digraph
is point-symmetric.

For a group $G$ let $G^*$\index{$G^*$} denote the automorphism group of $G$.  That is,
$\xi \in G^*$ if and only if $\xi : G \to G$ is a group isomorphism.
For the ring of integers mod $2n + 1$ the group of units $\Z_{2n+1}^*$ consists of those $i \not= 0$ which
are relatively prime to $2n + 1$ of which there are $\phi(2n+1)$ elements
(defining the Euler $\phi$-function). For $a\in \Z_{2n+1}$ we let $m_a$ denote
multiplication by $a$ so that $m_a(i) = ai$.  If $\r : \Z_{2n+1} \to \Z_{2n+1}$ is an additive group homomorphism and
$a = \r(1)$ then since $i = 1 + 1 \dots  + 1$ ($i$ times) it follows that $\r(i) = ai = m_a(i)$. In particular, identifying
$a \in \Z_{2n+1}^*$ with $m_a$ in the automorphism group of $\Z_{2n+1}$ we regard $\Z_{2n+1}^*$ as the automorphism group of
$\Z_{2n+1}$.

The \emph{adjoint action}\index{adjoint action}\index{action!adjoint} of $G$ on itself is given by the homomorphism $G \to G^*$ associating
to $g \in G$ the inner automorphism $Ad_g \in G^*$\index{$Ad_g$} defined by $Ad_g(h) = g h g^{-1}$. We write $Ad_G(h) = \{ Ad_g(h) : g \in G \}$
\index{$Ad_G(h)$} for the orbit of $h \in G$ under the adjoint action.
Of course the action is trivial when the group is abelian.

With $\r_g(h) = hg$, the right translation action \index{right translation action}\index{action!right translation} of $G$ on itself is
given by $g \mapsto \r_{g^{-1}}$.

\begin{lem}\label{lem11norma} Let $G$ be a group in which every non-identity element has finite, odd order and let $g,h \in G$.
\begin{enumerate}
\item[(a)]  If $g^2 = h^2$, then $g = h$.
\item[(b)]  If $g h g^{-1} = h^{-1}$, then $h = e$.
\end{enumerate}
\end{lem}

\begin{proof} (a) The cyclic group generated by $g$ has odd order and so is generated by $g^2$ as well. Hence, $h$ and $g$ generate cyclic group generated by
$h^2 = g^2$. In particular, $h$ and $g$ commute. Hence, $e = g^2 h^{-2} = (g h^{-1})^2$. It follows that $g h^{-1} = e$.

(b) If $g h g^{-1} = h^{-1}$, then $hgh = g$ and so $(hg)^2 = g^2$.  By (a) $hg = g$ and so $h = e$.

\end{proof} \vspace{.5cm}

A game subset $A$ for $G$ is a \emph{normal game subset}\index{normal game subset}\index{game subset!normal} if $A = g A g^{-1}$ for all $g \in G$, i.e.
$A$ is invariant under the adjoint action. Of course, if $G$ is abelian, every game subset is normal.
%Equivalently, if $h \in A$,  then the orbit $Ad_G(h) = \{ Ad_g(h) : g \in G \}$ is contained in $A$.

\begin{prop}\label{prop11normb} (a) If  $A$ is a game subset for a group $G$ with associated game $\Gamma[A]$, then the following are equivalent.
\begin{itemize}
\item[(i)] The game subset $A$ is normal.
\item[(ii)] If $h \in A$,  then the orbit $Ad_G(h)$ is contained in $A$.
\item[(iii)] If $i \to j$, then $g i g^{-1} \to g j g^{-1} $ for all $g \in G$.
\item[(iv)] If $i \to j$, then $i g \to j g $ for all $g \in G$.
\end{itemize}

If $A$ is normal, then $i \to j$ implies $j^{-1} \to i^{-1}$, i.e. the map $i \to i^{-1}$ is an isomorphism from $\Gamma[A]$ to the reverse game.

(b) The orbits under the adjoint action partition $G \setminus \{ e \}$ by an even number of subsets. If this number is $2m$ then there
are $2^m$ normal game subsets.
\end{prop}

 \begin{proof} (a) (i) $\Leftrightarrow$ (ii) is obvious.

 From (iii) with $i = e$ we obtain (i). On the other hand, given (i) $i^{-1} j \in A$
 implies $Ad_g(i^{-1}j) = Ad_g(i)^{-1}Ad_g(j) \in A$. Thus, (i) $\Leftrightarrow$ (iii).

 (i) $\Leftrightarrow$ (iv) because $\Gamma[A]$ is always invariant under left translation.

 If $i^{-1}j \in A$ and $A$ is normal we apply $Ad_i$ to obtain $(j^{-1})^{-1} i^{-1} = j i^{-1} \in A$ and so $(i^{-1})^{-1}j^{-1} \in A^{-1}$.

 (b) Lemma \ref{lem11norma}(b) implies that the orbits of $i$ and $i^{-1}$ are disjoint unless $i = e$. We obtain a normal game
 subset by choosing one orbit from each pair $Ad_G(i), Ad_G(i^{-1})$ with $i \not= e$.

\end{proof} \vspace{.5cm}
%
%By analogy with the $\Z_{2n+1}$ case, we define the \emph{affine group} on $G$ to be the subgroup of $S(G)$
%which is generated by the translations on $G$ together
%with the automorphisms in $G^*$.

\begin{theo}\label{theo12} Let $A$ be a graph subset of $G$ with $\Gamma[A]$ the associated digraph.

(a) If $\xi \in G^*$, then $\xi(A)$ is a graph subset and $\xi$ is an isomorphism
from $\Gamma[A]$ to $\Gamma[\xi(A)]$.

(b) The group $Aut(\Gamma[A])$ is commutative if and only if $G$ is commutative and $Aut(\Gamma[A]) = G$, i.e. the left translations are the only
automorphisms of $\Gamma[A]$.

(c) Assume that $Aut(\Gamma[A]) = G$. If $B$ is a graph subset
with $\Gamma[B]$ isomorphic to $\Gamma[A]$, then there is a unique $\xi \in G^*$
such that $B = \xi(A)$ and $\rho = \xi$ is the unique
isomorphism $\rho : \Gamma[A] \to \Gamma[B]$ such that $\rho(e) = e$. In particular,
the  set $\{ \xi(A) : \xi \in G^* \}$ is the set of
graph subsets $B$ such that $\Gamma[B]$ is isomorphic to $\Gamma[A]$. Thus, there are exactly $|G^*|$
such graph subsets.

(d) If $A \subset \Z_{2n+1}$ is a graph subset such that $Aut(\Gamma[A]) = \Z_{2n+1}$ then the
 set $\{ m_a(A) : a \in \Z_{2n+1}^* \}$ is the set of
graph subsets $B$ such that $\Gamma[B]$ is isomorphic to $\Gamma[A]$. Thus, there are exactly $\phi(2n+1)$
such game subsets.
\end{theo}

\begin{proof} (a):  Since $\xi$ is a bijection with
$\xi(e) = e$ and $\xi(i^{-1}) = \xi(i)^{-1}$, it follows that $\xi(A) \cap\xi(A)^{-1} = \xi(A \cap A^{-1}) = \emptyset$ and so
$\xi(A)$ is a graph subset.

We have $i^{-1}j \in A$ if and only if
$\xi(i)^{-1}\xi(j) = \xi(i^{-1} j) \in \xi(A)$.  Hence, $\xi$ is an isomorphism
from $\Gamma[A]$ to $\Gamma[\xi(A)]$.

(b): As is well-known, if $\rho$ is a permutation of $G$ which commutes with left translations, then $\rho$ is a right translation, because
$\rho(i) = \rho(\ell_i(e)) = \ell_i \rho(e) = i \rho(e)$. If $Aut$ is commutative, then the subgroup $G$ is commutative and every element is
a translation, i.e. $Aut = G$.

(c): By composing with a translation we can assume that $\r : \Gamma[A] \tto \Gamma[B]$
satisfies $\r(e) = e$. Consider the
translation $\ell_{\r(i)}$ on $\Gamma[B]$. $\r^{-1} \circ \ell_{\r(i)} \circ \r$
is an automorphism of $\Gamma[A]$ and so is a left translation.
Since $e \in A$ is mapped to $i$, it is $\ell_i$. That is, $\r \circ \ell_i = \ell_{\r(i)} \circ \r$. Applied to $j$, this says
$\r(ij) = \r(i)\r(j)$. That is, $\r \in G^*$.

If $\hat \r : \Gamma[A] \tto \Gamma[B]$ is an isomorphism fixing $e$ then
$\r^{-1} \circ \hat \r$ is an automorphism of $\Gamma[A]$
which fixes $e$. However, the identity is the only such automorphism.  Hence, $ \hat \r = \r$.

(d): This is just (c) with $G = \Z_{2n+1}$. Observe that $|G^*| = \phi(2n+1)$.

\end{proof} \vspace{.5cm}

In \cite{BI} Babai and Imrich call a tournament $\Pi$ on a group $G$ a \emph{tournament regular representation}\index{tournament regular representation}
or TRR \index{TRR} when the automorphism group $Aut(\Pi)$ is equal to the  group $G$ of left translations on $G$. From Theorem \ref{theo11} it
follows that $\Pi$ is isomorphic to a group game $\Gamma[A]$ with $A$ a game subset of $G$. In \cite{BI} and \cite{G} it is proved that every
group $G$ of odd order except $\Z_3 \times \Z_3$ and $\Z_3 \times \Z_3 \times \Z_3$ admits a TRR.

In $\Z_{2n+1}$ we let $[1,n] = \{ 1, \dots, n \}$ and $Odd_n = \{ 2k - 1: k = 1, \dots, n \}$.

\begin{theo}\label{theo13} The set $[1,n]$ is a game subset of $\Z_{2n+1}$
with $Aut(\Gamma[[1,n]])$ $ = \Z_{2n+1}$. That is, $\Gamma[[1,n]]$ is a TRR for $\Z_{2n+1}$. \end{theo}

\begin{proof} Since $2n + 1 - k = n + (n + 1 - k)$, $-[1,n] = [n+1,2n]$ and so $[1,n]$ is a game subset.

If $\rho \in Aut(\Gamma[[1,n]]$ with $\r(0) = i$ we may compose with the translation $\ell_{-i}$ to obtain an
automorphism which fixes $0$. It suffices to show that if $\r$ is an automorphism
which fixes $0$ then $\r $ is the identity.

Let $A = [1,n]$ and $\Gamma = \Gamma[A]$. By Lemma \ref{lem09}  $\r(A + i) = A + \r(i)$
and since $\r$ fixes $0$, $\r(A) = A$.

For $p = 1, \dots, n $, $A \cap (A + p) = \{ i : p < i \leq n \}$ which contains $n - p$ elements.

Thus, $p$ is the unique element $i$ of $A$ such that $|A \cap (A + i)| = n - p$.
 Since $\r(A \cap (A + i)) = A \cap (A + \r(i))$,
it follows that $\r $ fixes every element of $[1,n]$. $-(A \cap A + i) = -A \cap (-A + (-i))$.
Hence, $- p = 2n + 1 - p$ is the
unique element $i$ of $-A$ such that $|-A \cap (-A + i)| = n - p$.
Thus, $\rho$ fixes every element of $\Z_{2n+1}$.

\end{proof} \vspace{.5cm}

Notice that $n \in \Z_{2n+1}^*$ and $m_n(2k - 1) = n - (k - 1)$ mod $2n + 1$. Hence,
$m_n(Odd_n) = [1,n]$. Thus, $Odd_n$ is a game subset of $\Z_{2n+1}$ with $\Gamma[Odd_n]$ isomorphic to $\Gamma[[1,n]]$ via $m_n$.

Observe that for $A = [1,n]$ the tournament which is the restriction $\Gamma[A]|A$ is the standard order on $[1,n]$.

\begin{prop}\label{prop13a} If $\Pi$ is a game of size $2n + 1$, then the domination graph $dom(\Pi)$ is a cycle if and only if
$\Pi$ is isomorphic to the group game $\Gamma[[1,n]]$ on $\Z_{2n +1}$. \end{prop}

\begin{proof} Using Proposition  \ref{prop07c}(b) it is easy to check that $\Pi$ is a game with $dom(\Pi) = \langle i_0, \dots, i_{2n} \rangle$ if and
only if $p \mapsto i_p$ is an isomorphism from $\Gamma[Odd_n]$ to $\Pi$.

\end{proof} \vspace{.5cm}

Recall that there are $2^n$ game subsets of $\Z_{2n+1}$. With $n = 1,2$, we have
$2^n = \phi(2n+1)$, which must be true since in those cases there is, up to
isomorphism, only one game.

\begin{cor}\label{cor14} If $n > 2$ then there exists a game subset $A \subset \Z_{2n+1}$
such that $\Gamma[A]$ is not isomorphic to
$\Gamma[[1,n]]$. \end{cor}

\begin{proof} By Theorem \ref{theo13} and Theorem \ref{theo12}(d) there are
$\phi(2n + 1)$ game subsets $A$ such that $\Gamma[A]$ is  isomorphic to
$\Gamma[[1,n]]$. On the other hand, there are $2^n$ game subsets. For $n > 2, \ \phi(2n + 1) \leq 2n  < 2^n$.

\end{proof} \vspace{.5cm}

\begin{theo}\label{theo15} For a finite group $G$ of  order $2n + 1$, the following are equivalent.
\begin{itemize}
\item[(i)] $|G^*|$ is a power of $2$.
\item[(ii)] $|G^*|$ divides $2^n$.
\item[(iii)] The order of every element of $G^*$ is a power of $2$.
\item[(iv)] For every game subset $A \subset G$, and $\xi \in G^*$, $\xi(A) = A$ implies $\xi = 1_G$.
\item[(v)] The group $G^*$ acts freely on the set of game subsets of $G$.
\end{itemize}
\end{theo}

\begin{proof}
(ii) $\Rightarrow$ (i):  Obvious.

(i) $\Rightarrow$ (iii): The order of each element divides the order of the group.

(iii) $\Rightarrow$ (iv): If $\xi(A) = A$ then $\xi$ is an automorphism of $\Gamma[A]$.
Since the order of $\xi$ is a power of $2$, the
order must be $1$ by  Proposition \ref{prop04}, i.e. $\xi$ is the identity $1_G$.

(iv) $\Rightarrow$ (v):  This is the definition of a free action.

(v) $\Rightarrow$ (ii): By Proposition \ref{prop10} (a), $|G^*|$ divides the
cardinality of the set of game subsets which is $2^n$.

\end{proof} \vspace{.5cm}

A prime of the form $2^m + 1$ is called a \emph{Fermat prime}\index{Fermat prime}. For such a prime,
$m$ itself must be a power of two. To see this, observe that if $a > 1$ is an odd divisor of $m$ then
$x + 1$ divides $x^a + 1$ and the quotient has integer coefficients. Hence, with $x = 2^{m/a}$ we see that
$2^{m/a} + 1$ divides $2^m + 1$.   It follows that that
$F_k = 2^{2^k} + 1$ are the only possibilities.  Fermat discovered that for
$k = 0, 1, 2, 3, 4$ these are indeed primes: 3, 5, 17, 257, 65537. See, e. g. \cite{HW} Section 2.5.
However, these are the only Fermat primes which are known to exist.

\begin{theo}\label{theo16} For the odd number $2n + 1 > 1$ the following conditions are equivalent.
\begin{itemize}
\item[(i)] $\phi(2n + 1)$ is a power of $2$.
\item[(ii)] $\phi(2n + 1)$ divides $2^n$.
\item[(iii)] The order of every element of $\Z_{2n+1}^*$ is a power of $2$.
\item[(iv)] For every game subset $A \subset \Z_{2n+1}$, and $a \in \Z_{2n+1}^*$, $m_a(A) = A$ implies $a = 1$.
\item[(v)] The group $\Z_{2n+1}^*$ acts freely on the set of game subsets of $\Z_{2n+1}$.
\item[(vi)] $2n + 1$ is a square-free product of Fermat primes.
\end{itemize}
\end{theo}

\begin{proof} The equivalence of (i)-(v) is the special case of Theorem \ref{theo15} with $G = \Z_{2n+1}$.

(i) $\Leftrightarrow$ (vi): Write the prime factorization of $2n + 1$ as
$\Pi_{a=1}^{k} p_a^{e_a}$ with $e_a \geq 1$. Each $p_a$ is an
odd prime and $\phi(2n + 1) = \Pi_{a=1}^{k} p_a^{e_a-1}(p_a - 1)$. If $e_a > 1$
then $p_a|\phi(2n + 1)$. Hence, $\phi(2n + 1)$
can be a power of $2$ only when $e_a = 1$ for all $a$, i.e. $2n + 1$ is square-free. In that case,
$\phi(2n + 1) = \Pi_{a=1}^{k}(p_a - 1)$ and so $\phi(2n + 1)$ is a power of $2$
if and only if each $p_a - 1$ is a power of $2$, i.e. each $p_a$ is a Fermat prime.

\end{proof} \vspace{.5cm}

If $2n + 1$ is not a square-free product of Fermat primes, then $\Z_{2n+1}^*$  does not act freely on the
game subsets and so there exists a game subset
$A \subset \Z_{2n+1}$ and $a \in \Z_{2n+1}^*$ with $a \not= 1$ such that $m_a \in Aut(\Gamma[A])$. However, we can obtain a sharper result.

\begin{theo}\label{theo17} If $G$ is a finite group and $H$ is a non-trivial subgroup of $G^*$ with both $G$ and $H$ of odd order, then
there exists a game subset $A \subset G$ such that $H \subset Aut(\Gamma[A])$. \end{theo}

\begin{proof}  The group $H$ acts on $G$ and for $i \in G$ we let $H(i)$ denote $\{ \xi(i) : \xi \in H \}$, the $H$ orbit of $i$ in $G$. Distinct orbits
are disjoint. The identity element $e$ is a fixed point, i.e. $\{ e \} = H(e)$. If $\xi \in G^*$ and $i \not= e$ in $G$, then
$\xi(i) = i^{-1} $ implies that $\xi$ has even order. To see this note that $i \in G$ has odd order and so $i \not= i^{-1}$. Furthermore,
$\xi^2(i) = \xi(i^{-1}) = \xi(i)^{-1} = i$. Thus, $\xi^k(i) = i$ for $k$ even while $\xi^k(i) = i^{-1}$ for $k$ odd. Hence, $\xi$ has even order. In particular,
$\xi$ cannot be an element of $H$. Thus, for $i \not= e$ the orbits $H(i)$ and $H(i^{-1})$ are disjoint. Furthermore, $j \to j^{-1}$ is a bijection
from $H(i)$ to $H(i^{-1})$ and so $|H(i)| = |H(i^{-1})|$.

Let $i_1,\dots,i_m \in G \setminus \{ e \} $ so that $\{ H(i_1) \cup H(i_1^{-1}), \dots, H(i_m) \cup H(i_m^{-1}) \} $ is a partition of $ G \setminus \{ e \}$.
Define $A = H(i_1) \cup \dots \cup H(i_m)$ to obtain a game subset which is invariant with respect to the $H$ action. By using $H(i_k^{-1})$ instead
 of $H(i_k)$ for $k$ in an arbitrary subset of $\{1, \dots, m \}$ we obtain the $2^m$
game subsets which are $H$ invariant.

\end{proof} \vspace{.5cm}

{\bfseries Remark:} Conversely, if the subgroup $H$ of $G^*$ is contained in \\ $Aut(\Gamma[A])$ then since, $\xi(e) = e$ for $\xi \in H$, it follows that
$\xi(A) = A$. Hence, if $i \in A$, then the orbit $H(i)$ is a subset of $A$ and $H(i^{-1}) \subset A^{-1}$. It follows that the above construction
yields all game subsets $A$ such that $H \subset Aut(\Gamma[A])$. \vspace{.5cm}

In particular, if $p$ is an odd prime which divides $|G^*|$ then there exists $\xi \in G^*$ of order $p$ by the first Sylow Theorem. Theorem
\ref{theo17} implies that there exist game subsets $A \subset G$ such that $\xi \in Aut(\Gamma[A])$. With $H$ the order $p$ cyclic group
generated by $\xi$, each orbit $H(i)$ has cardinality $p$ or $1$. Since $\xi$ is not the identity, not all points are fixed and so at least
some orbits have cardinality $p$. If all have cardinality $p$, then $H$ acts freely on $G \setminus \{ e \}$ and on $A$. It then follows
that $p$ divides $|A| = n$. That is, $p$ is a common factor of $n$ and $|G^*|$ (which is $\phi(2n+1)$ when $G = \Z_{2n+1}$).

With $G = \Z_{2n+1}$, if $p$ is a prime such that $p^2|2n+1$, then $p|\phi(2n+1)$, but $p \nmid n$ since $n$ and $2n+1$ are relatively
prime.

With $2n+1 = 21$, $n = 10$ and $\phi(2n+1) = 12$ with no odd common factor.
With $2n+1 = 35$, $n = 17$ and $\phi(2n+1) = 24$ which are relatively prime. Nonetheless in both cases, there exist game subsets invariant with
respect to $m_3$.

We will see below that if $2n + 1$ is composite, then there exists a game subset $A$ such
that
 the inclusion $\Z_{2n+1} \subset Aut(\Gamma [A])$ is proper and so $Aut(\Gamma [A])$ is non-abelian. It follows that
the only possible numbers $n$ for which it may happen that
$\Z_{2n+1} = Aut(\Gamma [A])$, or, equivalently, that $Aut(\Gamma[A])$ is abelian, for every game subset $A \subset \Z_{2n+1}$, are those
with $2n + 1$ a Fermat prime. These are the only cases where it may happen that $\Gamma[A]$ is a TRR for every game subset $A$. This is trivially true
for $3$ and $5$. We do not know the answer for $17$ or the other known Fermat primes.

A digraph $\Pi$ on $I$ is called \emph{symmetric}\index{symmetric} when $Aut(\Pi)$ acts transitively on the edges of $\Pi$, i.e. the induced
action on $I \times I$ is transitive on $\Pi$ itself. When $|\Pi(i)| > 0$ for all $i$ then a symmetric digraph is point-symmetric.
 Notice that if $\Pi = \{ (1,2) \} $ on $\{1, 2 \}$ then $\Pi$ is trivially symmetric but is
not point-symmetric. In particular, a symmetric tournament is point-symmetric and so is a game. While symmetry is a demanding condition, there do exist
symmetric games.

\begin{ex}\label{ex17aa} The quadratic residue games \end{ex}

Let $p$ be a prime number congruent to $-1$ mod $4$ so that $p = 2n + 1$ with $n$  odd and let $k$ be an odd number so that
$p^k = 2m + 1$ with $m$ odd. Notice that $(2n_1 + 1)(2n_2 + 1) = 2n_3 + 1$ with $n_3 \equiv n_1 + n_2 $ mod $2$.

We use some elementary results from the theory of finite fields\index{finite field}, see, e.g. \cite{MM}. By Theorem 1.2.5 of \cite{MM} there
exists a finite field $F$ of order $p^k$ which is unique up to isomorphism. The additive group $G$ of the field is a $\Z_p$ vector space
of dimension $k$ and so we can identify it with  $(\Z_p)^k$. Let $F'$ be the multiplicative group of nonzero elements of $F$ so that $F'$ has
order $2m$.

For $a \in F$ the quadratic equation $x^2 = a$ has at most two roots in $F$. Since the characteristic of the field is odd, $a \not= -a$ if $a \not= 0$.
Hence, the map $sq$ on $F$ given by $sq(a) =  a^2$ restricts to a two-to-one map on $F'$ and so its image on $F'$ is a
subgroup $H$ of $F'$ with order $m$. If $i \in H$ then
the order $o$ of $i$ divides $m$ and so is odd. Hence, the order of $-i$ is $2o$ and so $-i \not\in H$. It follows that $H$ is a game subset of
the additive group $G$ and we let $\Gamma = \Gamma[H]$.

\begin{theo}\label{theo17quad} The group game $\Gamma$ on $G$ is symmetric and there is a group game $\Pi$ on $\Z_m$ such that the restriction of
$\Gamma$ to every $\Gamma(i)$ and to every $\Gamma^{-1}(i)$ is isomorphic to $\Pi$. \end{theo}

\begin{proof} The set $H$ itself is invariant under the  group $H$ acting by multiplication.
Hence, as in the proof of Theorem \ref{theo17}, the multiplications by elements
 of $H$ are automorphisms of $\Gamma$. If we let $\Pi$ denote the restriction of $\Gamma$ to $H = \Gamma(0)$, then $\Pi$ is
 a tournament on the multiplicative group
$H$ and the translation maps by elements of $H$ on $H$ (by multiplication) are automorphisms of $\Pi$ because they are
restrictions of the automorphisms of $\Gamma$. It follows from Theorem \ref{theo11} (a) that $\Pi$ is a group game on $H$.
By Theorem 1.2.8 of \cite{MM} the multiplicative group $F'$ is cyclic and so the subgroup $H$ is cyclic. Thus, $H$ is isomorphic to the
additive group $\Z_m$ and we can identify $\Pi$ with a group game on $\Z_m$.

Recall that $\Gamma[-H]$ is the reverse game $\Gamma[H]^{-1}$. The map $i \mapsto -i$ is an isomorphism from
$\Gamma[H]$ to $\Gamma[-H]$ which maps $-H$ to $H$. So it maps $\Gamma|(-H) = \Gamma|(\Gamma^{-1}(0))$ isomorphically
to $(\Gamma|H)^{-1} = \Pi^{-1}$. But $\Pi$ is a group game on a commutative group and so it is isomorphic to its reversed game.
Thus, $\Gamma|(\Gamma^{-1}(0))$ is
isomorphic to $\Pi$.

Translation in $G$ (by addition) is an automorphism of $\Gamma$ and so $\ell_i$ restricts to an isomorphism from $\Gamma|(\Gamma(0)) $
to $\Gamma|(\Gamma(i))$ and from $\Gamma|(\Gamma^{-1}(0))$ to $\Gamma|(\Gamma^{-1}(i))$. Hence, all of these restrictions are isomorphic to $\Pi$.

The multiplication action of $H \subset Aut(\Gamma)$ is transitive on the edges $\{ (0,h) : h \in H \} $ in $\Gamma$. Translation by $-i$
maps the edge $(i,j)$ to $(0,j-i)$ with $j-i \in H$. It follows that the action of $Aut(\Gamma)$ on the edges of $\Gamma$ is transitive and so
$\Gamma$ is symmetric.

\end{proof} \vspace{.5cm}

{\bfseries Remark:} If $\rho$ is a nontrivial field automorphism of $F$, then $\rho$ commutes with $sq$ and so leaves $H$ invariant. Hence, $\rho$ is an
automorphism of the group game $\Gamma[H]$. Observe that since $\rho$ is the identity on the subfield $\Z_p$, $\rho$ is neither a translation by an
element of $G$ nor multiplication by an element of $F'$. \vspace{.5cm}

A tournament $\Pi$ on $I$ is called \emph{doubly regular}\index{doubly regular}\index{tournament!doubly regular}
when there exist positive integers $n,k$ such that
\begin{equation}\label{eqdr1}
|\Pi(i)| = n, \quad \text{and} \quad |\Pi(i) \cap \Pi(j)| = k
\end{equation}
for all $(i,j) \in \Pi$. By Proposition \ref{prop06euler} (b) a doubly regular tournament is regular with $|I| = 2n + 1$ and so is a
game of size $2n + 1$. In addition, Proposition \ref{prop06euler} (b) implies that for all $i \in I$
the restriction $\Pi|\Pi(i)$ is a game with $n = 2k +1$
and so $|I| = 4k +3$.

Clearly, a symmetric tournament is doubly regular.\vspace{.5cm}

\begin{lem}\label{lem17dr} Let $\Pi$ be a tournament on $I$ with $(i,j) \in I$.
  \begin{equation}\label{eqdoub}
 |\Pi(i)| \ = \ |\Pi(j)| \quad \Longrightarrow \quad |\Pi^{-1}(i) \cap \Pi(j)| \ = \ |\Pi(i) \cap \Pi^{-1}(j)| + 1.
\end{equation}

If $|\Pi(i)| = |\Pi(j)| = 2k +1$ then the following
are equivalent.
\begin{itemize}
\item[(a)] $ |\Pi(i) \cap \Pi(j)| \ = \ k.$
\item[(b)] $ |\Pi(i) \cap \Pi^{-1}(j)| \ = \ k.$
\item[(c)] $ |\Pi^{-1}(i) \cap \Pi^{-1}(j)| \ = \ k.$
\item[(d)] $ |\Pi^{-1}(i) \cap \Pi(j)| \ = \ k+1.$
\end{itemize}
\end{lem}

\begin{proof} Define
\begin{equation}\label{eqdr2}\begin{split}
 A = \Pi(i) \cap \Pi(j), \quad B = \Pi(i) \cap \Pi^{-1}(j), \hspace{.5cm} \\
 C = \Pi^{-1}(i) \cap \Pi^{-1}(j), \quad D = \Pi(i)^{-1} \cap \Pi(j).
 \end{split}
 \end{equation}

 Clearly, we obtain the disjoint unions:
 \begin{equation}\label{eqdr3}\begin{split}
 A \cup D \ = \ \Pi(j), \quad A \cup B \cup \{ j \} = \Pi(i), \hspace{.5cm} \\
  C \cup D \ = \ \Pi^{-1}(i), \quad C \cup B \cup \{ i \} = \Pi^{-1}(j),
   \end{split}
 \end{equation}

So (\ref{eqdoub}) is clear and  (a) $\Leftrightarrow$ (d), (b); (c) $\Leftrightarrow$ (b),(d).

 \end{proof} \vspace{.5cm}

 From the lemma it follows that if the game $\Pi$ is doubly regular then the reverse game $\Pi^{-1}$ is doubly regular.

 Notice that for $(i,j) \in \Pi$, $|\Pi^{-1}(i) \cap \Pi(j)|$ is the number of $3$-cycles through the edge $(i,j)$. A. Kotzigs call a tournament
 $\Pi$ on $I$ \emph{homogeneous}\index{homogeneous}\index{tournament!homogeneous} when for some positive integer $t$ there are exactly $t$
 $3$-cycles through each edge, see \cite{RB}. For a homogeneous game of size $2n + 1$, $t$ times the number of edges is $3$ times the number of $3$-cycles. So
 from Theorem \ref{theo08} it follows that $ t \cdot n(2n+1) = 3 \cdot \frac{(2n+1)n(n+1)}{6}$ and so $t = \frac{n+1}{2}$. Hence, $n$ is odd and with
 $t = k +1$, $n = 2k +1$. It follows from Lemma \ref{lem17dr} that a game is homogeneous if and only if it is doubly regular. Reid and Brown
 prove that a homogeneous tournament is necessarily regular and so is doubly regular, Theorem 1 of \cite{RB}. Using Hadamard matrices
 they construct doubly regular tournaments for an infinite family of sizes. It apparently remains an open
 question whether there exist doubly regular tournaments of size $4k + 3$ for every $k$.

Finally, we  observe that the only group games which are reducible are those on  $\Z_{2n+1}$  which are isomorphic to $\Gamma[Odd_n]$ or, equivalently,
isomorphic to $\Gamma[[1,n]]$.

\begin{theo}\label{theo28} (a) If $A \subset \Z_{2n+1}$ is a game subset with
$\Gamma[A]$ reducible, then $A = m_a(Odd_n)$ for a unique
$a \in  \Z_{2n+1}^*$. Conversely, for $a \in \Z_{2n+1}^*$, $a$ is the unique element of $\Z_{2n+1}$
 such that $\Gamma[m_a(Odd_n)]$ is reducible via $0 \to a$.

 (b) If a group  $G$ is not cyclic, then for no game subset $A $ of $G$ is $\Gamma[A]$ reducible.

 \end{theo}

\begin{proof} Assume $A$ is a game subset for a group $G$.  If $\Gamma[A]$ is reducible via $u \to v$, then it is reducible via $e \to u^{-1}v$.

 Notice that $e \to a$ if and only if $a\in A$. The following are equivalent.
 \begin{itemize}
 \item $\Gamma[A]$ is reducible via $e \to a$.
 \item $ a\in A$ and $A \cap aA = \emptyset$.
 \item $A^{-1} = aA$.
 \end{itemize}
 Observe that $c \in A \cap aA$ if and only if $e \to c$ and $a \to c$.  By Proposition \ref{prop07} (d), the set of such $c$ is
 empty if and only if $\Gamma[A]$ is reducible  via $\{ e, a \}$.
 If $a \in A$ then $e \not\in aA$ and so $aA = A^{-1}$ if and only if $A$ and $aA$ are
 disjoint.

 From Proposition \ref{prop07} (f), it follows that $\{ a \in A : \Gamma[A]$ is reducible via $e\to a \}$
 consists of at most one element.

 (a): In this case $G = \Z_{2n+1}$ and so $\Gamma[A]$ is reducible via $0 \to a$ if and only if $-A = a + A$.

 The game $\Gamma[Odd_n]$ is reducible via $0 \to 1$ because $Odd_n + 1 = -Odd_n$.
 For $a \in  \Z_{2n+1}^*$, $m_a$ is an isomorphism
 from $\Gamma[Odd_n]$ to $\Gamma[m_a(Odd_n)]$ and so $\Gamma[m_a(Odd_n)]$ is reducible via $0 \to a$.

  Now assume that $\Gamma[A]$ is reducible via $0 \to a$. If for some $i$, both $i$ and $i + a$ are elements of $ A$, then $i + a \to 0, a$ and
 and if $-i, -i - a \in A $, then $-i \to 0, a$.  So Proposition \ref{prop07} (d) implies that it cannot happen that $i, i + a \in A$ or
 $i, i + a \in -A$.

If $a \in  \Z_{2n+1}^*$ then $a$ is a generator of the additive group $\Z_{2n+1}$ and so $m_a$ is a permutation
 of $\Z_{2n+1}$ mapping $0,1,2,\dots,2n$ to
 $0,a,2a,\dots, 2na$. Since $a \in A$, $(2k - 1)a \in A, 2ka \in -A$ for $a = 1,\dots,n$. That is, $A = m_a(Odd_n)$.

 If $a \not\in Z_{2n+1}^*$ then the order of $a$ in $\Z_{2n+1}$ is a proper divisor of $2n + 1$ and the cyclic subgroup generated by $a$ is a proper
 subgroup of $\Z_{2n+1}$. The result follows by the proof in part (b).

 (b): Now assume that $G$ is a general group of odd order,  and that the cyclic subgroup $H$ generated by $a$ is a proper subgroup of $G$.
This applies to all $a \in G$ when $G$ is not cyclic. Let $2k + 1 = |H|$ and $2n + 1 = |G|$.

 In general, if $\Gamma[A]$ is reducible via $e \to a$ then $aA \cap A = \emptyset$. Consequently, $b \in A$ implies $ab \not\in A$.
 So for any $b \in G$, no successive members of the sequence $b, ab, a^2b, \dots$ lie in $A$. Since
 $ \{ b, ab, a^2b, \dots \} = Hb$ has cardinality $2k+1 $, it follows that $|A \cap Hb|$ is at most $k$.
 The number of right cosets $Hb$ is $(2n + 1)/(2k + 1)$. Since $k < n$, this would imply that $n > [(2n + 1)/(2k + 1)] \cdot k \geq |A|$ which
 contradicts $|A| = n$.

\end{proof} \vspace{.5cm}

Notice that $[1,n] + n = [n+1,2n] \subset \Z_{2n+1}$ so $\Gamma[[1,n]]$ is reducible
via $0 \to n$. Thus, we see again that $[1,n] = m_n(Odd_n)$.
\vspace{1cm}

\section{Inverting Cycles}\label{secinvert}

While we are most interested in games, the results of this section are equally applicable for more general tournaments and so we will
consider the more general case.

Assume that $\Pi$ and $\Gamma$ are tournaments on the set $I$. Let $\r \in S(I)$, i.e. a permutation on $I$.
Recall that $\bar \r = \r \times \r$ is the
product permutation on $I \times I$.

We define
\begin{equation}\label{eq4}
\Delta(\r,\Pi,\Gamma) \ = \ \Pi \cap (\bar \r)^{-1}(\Gamma^{-1}) \ = \ \{ (i,j) \in \Pi : (\r(j),\r(i)) \in \Gamma \}.
\end{equation}
That is, $\Delta(\r,\Pi,\Gamma)$\index{$\Delta(\r,\Pi,\Gamma)$} consists of those edges of $\Pi$ which are reversed by $\bar \r$.
For the special case when
$\r = 1_I$ we will write
\begin{equation}\label{eq4a}
\Delta(\Pi,\Gamma) \ = \ \Delta(1_I, \Pi,\Gamma) \ = \ \{ (i,j) \in \Pi : (j, i) \in \Gamma \} \ = \  \Delta(\Gamma,\Pi)^{-1}.
\end{equation}
\index{$\Delta(\Pi,\Gamma)$}

If $\r_{12} \in S(I)$ is an isomorphism from $\Gamma_1$ to $\Gamma_2$ and $\r \in S(I)$,  then
\begin{equation}\label{eq5}
\Delta(\r,\Pi,\Gamma_1) \ = \ \Delta(\r_{12} \circ \r,\Pi,\Gamma_2).
\end{equation}
\vspace{.5cm}

We will say that $\r \in S(I)$ \emph{preserves scores}\index{preserves scores}, or is a
\emph{score-preserving permutation}\index{score-preserving permutation}, from $\Pi$ to $\Gamma$ if $|\Gamma(\r(i))| = |\Pi(i)|$ for all $i \in I$.
It then follows that $|\Gamma^{-1}(\r(i))| = |I|  - 1 - |\Gamma(\r(i))| =  |I|  - 1 - |\Pi(i)| =  |\Pi^{-1}(i)|$. We will say that
$\Pi$ and $\Gamma$ have the same scores when the identity on $I$ preserves scores from $\Pi$ to $\Gamma$, i.e.
$|\Gamma(i)| = |\Pi(i)|$ for all $i \in I$. Between two games on $I$ any permutation preserves scores.

\begin{prop}\label{prop24} Assume that $\Pi$ and $\Gamma$ are tournaments on the set $I$ and $\r \in S(I)$.
The permutation $\r$ preserves scores from $\Pi$ to $\Gamma$ if and only if $\Delta(\r,\Pi,\Gamma)$ is Eulerian. \end{prop}

\begin{proof}
Let $\Delta = \Delta(\r,\Pi,\Gamma)$. Thus, $\Delta$ is a subgraph of $\Pi$ and $\bar \r(\Delta^{-1})$ is a
subgraph of $\Gamma$.

The permutation $\r$ is an isomorphism from $\Pi \setminus \Delta$ to $\Gamma \setminus \bar \r(\Delta^{-1})$ and
is an isomorphism from $\Delta^{-1}$ to $\bar \r(\Delta^{-1}) \subset \Gamma$.

\begin{align}\label{eq7}
\begin{split}
\Pi(i)  \ &= \ (\Pi \setminus \Delta)(i) \cup \Delta(i),  \\
\Gamma(\r(i))  \ &= \ \r[(\Pi \setminus \Delta)(i)] \cup \r[\Delta^{-1}(i)].
\end{split}
\end{align}
The unions are disjoint and the permutation $\r$ preserves cardinality. So it follows that $|\Pi(i) | = |\Gamma(\r(i))|$ if and only if
$|\Delta(i)| = |\Delta^{-1}(i)|$.

\end{proof} \vspace{.5cm}

If $\Delta$ is a subgraph of $\Pi$ then we define \emph{$\Pi$ with $\Delta$ reversed} by
\begin{equation}\label{eq8}
\Pi/\Delta \ = \ (\Pi \setminus \Delta) \ \cup \ \Delta^{-1}.
\end{equation}

If $\Delta_1$ and $\Delta_2$ are subgraphs of $\Pi$ then clearly
\begin{equation}\label{eq9}
\Delta_1 \cap \Delta_2 \ = \ \emptyset \quad \Longrightarrow \quad \Pi/(\Delta_1 \cup \Delta_2) \ = \ (\Pi/\Delta_1)/ \Delta_2.
\end{equation}
Notice that if, as above, $\Delta_1$ and $\Delta_2$ are disjoint, then $\Delta_2$ is a subgraph of  $\Pi/\Delta_1$. Recall that disjoint
subgraphs may have vertices in common.

\begin{prop}\label{prop25} Assume that  $\Pi$ and $\Gamma$ are tournaments on $I$. If $\Delta$ is a subgraph of $\Pi$,
then $\Pi/\Delta = \Gamma$ if and only if $\Delta = \Delta(\Pi,\Gamma)$. In particular,
\begin{equation}\label{eq10}
\Delta \ = \  \Delta(\Pi,\Pi/\Delta)
\end{equation}

The tournament $\Pi/\Delta$ has the same scores as  $\Pi$
 if and only if $\Delta$ is Eulerian.

\end{prop}

\begin{proof} The first part is obvious. The rest follows from Proposition \ref{prop24}.

\end{proof} \vspace{.5cm}

\begin{prop}\label{prop27}  Let $\Pi$ and $\Gamma$ be tournaments on $I$ and $\r \in S(I)$. If $\Delta$ is a subgraph of $\Pi$, then
$\r$ is an isomorphism from $\Pi/\Delta$ to $\Gamma$ if and only if $\Delta = \Delta(\r, \Pi, \Gamma)$. \end{prop}

\begin{proof} If $\Delta = \Delta(\r, \Pi, \Gamma)$, then it is clear that $\r$ is an isomorphism from $\Pi/\Delta$ to $\Gamma$.

On the other hand, if $\r$ is an isomorphism from $\Pi/\Delta$ to $\Gamma$, then (\ref{eq5}) and (\ref{eq10}) imply that.
\begin{equation}\label{eq11}
\Delta \ = \  \Delta(\Pi,\Pi/\Delta) \ = \ \Delta(\r \circ 1_I, \Pi, \Gamma) \ = \ \Delta(\r, \Pi, \Gamma).
\end{equation}

\end{proof} \vspace{.5cm}

 \begin{cor}\label{cor25aa}If  $\Pi$ is a game on $I$, $\Delta \subset \Pi$ is Eulerian and $u, v \in I$, then any two of the
 following conditions implies the third.
 \begin{enumerate}

 \item[(i)] The game $\Pi$ is reducible via $\{ u, v \}$.

  \item[(ii)] The game $\Pi/\Delta$ is reducible via $\{ u, v \}$

   \item[(iii)] The graph  $\Delta|(I \setminus \{ u, v \})$ is Eulerian.

 \end{enumerate}
 \end{cor}

 \begin{proof} It is clear that
 \begin{equation}\label{eq10aa}
 (\Pi/\Delta)|(I \setminus \{ u, v \}) \ = \ (\Pi|(I \setminus \{ u, v \}))/ (\Delta|(I \setminus \{ u, v \}))
 \end{equation}

Condition (i) is equivalent to the assumption that $\Pi|(I \setminus \{ u, v \})$ is a game.

By Equation (\ref{eq10aa}), condition (ii) is equivalent to the
assumption that $(\Pi|(I \setminus \{ u, v \}))/ (\Delta|(I \setminus \{ u, v \}))$
 is a game.

 So assuming (i), (ii) $\Leftrightarrow$ (iii) by Proposition \ref{prop25}.  Similarly, assuming (ii), (i) $\Leftrightarrow$ (iii) by
 applying the previous result, replacing $\Pi$ by $\Pi/\Delta$ and $\Delta$ by $\Delta^{-1}$.

\end{proof} \vspace{.5cm}

%It follows from Proposition \ref{prop24}  that if $\Pi$ is a tournament, then $\Pi/\Delta$ is a tournament with the same score vector
% if and only if $\Delta$ is Eulerian.

 It follows from Proposition \ref{prop25} that if $\Pi$ and $\Gamma$ are any two tournaments on $I$ with the same scores, then $\Gamma$ is equal to
 $\Pi$ with $\Delta$ reversed for $\Delta$ the Eulerian subgraph $\Delta(\Pi,\Gamma)$.

 \begin{cor}\label{cor25a} If $\Pi$ is a tournament on $I$, then the number of tournaments on $I$ with the same scores as $\Pi$ is exactly
 the number of Eulerian subgraphs of $\Pi$.  In particular, the number of Eulerian subgraphs of a tournament depends only on the score vector.
 \end{cor}

\begin{proof} If $\Gamma$ is a tournament on $I$, then by Proposition \ref{prop25} $\Gamma = \Pi/\Delta$ with $\Delta = \Delta(\Pi,\Gamma)$ and
by Proposition \ref{prop24} $\Gamma$ has the same scores if and only if $\Delta$ is Eulerian. Hence, the set of tournaments $\Gamma$ with the
same scores is in a one-to-one correspondence with the set of Eulerian subgraphs of $\Pi$ via $\Gamma \leftrightarrow \Delta(\Pi,\Gamma)$.

Now suppose that $\Pi_1$ and $\Pi_2$ are tournaments on $I_1$ and $I_2$ with the same score vector.  There exists a bijection
$\r : I_1 \to I_2$ which preserves scores, i.e. $|\Pi_1(i)| = |\Pi_2(\r(i))|$ for all $i \in I_1$. Define $\Gamma$ on $I_1$ so that
$\r : \Gamma \to \Pi_2$ is an isomorphism. Clearly, $\Pi_1$ and $\Gamma$ have the same scores and so have the same number of Eulerian subgraphs.
Since $\r : \Gamma \to \Pi_2$ is an isomorphism, $\Gamma$ and $\Pi_2$ have the same number of Eulerian subgraphs.
 \end{proof} \vspace{.5cm}

 The following is a theorem of Ryser \cite{R2}, see also
 \cite{M} Theorem 35 and  \cite{BL}.

 \begin{theo}\label{theo26} Assume that  $\Pi$ and $\Gamma$ are tournaments with the same score on $I$. There exists a finite sequence
 $\Pi_1, \dots, \Pi_k$ of tournaments on $I$ with $\Pi_1 = \Pi$ and $\Pi_k = \Gamma$ and such that for $p = 1,\dots, k-1$,
 $\Pi_{p+1}$ is $\Pi_p$ with some $3$-cycle reversed. In particular, all of these tournaments have the same score.

 If $\Gamma$ is obtained from $\Pi$ be reversing a single cycle of length $\ell$ then  a sequence can be chosen
 with $k = \ell - 2$. \end{theo}

 \begin{proof} In one step we can get from $\Pi$ to $\Gamma$ with $\Gamma$ equal to $\Pi$ with the Eulerian subgraph
  $\Delta(\Pi,\Gamma)$ reversed.

  By Theorem \ref{theo02},  $\Delta(\Pi,\Gamma)$ is a disjoint union of cycles. From (\ref{eq9}) and induction
  it follows that there exists a finite sequence
 $\Pi_1, \dots, \Pi_k$ of games on $I$ with $\Pi_1 = \Pi$ and $\Pi_k = \Gamma$ and such that for $p = 1,\dots, k-1$
 $\Pi_{p+1}$ is $\Pi_p$ with some cycle reversed.

 Thus, we are reduced to the case when $\Gamma = \Pi/\Delta$ and $\Delta$ is a cycle $\langle i_1,\dots,i_{\ell} \rangle$ with $\ell \geq 3$.
 We obtain the result by induction on the length $\ell$ of the cycle. If $\ell = 3$, then we get from $\Pi$ to $\Gamma$ immediately
 by reversing a single $3$-cycle. Now assume that $\ell > 3$ and assume that the result holds for shorter cycles.
 \vspace{.5cm}

 {\bfseries Case 1} ($i_1 \to i_3$ in $\Pi$): In this case, $\langle i_1,i_3,\dots,i_{\ell} \rangle$ is an $\ell - 1$ cycle in $\Pi$ and so we can get
 from $\Pi$ to $\tilde \Gamma = \Pi/\langle i_1,i_3,\dots,i_{\ell} \rangle$ via a sequence of $\ell - 3$  $3$-cycles by inductive hypothesis.

 Observe that $i_3 \to i_1$ in $\tilde \Gamma$ and so $\langle i_1, i_2, i_3 \rangle$ is a $3$-cycle in $\tilde \Gamma$. Furthermore,
 $\Gamma = \Pi/ \langle i_1,\dots,i_{\ell}\rangle = \tilde \Gamma/\langle i_1, i_2, i_3 \rangle$ because
 the edge between $i_1$ and $i_3$ has been reversed twice.
 Thus, extending the $\tilde \Gamma$ sequence by one we obtain the sequence from $\Pi$ to $\Gamma$.
 \vspace{.5cm}

 {\bfseries Case 2} ($i_3 \to i_1$ in $\Pi$): We proceed as before reversing the order of operations.
 This time $\langle i_1, i_2, i_3 \rangle$ is a $3$-cycle
 in $\Pi$ and $\langle i_1,i_3,\dots,i_{\ell}\rangle $  is an $\ell - 1$ cycle in $\tilde \Gamma = \Pi/\langle i_1, i_2, i_3 \rangle$. Now
 $\Gamma = \tilde \Gamma/\langle i_1,i_3,\dots,i_{\ell} \rangle$ and so by inductive hypothesis we can get from $\tilde \Gamma$ to $\Gamma$ via
 a sequence of $\ell - 3$ $3$-cycles.

\end{proof} \vspace{.5cm}

A \emph{decomposition}\index{digraph!decomposition} for an Eulerian digraph $\Pi$  is a pairwise disjoint set of cycles with union $\Pi$.
A \emph{maximum decomposition}\index{maximum decomposition}  is a decomposition having maximum
cardinality. We call the cardinality of a maximum decomposition the \emph{span}\index{span} of $\Pi$, denoted $\s(\Pi)$\index{$\s(\Pi)$}.
The complement in $\Pi$ of
the union of any pairwise disjoint collection of cycles
in $\Pi$ is Eulerian by  Lemma \ref{lem01} and so by Theorem \ref{theo02} any such collection can be
extended to a decomposition. Thus, the span is the maximum
size of such a collection. Furthermore, such a collection has cardinality equal to the span if and only if it is a maximum decomposition.

We define the \emph{balance invariant}\index{balance invariant}\index{digraph!balance invariant} $\b(\Pi)$ of the Eulerian digraph $\Pi$ by
\begin{equation}\label{eq12x}
\b(\Pi) \ = \ |\Pi| - 2 \s(\Pi). \hspace{3cm}
\end{equation}
That is, $\b(\Pi)$\index{$\b(\Pi)$} is the number of edges minus twice the size of the maximum
decomposition.  Observe that if $\{ C_1, C_2, \dots, C_p \}$
is a maximum decomposition of $\Pi$, then
\begin{equation}\label{eq12xa}
\b(\Pi) \ = \ \sum_{r=1}^p \ \b(C_r), \hspace{3cm}
\end{equation}
because if $C$ is a cycle, $\b(C) = |C| - 2$. Also, it is clear that if $\{ C_1, C_2, \dots, C_p \}$ is a decomposition of $\Pi$ then
$\{ C_1^{-1}, C_2^{-1}, \dots, C_p^{-1} \}$ is a decomposition of  $\Pi^{-1}$ and so $\b(\Pi) = \b(\Pi^{-1})$.

The balance invariant and the following result were shown to me by my colleague Pat Hooper.

\begin{theo}\label{theo28x} Let $\Pi$ and $\Gamma$ be tournaments with the same scores on a set $I$.
\begin{enumerate}
\item[(a)] If a tournament $\Gamma'$ is obtained by reversing a $3$-cycle in $\Gamma$, then
\begin{equation}\label{eq13x}
|\b(\Delta(\Gamma',\Pi)) - \b(\Delta(\Gamma,\Pi))| \ = \ 1.
\end{equation}

\item[(b)] If $\Pi \not= \Gamma$, then there exists a game $\Gamma'$ obtained by reversing a $3$-cycle in $\Gamma$, such that
\begin{equation}\label{eq14x}
\b(\Delta(\Gamma',\Pi)) \ = \ \b(\Delta(\Gamma,\Pi)) - 1.
\end{equation}

\item[(c)] Assume that $C$ is a $3$-cycle contained in $\Delta(\Gamma,\Pi)$ but which is
not contained in any maximum decomposition of $\Delta(\Gamma,\Pi)$. If the tournament
$\Gamma'$ is obtained by reversing the $3$-cycle $C$ in $\Gamma$, then
\begin{equation}\label{eq14xx}
\b(\Delta(\Gamma',\Pi)) = \b(\Delta(\Gamma,\Pi)) + 1.
\end{equation}
\end{enumerate}
\end{theo}

\begin{proof} Let the reverse of the given $3$-cycle be $\langle i_1, i_2, i_3 \rangle$ so that $\langle i_1, i_2, i_3 \rangle$
is contained in $\Gamma'$. Let $\Delta = \Delta(\Gamma,\Pi)$ and
$\Delta' = \Delta(\Gamma',\Pi)$
\vspace{.5cm}

(a): {\bfseries Case 1}: The $3$-cycle is disjoint from $\Delta$ and so its reverse is contained in $\Delta'$.

Clearly, $|\Delta'| = |\Delta| + 3$. Given a maximum decomposition for $\Delta$ we can adjoin the $3$-cycle to obtain a decomposition
for $\Delta'$.  Hence,
\begin{equation}\label{eq15x}
\s(\Delta') \ \geq \ \s(\Delta) + 1. \hspace{3cm}
\end{equation}

Now take a maximum decomposition for $\Delta'$. If the $3$-cycle is one of the cycles of this decomposition then by removing it we
obtain a decomposition for $\Gamma$ and so see that $\s(\Delta) \ \geq \ \s(\Delta') - 1$ and so we obtain equality in (\ref{eq15x}).
Hence,
\begin{equation}\label{eq16x}
\b(\Delta') =   |\Delta'| - 2 \s(\Delta') =  |\Delta| + 3 - 2 (\s(\Delta) + 1) = \b(\Delta) + 1.
\end{equation}
Thus, (\ref{eq13x}) follows.

Next suppose that $(i_1,i_2)$ is contained in one cycle of the $\Delta'$ decomposition and that $(i_2,i_3)$ and $(i_3,i_1)$ are contained in another.
So we can write these cycles $\langle i_2, r_1, \dots, r_{k}, i_1 \rangle$ and $\langle i_1, s_1, \dots, s_{\ell}, i_2, i_3 \rangle$.
In particular, $i_1$ and $i_2$ are vertices of $\Delta$ as well as $\Delta'$. Thus, $[i_2, r_1, \dots, r_{k}, i_1]$ and
$[i_1, s_1, \dots, s_{\ell}, i_2]$ are simple paths in $\Delta$ with no edges in common. It follows from Theorem \ref{theo02a} that the
closed path $[i_1, s_1, \dots, s_{\ell}, i_2,  r_1, \dots, r_{k}, i_1]$ is an Eulerian subgraph of $\Delta$. Writing it as a disjoint union
of cycles and using the remaining cycles of the $\Delta'$ decomposition, we again obtain  $\s(\Delta) \ \geq \ \s(\Delta') - 1$
 and so again we obtain equality in (\ref{eq15x}) and (\ref{eq13x}) follows as before.

Now suppose that  $(i_1,i_2), (i_2,i_3)$ and $(i_3,i_1)$ are each contained in separate cycles of the $\Delta'$ decomposition.
We can write these cycles
$$\langle i_2, r_1, \dots, r_{k}, i_1 \rangle, \ \langle i_1, s_1, \dots, s_{\ell}, i_3 \rangle, \
\langle i_3, t_1, \dots, t_{p}, i_2 \rangle.$$
 So $i_1, i_2$ and $i_3$ are vertices of $\Delta$. Again we concatenate simple paths with no edges in common
to obtain the closed path
$$ [i_2, r_1, \dots, r_{k}, i_1, s_1, \dots, s_{\ell}, i_3, t_1, \dots, t_{p}, i_2].$$
 This time we obtain
$\s(\Delta) \ \geq \ \s(\Delta') - 2$. Thus, $ \s(\Delta) + 1 \leq \s(\Delta') \leq \s(\Delta) + 2.$  If $ \s(\Delta) + 1 = \s(\Delta')$
then (\ref{eq16x}) holds and (\ref{eq13x}) follows. If, instead $\s(\Delta') = \s(\Delta) + 2$ then we obtain
\begin{equation}\label{eq17x}
\b(\Delta') =   |\Delta'| - 2 \s(\Delta') =  |\Delta| + 3 - 2 (\s(\Delta) + 2) = \b(\Delta) - 1,
\end{equation}
which still implies (\ref{eq13x}).  This case does occur if the three vertices lie on a cycle in the maximum decomposition of $\Delta$. The length of
such a cycle would have to be at least 6.
\vspace{.5cm}

 {\bfseries Case 2}: The $3$-cycle is contained in $\Delta$ and so its reverse is disjoint from $\Delta'$.

 This follows from Case 1 by reversing the roles of $\Gamma$ and $\Gamma'$.
 \vspace{.5cm}

  {\bfseries Case 3}: A single edge of the $3$-cycle is contained in $\Delta$ and so the reverse of the other two edges is
  contained in $\Delta'$.

This time $|\Delta'| = |\Delta| + 1$.

We may suppose that the edge $(i_2,i_1) \in \Delta$. If $\langle i_1,r_1,\dots,r_k, i_2 \rangle$ is
a cycle of a maximum decomposition for $\Delta$ then $\langle i_1,r_1,\dots,r_k, i_2, i_3 \rangle$ is a cycle in $\Delta'$ and together
with the remaining cycles from the $\Delta$ decomposition yields a decomposition of $\Delta'$. Hence, $\s(\Delta') \geq \s(\Delta)$.

 Now take a maximum decomposition for $\Delta'$. If the edges $(i_2,i_3)$ and $(i_3,i_1)$ are contained in a single cycle,
 $\langle i_1,r_1,\dots, r_k, i_2, i_3 \rangle$, then reversing the above procedure we obtain a decomposition for $\Delta$
 containing $\langle i_1,r_1,\dots, r_k, i_2 \rangle$. Hence, $\s(\Delta) = \s(\Delta')$, and we get $\b(\Delta') = \b(\Delta) + 1$.

 Suppose instead that the edges $(i_2,i_3)$ and $(i_3,i_1)$ are contained separate cycles
 $\langle i_1,r_1,\dots,r_k, i_3 \rangle$ and $\langle i_3,s_1,\dots,s_{\ell}, i_2 \rangle$  of a $\Delta'$ maximum decomposition. It
 follows that
 $$[i_1,r_1,\dots,r_k, i_3,s_1,\dots,s_{\ell}, i_2, i_1]$$ is a closed path in $\Delta$ with distinct edges and so by Theorem \ref{theo02a} again
 it is a union of cycles. Adjoining the remaining $\Delta'$ cycles we see that $\s(\Delta) \geq \s(\Delta') - 1.$ That is,
 $\s(\Delta) \leq \s(\Delta') \leq \s(\Delta) + 1$.  If $\s(\Delta) =\s(\Delta')$ then, as before, $\b(\Delta') = \b(\Delta) + 1$.
 If $\s(\Delta') = \s(\Delta) + 1$, then $\b(\Delta') = \b(\Delta) - 1$.
  \vspace{.5cm}

  {\bfseries Case 4}: Two edges of the $3$-cycle are contained in $\Delta$ and so the reverse of the other edge is
  contained in $\Delta'$.

 This follows from Case 3 by reversing the roles of $\Gamma$ and $\Gamma'$.
 \vspace{.5cm}

(b) $\Delta$ is nonempty since $\Pi \not= \Gamma$.  Among the maximum decompositions for $\Delta$ we choose one
such that the shortest cycle has the smallest length possible. We label this shortest cycle $\langle i_1, i_2, \dots, i_k  \rangle$. Thus,
$k$ is assumed to be the minimum length of any cycle in $\Delta$ which occurs in a maximum decomposition.
 \vspace{.5cm}

\textbf{Claim:} No $(i_p, i_q)$ with $q \not= p \pm 1$ mod $k$ lies in $\Delta$.

\begin{proof} If it did, then by relabeling we can assume that
$p = 1$ and $2 < q < k$. In particular, $k \geq 4$. Then the cycle in the decomposition which contains $(i_1,i_q)$ is of the
form $\langle i_q, r_1,\dots, r_{\ell}, i_1 \rangle$ with $\ell \geq k - 2 \geq 2$. The union of $\langle i_1, i_2, \dots, i_k  \rangle$
and $\langle i_q, r_1,\dots, r_{\ell}, i_1 \rangle$ is the same
as the union of the cycles $\langle  r_1,\dots, r_{\ell}, i_1, i_2, \dots, i_q \rangle$ and $\langle i_1, i_q, i_{q+1},\dots, i_k \rangle$.
Replacing the initial two
cycles by these two we obtain a decomposition of $\Delta$ with cardinality $\s(\Delta)$ and so is a maximum decomposition.  Furthermore, it
 contains a cycle of length less than $k$, contradicting the minimality of $k$.
\end{proof} \vspace{.5cm}

{\bfseries Case 1}[$k = 3$]

Reverse this $3$-cycle to obtain $\Gamma'$. We see that $|\Delta'| = |\Delta| - 3$. The remaining cycles provide a decomposition of
$\Delta'$. Hence, $\s(\Delta') \geq \s(\Delta) - 1$. It follows that $\b(\Delta') \leq \b(\Delta) - 1$. From (\ref{eq13x}) equality holds.
 \vspace{.5cm}

{\bfseries Case 2}: [$k \geq 4$ and $(i_3,i_1) \in \Gamma$].

From the Claim $(i_3,i_1) \in \Gamma \setminus \Delta$.

Reverse the $3$-cycle $\langle i_1, i_2, i_3 \rangle$. This time $|\Delta'| = |\Delta| - 1$. Now $\langle i_3, \dots, i_k, i_1 \rangle$
is a cycle of $\Delta'$ which, together with the remaining $\Delta$ cycles, provides a decomposition for $\Delta'$. Thus, $\s(\Delta') \geq \s(\Delta)$
and so $\b(\Delta') \leq \b(\Delta) - 1$. Again, from (\ref{eq13x}) equality holds.

In general, by relabeling, Case 2 applies whenever for some $p = 1, \dots, k$, $(i_{p+2},i_p) \in \Gamma$ with $p+2$ reduced mod $k$.

In particular, if $k = 4$ and $(i_1, i_3) \in \Gamma$ then the cycle is $\langle i_3, i_4, i_1, i_2 \rangle$ with $(i_1, i_3) \in \Gamma$. Thus,
Case 2 applies when $k = 4$.
 \vspace{.5cm}

{\bfseries Case 3}: [$k \geq 5$, and
 $(i_p,i_{p+2}) \in \Gamma$ for $p = 1, \dots, k$ (reducing the indices mod $k$)].

Observe that, with respect to $\Gamma$, $i_5$ is an output for $i_3$ and $i_1$ is an input for $i_3$. So there will exist
$q$ with $5 \leq q \leq k$ such that $i_q$ is an output for $i_3$ and $i_{q+1}$ is an input for $i_3$. By the above remarks
$(i_3,i_q), (i_{q+1},i_3) \in \Gamma \setminus \Delta$.

Reverse the $3$-cycle $\langle i_3, i_q, i_{q+1} \rangle$. $|\Delta'| = |\Delta| + 1$. On the other hand, $\langle i_3, i_4, \dots, i_q \rangle$ and
$\langle i_{q+1},\dots,i_k, i_1, i_2, i_3 \rangle$ are disjoint cycles of $\Delta'$ which, together with the remaining $\Delta$ cycles, form
a decomposition of $\Delta'$. Hence, $\s(\Delta') \geq \s(\Delta) + 1$. Thus, $\b(\Delta') \leq \b(\Delta) - 1$. Again equality holds.
 \vspace{.5cm}

(c) Since $C$ is contained in $\Delta(\Gamma,\Pi)$ it follows that $\Delta(\Gamma',\Pi) = \Delta(\Gamma,\Pi) \setminus C$.
Hence, $|\Delta(\Gamma',\Pi)| = |\Delta(\Gamma,\Pi)| - 3$. On the other hand, the span $\s(\Delta(\Gamma',\Pi)) \leq \s(\Delta(\Gamma,\Pi)) - 2$.
For if we had a decomposition of size at least $\s(\Delta(\Gamma,\Pi)) - 1$ then adjoining $C$ we would obtain a decomposition
of $\Gamma$, containing $C$ and of size at least $\s(\Delta(\Gamma,\Pi))$ which would have to be a maximum decomposition, contrary to hypothesis.
It thus follows that $\b(\Delta(\Gamma',\Pi)) \geq \b(\Delta(\Gamma,\Pi)) + 1$ and so equality holds by part (a).

\end{proof} \vspace{.5cm}

\textbf{Remark:} For a game $\Pi$ it follows from the Claim in (b) that the shortest cycle which occurs in a maximum decomposition for
$\Pi = \Delta(\Pi,\Pi^{-1})$
has length 3.
On the other hand, it can happen that an Eulerian digraph $\Delta$ contains a unique $3$-cycle and it does not occur in a maximum decomposition.
For example, let
\begin{equation}\label{eq16xxy}
\begin{split}
\Delta \ = \ \langle 0, 1, \dots, 8 \rangle \cup \langle 6, 3, 0 \rangle \ = \\
\langle 0, 1, 2, 3 \rangle \cup \langle 3, 4, 5, 6 \rangle \cup \langle 6, 7, 8, 0  \rangle.
\end{split}
\end{equation}

From Theorem \ref{theo28x} we obtain

\begin{cor}\label{cor29x} If $\Pi$ and $\Gamma$ are two tournaments on $I$ with the same scores,
then the minimum number of $3$-cycle reversal steps needed to get from $\Gamma$ to $\Pi$ is $\b(\Delta(\Gamma,\Pi))$.\end{cor}

$\Box$ \vspace{.5cm}

We also have

\begin{cor}\label{cor29y} If $\Pi$ and $\Gamma$ are two tournaments on $I$ with the same scores,
and $\Pi$ is obtained from $\Gamma$ be reversing $k$ $3$-cycles, then $|\Delta(\Gamma,\Pi)|$ and $\b(\Delta(\Gamma,\Pi))$ are both congruent to
$k$ modulo $2$.\end{cor}

\begin{proof} Obviously, $|\Delta(\Gamma,\Pi)|$ is congruent to $\b(\Delta(\Gamma,\Pi))$ mod 2 and it is easy to check directly that reversing
a $3$-cycle changes $|\Delta(\Gamma,\Pi)|$ by adding or subtracting either 1 or 3. Alternatively, the congruence result follows from (\ref{eq13x}).

\end{proof} \vspace{.5cm}

A tournament $\Gamma$ has a decomposition into cycles if and only if it is Eulerian and so is a game.
If the decomposition consists entirely of $3$-cycles then
the decomposition is clearly maximum.

If $I$ is a set of size $p$, then a set of \emph{Steiner triples}\index{Steiner triples} for $I$ is a set of
three element subsets, called triples, such that each pair of elements is contained in exactly one triple.
That is, the $p(p-1)/2$ pairs are partitioned into $p(p-1)/6$ triples. For each triple $\{i,j,k \}$ there are two
possible orientations defining $3$-cycles $\langle i, j, k \rangle$ or $\langle k, j, i \rangle$. Choosing the
orientations arbitrarily leads to a game and so we obtain $2^{p(p-1)/6}$ games on
$I$. Since these are games, the size $p$ must be odd with $p(p-1)$ divisible by 6. Thus, $p$ must be
congruent to 1 or 3 modulo 6. Equivalently,
if $p = 2n + 1$, then $n$ is congruent to 0 or 1 mod 3 and the $n(2n +1)$ edges are partitioned into $n(2n+1)/3$ triples.
The question of which sizes $p$ admit sets of such triples was raised in 1833 and was solved in
1847 by Rev. T. P. Kirkman \cite{K} who showed that these congruence conditions are sufficient as well as necessary.
Explicit constructions are given in \cite{B},\cite{Sk} and \cite{Sp}. The latter shows that if $n$ is congruent to
2 modulo 3, then there exist games of size $2n + 1$ with a decomposition consisting of  $[n(2n+1)- 4]/3$ $3$-cycles and one 4-cycle.

We will call a game a \emph{Steiner game}\index{Steiner game}\index{game!Steiner} when it admits a decomposition consisting of $3$-cycles.

\vspace{1cm}

\section{Interchange Graphs}\label{secinterchange}

Recall that an undirected graph on a finite set $X$ is represented by a symmetric relation $R$ on $X$,
disjoint from the diagonal.  That is, $R = R^{-1}$ and
$R \cap 1_X = \emptyset$. There is an edge between $x$ and $y$ when $(x,y), (y,x) \in R$. The distance $d(x,y)$
between vertices $x, y$ of the graph is the smallest
$n$ such that there is a path $[x = x_0 , \dots, x_n = y]$.  We regard $x = x_0 = x$ as a path of length zero so that $d(x,x) = 0$.
A path $[x = x_0 , \dots, x_n = y]$ which achieves this minimum length is called a \emph{geodesic}\index{geodesic}
connecting $x$ and $y$.  It then follows
for $0 \leq i \leq j \leq n$ that $[x_i, \dots, x_j]$ is a geodesic connecting $x_i$ and $x_j$ and the distance $d(x_i,x_j) = j - i$.
We define the distance  $d(x,y)$ to be infinite if no such path exists. Thus, the
graph is connected when every distance is finite. If $d(x,y)$ is finite,  then $z \in X$ lies on a geodesic between $x$ and $y$ if and only if
$d(x,y) = d(x,z) + d(z,y)$.

Given a set $I$ with $|I| = p$ and a score vector $s$, Brualdi and Li \cite{BL} define the undirected
\emph{interchange graph}\index{interchange graph} $R$ on the set of tournaments on $I$ with
score vector $s$. The tournaments $\Gamma$ and $\Pi$ are connected by an
edge in the graph when each can be obtained from the other by reversing a single $3$-cycle.  That is, when $\Delta(\Gamma,\Pi) = \Delta(\Pi,\Gamma)^{-1}$
is a $3$-cycle. By Theorem  \ref{theo26} the graph is connected. Since the number of $3$-cycles in a tournament depends only on the score vector, see
(\ref{eqnewx}), the interchange graph is regular with $|R(\Gamma)|$ the number of $3$-cycles in $\Gamma$.

By Corollary \ref{cor29y}, the following conditions are equivalent for tournaments $\Gamma$ and $\Pi$ with the same scores.
\begin{itemize}
\item There exists a path of even length from $\Gamma$ to $\Pi$
\item Every path from $\Gamma$ to $\Pi$ has even length.
\item The distance $d(\Gamma,\Pi)$ is even.
\item $|\Delta(\Gamma,\Pi)|$ is even.
\item $\b(\Delta(\Gamma,\Pi))$ is even.
\end{itemize}

We then say that $\Gamma$ and $\Pi$ have the \emph{same parity}\index{parity}. Clearly, an edge in the graph always connects tournaments of opposite parity.
Using the two parity classes, we see , as was observed by Brualdi and Li, that the graph is  a \emph{bipartite graph}\index{bipartite graph}.

The score vector partitions $I$ into score value subsets. The group of score-preserving permutations is just the product of the permutation groups
of each score value subset. In particular, every score-preserving permutation is a product of disjoint score-preserving cycles and of score-preserving
transpositions. If $\r$ is a score-preserving permutation then $\bar{\r}(\Gamma)$ is another tournament with the same scores.
Recall that $\bar{\r} = \r \times \r$ on $I \times I$.  Thus, the
score-preserving permutations act on the interchange graph.

\begin{theo}\label{theoint00} If $\r$ is a permutation preserving the score vector of a tournament $\Gamma$ then the tournaments
$\Gamma$ and $\bar{\r}(\Gamma)$ have the same parity if and only if $\r$ is an even permutation,
i.e. a product of an even number of transpositions. \end{theo}

\begin{proof} It suffices to show that if $\r$ is a score-preserving transposition then $\Gamma$ and $\bar{\Gamma}$ have opposite parity.

Suppose that the transposition interchanges vertices $i$ and $j$ with $i \to j$ in $\Gamma$. Let $\Delta = \Delta(\Gamma,\bar{\r}(\Gamma))$.
The edge $(i,j) \in \Delta$ and if $r$ is a another vertex which is an input for one of $i$ and $j$ and an output for the other,
the the edges between $r$ and both $i$ and $j$ lie in $\Delta$ and these are the only edges in $\Delta$. Thus, $\Delta$ is the union of the
$3$-cycles  $\langle i, j, s \rangle $ and the 3-orders with  $i \to r \to j$. The inputs to $j$ in $\Delta$ are $i$ and the 3-order vertices $r$.
The outputs of $j$ in $\Delta$ are the cycle vertices $s$. Because $\Delta$ is Eulerian,
it follows that if $(i,j)$ is contained in $k+1$ $3$-cycles then there are $k$ 3-order
vertices $r$. Each decomposition for $\Delta$ consists of one $3$-cycle and $k$ 4-cycles of the form $\langle j, s, i, r \rangle$.
Hence, $\s(\Delta) = k +1$, $|\Delta| = 4k + 3$ and
$\b(\Delta) = 2k + 1$.  Since this is odd, $\Gamma$ and $\bar{\r}(\Gamma)$ have opposite parity.

\end{proof}

%Of course, for games with constant score vector the entire permutation group acts on the interchange graph. The orbit of a game $\Gamma$ under the
%action consists of those games which are isomorphic to $\Gamma$.

It is easy to check that reversing a $3$-cycle takes a strong tournament to a strong tournament.  In fact, whether the tournaments associated with
$s$ are strong or not is detectable by inequalities on the terms of $s$, see, e.g. Theorem 11.13 of \cite{HNC} or Theorem 9 of \cite{HM}.

Chen, Chang and Wang \cite{CCW} studied the interchange graph of the tournaments on $\{0, \dots, p-1 \}$ with score vector
$s = (1,1,2,3,\dots, p-3, p-2, p-2)$. From (\ref{eqnewx}) it follows that each game contains exactly $p - 2$ $3$-cycles. They proved that the
interchange graph is a \emph{hypercube} of dimension $p - 2$.

A hypercube of dimension $\ell$ is a graph on the set $\{ 0, 1 \}^{\ell}$ with $\ell$-tuples $x$ and $y$ connected by an edge if they differ in exactly one
place. If $x$ and $y$ differ in exactly $k$ places then the distance $d(x,y) = k$ and so the diameter of the hypercube is $\ell$.
There are exactly $k!$ geodesics connecting $x$ to $y$ when $d(x,y) = k$, each obtained
by choosing an ordering of the places on which the switches are successively made. In general, this is a lower bound for the number of geodesics.

\begin{theo}\label{theoint01} If $\Gamma$ and $\Pi$ are two tournaments with the same scores and $d(\Gamma, \Pi) = k$
then there are at least $k!$ geodesics
connecting $\Gamma$ and $\Pi$. \end{theo}

\begin{proof} What we must show is that there are $k$ distinct tournaments $\Gamma'$ such that $d(\Gamma, \Gamma') = 1$ and $d(\Gamma',\Pi) = k-1$.
For then, inductively, there are $(k - 1)!$ geodesics from each $\Gamma'$ to  $\Pi$ and each of these extends via the edge $(\Gamma, \Gamma')$ to a
geodesic from $\Gamma$ to $\Pi$ for a total of $k!$ geodesics of the latter type.

First we consider the case when $\Delta(\Gamma, \Pi)$ consists of a single cycle $C$ of length $p$ which we will label
$\langle 0, \dots, p-1 \rangle$ and use the integers mod $p$ as the labels. Then the distance $k$ is $p - 2$.  We must find
at least $p - 2$ distinct $3$-cycles in $\Gamma$ such that the reversal of each of which leads to a tournament $\Gamma'$ with $d(\Gamma',\Pi) = k-1$.
This is trivial if
$p = 3$ and so we assume $p \geq 4$.

For each vertex $i$ of $C$ we associate a $+$ if $i-1 \to i+1$ in $\Gamma$ and a $-$ if $i+1 \to i-1$ in $\Gamma$.
If $i$ has a $-$ then $\langle i-1, i, i+1 \rangle$ is a $3$-cycle in $\Gamma$ and reversing it leads to $\Gamma'$  with
 $\Delta(\Gamma', \Pi)$ the $p - 1$ cycle $\langle 0, \dots, i-1, i+1, \dots, p-1 \rangle$, omitting $i$. We will call this
 the \emph{near cycle}\index{near cycle}\index{cycle!near} associated with the $-$ at $i$.

 If $p = 4$ then two vertices have $-$ labels and so there are $p - 2 = 2$ near cycles. So we may assume $p \geq 5$.

For a vertex $i$ a \emph{far cycle}\index{far cycle}\index{cycle!far}  is $\langle i, j, j+1 \rangle$ in $\Gamma$ where
$i \to j$ and $j+1 \to i$ with $j \not= i+1$ and $j+1 \not= i-1$.
So this far cycle intersects $C$ only at the edge $(j,j+1)$.

As in the proof of Theorem  \ref{theo28x} (b), reversing a far cycle leads to $\Gamma'$ with $\Delta(\Gamma', \Pi)$ the union of two
 cycles with only the vertex $i$ in common and with $d(\Gamma',\Pi) = p-3$.

 We show that there are at least $p-2$ distinct $3$-cycles which are either near or far. We will call the near or far
 cycles the \emph{special cycles}\index{special cycle}\index{cycle!special}
 for $\langle 0, \dots, p-1 \rangle$ in the tournament which is the restriction  $\Gamma_C = \Gamma|\{ 0, \dots, p-1 \}$.

 For a vertex $i$ we will call the $\pm$ labels for the vertices $i-1, i, i+1$ the \emph{pattern}\index{pattern} for $i$.

 \begin{itemize}
 \item[(i) ]  If the pattern for $i$ is $+, \pm, +$, then $i+2$ is an output for $i$ and $i-2$ is an input for $i$ and so there is
 some $j$ between them with $j$ an output for $i$ and $j+1$ an input, leading to a far cycle for $i$.
 \vspace{.25cm}

 \item[(ii) ] If the pattern for $i$ is $+, \pm, -$, then $i+2$ and $i-2$ are both inputs for $i$. So there is a far cycle for $i$
 unless every $j \not= i-1,i,i+1$ is an input for $i$. In that case the $\Gamma_C$ score of $i$ is $1$.
  \vspace{.25cm}

 \item[(iii) ] If the pattern for $i$ is $-, \pm, +$ then $i+2$ and $i-2$ are both outputs for $i$. So there is a far cycle for $i$
 unless every $j \not= i-1,i,i+1$ is an output for $i$. In that case the $\Gamma_C$ score of $i$ is $p-1$.
  \vspace{.25cm}

 \item[(iv) ]  If the pattern for $i$ is $-, \pm, -$ then it may happen that there is no far cycle for $i$.
 \end{itemize} \vspace{.25cm}

 If every vertex has a $-$ then there are $p$ near cycles. If every vertex has a $+$ then from type (i) we see that every vertex has at least one
 far cycle and so there are at least $p$ far cycles.

 Now we assume that both $+$ and $-$ labels appear.

 For our preliminary estimate we neglect the possibility of scores $1$ or $p-1$ and assume that the patterns of types (i), (ii) and (iii) each
 lead to a far cycle and that type (iv) never does. Thus, every $-$ adjacent to a $+$ leads to both a near and a far cycle and every $+$ adjacent
 to a $+$ leads to a far cycle.

 For a run of $-$'s, each $-$ leads to a near cycle and at each end there is also a far cycle. So the count of cycles
 is the length of the run plus 1 and plus 1 more if the length of the run is greater than 1.

For a run of $+$'s, each $+$ leads to a far cycle unless the length of the run is 1. Thus, the count of the cycles is the length of the
run minus 1 and plus 1 if length of the run is greater than 1.

Notice that because we are on a cycle, the number of $+$ runs is equal to the number of $-$ runs.

Adding these up we obtain as our preliminary estimate the sum of the lengths of the runs, which is $p$, plus the number of runs of either sort which
are longer than 1.

Now we must correct for the scores $1$ and $p-1$.

Assume first that there is at most one vertex with score $1$ and at most one with score $p-1$. For each of these vertices we assumed there was a far
cycle where there need not be one.  Thus, we correct our preliminary estimate by subtracting $2$.  Thus, in this case there are at least $p - 2$
special cycles.

Now suppose the vertex $i$ has score $1$. Every vertex other than $i-1, i, i+1$ is an input to $i$ and so has at least two outputs.
If $i$ is associated with $-$ then $i+1$ has two outputs and so the only possibility for another vertex with score $1$ is $i - 1$.
Since the pattern for $i$ is $+, -, -$, the pattern for $i -1$ must be $+, +, -$ if it has score 1.  If $i$ is associated with $+$ then
only $i+1$ can have score $1$ in which case it has pattern $+, -, -$.   Observe that in either case, there is a run of $+$'s and a run of $-$'s of
length greater than one.

Similarly, if the vertex $i$ has score $p-1$ then only $i-1$ or $i+1$ could have score $p-1$ and which possibility could occur
depends on whether $i$ is associated with $-$ or $+$. When there are two vertices with score $p-1$ then again
there is a run of $+$'s and a run of $-$'s of length greater than one.

So if either there is more than one score $1$ vertex or more than on score $p-1$ vertex or both, our preliminary estimate is
at least $p + 2$ for the two long runs. We subtract at most $4$ to correct for the $4$ far cycles from types (ii) and (iii).
Thus again we have at least $p - 2$ special cycles.

Finally, notice that if the restriction $\Gamma_C = \Gamma|\{ 0, \dots, p-1 \}$ happens to have score vector $(1,1,\dots,p-1,p-1)$ then there
are only $p-2$ cycles in the restriction and these are the $p-2$ special cycles.

Now we return to the general case and suppose that $C = C_1$ with $C_1, C_2, \dots, C_{\ell}$ a maximum decomposition of
$\Delta(\Gamma,\Pi)$. Define $\widehat{\Gamma}$ by reversing $C$ in $\Gamma$. Thus, $\Delta(\widehat{\Gamma},\Pi) = \Delta(\Gamma,\Pi) \setminus C$
and $C_2,\dots,C_{\ell}$ is a maximum decomposition for $\Delta(\widehat{\Gamma},\Pi)$. Hence, $d(\Gamma, \widehat{\Gamma}) = p-2$
and $d(\widehat{\Gamma},\Pi) = d(\Gamma,\Pi) - (p-2)$. If we reverse a special cycle of $C$ then we obtain $\Gamma'$ with
$d(\Gamma,\Gamma') = 1$ and $d(\Gamma',\widehat{\Gamma}) = p-3$. Hence, $d(\Gamma',\Pi) = d(\Gamma, \Pi) - 1$.

Since we can rearrange the $C_r$'s it follows that by reversing any special cycle in any of the $C_r$ leads to a $\Gamma'$ of the required sort.
Furthermore, for the cycle $C_r$ there are at least $|C_r| - 2$ special cycles and the sum $\sum_{r=1}^{\ell} \ |C_r| - 2 =
\b(\Delta(\Gamma,\Pi)) = d(\Gamma,\Pi)$ which is what we want.

There is, however, a final problem. A $3$-cycle may occur as a special cycle for two different $C_r$'s leading to double counting.

To cure this, we choose our maximum decomposition with care.

Recall from the Claim in the proof of Theorem  \ref{theo28x} (b) that
if we choose the decomposition so that $C = C_1$ is the cycle of shortest length which occurs
in any maximum decomposition, then for vertices $i,j$ of $C_1$ with $i \to j$ the edge $(i,j)$ does not occur in $\Delta(\Gamma,\Pi) \setminus C_1$.
That is, the tournament $\Gamma_C$ is disjoint from $\Delta(\Gamma,\Pi) \setminus C_1$ and so $C_1$ and $\Delta(\Gamma,\Pi) \setminus C_1$
have at most a single vertex in common. In particular, no special cycle of
$C_1$ intersects $\Delta(\Gamma,\Pi) \setminus C_1$.

Inductively, we choose the decomposition
so that $C_r$ is the shortest cycle which occurs in any maximum decomposition of $\Delta(\Gamma,\Pi) \setminus (C_1 \cup \dots \cup C_{r-1})$. We thus
obtain a maximum decomposition such that no special cycle of $C_r$ intersects any $C_s$ for $s > r$. Thus, for this decomposition
the $d(\Gamma,\Pi)$ special cycles are all distinct.

\end{proof}  \vspace{.5cm}

From the above proof we see that for the cycle $C = \langle  0, \dots, p-1 \rangle$ if the tournament $\Gamma|\{ 0, \dots, p-1 \}$
has no vertices of score $1$ or $p-1$ then there are at least $p$ special cycles. In particular, this applies if $p = 2k + 1$ and
$\Gamma|\{ 0, \dots, 2k \}$ is a game.

\begin{ex}\label{exint02} With $k \geq 3$, let $A$ be a game subset of $\Z_{2k+1}$. Assume that
$\Gamma|\{ 0, \dots, 2k \}$ is the group game $\Gamma[A]$ and assume that
$1 \in A$ so that the cycle $C = \langle  0, \dots, 2k \rangle$ is contained in $\Gamma[A]$. Since $\Z_{2k+1}$ acts transitively on
$\Gamma[A]$ and preserves the cycle $C$, every vertex has the same number of near and far cycles. In particular, every vertex
is associated with $+$ if and only if $2 \in A$.\end{ex}

{\bfseries Example (a)} [$ A = Odd_k = \{ 1,3,\dots, 2k-1 \}$] The vertex $0$ is associated with $-$ and so has a near cycle. The far cycles
associated with $0$ are $\langle 0, j, j+1 \rangle$ for $j$ an odd number with $3 \leq j \leq 2k-3$.  That is, there are $k - 2$ far
cycles for a total of $k - 1$ cycles for each vertex. Thus, the number of special cycles is $(2k + 1)(k - 1)$.

{\bfseries Example (b)} [$A = \{1, 2, 4, \dots, 2k-2\}$] The vertex $0$ is associated with $+$ and so  has no near cycle. The far cycles
associated with $0$ are $\langle 0, p, p+1 \rangle$ for $p$ an even number with $2 \leq p \leq 2k-2$. That is, there are
$k - 1$ far cycles and so the number of special cycles is again $(2k + 1)(k - 1)$.

{\bfseries Example (c)} [$A = [1, k]$] The vertex $0$ is associated with $+$ and so  has no near cycle. The unique far cycle
associated with $0$ is $\langle 0, k, k+1 \rangle$ and so the number of special cycles is $2k + 1$. Notice that the games in (a) and (c) are
isomorphic via the multiplication map $m_k$.  However, $m_k$ maps the cycle $C$ to a different Hamiltonian cycle.

{\bfseries Example (d)} [$A = \{1, k+1, \dots, 2k-1 \}$] The vertex $0$ is associated with $-$ and so  has a near cycle. There are no far cycles
and so the number of special cycles is again $2k + 1$.

$\Box$  \vspace{.5cm}

\begin{theo}\label{theoint03} Let $C = \langle 0, \dots, p-1 \rangle$ be a cycle in a tournament $\Gamma$ with $p \geq 4$.
If $p = 2k + 1$, then the number of special cycles for $C$ is at most $(2k + 1)(k - 1)$. If $p = 2k$, then the
number of special cycles for $C$ is at most $k(2k - 3)$ and the inequality is strict unless $k \equiv 2$ mod $4$. \end{theo}

\begin{proof} We use the notation of the proof of Theorem \ref{theoint01}. In the following table we describe for each
pattern for a vertex $i$ the largest possible number of far cycles associated with a vertex. A far cycle is of the
form $\langle i, j, j+1 \rangle$ where $j \not= i-1,i, i+1$ and so $i \to j, j+1 \to i$.
\vspace{.5cm}

$\begin{array}{cccc}
\text{Vertex Pattern} &  \text{Far Cycle Max} & \text{if \ } p \ = \ 2k & \text{if \ } p \ = \ 2k + 1 \\ [0.5ex]
 + \ \pm \ +  &  [(p - 3)/2] & k - 2 & k - 1\\
 + \ \pm \ -  &  [(p - 4)/2] & k - 2 & k - 2\\
 - \ \pm \ +  &  [(p - 4)/2] & k - 2 & k - 2\\
 - \ \pm \ -  &  [(p - 5)/2] & k - 3 & k - 2
 \end{array}$
 \vspace{.5cm}

 It follows that if either every vertex is associated with a $+$ or every vertex is associated with a $-$,
 then the upper bound of the number of cycles is $2k(k - 2)$ when $p = 2k$ and is $(2k + 1)(k - 1)$ when $p = 2k +1$.
  Notice that this case cannot occur when $2k = 4$.

 Now assume that both $+$'s and $-$'s occur.

 Let $p = 2k + 1$. For a run of $+$'s every vertex may have $k - 1$ far cycles except at the ends and so the number of cycles
 is $(k - 1)$ times the length of the run minus $1$ and minus $1$ more if the length of the run is greater than 1.  For a run of $-$'s
 the bound on number of cycles (near and far) is $(k - 1)$ times the length of the run plus $1$ and minus $1$
 if the length of the run is greater than 1. The number of $+$ runs equals the number of $-$ runs and so
  the total upper bound on the number of cycles is $(2k + 1)(k - 1)$ minus the number of runs of length greater than one.
  Since $p$ is odd, there is at least
  one run of length greater than one.

 Let $p = 2k$. For a run of $+$'s every vertex may have $k - 2$ far cycles unless the run is a singleton in which case the maximum is
 $k - 3$. Thus, the bound is $(k - 2)$ times the length of the run minus $1$ and plus $1$ if the length of the run is greater than 1.
 For a run of $-$'s the bound on number of cycles (near and far) is $(k - 2)$ times the length of the run plus $1$ and plus $1$ more
 if the length of the run is greater than 1. Again the number of $+$ runs equals the number of $-$ runs and so
  the total upper bound on the number of cycles is $2k(k - 2)$ plus the number of runs of length greater than one. Since such a long run has
  at least two elements there are at most $k$ such runs.  Thus, the upper bound is $2k (k - 2) + k = k(2k - 3)$. Since the number of the
  two sorts of runs are equal, there are fewer than $k$ such runs if $k$ is odd. Thus, the bound is strict unless $k \equiv 0$ or $2$ mod $4$.
  With exactly $k$ long runs the pattern on the cycle consists of pairs of $+$'s alternating with pairs of $-$'s. Now suppose that $k \equiv 0$
  mod $4$. By relabeling, we may assume that associated with $0, 1, 2, 3$ are $-, - , +, +$. It follows that $0$ has pattern $+ - -$. Since
  $k$ is divisible by $4$ the vertex at $k$ will have the same pattern. In order to have $k - 2$ far cycles associated with vertex $0$,
   it must happen that the odd vertices (other than $2k - 1$) are outputs of $0$ and the even vertices are inputs of $0$.
   In particular, $k \to 0$. Similarly, in order to have $k - 2$ far cycles associated with vertex $k$,
   it must happen that the odd vertices (other than $k - 1$) are outputs of $k$ and the even vertices are inputs of $k$. In particular,
   $0 \to k$.  But for a tournament, it cannot happen that both $0 \to k$ and $k \to 0$. Hence, the inequality is strict when
   $k \equiv 0$ mod $4$.

 \end{proof}

\begin{ex}\label{exint04} The inequality is achieved when $k \equiv 2$ mod $4$. \end{ex}

Define the digraph on $\Z$ by
\begin{equation}\label{int01}
\begin{cases} \hspace{3cm} t \to t+1 \\
t \to t + 2i, \ t + (2i+1) \to t \quad  \text{for all} \ \ i \geq 1 \ \text{and} \ t \equiv 0, 1  \text{\ mod  }4 \\
t + 2i \to t, \ t \to t + (2i+1)\quad  \text{for all} \ \ i \geq 1 \ \text{and} \ t \equiv 2, 3  \text{\ mod  }4 \end{cases}
\end{equation}
If $k > 2$ and $k \equiv 2$ mod $4$, then this induces a tournament on $\Z_{2k}$ for which translation by $4$ is an automorphism.
Assume that for the cycle $\langle 0, \dots, p-1 \rangle$ with $p = 2k$ $\Gamma|\{ 0, \dots, p-1 \}$
 is this tournament. We need only examine the vertices $0, 1, 2, 3$. It is easy to see that that there are $2k (k - 2) + k = k(2k - 3)$ special cycles.

$\Box$  \vspace{.5cm}

Now we focus on games on $I$ with $|I| = 2n + 1$. In that case, the score vector is $(n, n, \dots, n)$ and the entire permutation group on $I$ acts
on the interchange graph. The orbit of a game $\Gamma$ under the action consists of those games which are isomorphic to $\Gamma$. Those permutations
which map $\Gamma$ to itself are precisely the automorphisms of $\Gamma$.
%
%Following Dixon \cite{D} we see that the order of $Aut(\Gamma)$ is at most
%$3^n$. Hence, there are at least $(2n + 1)!/3^n$ distinct games on $I$ which are isomorphic to $\Gamma$.

\begin{lem}\label{lemint05} Let $\Pi$ be a game of size $2n - 1$ on the set of vertices $J$ and $K \subset J$ with $|J| = 2n - 1$ and $|K| = n$.
If $\Gamma$ is the extension of $\Pi$ via $u \to v$ and $K$, so that $\Gamma$ is a game  on $I = J \cup \{ u,v \}$,
then
\begin{equation}\label{int02}
\b(\Gamma) \ \leq \ \b(\Pi) + 2n -1.
\end{equation}
\end{lem}

\begin{proof} Let $K = \{ a_0,\dots,a_{n-1} \}, J \setminus K = \{b_1, \dots, b_{n-1} \}$.
Given a decomposition for $\Pi$ we build a decomposition for $\Gamma$.

For $r = 1, \dots, n-1$ if $a_r \to b_r$ in $\Pi$, then replace the edge $(a_r,b_r)$
 by $(a_r,u), (u,b_r)$ to get a cycle in $\Gamma$ with the length increased by $1$. In addition, define the
$3$-cycle $D_r = \langle a_r, b_r, v \rangle$. If, instead, $b_r \to a_r$ in $\Pi$, then replace
the edge $(b_r,a_r)$ by $(b_r,v), (v,a_r)$ to get a cycle in $\Gamma$ with the length increased by $1$. In addition, define the
$3$-cycle $D_r = \langle b_r, a_r, u \rangle$.  Finally, define the $3$-cycle $\langle a_0, u, v \rangle$.

Notice that a single cycle $C$ may contain edges $(a_r,b_r)$ or $(b_r,a_r)$ for more than one $r$. Thus, instead of two or more cycles each
extended in length by $1$ we have a single cycle with the length extended by two or more.

If the original decomposition for $\Pi$ was maximum with cardinality $k$ then  $\b(\Pi) = (n - 1)(2n - 1) - 2k$. The decomposition
we have constructed consists of the $k$ extended cycles and $n$ new $3$-cycles. Thus, the span of $\Gamma$ is at least $k + n$.
Hence, $\b(\Gamma) \leq n(2n + 1) - 2(k + n) = (n - 1)(2n - 1) - 2k + (2n - 1)$.

\end{proof}

{\bfseries Remark:} Equality holds in (\ref{int02}) if and only if the span of
$\Gamma$ is equal to $k + n$ and so if and only if the decomposition for $\Gamma$
constructed above from a maximum decomposition of $\Pi$ is maximum for $\Gamma$.
\vspace{.5cm}

\begin{theo}\label{theoint06} Let $A$ be the game subset $[1,n] \subset \Z_{2n+1}$. For the associated group game $\Gamma[A]$,
\begin{equation}\label{int03}
d(\Gamma[A],\Gamma[A]^{-1}) \ = \ \b(\Gamma[A]) \ = \ n^2.
\end{equation}
\end{theo}

\begin{proof} On the cycle $C = \langle 0, 1, \dots, 2n \rangle$ define the group of 1-chains on $C$ to be the free abelian
group generated by the edges of $C$ so that an element is a formal sum
$\xi = \sum_{r=0}^{2n-1} \ m_r (r, r+1)$ with $m_r \in \Z$. The group 0-chains is the free abelian group on the vertices and
the boundary map from the 1-chains to the 0-chains is given by $\partial (r, r+1) = 1(r+1) - 1(r)$. Clearly, the boundary of a 1-chain
is $0$ if and only if the chain is a constant multiple of $\sum_{r=0}^{2n-1} \ 1 (r, r+1)$.

Let $\Gamma$ be any tournament on $\Z_{2n+1}$. Each edge of $\Gamma$ can be written uniquely as $(t, t+s)$ with
$t = 0, \dots, 2n, s = 1, \dots, 2n$ and with addition mod $2n+1$. Define the associated chain $\xi(t,t+s) = \sum_{r=0}^{s-1} 1(t+r,t+r+1)$.
If $Q$ is a subgraph of $\Gamma$, then the chain $\xi(Q)$ is the sum of $\xi$ applied to the edges of $Q$.  Observe first that the
coefficients of $\xi(Q)$ are all non-negative. Next, if $Q_1$ and $Q_2$ are disjoint subgraphs, then
$\xi(Q_1 \cup Q_2) = \xi(Q_1) + \xi(Q_2)$. Furthermore, if $Q$ is a cycle, then the boundary of $\xi(Q)$ is $0$ and so
$\xi(Q) = m(Q) \cdot \sum_{r=0}^{2n-1}  (r, r+1)$ with $m(Q) > 0$. The number $m(Q)$ is the number of times the cycle $Q$ wraps around $C$.
If $\Pi$ is an Eulerian subgraph, then since it is a disjoint union of cycles, it follows that there exist a positive integer $m = m(\Pi)$
such that $\xi(\Pi) = m(\Pi) \cdot \sum_{r=0}^{2n-1} \ (r, r+1)$. Furthermore, for any decomposition of $\Pi$ by disjoint cycles
$Q_1, \dots, Q_k$, $\xi(\Pi) = \xi(Q_1) + \dots + \xi(Q_k)$ and so $m(\Pi) = m(Q_1) + \dots + m(Q_k)$.
Since each $m(Q_t) \geq 1$, it follows that  $k \leq m(\Pi)$ and so the
span $\s(\Pi)$ is bounded by $m(\Pi)$.

In general, this estimate too crude to be of much use. For example, if $1 \to 0$ in $\Gamma$ then $\xi(1, 0) = \xi(1, 1 + (2n))$
covers the entire cycle except for $(0,1)$. But for $\Gamma[A]$ it gives us what we need.

To compute $m(\Gamma[A])$ we count the  edges $(r, r+s)$ such that $\xi(r, r+s)$ has a coefficient of $1$ on $(0,1)$.
These are $(0,1), (0,2),  \dots $ with $s$ translates of the edge $(0,s)$ hitting $(0,1)$ for $s = 1,\dots, n$. Hence,
$m(\Gamma[A]) = 1 + 2 + \dots + n = n(n+1)/2$. Thus, $\b(\Gamma[A]) \geq n(2n+1) - n(n+1) = n^2$.

Now we prove that $\b(\Gamma[A]) \leq  n^2$ by induction on $n$. This is trivial for $n = 1$.

Now let $\Pi = \Gamma[[1,n-1]]$ on $J = \Z_{2n-1}$ and let $K = [0,n-1]$ and let $\Gamma$ be the extension of $\Pi$
via $u \to v$ and $K$. Define a map by
\begin{equation}\label{int04}
\begin{cases}  i \mapsto i \quad \text{for} \ i = 0, \dots, n-1,\\
\hspace{1cm} u \mapsto n, \\
i \mapsto i+1 \ \text{for} \ i = n, \dots, 2n-2, \\
\hspace{1cm}v \mapsto 2n. \end{cases}
\end{equation}
This is an isomorphism from the extension $\Gamma$ of $\Gamma[[1,n-1]]$ onto $\Gamma[[1,n]]$.

By induction hypothesis and (\ref{int02})
\begin{equation}\label{int05}
\b(\Gamma[[1,n]]) \ \leq \ \b(\Gamma[[1,n-1]]) + 2n - 1 \ \leq \ (n - 1)^2 + 2n - 1 \ = \ n^2.
\end{equation}

Finally, since $\Gamma = \Delta(\Gamma,\Gamma^{-1})$ it follows that $\b(\Gamma)$ is the distance from a game $\Gamma$ to its inverse.

\end{proof}

{\bfseries Remark:} In this case equality holds in (\ref{int05}) and so by the Remark after Lemma \ref{lemint05} the decompositions
constructed from a maximum decomposition of $\Gamma[[1,n-1]]$ are maximum decompositions for $\Gamma[[1,n]]$.

The game $\Gamma = \Gamma[[1,n-1]]$ may admit smaller decompositions. For example, if $2n + 1$ is prime, then
$\{ \langle 0, j, 2j, \dots, 2nj \rangle : j = 1,\dots,n \}$ is a decomposition of $\Gamma$ by $n$ cycles, each of length $2n + 1$.
\vspace{.5cm}

I conjecture that for any game $\Gamma$ on $I$ with $|I| = 2n + 1$, the distance $d(\Gamma,\Gamma^{-1}) \leq n^2$. In fact, I suspect
that for any pair of games $\Gamma, \Pi$ on $I$, $d(\Gamma,\Pi) \leq n^2$, i.e. the diameter of the interchange graph is
$n^2$. Alon, McDiarmid and Molloy conjecture in \cite{AMM} that any $k$-regular digraph contains a set of $k(k+1)/2$ disjoint cycles.
Since a game on $I$ is $n$-regular their conjecture would imply that $d(\Gamma,\Gamma^{-1}) \leq n^2$. In addition, if
the Eulerian digraph $\Delta(\Gamma,\Pi)$ happens to be $k$-regular then it contains at most $k(2n + 1)$ edges and, if their conjecture is
true, its span is at least $k(k+1)/2$. So $ d(\Gamma,\Pi) = \b(\Delta(\Gamma,\Pi))$ is bounded by $k(2n + 1) - k(k + 1) = 2nk - k^2 = n^2 - (n - k)^2$.
The diameter conjecture would require $\b(\Delta) \leq n^2$ for every Eulerian graph on at most $2n +1$ vertices.

The diameter is certainly bounded by $n(2n - 1)$ because for any Eulerian graph $\Delta$ on $I$, $\b(\Delta) \leq |\Delta| \cdot \frac{2n - 1}{2n + 1}$.
This follows because $l_1 + \dots + l_{\s} = |\Delta|$ and each $l_i \leq 2n + 1$, where $\s$ is the span of $\Delta$ and $l_1, \dots, l_{\s}$ are
the lengths of the cycles in some maximum decomposition.  It follows that $\s \geq |\Delta|/(2n +1)$ and so $\b = |\Delta| - 2 \s$ is at most
$|\Delta| \cdot ( 1 - \frac{2}{2n + 1})$.

If $n \equiv 0, 1$ mod $3$, then there exist Steiner games $\Gamma$ on $I$ for which the maximum decomposition for
$\Gamma$ consists of $n(2n + 1)/3$ $3$-cycles. Thus, for a Steiner game any maximum decomposition consists of $n(2n + 1)/3$ cycles
and so they all must be $3$-cycles.
Hence, $d(\Gamma,\Gamma^{-1}) = \b(\Gamma) = n(2n + 1)/3$. Since
$d(\Gamma[[1,\dots,n]],\Gamma[[1,\dots,n]]^{-1}) = n^2$ it is clear that $\Gamma[[1,\dots,n]]$ is not a Steiner game.
Because the interchange graph is connected, there exists a Steiner game $\Gamma$ and a $3$-cycle $C$ in $\Gamma$
such that the game $\Gamma' = \Gamma/C$ with just $C$ reversed is not a Steiner game. It follows that $C$ cannot be an element of
any maximum decomposition for $\Gamma$ because from such a maximum decomposition we would obtain a $3$-cycle decomposition for $\Gamma'$.
From Theorem \ref{theo28x} (c) applied with $\Pi = \Gamma^{-1}$ we see that
\begin{equation}\label{int06}
d(\Gamma',\Gamma^{-1}) \ = \ 1 + d(\Gamma,\Gamma^{-1}) \ = \ 1 + n(2n + 1)/3.
\end{equation}
That is, in the interchange graph, $\Gamma'$ is farther away from $\Gamma^{-1}$ than is its inverse $\Gamma$.

From the proof of Theorem \ref{theoint06} we see that the game $\Gamma[[1,n]]$ on $\Z_{2n+1}$ is reducible to
$\Gamma[[1,n-1]]$ on $\Z_{2n-1}$.
Thus, $\Gamma[[1,n]]$ is completely reducible, where

\begin{df}\label{defint07} A game $\Gamma$ on a set  $I$ with $|I| = 2n + 1$  is
\emph{completely reducible}\index{completely reducible}\index{game!completely reducible} when
there is a sequence of subsets $I_1 \subset I_2 \dots \subset I_n = I$ with $|I_k| = 2k + 1$ such that
the restriction $\Gamma|I_k$ is a game on $I_k$. \end{df}
\vspace{.5cm}

\begin{prop}\label{propint08} If $\Gamma$ is a completely reducible game of size $2n + 1$ then
\begin{equation}\label{int07}
d(\Gamma,\Gamma^{-1}) \ = \\b(\Gamma) \ \leq \ n^2.
\end{equation}
\end{prop}

\begin{proof} By induction on $n$ using
Lemma \ref{lemint05}.

\end{proof}

\vspace{1cm}

\section{The Double Construction and \\ the Lexicographic Product}\label{secdoublelex}

If $\Pi$ is a tournament on a set $J$ with $|J| = n$, then we define
the \emph{double}\index{double} of $\Pi$ to be the game $2\Pi$\index{$2\Pi$} on
$I = \{ 0 \} \cup J \times \{ -1, +1 \}$. With
\begin{align}\label{eq12}
\begin{split}
2\Pi(0) \ &= \ J \times \{ -1 \}, \quad (2\Pi)^{-1}(0) \ = \ J \times \{ +1 \}, \\
2\Pi(i+) \ &= \ \Pi(i) \times \{ +1 \} \cup \Pi^{-1}(i) \times \{ -1 \} \cup \{ 0 \}, \\
2\Pi(i-) \ &= \ \Pi(i) \times \{ -1 \} \cup \Pi^{-1}(i) \times \{ +1 \} \cup \{ i+ \},
\end{split}
\end{align}
where we will write $i\pm$ for $(i, \pm 1)$.

That is, if $i \to j$ in $\Pi$, then in $2 \Pi$
\begin{align}\label{eq13}
\begin{split}
i- \to j-, \quad i+ \to j+, \\
j+ \to i-, \quad j- \to i+.
\end{split}
\end{align}

In addition, for every $i \in J$,
\begin{equation}\label{eq13aa}
i+ \to 0 \to i- \  \text{and} \  \ i- \to i+.
\end{equation}

In passing we note the following consequence of this construction.

\begin{prop}\label{prop29aa} If $\Pi$ is a digraph with $n$ vertices, then $\Pi$ is a subgraph of
a game of size $2n + 1$. \end{prop}

\begin{proof} It is clear that $\Pi$ can be included as a subgraph of some tournament $\Pi_1$ on
the set $J$ of the vertices of $\Pi$.  So, up to isomorphism, $\Pi$ is a subgraph of the game $2 \Pi_1$.

\end{proof} \vspace{.5cm}

It is easy to check that $2\Pi$ is reducible via each pair $i- \to i+$. It reduces to the double of
the restriction of $\Pi|(J \setminus \{ i \})$. In fact, if $J_1 \subset J$ is nonempty and $\Pi_1$ is the restriction
$\Pi|J_1$ then $2 \Pi_1 = 2 \Pi|(\{ 0 \} \cup J_1 \times \{ -1, +1 \})$ is a subgame of $2 \Pi$.

Thus, a double is completely reducible. It follows from Proposition \ref{propint08} that $d(2\Pi,(2\Pi)^{-1}) \leq n^2$.
To see this directly, observe that if $i \to j$ in $\Pi$ then $\langle i-,j-,i+,j+ \rangle$ is a 4-cycle and
$\langle i+,0,i- \rangle$ is $3$-cycle, both in $2\Pi$. Thus, we obtain a decomposition of $2\Pi$ with
$\frac{n(n-1)}{2} + n =  \frac{n(n+1)}{2}$ cycles.

\begin{lem}\label{lem29aab} Let $\Pi$ be a tournament on  $J$  with $i \to j$ in $\Pi$.
\begin{itemize}
\item[(a)] $2\Pi$ is reducible via $i- \to u$ only for $u = i+$.

\item[(b)] $2\Pi$ is reducible via $0 \to j-$ if and only if $\Pi(j) = \emptyset$, i.e. $j$ has score $0$ in $\Pi$.

\item[(c)] $2\Pi$ is reducible via $i+ \to 0$ if and only if $\Pi^{-1}(i) = \emptyset$, i.e. $i$ has score $n-1$ in $\Pi$.

\item[(d)] $2\Pi$ is not reducible via $i- \to j-$ or via $i+ \to j+$.

\item[(e)] $2\Pi$ is reducible  via $j+ \to i-$ if and only if for all $k \in J \setminus \{i,j \}$
either $i, j \to k$ or $k \to i, j$.
\end{itemize}\end{lem}

\begin{proof} (a) By  Proposition \ref{prop07} (f), $i+$ is the only vertex $u$ of $2 \Pi$ such that $2\Pi$ is reducible via each pair $i- \to u$.

The remaining results all use Proposition \ref{prop07} (d).

(b), (c) If $j \to k$ then $j- \to k-$ and $0 \to k-$ and if $k \to i$ then $k+ \to i+$ and $k+ \to 0$. So reducibility fails in these cases.
The converse is easy to check.

(d) Reducibility fails because $0 \to i-, j-$ and $i+, j+ \to 0$.

(e) If, for example, $j \to k \to i$, then  $j+, i- \to k+$.

\end{proof} \vspace{.5cm}

\begin{prop}\label{prop29bbbb} If $\Pi$ is a game on $I$, then $dom(2\Pi) = \{ (i-,i+) : i \in I \}$. \end{prop}

\begin{proof} By Proposition \ref{prop07}  every edge of $\Pi$ is contained in a $3$-cycle of $\Pi$.
It follows from Lemma \ref{lem29aab} that $2\Pi$ is not reducible
via any pair other than some  $i- \to i+$.

\end{proof} \vspace{.5cm}

\begin{ex}\label{ex29ac} Let $\Pi$ be the tournament on $\{1, 2, 3, 4 \}$ given by $\langle 2, 3, 4 \rangle \ \cup \ (\{ 1 \} \times \{2, 3, 4 \})$.
Let $\Gamma = 2\Pi$ and $\bar \Gamma = \Gamma/\langle 2+, 3+, 4+ \rangle $.  The domination graphs are given by the paths:
\begin{align}\label{13ac}
\begin{split}
dom(\bar \Gamma) \ = \ \{ [1-, &1+, 0 ]\}, \ \text{and} \\  dom(\Gamma) \ = \ dom(\bar \Gamma)  \cup \{ [k-, &k+] : k = 2,3,4 \}.
\end{split}\end{align}
\end{ex}

$\Box$ \vspace{.5cm}

\begin{ex}\label{ex29ab} If $J = [1,n]$ and $i \to j$ if and only if $i < j$, then  $0 \mapsto 0, i- \mapsto i$
and $i+ \mapsto i + n$ is an isomorphism
from $2\Pi$ to $\Gamma[[1,n]]$. \end{ex}

Since the $\phi(2n+1)$ isomorphs of $\Gamma[[1,n]]$, or,
equivalently, the isomorphs of $\Gamma[Odd_n]$, are the
only reducible group games, see Theorem \ref{theo28},  they are the only group games which can be expressed as doubles.

$\Box$ \vspace{.5cm}

We let $\Pi_+$\index{$\Pi_+$} and $\Pi_-$\index{$\Pi_-$} denote the restrictions
of $2 \Pi$ to $J \times \{ +1 \}$ and to  $J \times \{ -1 \}$, respectively. Of course, each of these subgraphs is isomorphic to
$\Pi$.  In addition, we define $X(2 \Pi)$\index{$X(2 \Pi)$} to be
\begin{equation}\label{eq14}
\{ (j+,i-) : (i,j) \in \Pi \} \ \cup \  \{ (j-,i+) : (i,j) \in \Pi \}.
\end{equation}
%\vspace{.5cm}

If $\g : \Pi \tto \Pi_1$ is an isomorphism, then we define the isomorphism $2\g : 2\Pi \tto 2\Pi_1$ by
\begin{equation}\label{eq14aa}
2\g(0) \ = \ 0, \ \text{and} \quad   2\g(j\pm) \ = \ \g(j)\pm \ \ \text{for} \ j \in J.
\end{equation}
The map $\g \mapsto 2\g$ defines an injective group homomorphism  $2: Aut(\Pi) \to Aut(2\Pi)$.
%In some cases, it is an isomorphism.

\begin{prop}\label{prop30} Let $\Pi$ be a tournament on a set $J$ with $|J| = n$
\begin{itemize}
\item[(a)] If $\Pi_1$ is a tournament and $\r : 2 \Pi \tto 2 \Pi_1$ is an isomorphism
such that $\r(0) = 0$, then there exists a unique
isomorphism $\gamma : \Pi \tto \Pi_1$ such that $\rho = 2\gamma$.

\item[(b)] Assume that for $\Pi$
no $i \in J$ has score $0$ or $n-1$. If
 $\Pi_1$ is a tournament and $\r : 2 \Pi \tto 2 \Pi_1$ is an isomorphism
then $\r(0) = 0$.  In particular,
 $2: Aut(\Pi) \to Aut(2 \Pi)$ is an isomorphism.
 \end{itemize}\end{prop}

\begin{proof} (a) If $\r(0) = 0$, then $\r$ restricts to bijection  from  $(2 \Pi)(0) = J \times {-1}$
to $J_1 \times {-1}$, yielding an isomorphism from $\Pi_-$ to $\Pi_{1-}$. Thus, there is an
isomorphism $\gamma : \Pi \to \Pi_1$ such that $\r(i-) = \gamma(i)-$ for all $i \in J$.

For each $i \in J$, $2 \Pi_1$ is reducible via $\gamma(i)- = \r(i-) \to \r(i+)$.
On the other hand, $2 \Pi_1$ is
reducible via $\gamma(i)-  \to \gamma(i)+$. From Proposition \ref{prop07} (f)
it follows that $\r(i+) = \gamma(i)+$. Thus, $\r = 2\g$.

(b) Every vertex in $J$ has an input and an output, and so $2 \Pi$ is not
reducible via any pair which contains $0$. On the other hand,
$2 \Pi$ is reducible via $i- \to i+$ for all $i \in J$. Since  $2 \Pi_1$ is
reducible via $j- \to j+$ for all $j \in J_1$ it follows that $\r(0) = 0$.

\end{proof} \vspace{.5cm}

We will see below in Theorem \ref{theoiso03} that  $2: Aut(\Pi) \tto Aut(2\Pi)$ is almost always an isomorphism.

In the special case when $n$ is  odd and $\Pi$ is itself a game, we note that $\Pi_+$ and $\Pi_-$ are subgames of $2 \Pi$ and so
are Eulerian subgraphs. In addition, in this case, $X(2 \Pi)$, defined in (\ref{eq14}), is an Eulerian subgraph. Thus, we can obtain
additional examples, by reversing one or more of the three disjoint Eulerian subgraphs $\Pi_+$, $\Pi_-$ and $X(2 \Pi)$.

\begin{theo}\label{theo30aa} If $\Pi$ is a Steiner game, and $\Gamma = (2 \Pi)/\Delta$ with $\Delta$ equal to
$\Pi_+$, $\Pi_-$, $X(2 \Pi)$, or a union of any two of these, then $\Gamma$ is a Steiner game. \end{theo}

\begin{proof}  We will do the case with $\Delta = \Pi_+$ as the others are similar. Notice that if $i \to j$ in $\Pi$ then
$j- \to i+, j+ \to i-$ and also $j+ \to i+$ in $\Gamma$.  It is this coherence which is the basis of the construction.

Assume that $\langle k, j, i \rangle$ is one of the $3$-cycles in a maximum decomposition for $\Pi$. Associated to it we use the following
$3$-cycles in $\Gamma$
\begin{align}\label{cyc01}
\begin{split}
\langle i+, j+, k- \rangle,   &\langle i-, j+, k+ \rangle, \langle i+, j-, k+ \rangle, \\
\langle i+, 0, i-  \rangle, \  &\langle j+, 0, j-  \rangle, \ \langle k+, 0, k-  \rangle, \\
 &\langle k-, j-, i- \rangle.
\end{split}
\end{align}
Notice that if the vertex $i$ occurs in two $3$-cycles of the decomposition, the vertical edge $(i-,i+)$ occurs in the same $3$-cycle
$\langle i+, 0, i-  \rangle$ associated with both of the decomposition $3$-cycles.

Thus, we obtain a decomposition of $\Gamma$ by $3$-cycles.

\end{proof}

 Using the doubling construction we can build some interesting examples.

 Call a tournament $\Pi$ \emph{rigid}\index{rigid tournament}\index{tournament!rigid} when the identity is the
 only automorphism, i.e. the automorphism group is trivial.

 \begin{lem}\label{lemrigid} If $\Pi$ is a tournament with score vector $(s_1,\dots,s_p)$ and for every
 $k \in \N$, $s_i = k$ for at most two distinct vertices $i$, then $\Pi$ is rigid. \end{lem}

 \begin{proof} For any $k$, the set $\{ i : s_i = k \}$ is invariant for any automorphism $\r$ of $\Pi$.
 So $\r$ fixes the vertices with a unique score value. If $\{ i : s_i = k \}$ consists of exactly two vertices, then
 $\r$ fixes each of them by  Proposition \ref{prop04}.  Hence, $\r$ is the identity.

 \end{proof} \vspace{.5cm}

 For example,  if $\Pi$ has
 score vector $(1,1,2,\dots,p-3,p-2,p-2)$, then it is rigid.

 %Because of Theorem \ref{theo12} we are interested in the case when translations are the only automorphisms of a group game.

 \begin{theo}\label{theo12rigid} Let $A$ be a game subset of a commutative group $G$ with $\Gamma[A]$ the associated group game.
If the tournament on $A$ obtained by restricting $\Gamma[A]$ is rigid, then $Aut(\Gamma[A]) = G$, i.e. the left translations are the only
automorphisms of $\Gamma[A]$. That is, $\Gamma[A]$ is a tournament regular representation of $G$. \end{theo}

\begin{proof}  Write $\Gamma$ for $\Gamma[A]$. It suffices to show that if $\r$ is an automorphism which fixes $e$,
then $\r$ is the identity. For such an automorphism
$A = \Gamma(e)$ and $A^{-1} = \Gamma^{-1}(e)$ are invariant sets and so $\r$ restricts to an automorphism of $\Gamma|A$ and of $\Gamma|A^{-1}$.
Since $\Gamma|A$ is rigid, $\r$ fixes every element of $A$. Since the group is commutative, $\Gamma|A^{-1}$ is the reverse game of $\Gamma|A$ and
so is rigid as well.  Hence, $\r$ fixes every element of $A^{-1}$ and so is the identity.

 \end{proof} \vspace{.5cm}

 For $G = \Z_{2n+1}$ the score vector of the restriction of $\Gamma[A]$ to $A$ is $(0,1,2, \dots,n-2,n-1)$ for $A = [1,n]$ and is
 $(1,1,2,3,\dots,n-3,n-2,n-2)$ for $A = [1,n-1] \cup \{n+1\}$ for $n \geq 4$. From Lemma \ref{lemrigid} and Theorem \ref{theo12rigid}
 it follows that $Aut(\Gamma[A]) = \Z_{2n+1}$ in each of
 these cases.  Note that the first example provides a reproof of  Theorem \ref{theo13}.

 \begin{ex}\label{ex31a} There exists a game $\Gamma_1$ of size $9$ which is rigid.
 There exists a game $\Gamma_2$ of size $13$ which is rigid and is not isomorphic to its reversed game.
 \end{ex}

\begin{proof}   Let $J_1 = \{ 1, 2, 3, 4 \}$ and define $\Pi_1$ to contain the 4-cycle $\langle 1, 2, 3, 4 \rangle$ and with $3 \to 1, 4 \to 2$.
The score vector is $(1,1,2,2)$ and so $\Pi_1$ is rigid. Since no score value is $0$ or $3$, it follows from
 Proposition \ref{prop30} that $\Gamma_1 = 2 \Pi_1$ is rigid.

 We saw above that $2 (\Pi^{-1})$ is isomorphic to $(2 \Pi)^{-1}$ for any tournament $\Pi$. It is easy to check that $\Pi_1$ is isomorphic to
 $\Pi_1^{-1}$ and so the game $\Gamma_1$ is isomorphic to its reversed game.

 Now let $\Pi_0$ be a game of size $5$ on $\{ 0, 1, 2, 3, 4 \}$. On $J_2 = \{ 0, 1, 2, 3, 4, 5 \}$ define $\Pi_2$ by
 \begin{align}\label{eq15}
\begin{split}
\Pi_2(4) \ = \ \Pi_0(4) \cup \{ 5 \}, \qquad &\Pi_2(5) \ = \ \{ 0, 1, 2, 3 \}, \\
\Pi_2(i) \ = \ \Pi_0(i) \qquad &\text{for} \ i = 0, 1, 2, 3.
  \end{split}
\end{align}

Thus, the score vector is $(2,2,2,2,3,4)$. Since $\Pi_2^{-1}$ has score vector $(3,3,3,3,2,1)$ it follows that
$\Pi_2$ is not isomorphic to $\Pi_2^{-1}$.

If $\r$ is an automorphism of
$\Pi_2$ then it must fix the vertices $5$ and $4$. Hence, $\Pi_2(5)$ is invariant and so
$\r$ restricts to an automorphism of $\Pi_0$. For the unique game of
size $5$, the automorphism group is $\Z_5$ acting
freely by translation. Since $\r$ fixes $4$, it is the identity.

Since the score values $0$ and $6$ do not occur for $\Pi_2$ it follows from Proposition \ref{prop30} again
that $\Gamma_2 = 2 \Pi_2$ is rigid and is not isomorphic to its
reversed game.

\end{proof} \vspace{.5cm}

Another important construction is the \emph{lexicographic product}\index{lexicographic product} of two digraphs.
The lexicographic product for undirected graphs is described in
\cite{S1} and \cite{S2}. For digraphs it was introduced in \cite{GM2}. It is also referred to as the
\emph{wreath product}\index{wreath product} of digraphs.

Let $\Gamma$ be a digraph on a set $I$
and $\Pi$ be a digraph on a
set $J$.
Define $\Gamma \ltimes \Pi$\index{$\Gamma \ltimes \Pi$} on the set $I \times J$ so that for $p, q \in I \times J$
\begin{equation}\label{eq18}
p \to q \quad \Longleftrightarrow \quad \begin{cases} p_1 \to q_1 \ \text{in} \ \Gamma,
\quad \text{or}\\ p_1 = q_1 \ \text{and} \ p_2 \to q_2 \ \text{in} \ \Pi. \end{cases}
\end{equation}

The map given by  $p \mapsto p_1$ is a surjective morphism from $\Gamma \ltimes \Pi$ to $\Gamma$.

Clearly, if both $\Pi$ and $\Gamma$ are Eulerian, or if both are tournaments, then $\Gamma \ltimes \Pi$ satisfies the corresponding
property. In particular, the lexicographic product of two games is a game. Also, the lexographic product of two orders is an order. Furthermore, it is clear that
\begin{equation}\label{eq18a}
(\Gamma \ltimes \Pi)^{-1} \ = \ (\Gamma^{-1}) \ltimes (\Pi^{-1}).
\end{equation}

For each $i \in I$, let $J_i = \{ p : p_1 = i \}$. On each $J_i$, $\Gamma \ltimes \Pi$ restricts to a
digraph $\Pi_i$, so labeled because it is clearly isomorphic to $\Pi$ via $p \mapsto p_2$.

We will call an edge $(p,q)$ \emph{vertical} \index{vertical edge}\index{edge!vertical} when $p_1 = q_1$ and
\emph{horizontal}\index{horizontal edge}\index{edge!horizontal} otherwise, i.e. when $p_1 \to q_1$. We will call a
subgraph $\Theta \subset \Gamma \ltimes \Pi$
vertical (or horizontal) when it contains only vertical (resp. only horizontal) edges. Thus, $\Theta$ is vertical if and only if it is contained in
$\bigcup_{i \in I} \Pi_i$ and it is horizontal if and only if it is disjoint from $\bigcup_{i \in I} \Pi_i$.

For $p \in I \times J$ some of the outputs $(\Gamma \ltimes \Pi)(p)$ are in the $\Pi_{p_1}$ subgame.
These are the vertical outputs\index{vertical outputs}. The remaining - horizontal - outputs \index{horizontal outputs}
 are the elements of the $J_i$'s
with $ i \in \Gamma(p_1)$. That is,
\begin{equation}\label{eq19}
(\Gamma \ltimes \Pi)(p) \ = \ \Pi_{p_1}(p) \cup (\bigcup \{ J_i : i \in \Gamma(p_1) \}).
\end{equation}

If $\r \in Aut(\Gamma)$ and $\gamma : I \tto Aut(\Pi)$ is a map, so that for each $i \in I$
$\gamma_i \in Aut(\Pi)$, then we define $\rho \ltimes \gamma \in Aut(\Gamma \ltimes \Pi)$ by
\begin{equation}\label{eq20}
(\rho \ltimes \gamma)(p) \ = \ (\rho(p_1),\gamma_{p_1}(p_2)).
\end{equation}

Any permutation $\r$ of $I$ induces an automorphism of the product group $G^I$ by $(\g \circ \r)_i = \g_{\r(i)}$.
This provides a right action of $S(I)$ by group homomorphisms on the product group $G^I$.

Suppose a group $T$ acts on the right  by group homomorphisms on a group $K$.  We define
the \emph{semi-direct product}\index{semi-direct product} $T \ltimes K$\index{$T \ltimes K$} to be $T \times K$ with the
multiplication $(t_1,k_1)\cdot(t_2,k_2) = (t_1t_2,(k_1 \cdot t_2)k_2)$
The group homomorphisms $i : K \to T \ltimes K,
j : T \to T \ltimes K$  are defined by $i(k) = (e_T,k), j(t) = (t,e_K)$ where $e_T, e_K$ are the identity elements
of $T$ and $K$,respectively, inject $T$ and $K$ as subgroups of $T \ltimes K$. The first coordinate projection
$p : T \ltimes K \to T$ is a group homomorphism with $p \circ j$ the identity on $T$ and with
$i(K)$ the kernel of $p$. Observe that when the action of $T$ on $K$ is non-trivial,
the semi-direct product is non-abelian, because, e.g. $j(t)i(k) = (t,k)$, but $i(k)j(t) = (t, k \cdot t)$.

A \emph{short exact sequence}\index{short exact sequence} is a diagram of group homomorphisms
\begin{equation}
\begin{tikzcd}
K \arrow[r,"i"] & G \arrow[r,"p"] & T,
\end{tikzcd}
\end{equation}
where $i$ is an injection onto the kernel of the surjection $p$. We then say that $G$ is an \emph{extension} of $K$ by $T$.

The short exact sequence \emph{splits} when there is a homomorphism
$j : T \tto K$ such that $p \circ j = 1_T$. In that case, $T$ acts on $K$ by $i(k \cdot t) = j(t)^{-1} i(k) j(t)$,
and $(t,k) \mapsto j(t)i(k)$ is an isomorphism of $T \ltimes K$ onto $G$.

 If $T$ is a subgroup of $S(I)$, then the  semi-direct product $T \ltimes G^I$ consists of the set
$T \times G^I$ with the group composition $(\r_1,\g_1)(\r_2,\g_2) = (\r_1 \circ \r_2, (\g_1 \circ \r_2)\g_2)$.
This is also called the \emph{wreath product} \index{wreath product} of the groups $T$ and $G$.

Thus, we see that $Aut(\Gamma \ltimes \Pi)$ contains $Aut(\Gamma) \ltimes Aut(\Pi)^I$. It is shown in \cite{AGM} that
this is the entire automorphism group. We provide a somewhat simpler proof.

\begin{theo}\label{theo31} If $\Gamma$ and $\Pi$ are tournaments on sets $I$ and $J$, respectively,
then the automorphism
group of $\Gamma \ltimes \Pi$ is the semi-direct product, i.e.
\begin{equation}\label{eq21}
Aut(\Gamma \ltimes \Pi) \ = \ Aut(\Gamma)\ltimes [Aut(\Pi)]^I.
\end{equation}
\end{theo}

\begin{proof} We show that if $\theta$ is an automorphism of $\Gamma \ltimes \Pi$ then there exist unique
$\rho \in Aut(\Gamma)$ and $\gamma : I \tto Aut(\Pi)$ such that $\theta = \rho \ltimes \gamma$.

Notice that $p \not\in \Pi_{p_1}(p)$ implies
for all $p$
\begin{equation}\label{eq21a}
 |\Pi_{p_1}(p)| < |J|. \hspace{3cm}
\end{equation}

For a fixed $p \in I \times J$, let $i = p_1$ and $\bar i = \theta(p)_1$. From (\ref{eq19}) and (\ref{eq21a})
we see that
\begin{equation}\label{eq21b}
|\Gamma(i)| \cdot |J| \leq |(\Gamma \ltimes \Pi)(p)| < (|\Gamma(i)| + 1) \cdot |J|
\end{equation}

Since $|(\Gamma \ltimes \Pi)(p)| = |(\Gamma \ltimes \Pi)(\theta(p))|$ it follows that
$|\Gamma(i)| = |\Gamma(\bar i)|$.

For the product game, assume that  $p \to q$ and consider the intersection of the output sets.
If $q$ is a vertical output, i.e. $p_1 = q_1 = i$ then
\begin{equation}\label{eq22}
(\Gamma \ltimes \Pi)(p) \cap (\Gamma \ltimes \Pi)(q) \ = \
[\Pi_{i}(p) \cap \Pi_{i}(q)] \cup [\bigcup \{ J_k : k \in \Gamma(i) \}].
\end{equation}
Thus, if $q$ is a vertical output, we have as in (\ref{eq21b})
\begin{equation}\label{eq22a}
|\Gamma(i)| \cdot |J| \leq  |(\Gamma \ltimes \Pi)(p) \cap (\Gamma \ltimes \Pi)(q)| < (|\Gamma(i)| + 1) \cdot |J|.
\end{equation}

On the other hand, if $p \to q$ with $p_1 \not= q_1$ and hence $i = p_1 \to q_1= j$, then all
of the vertical outputs of $q$ are outputs of $p$, but
$J_{i}$ contains no outputs of $q$. So in that case
\begin{equation}\label{eq23}
(\Gamma \ltimes \Pi)(p) \cap (\Gamma \ltimes \Pi)(q) \ = \
[ \Pi_{j}(q)] \cup [\bigcup \{ J_k : k \in \Gamma(i)\cap \Gamma(j) \}].
\end{equation}
Since $j \in \Gamma(i) \setminus \Gamma(j)$, in this case
\begin{equation}\label{eq23a}
|(\Gamma \ltimes \Pi)(p) \cap (\Gamma \ltimes \Pi)(q)| < |\Gamma(i)|  \cdot |J|.
\end{equation}

Since $\theta$ is an automorphism, it maps $(\Gamma \ltimes \Pi)(p)$ to $(\Gamma \ltimes \Pi)(\theta(p))$ and
commutes with intersection. Since $|\Gamma(i)| = |\Gamma(\bar i)|$
it now follows from (\ref{eq22a}) and (\ref{eq23a})
that $\theta$ maps the vertical outputs of $p$ to the
vertical outputs of $\theta(p)$.

 That is, $p \to q$ and $p_1 = q_1$ implies
$\theta(p) \to \theta(q)$ and $\theta(p)_1 = \theta(q)_1$.
 Since $\Pi$ is a tournament, it follows that
for every $p \not= q \in J_i$ either $p \to q$ or $q \to p$.
So we can define $\r$ on $I$ so that $\theta(p)_1 = \r(p_1)$ and
$\theta$ maps $J_i$ to $J_{\r(i)}$. Hence, we can define
$\gamma_i$ on $J$ so that $\theta(p) = (\r(p_1),\gamma_{p_1}(p_2))$. That is,
$\theta = \r \ltimes \gamma$ at least as
set maps. Since $\theta$ is a bijection, $\rho$ and all the $\gamma_i$ are
bijections. Uniqueness of $\r$ and the $\gamma_i$ is obvious.
Finally, it is clear that
each $\gamma_i$ and $\r$ preserve the output relation and so are themselves automorphisms.

Using the maps $i(\gamma) = 1_I \ltimes \gamma$ and $p(\r \ltimes \gamma) = \r$ we obtain the short exact
sequence
\begin{equation}\label{eq24}
[Aut(\Pi)]^I \ \longrightarrow \  Aut(\Gamma \ltimes \Pi) \ \longrightarrow \  Aut(\Gamma),
\end{equation}
which splits by using $j(\r) = \r \ltimes (1_J)^I$.

% Thus, $Aut(\Gamma \ltimes \Pi)$ is the
%semi-direct product  $Aut(\Gamma)\ltimes [Aut(\Pi)]^I$.

\end{proof} \vspace{.5cm}

It follows that
\begin{equation}\label{eq24a}
|Aut(\Gamma \ltimes \Pi)| \ = \ |Aut(\Gamma)|\cdot |Aut(\Pi)|^{|I|}.
\end{equation}

Now let $\Gamma_1$ be the game of size 3 so that $|Aut(\Gamma_1)| = 3$. Inductively for $k = 2, 3, \dots$ define $\Gamma_{k} =
\Gamma_{k-1} \ltimes \Gamma_1$ so that $\Gamma_k$ is a game on a set of size $3^k$.
So from (\ref{eq24a}) it follows that $|Aut(\Gamma_k)|$ is $3^r$ with $r = \sum_{t=1}^{k} 3^{k-t}$.  That is,
\begin{equation}\label{eq24b}
|Aut(\Gamma_k)| \ = \ (3)^{(3^k - 1)/2}.
\end{equation}

On the other hand, following \cite{GM1}, Dixon proved in \cite{D} that for a tournament of size $p$ the automorphism group has
order at most $(\sqrt{3})^{p - 1}$ with the inequality strict unless $p = 3^n$, see also \cite{AB}.

Thus, the order of the automorphism of a game on a set $I$ of size $2n + 1$ is at most $3^n$ and the above construction yields a game with the
largest possible automorphism group.

Each $\Gamma_k$ is a Steiner game.  In fact we have the following, which is essentially Theorem 1.2 of Chapter 8 of \cite{R1}.

\begin{theo}\label{theo31aa} If $\Gamma$ and $\Pi$ are Steiner games then the product $\Gamma \ltimes \Pi$ is a Steiner game. \end{theo}

\begin{proof} Assume that $\Gamma$ and $\Pi$ are decomposed into $3$-cycles. For $p, q, r \in  I \times J$ we define the Steiner $3$-cycles for the
product via four cases.
\begin{itemize}
\item $p_1 = q_1 = r_1$ and $\langle p_2, q_2, r_2 \rangle$ is one of the Steiner cycles for $\Pi$,
\item $\langle p_1, q_1, r_1 \rangle$ is one of the Steiner cycles for $\Gamma$ and $p_2 = q_2 = r_2$,
\item $\langle p_1, q_1, r_1 \rangle$ is one of the Steiner cycles for $\Gamma$ and $\langle p_2, q_2, r_2 \rangle$ is one of the Steiner cycles for $\Pi$.
\item $\langle p_1, q_1, r_1 \rangle$ is one of the Steiner cycles for $\Gamma$ and $\langle r_2, q_2, p_2 \rangle$ is one of the Steiner cycles for $\Pi$.
\end{itemize}

Notice that if $p_1 \to q_1$ in $\Gamma$ then either $p_2 = q_2$, $p_2 \to q_2$ or $q_2 \to p_2$.
\end{proof}

Let $\Theta$ be a tournament on $K$, $\Gamma$ be a  digraph on $I$ and
$\pi : \Theta \to \Gamma$ a morphism of digraphs with $\pi : K \to I$  surjective.
For $i \in I$,  define $K_i = \pi^{-1}(i)$ and let $\Pi_i$ be the restriction $\Theta|K_i$, which is a tournament on $K_i$. The morphism
lets us regard $\Theta$ as a generalization of the lexicographic product in that for $p, q \in K:$
\begin{equation}\label{eq18ab}
p \to q \quad \Longleftrightarrow \quad \begin{cases} \pi(p) \to \pi(q) \ \text{in} \ \Gamma,
\quad \text{or}\\ \pi(p) = \pi(q)  \ \text{and} \ p \to q \ \text{in} \ \Pi_{\pi(p)}. \end{cases}
\end{equation}

Clearly, if there exists a tournament $\Pi$ such that $\Pi_i$ is isomorphic to $\Pi$ for all $i \in I$ then
$\Theta$ is isomorphic to $\Gamma \ltimes \Pi$.

If for each $i \in I$, $\g_i$ is an automorphism of $\Pi_i$, then $\g \in \prod_{i\in I} \g_i$ is defined by
$g(p) = \g_i(p)$ for $p \in K_i$. Thus, as in the lexicographic case we obtain at least the inclusion of groups
\begin{equation}\label{eq18ax}
\prod_{i \in I} \ Aut(\Pi_i)  \ \subset \ Aut(\Theta).
\end{equation}

\begin{theo}\label{theo31ab}  Let $\Theta$ be a tournament on $K$, $\Gamma$ be a  digraph on $I$ and
$\pi : \Theta \tto \Gamma$ a morphism of digraphs with $\pi : K \tto I$  surjective.

(a) The digraph $\Gamma$ is a tournament on $I$.

(b) If $\Theta$ is a game, then for each $i \in I$, $\Pi_i$ is a subgame of $\Theta$.

(c) For the following three conditions, any two imply the third.
\begin{itemize}
\item[(i)] $\Theta$ is a game.
\item[(ii)] $\Gamma$ is a game.
\item[(iii)] $\Pi_i$ is a game for each $i \in I$ and for $i, j \in I$, $|\Pi_i| = |\Pi_j|$.
\end{itemize}
\end{theo}

\begin{proof} (a) If $i$ and $j$ are distinct elements of $I$ then
there exist $p, q \in K$ such that $\pi(p) = i, \pi(q) = j$. Since $\Theta$
is a tournament, either $p \to q$ or $q \to p$ which imply $i \to j$ or $j \to i$, respectively.

(b) From (\ref{eq18ab}) it follows that for $p \in K$ with $\pi(p) = i$
\begin{equation}\label{eq19ab}
\Theta(p) \ = \ \Pi_i(p) \cup (\bigcup \{ K_j : j \in \Gamma(i) \}).
\end{equation}
If $\pi(q) = i$, then $|\Theta(p)| = |\Theta(q)|$ and (\ref{eq19ab}) imply
that $|\Pi_i(p)| = |\Pi_i(q)|$. That is, the scores of all of the
elements of the tournament $\Pi_i$ are the same.  Hence, $\Pi_i$ is a game.

(c) Let $k_i = |K_i|$ for $i \in I$. By (b) either (i) or (iii) implies that each
$\Pi_i$ is a game. So (\ref{eq19ab}) implies for $p \in K$ with $\pi(p) = i$
\begin{equation}\label{eq19ac}
|\Theta(p)|  \ = \ \frac{k_i - 1}{2}  + \sum \{ k_j : j \in \Gamma(i) \}).
\end{equation}

Assume (iii) so that $k_i = k$ is independent of $i \in I$. Then (\ref{eq19ac}) becomes $|\Theta(p)|  \ = \ \frac{k - 1}{2}  + |\Gamma(i)|\cdot k$.
So $|\Theta(p)|$ is the same for all $p \in K$ if and only if $|\Gamma(i)|$ is
the same for all $i \in I$. That is, $\Theta$ is a game if and only if $\Gamma$ is a game.

Finally, we assume that $\Theta$ and $\Gamma$ are both games and prove (iii).
We know from (b) that each $\Pi_i$ is a game and so it suffices to show that
$k_i$ is independent of $i \in I$.

Since $\Theta$ and $\Gamma$ are games, $|K| = 2m + 1$ and $|I| = 2n + 1$ for some natural numbers $m,n$ and from (\ref{eq19ac})
 we have
 \begin{equation}\label{eq19ad}
 2m + 1 = k_i + \sum  \{ 2k_j : j \in \Gamma(i) \} \quad \text{and} \quad 2n + 1 = 1 + 2 |\Gamma(i)|
 \end{equation}
 for all $i \in I$.

We require a little linear algebra result.

\begin{lem}\label{lem31ac} Let $A = (a_{ij})$ be a square matrix with integer entries.  If $a_{ij}$ is even whenever $i \not= j$ and
$a_{ii}$ is odd for all $i$ then $A$ is nonsingular.\end{lem}

\begin{proof} Using the ring homomorphism $\Z \to \Z_2$ we reduce $A$ mod $2$ and get the identity matrix over the field  $\Z_2$. The determinant
is $1$ in $\Z_2$ and so the determinant of $A$ is congruent to $1$ mod $2$, i.e. it is odd and so is non-zero.

\end{proof}

Define the matrix $A$ on $I \times I$ by
\begin{equation}\label{eq19ae}
a_{ij} \ = \ \begin{cases} 2 \ \text{if} \ i \to j, \\  0 \ \text{if} \ j \to i, \\ 1 \ \text{if} \ i = j. \end{cases}
\end{equation}
Let $u, q$ be the $I \times 1$ matrices with $u_i = 1$ and $q_i = k_i$ for $i \in K$. From (\ref{eq19ad}) we have
$(2m + 1)u = Aq $ and $(2n + 1)u = Au$. Because the matrix $A$ is nonsingular, we obtain $q = \frac{2m + 1}{2n + 1}u$. Thus,
$q_i = k_i$ is independent of $i$.

\end{proof}

It is clear that if $\Gamma$ is a tournament on a set $I$ and for each $i \in I$ $\Pi_i$ is a tournament on a set $K_i$,
then we can let $K = \bigcup_{i \in I} \{ i \} \times K_i$ and $\pi : K \to I$ be the first coordinate projection. We can define the tournament
$\Theta$ on $K$ uniquely so that (\ref{eq18ab}) holds and thus so that $\pi$ is a surjective morphism. We will call this a
\emph{generalized lexicographic product}\index{generalized lexicographic product}\index{lexicographic product!generalized}.  If $\Gamma$ is a game on  $I$
and for each $i \in I$ $\Pi_i$ is a game on  $K_i$ with $|K_i|$ independent of $i$, then $\Theta$ is a game on $K$.

\begin{ex}\label{ex31ad} There exists a morphism $\pi : \Theta \tto \Gamma$ with $\Theta$ but not $\Gamma$ a game and there exists
a $3$-cycle $\Gamma_0 \subset \Gamma$ such that $\pi^{-1}(\Gamma_0)$ is not a subgame of $\Theta$. \end{ex}

\begin{proof} Let $\Gamma_3$ be the $3$-cycle game $\langle 0, 1, 2 \rangle$ and let $\Theta = \Gamma_3 \ltimes \Gamma_3$ which maps to
$\Gamma_3$ with $3$-cycle fibers $\Pi_i$ for $i = 0, 1, 2$. For $i = 1, 2$ map $\Pi_i$ to a single vertex $a_i$. This maps
$\Theta$ onto a tournament $\Gamma$ of size $5$ with score vector $(3,2,2,2,1)$. Select a vertex $a_0 \in \Pi_0$ to get
a $3$-cycle $\langle a_0, a_1, a_2 \rangle$ in $\Gamma$ whose preimage in $\Theta$ is a tournament of size $7$ which is not a game.

\end{proof}

Now assume that $\Theta$ is a subgraph of $\Gamma \ltimes \Pi$.  If $\Theta$ is vertical, then it is the  union of the separated
subgraphs $\{ \Theta \cap \Pi_i : i \in I \}$. So $\Theta$ is Eulerian if and only if
each $\Theta \cap \Pi_i$ is Eulerian. If $\Theta$ is horizontal, then it
is a disjoint union of horizontal cycles. Now assume that $\langle i_1,\dots,i_n \rangle$
is a cycle in $\Gamma$. If $\{ j_1,\dots,j_n \}$ is an arbitrary
sequence of length $n$ in $J$, then $\langle (i_1,j_1),\dots,(i_n,j_n) \rangle$ is a horizontal cycle in $\Gamma \ltimes \Pi$. Thus, there are
$|J|^n$ distinct cycles which project to $\langle i_1,\dots,i_n \rangle$ via $\pi : \Gamma \ltimes \Pi \to \Gamma$.
%If $\{ \{ j_1^-,j_1^+ \},\dots
%\{ j_n^-,j_n^+ \} \}$ is a sequence of distinct pairs in $J$ then $\langle (i_1,j_1^-),\dots,(i_n,j_n^-), (i_1,j_1^+),\dots,(i_n,j_n^+)\rangle$
%is a cycle in $\Gamma \ltimes \Pi$ of length $2n$.
Using these we can construct explicit examples of large numbers of distinct games which are isomorphic.

Assume that $\langle i_1,\dots,i_n \rangle$ is a Hamiltonian cycle in the game $\Gamma$ so that $n = |I|$, and assume that $\Pi$ is a group
game with $e \in J$ the identity. For $i = i_k$ let $\gamma_i$ be the translation of $J$ taking $e$ to $j_k$ and let $\r$ be the identity on $I$.
Then $\rho \ltimes \gamma$ is an automorphism of $\Gamma \ltimes \Pi$ taking $\langle (i_1,e),\dots,(i_n,e) \rangle$
to $\langle (i_1,j_1),\dots,(i_n,j_n) \rangle$.  So it is an isomorphism from $\Gamma \ltimes \Pi/\langle (i_1,e),\dots,(i_n,e) \rangle$ to
$\Gamma \ltimes \Pi/\langle (i_1,j_1),\dots,(i_n,j_n) \rangle$. In $\Gamma \ltimes \Pi/\langle (i_1,j_1),\dots,(i_n,j_n) \rangle$ the edges of
the reversed cycle are the only edges reversed by $\pi$. Hence, the $|J|^{|I|}$ games obtained by reversing the distinct cycles are distinct
isomorphic games.

In fact, if $\Gamma$ is a game on $I$ with $|I| = 2n + 1$, then by using the permutations of $I$ we can construct $(2n + 1)!$ games. If two of these
permutations $\r_1, \r_2$ map to the same game then they differ by the automorphism
$(\r_2)^{-1}\r_1$ of $\Gamma$. It follows that $\Gamma$ is isomorphic to
exactly $(2n + 1)! \div |Aut(\Gamma)|$ distinct games. In particular, if $\Gamma$ is rigid so that $|Aut(\Gamma)| = 1$, then $\Gamma$ is isomorphic to
$(2n + 1)!$ distinct games. In general, by the result of Dixon from \cite{D} quoted above,
$|Aut(\Gamma)| \leq 3^n$. Consequently, any game of size $2n +1$ is isomorphic to at least $(2n + 1)! \div 3^n$ distinct games.

If $G$ is a subgroup of the permutation group $S(I)$ with order $|G|$ a power of $2$, then by Proposition \ref{prop04} the identity is the only
element of $G$ which is an automorphism of any tournament on $I$. It follows that by applying the elements of $G$ to any game we obtain $|G|$ distinct
games which are isomorphic.

We can use the above surjective morphism construction to build some illustrative examples of group actions on games.

\begin{ex}\label{ex31ae} If $G$ is a group of odd order and $n \geq |G|$, then there exists a game $\Theta$ on a set $K$ of size
$2n + 1$ on which $G$ acts.  The set $K$ contains two copies of $G$ on each of which $G$ acts by translation while the remaining
$2(n - |G|) + 1$ points are fixed points of the action. \end{ex}

 \begin{proof} With $m = n - |G|$ let $\Gamma$ be an arbitrary tournament on the set $I = \{0, 1, \dots, m \}$. Let $K_0 = G$ and
 $\Pi_0$ be a group game on $G$.  So $G$ acts be translation on $\Pi_0$. For $i = 1, \dots, m$ let $K_i = \{ i \}$ and let
 $\Pi_i$ be the trivial -empty- game on the singleton $K_i$. Let $\bar K =  \bigcup_{i \in I} \{ i \} \times K_i$ and let
 $\bar \Theta$ be the generalized lexicographic product. From (\ref{eq18ax}) we see that $G$ acts on the tournament $\bar \Theta$.
 $\bar K$ contains a copy of $G$ on which the action is by translation and $m$ fixed points.

 Let $\Theta$ be the double $2\bar \Theta$ on $K = \{ 0 \} \cup \bar K \times \{ -1, +1 \}$. By the injection
 $2 : Aut(\bar \Theta) \to Aut(2 \bar \Theta)$ we obtain the action of $G$ on $\Theta$.

\end{proof}\vspace{.5cm}

Notice that if $G$ is cyclic, then the translation by a generator is a permutation of $G$ with a cycle of length $|G|$. A fixed point imposes a bound
on the length of such a cycle.

\begin{prop}\label{prop31af} Let $\Gamma$ be a game on a set $I$ with $|I| = 2n + 1$. If $\r$ is an automorphism of $\Gamma$ which fixes some vertex, then
the length of any cycle in the associated permutation is at most $n$.\end{prop}

 \begin{proof} If the vertex $i$ is fixed by $\r$ then each of the size $n$ sets $\Pi(i)$ and $\Pi^{-1}(i)$ is invariant.
 %Let $i_0$ be a fixed vertex and $(i_1,\dots,i_m)$ be a cycle in the permutation $\r$. If $i_0 \to i_1$, then
% $i_0 = \r^j(i_0) \to \r^j(i_1) = i_{j+1}$. Hence, $\{i_1, \dots, i_m \} \subset \Gamma(i_0)$ and because $\Gamma$ is a game
% $m \leq n$.  Similarly, if $i_1 \to i_0$, then $\{i_1, \dots, i_m \} \subset \Gamma^{-1}(i_0)$ and so $m \leq n$.

\end{proof}\vspace{.5cm}

It can happen, however, that a large cyclic group acts effectively on a game.

\begin{ex}\label{ex31ag} There exists a game of size $29$ on which $\Z_{33}$ acts effectively.\end{ex}

 \begin{proof} Let $I = \{ 1, 2 \}$ with $\{ (1,2) \} = \Gamma$. Let $K_1 = \Z_3$ and $K_2 = \Z_{11}$ and let $\Pi_1, \Pi_2$ be group games.
 Again let $\bar \Theta$ be the generalized lexicographic product on $\bar K = (\{ 1 \} \times K_1)  \cup (\{ 2 \} \times K_2)$.
 Thus, $\bar \Theta$ is a tournament on $\bar K$ with $|\bar K| = 14$. Since $\Z_{33}$ is isomorphic to $\Z_3 \times \Z_{11}$,
 it acts effectively on
 $\bar \Theta$. Again let $\Theta = 2\bar \Theta$. If $\r$ is the generator of $\Z_{33}$, then the associated permutation fixes $0$ and
 contains two $3$-cycles and two 11-cycles.

 \end{proof} \vspace{.5cm}

 We conclude the section with an obvious remark.

 \begin{prop}\label{prop31ah} Let $G$ act on a game $\Gamma$ if $\Delta \subset \Gamma$ is an Eulerian subgraph which is $G$ invariant, then
 $G$ acts on the game $\Gamma/\Delta$ with the same action on the vertices. \end{prop}

% $\Box$

\vspace{1cm}

\section{Bipartite Tournaments and Pointed Games}\label{secpointed}

Recall that a digraph $\Pi$ on $I$ is bipartite\index{bipartite} when $I$ is the disjoint union of
sets $J, K$ and $\Pi \subset (J \times K) \cup (K \times J)$. A cycle in a bipartite digraph has even
length and so a  4-cycle is the cycle of shortest length in a bipartite digraph. As an Eulerian digraph is a
disjoint union of cycles, it follows that a bipartite Eulerian digraph has even cardinality.

We will call $\Pi$ a \emph{bipartite tournament}\index{bipartite tournament} \index{tournament!bipartite} on the pair $\{ J,K \}$ of disjoint sets
when $\Pi \cup \Pi^{-1} = (J \times K) \cup (K \times J)$.  That is, each $i \in J$ has an edge connecting it to every element of $K$ and vice-versa.
Hence, $|\Pi| = |J| \cdot |K|$.  We say that two
bipartite tournaments $\Pi$ and $\Gamma$ on the pair $\{ J, K \}$ have the same scores when each vertex
$i \in J \cup K$ has the same number of outputs in $\Pi$ and $\Gamma$, and consequently the same number
of inputs. The following is the bipartite analogue of Proposition \ref{prop24} and Theorem \ref{theo26}.

 \begin{theo}\label{theo26bi} Assume that  $\Pi$ and $\Gamma$ are bipartite tournaments on the pair $\{ J,K \}$.
 \begin{itemize}
 \item[(a)] The difference graph $\Delta = \Delta(\Pi,\Gamma)$ is Eulerian if and only if $\Pi$ and $\Gamma$ have the same scores.

 \item[(b)] If $\Pi$ and $\Gamma$ have the same scores, then there exists a finite sequence
 $\Pi_1, \dots, \Pi_k$ of bipartite tournaments on $\{ J,K \}$ with $\Pi_1 = \Pi$ and $\Pi_k = \Gamma$ and such that for $p = 1,\dots, k-1$,
 $\Pi_{p+1}$ is $\Pi_p$ with some 4-cycle reversed. In particular, all of these bipartite tournaments have the same scores.

 If $\Delta$ a single cycle of length $2\ell$ then  a sequence can be chosen
 with $k = \ell - 1$.
 \end{itemize}
 \end{theo}

 \begin{proof} (a) is proved exactly as is Proposition \ref{prop24}.

 (b) As in the proof of Theorem \ref{theo26} we reduce to the case when $\Delta$ is  a single cycle whose necessarily even length we write as $2 \ell$.
 We proceed by induction beginning with $\ell = 2$ in which case $\Gamma$ is obtained by reversing $1 = \ell - 1$ 4-cycle.

 Assume $\ell > 2$. If  $\Delta$ consists of the cycle
cycle $\langle i_1,\dots,i_{2\ell} \rangle$ with $\ell \geq 3$, we use two cases as before.

 {\bfseries Case 1} ($i_1 \to i_4$ in $\Pi$): In this case, $\langle i_1,i_4,\dots,i_{2\ell} \rangle$ is an $2(\ell - 1)$ cycle in $\Pi$ and so we can get
 from $\Pi$ to $\tilde \Gamma = \Pi/\langle i_1,i_4,\dots,i_{2\ell} \rangle$ via a sequence of $\ell - 2$  4-cycles by inductive hypothesis. Then
 we go from $\tilde \Gamma$ to $\Gamma$ by reversing the 4-cycle
$\langle i_1, i_2, i_3, i_4 \rangle$  in $\tilde \Gamma$. Furthermore,
 $\Gamma = \Pi/ \langle i_1,\dots,i_{\ell}\rangle = \tilde \Gamma/\langle i_1, i_2, i_3, 1_4 \rangle$ because
 the edge between $i_1$ and $i_4$ has been reversed twice.

 {\bfseries Case 2} ($i_4 \to i_1$ in $\Pi$): We proceed as before reversing the order of operations.
 This time $\langle i_1, i_2, i_3, i_4 \rangle$ is a 4-cycle
 in $\Pi$ and $\langle i_1,i_4,\dots,i_{\ell}\rangle $  is an $2(\ell - 1)$ cycle in $\tilde \Gamma = \Pi/\langle i_1, i_2, i_3, i_4  \rangle$.
 Proceed as before.

\end{proof} \vspace{.5cm}

The following is now obvious.

\begin{cor}\label{cor27bi} If $\Pi$ is a bipartite tournament on the pair $\{ J,K \}$ then $\Gamma \leftrightarrow \Delta(\Pi,\Gamma)$ is a
bijective correspondence between the set of bipartite tournaments $\Gamma$ on the pair $\{ J,K \}$ with the same scores as $\Pi$ and the set
Eulerian subgraphs of $\Pi$. In particular, the cardinality of the set of Eulerian subgraphs depends only on the set of scores.  \end{cor}

 $\Box$ \vspace{.5cm}

 We can form the interchange graph\index{interchange graph} on the set of bipartite tournaments
 on the pair $\{ J,K \}$, connecting $\Pi$ and $\Gamma$ by an undirected
 edge if $\Delta = \Delta(\Pi,\Gamma)$ is a single 4-cycle. In general, it is clear that the distance between $\Pi$ and $\Gamma$ is bounded by
 $\frac{|\Delta|}{2} - \s(\Delta)$ where, as before, the span $\s(\Delta)$ is the size of a maximum decomposition of $\Delta$ by disjoint cycles.

 Our brief consideration of bipartite digraphs is motivated by their application to the pointed games which we will now consider.

 A \emph{pointed game}\index{pointed game} \index{game!pointed} $\Pi$ on $I$ is a game with a chosen vertex which we label $0$.
 Let $I_+ = \Pi^{-1}(0)$ and $I_- = \Pi(0)$. With $I_+, I_-$ fixed we call $\Pi$ a pointed game on the pair $(I_+,I_-)$.
 We denote by $\Pi_+$ and $\Pi_-$ the tournaments which are the restrictions of $\Pi$ to
 $I_+$ and $I_-$ respectively. We let $\Xi$ denote $\Pi|(I_+ \cup I_-) \setminus (\Pi_+ \cup \Pi_-)$. It consists of those edges
 which connect elements of $I_+$ with those of $I_-$. Thus, $\Xi$ is a bipartite tournament
 on the pair $\{ I_+, I_- \}$.  The union $\Pi_+ \cup \Pi_- \cup \Xi$ consists of all edges of $\Pi$ except those which connect to $0$.

 Recall that a subset $A \subset I$ is invariant for a relation $\Pi$ on $I$ when $\Pi(A) \subset A$, in which case, $B = I \setminus A$ is
 invariant for $\Pi^{-1}$. We think of the pair $\{ B, A \}$ as a \emph{splitting}\index{splitting} for $\Pi$ when neither $A$ nor $B$ is empty.
 A splitting for  a relation $\Pi$ exists if and only $\Pi$ is not strong.  Since any game is strong, no game admits a splitting, but the tournaments
 $\Pi_{\pm}$ and the bipartite tournament $\Xi$ need not be strong.

 \begin{lem}\label{lemsplit}({\bfseries Splitting Lemma}) Let $\Pi$ be a pointed game on the pair $(I_+,I_-)$ and $A \subset I \setminus \{ 0 \}$.
 Let $A_{\pm} = A \cap I_{\pm}$.
 The following are equivalent.
 \begin{itemize}
 \item[(i)] $A$ is invariant for $\Xi^{-1}$.
 \item[(ii)] $A_{\pm}$ is invariant for $\Pi_{\pm}$  and, in addition, either \\ $|A_+| = |A_-|$ or $|A_+| = |A_-| - 1$.
 \end{itemize}
 \end{lem}

 \begin{proof} Let $|I| = 2n + 1$. With $B = I \setminus (A \cup \{ 0 \})$
 let $a_{\pm} = |A_{\pm}|$ and $b_{\pm} = |B_{\pm}|$ with $B_{\pm} = B \cap I_{\pm}$.
 If $2n + 1 = |I|$, then $n = |I_{\pm}|$ and so $a_{\pm} + b_{\pm} = n$.
 Furthermore, $|A| = a_+ + a_-$ and $|B| = b_+ + b_-$.  Let $\Pi|A$ be the restriction of $\Pi$ to $A$.

 (i) $\Rightarrow$ (ii): Let $i \in A_+$. By assumption, every $\Pi$ edge between $i$ and an element of $B_-$, i. e. every $\Xi$ edge
 between $i$ and an element of $B$, is an output from $i$, accounting for $b_-$ outputs from $i$. There is also one output to $0$.
 Because $\Pi$ is a game there are a total of $n$ outputs from $i$.  Thus, in $\Pi|A$, the element $i$ has at most $n - b_- - 1 = a_- - 1$
 outputs. Thus, the total number of outputs in $\Pi|A$ from elements of $A_+$ is bounded by $a_+ \cdot (a_- - 1)$.

 Similarly, if $i \in A_-$, then every $\Pi$ edge between $i$ and an element of $B_+$ is an output. Thus, in $\Pi|A$ the element $i$ has at
 most $n - b_+ =  a_+$ outputs. Thus, the total number of outputs in $\Pi|A$ from elements of $A_-$ is bounded by $a_- \cdot a_+$.

 Every $\Pi|A$ output comes from an element of $A_+$ or $A_-$.  The total number of outputs is the total number of edges which is
 $\frac{1}{2}(a_+ + a_-)(a_+ + a_- - 1)$. Thus, we obtain
 \begin{equation}\label{eqsplit1}
 \frac{1}{2}(a_+ + a_-)(a_+ + a_- - 1) \quad \leq \quad a_+ (a_- - 1) \ + \ a_- a_+.
 \end{equation}

 Furthermore, if any $i \in A_{\pm}$ has an output to $B_{\pm}$, i.e. if either $A_+$ is not $\Pi_+$ invariant or $A_-$ is not $\Pi_-$ invariant,
 then the inequality is strict.  But this inequality can be rewritten as:
 \begin{equation}\label{eqsplit2}
 (a_- - a_+)^2 \quad \leq \quad (a_- - a_+).
 \end{equation}
 Because $a_+$ and $a_-$ are integers this equation cannot hold strictly and the only way it can hold is if $a_- - a_+$ equals $0$ or $1$.

 (ii) $\Rightarrow$ (i): If $i \in A_+$, then every edge from $B_+$ is an input. In $\Pi|A$ $i$ has at most $n - b_+ = a_+$ inputs.
 Similarly, if $i \in A_-$, then every edge from $B_-$ is an input and it has an input from $0$. In $\Pi|A$ it has at most
 $n - b_- -1 = a_- - 1$ inputs. Again the total number of inputs is the total number of edges.  This time we get
  \begin{equation}\label{eqsplit3}
 \frac{1}{2}(a_+ + a_-)(a_+ + a_- - 1) \quad \leq \quad (a_+)^2 \ + \ (a_-)^2 \ - \ a_-
 \end{equation}
 and if $A$ is not $\Xi^{-1}$ invariant then this inequality is strict.

 This time the inequality can be rewritten as the reverse of (\ref{eqsplit2})
  \begin{equation}\label{eqsplit4}
 (a_- - a_+)^2 \quad \geq \quad (a_- - a_+).
 \end{equation}
 This is always true and can, of course, be strict. We will see below that $A_{\pm}$ can both be $\Pi_{\pm}$ invariant without $\Xi^{-1}$ invariance of
 $A$.

 But if, in addition, $a_- - a_+$ equals $0$ or $1$, then this inequality is not strict and so $A$ must be $\Xi^{-1}$ invariant.

 \end{proof}\vspace{.5cm}

 For a pointed game $\Pi$ on $2n + 1$ vertices, $\Pi_+$ and $\Pi_-$ are tournaments on $n$ vertices.  We now show that these tournaments can be chosen
 arbitrarily.

 \begin{theo}\label{theopoint} Let $\Gamma_+$ and $\Gamma_-$ be tournaments on the set $J$ of size $n$. Let
  $I_+ = J \times \{ +1 \}, I_- = J \times \{ -1 \}$ and $I = \{ 0 \} \cup I_- \cup I_+$.

 There exists a pointed game $\Pi$ on the pair $(I_+,I_-)$ such that with the identification $j \mapsto (j, -1)$
 $\Pi_- = \Gamma_-$ and with the identification $j \mapsto (j, +1)$ $\Pi_+ = \Gamma_+$. \end{theo}

  \begin{proof} Let $\Delta = \Delta(\Gamma_-,\Gamma_+)$.  We prove the result by induction on $k = |\Delta|$. Notice that we make no
  assumption about the scores and so $\Delta$ need not be Eulerian.

  $(k = 0)$ In this case, $\Gamma_- = \Gamma_+$ and we use $\Pi = 2\Gamma_-$.

  $(k > 0)$ Choose $(i,j) \in \Delta$. Let $\hat \Delta = \Delta \setminus \{ (i,j) \}$ and $\hat \Gamma = \Gamma_-/\hat \Delta$.
  By induction hypothesis there exists a pointed game $\hat \Pi$ with $\hat \Pi_- = \Gamma_-$ and  $\hat \Pi_+ = \hat \Gamma$.
  Let $i+ = (i,+1), j+ = (j,+1)$. Then $(i+,j+)$ is an edge of $\hat \Pi_+$ and with the identification above $j \mapsto (j, +1)$
  we have $\hat \Pi_+/\{ (i+,j+) \} = \Gamma_+$.

  Let $\hat \Xi$ be the bipartite tournament of connections between $I_+$ and $I_-$ in $\hat \Pi$. We first show that there exists a
  $\hat \Xi$ path from $j+$ to $i+$.

  Let $A = \{ i+ \} \cup \O \hat \Xi^{-1}(i+)$. This is a
  $\hat \Xi^{-1}$ invariant set which contains $i+$. The Splitting Lemma \ref{lemsplit} implies that
  $A_+ = A \cap I_+$ is $\hat \Pi_+$ invariant. Since $(i+,j+) \in \hat \Pi_+$ it follows that $j+ \in A_+ \subset A$. This exactly says that
  there is a $\hat \Xi$ path from $j+$ to $i+$.

  By eliminating repeated vertices as usual we get a simple $\hat \Xi$ path. Concatenating
  with the edge $(i+,j+)$ we obtain a cycle $C$ in $\hat \Pi$. The cycle is disjoint from $\hat \Pi_-$ and does not contain the vertex
  $0$.  Furthermore, it intersects $\Pi_+$ only in the edge $(i+,j+)$.  It follows that $\Pi = \hat \Pi/C$ is a pointed game with
   $\Pi_- = \Gamma_-$ and  $\Pi_+ = \Gamma_+$ as required.

 \end{proof}\vspace{.5cm}

 While the input and output tournaments in a game can be arbitrary, this is not true for group games.

 \begin{lem}\label{lempointgame} Let $A$ be a game subset for the group $\Z_{2n+1}$ with $\Gamma[A]$ the associated game.
 If  $i \in A$ with $i$ relatively prime to $2n + 1$  and the score of $i$ in the tournament $\Gamma[A]|A$ is $n - 1$, i.e. $|A \cap (i + A)| = n - 1$,
 then $A = m_i([1,n])$. \end{lem}

   \begin{proof}  First, assume that $i = 1$. Every element of $A$ other than $1$ itself is an output of $1$.  That is, if $a \not= 1$ is an
   element of $A$, then $a - 1 \in A$. So if
   $m = \max \{ i \in A : i = 1,2,\dots,2n \}$
   then, inductively, $m-1, m-2, \dots, 1 \in A$ and $m+1,\dots,2n \not\in A$. Since $|A| = n$, $m = n$.

   In general, if $i$ is relatively prime to $2n + 1$, then multiplication $m_i$ is a group isomorphism on $\Z_{2n+1}$ and with
   $B = (m_i)^{-1}(A)$ it is an isomorphism from $\Gamma[B]$ to $\Gamma[A]$ taking $B$ to $A$. So $1 \in B$ with score $n-1$ for the restriction
   to $B$.  Hence, $B = [1,n]$ and so $A = m_i([1,n])$.

  \end{proof}\vspace{.5cm}

 \begin{theo}\label{theopointgame} Let $A$ be a game subset for a group $G$ with $|G| = 2n + 1$ prime. If there exists $i \in A$ such that
  the score of $i$ in the tournament $\Gamma[A]|A$ is $n - 1$, then there is an isomorphism of $G$ with $\Z_{2n+1}$ mapping $i$ to $1$ and
  $A$ to $[1,n]$. In particular, the score vector for the tournament $\Gamma[A]|A$ is $(0,1,2,\dots,n-1)$. \end{theo}

  \begin{proof} A group of prime order is cyclic and so $G$ is isomorphic to $\Z_{2n+1}$ by an isomorphism which maps $i$ to some element
  $j \in [1,2n]$ and so is relatively prime to $2n+1$. Apply the previous lemma and compose with $(m_j)^{-1}$. So the score vector of $\Gamma[A]|A$
  is the score vector of the restriction to $[1,n]$ of the associated game on $\Z_{2n+1}$.

  \end{proof}

  {\bfseries Remarks:} (a) If $2n + 1$ is not prime there can be other games on $\Z_{2n+1}$ with an element of
  having score $n-1$ in the tournament $\Gamma[A]|A$.  For example, with $2n + 1 = 9$ and $A = \{ 1, 3, 4, 7 \}$ the element
  $3$ has score $3$ and the restricted tournament has score vector $(1,1,1,3)$.

  (b) It follows that if $2n + 1$ is a prime then a tournament $\Pi$ with score vector $(1,1,2,\dots,n-3,n-3,n-1)$ cannot occur as
  the restriction $\Gamma[A]|A$ for a game subset $A$ of a group of order $2n + 1$. \vspace{.5cm}

  From Theorem \ref{theopoint} it follows that the number of pointed games $Games(I_+,I_-)$ on the pair $(I_+,I_-)$ with  $|I| = 2n + 1$ is
  bounded below by the square of the number of tournaments on a set of size $n$,
 It is bounded above by the number of tournaments on a set of size $2n$, i. e.  by
 \begin{equation}\label{pointest}
 2^{n(n-1)} \ \leq \ |Games(I_+,I_-)| \ \leq \ 2^{n(2n-1)}.
 \end{equation}

  It is clear than neither estimate is sharp.

  If $\Gamma_+$ and $\Gamma_-$ are tournaments of size $n$, and $\Pi$ is a pointed game with $\Pi_- = \Gamma_-$ and  $\Pi_+ = \Gamma_+$ then
  the scores of the elements of the associated bipartite tournament $\Xi$ are clearly determined by the score vectors of $\Gamma_-$ and  $ \Gamma_+$.
  Thus, the number of pointed games $\Pi$ which similarly satisfy $\Pi_- = \Gamma_-$ and  $\Pi_+ = \Gamma_+$ is the number of Eulerian
  subgraphs of $\Xi$ for any one of them.

  \begin{ex}\label{expoint} If $\Gamma = \Gamma_-  = \Gamma_+$ and $\Pi = 2 \Gamma$, then from the cycles in $\Gamma$ we can obtain
  cycles in $\Xi$. \end{ex}

  \begin{proof} If $\langle i_1,i_2\dots,i_{2\ell}\rangle$ is a
  cycle in $\Gamma$ of even length, then $\langle i_{2\ell}+,i_{2\ell-1}-,$ $\dots,i_{1}-\rangle$ and $\langle i_{2\ell}-,i_{2\ell-1}+,\dots,i_{1}+\rangle$
  are disjoint cycles of this same length in $\Xi$.

  If $\langle i_1,i_2\dots,i_{2\ell+1}\rangle$ is a cycle of odd length in $\Gamma$, then
  $$\langle i_{2\ell+1}+,i_{2\ell}-,\dots,i_{1}+, i_{2\ell+1}-,i_{2\ell}+,\dots,i_{1}-\rangle$$ is a cycle in $\Xi$ of double the length.
  In addition, $$\langle i_{2\ell+1}-, i_{2\ell+1}+,i_{2\ell}-,i_{2\ell}+,\dots,i_{1}-,i_{1}+  \rangle$$ is a cycle in $\Xi$ of double the
  length. Furthermore,
   for each $1 \leq k \leq \ell$
  $$\langle i_{2\ell+1}+,i_{2\ell}-,\dots,i_{2k}-,i_{2k}+,i_{2k-1}-,\dots,i_1- \rangle$$ and
  $$\langle i_{2\ell+1}-,i_{2\ell}+,\dots,i_{2k-1}-,i_{2k-1}+,i_{2k-2}-,\dots,i_1+ \rangle$$
  are cycles of length $2\ell + 2$ in $\Xi$. Observe that each of these has $\ell + 1$ edges of the form $(i_p+,i_{p-1}-)$ (mod $2 \ell + 1$). It follows
  that no pair of theses cycles associated with $\langle i_1,i_2\dots,i_{2\ell+1}\rangle$ are disjoint.

 If $n = 3$ and $\Gamma = \langle i_1,i_2,i_3\rangle$, then then for $\Pi = 2 \Gamma$ these are the only cycles in $\Xi$.  That is, $\Xi$ contains
 two 6-cycles and three 4-cycles, no two of which are disjoint.  Thus, $\Xi$ contains six Eulerian subgraphs (including the empty subgraph).
   \end{proof} \vspace{.5cm}

 When a pointed game is the double of $\Pi$ on $I$ then $i- \to i+$ for all $i \in I$. This convenient pairing need not be possible for all
 pointed games.

 \begin{ex}\label{exmarriage} There is a pointed game $\Pi$ on the pair $(I_+,I_-)$ for which there does not exist a bijection
 $\r : I_- \tto I_+$ such that $i \to \r(i)$ for all $i \in I_-$. \end{ex}

 \begin{proof}
 Let $\Gamma$ and $\Theta$ be games on disjoint sets $J$ and $K$ with $|J| = 3$ and $|K| = 5$. With $I = J \cup K$  let $I_{\pm} = I \times \{ \pm 1 \}$
 with similar notation for $J_{\pm}, K_{\pm} \subset I_{\pm}$.

 Let $\Gamma_{\pm}, \Theta_{\pm}$ be copies of the games $\Gamma$ and $\Theta$ on $J_{\pm}$ and
 $K_{\pm}$. Define the pointed game $\Pi$ on $\{ 0 \} \cup I_+ \cup I_-$ to be $I_+ \times \{ 0 \} \cup \{ 0 \} \times I_-$ together with
 \begin{align}\label{eqmarriage01}
 \begin{split}
2 \Gamma \ \cup \ \Theta_+ \ \cup \ \Theta_- \ \cup & \ [K_+ \times K_-] \\
\cup \ [(J_- \cup J_+) \times K_+]  \ \cup & \ [K_- \times (J_- \cup J_+)].
\end{split}\end{align}

We can represent $\Pi$ by the diagram

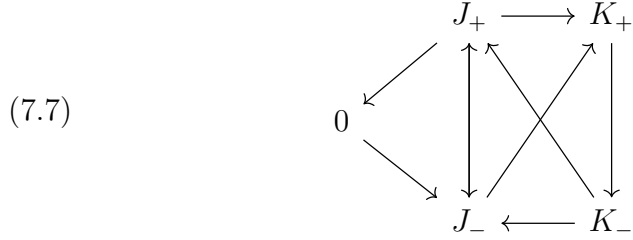
\begin{equation}\label{cdmarriage}
\begin{tikzcd}
& |[alias=T]| J_+ \arrow[dl]\arrow[r]\arrow[d,from=T,to=O]   & |[alias=S]| K_+\arrow[d,from=S,to=W]  \\
0\arrow[dr]  \\
& |[alias=O]| J_- \arrow[u,from=O,to=T]\arrow[ur,from=O,to=S] & |[alias=W]| K_- \arrow[l]\arrow[ul,from=W,to=T]
 \end{tikzcd}
\end{equation}

Since the five elements of $K_-$ have upward outputs only among the three elements of $J_+$, the required bijection cannot exist.

   \end{proof}

 On the other hand, a strengthening of this condition characterizes the games which are doubles.

 \begin{theo}\label{theomarriage} Let $\Pi$ be a pointed game on the $(I_+,I_-)$. If there exists a bijection
 $\r : I_- \to I_+$ such that $i \to \r(i)$  and $\Pi$ is reducible via $i \to \r(i)$ for all $i \in I_-$, then as a pointed game
 $\Pi$ is isomorphic to the double $2 (\Pi|I_-)$. \end{theo}

 \begin{proof}  By relabeling, we may assume that $I_{\pm} = J \times \{\pm 1 \}$ and $\r(i-) = i+$. In that case, we will show that
 $\Pi = 2 \Pi_-$ with $\Pi_-$ the game on $J$ which is identified with $\Pi|I_-$ via the identification $i \mapsto i-$.

 Let $\Delta = \Delta(2 \Pi_-,\Pi)$. Since the two games agree on $I_-$, $\Delta$ is disjoint from $\Pi|I_-$. Furthermore,
 $i- \to i+$ for all $i \in J$ in both games and so no such edge is in $\Delta$. Furthermore, no edge containing $0$ is in $\Delta$.
 We will show that if $\Delta$ is nonempty, then $\Pi$ is not reducible via $j- \to j+$ for some $j \in J$.
 \vspace{.25cm}

\textbf{ Case 1 }(\ $\Delta \cap \Xi \ \not= \ \emptyset$ \ ): That is, there exists $i \to j$ in $\Pi_-$ such that either $j+ \to i-$  or
 $j- \to i+$ is in $\Delta$. If $j+ \to i-$ in $\Delta$ then $i- \to j+$ in $\Pi$. Since $i- \to j-$ in $\Pi$
  it follows from Proposition \ref{prop07} that $\Pi$ is not reducible via $j- \to j+$. If $j- \to i+$ is in $\Delta$,
  then, since $\Delta$ is Eulerian,
 there exists an edge to $j-$ in $\Delta$. Since $0$ is not a vertex of $\Delta$ and $\Delta$ is disjoint from $\Pi|I_-$,
 it follows that for some $k \in J$, $k+ \to j-$ in $\Delta$. As before, $\Pi$ is not reducible via $k- \to k+$.
  \vspace{.25cm}

 \textbf{ Case 2 } (\ $\Delta \ \subset \ \Pi|I_+$ \ ): There  exists $i \to j$ in $\Pi_-$ such that  $i+ \to j+$ in $\Delta$ and
 so $j+ \to i+$ in $\Pi$. Since $\Delta$ is disjoint from $\Xi$, $j+ \to i-$ in $\Pi$. Hence, $\Pi$ is not reducible via
 $i- \to i+$.

 \end{proof}

 Finally, we have some  observations about reducibility, extending Lemma \ref{lem29aab}(b),(c) and (d).

 \begin{prop}\label{propmarriage2} Let $\Pi$ be a pointed game on the $(I_+,I_-)$. \begin{itemize}
 \item[(a)] If $i,j \in I_+$ or $i,j \in I_-$, then
 $\Pi$ is not reducible via $\{ i, j \}$.
 \item[(b)]If $i \in I_+$, then $\Pi$ is reducible via $i \to 0$ if and only if $i \to k$ for all
 $k \in I_+ \setminus \{ i \}$.
 \item[(c)]If $i \in I_-$, then $\Pi$ is reducible via $0 \to i$ if and only if $k \to i$ for all
 $k \in I_+ \setminus \{ i \}$.
 \end{itemize}\end{prop}

  \begin{proof} If $i, j \in I_+$, then $i, j \to 0$ and so $\Pi$ is not reducible via $\{ i, j \}$ by Proposition \ref{prop07}(b). If
  $i, j \in I_-$, then $0 \to i, j$. If $i, j \in I_+$ and $j \to i$, then since $j \to 0$, $\Pi$ is not reducible via $\{ 0, i \}$.
  Conversely, if $i \to k$ for all $k \in I_+ \setminus \{ i \}$, then $\Pi(i) = \{ 0 \} \cup (I_+ \setminus \{ i \})$ while
  $\Pi(0) = I_-$. So $\Pi$ is reducible via $\{ i, 0 \}$ by Proposition \ref{prop07}(b) again. Similarly, for $i \in I_-$.
  \end{proof}  \vspace{.5cm}

  \begin{cor}\label{cormarriage3} If $\Pi$ is a game, then $\Gamma = (2 \Pi)/\Pi_+$ is a non-reducible game. \end{cor}

  \begin{proof}  If $i \to j$ in $\Pi$ then
  $$i- \to j-, \ \ j- \to i+, \ \ j+ \to i-, \ \ j+ \to i+, \ \ i- \to i+, \ \ j- \to j+$$
  in $\Gamma$. Hence,
  $$ j+ \to i-, i+, \ \ i- \to j-, i+, \ \ j+, i- \to i+,$$
  and so $\Gamma$ is not reducible via any of these pairs. The remaining pairs are excluded by Proposition \ref{propmarriage2}.
    \end{proof}  \vspace{.5cm}

     We close the section with an improvement of Proposition \ref{prop29aa}.

     \begin{prop}\label{prop29aaimprove} If $\Pi$ is a digraph with $n$ vertices, then $\Pi$ is a subgraph of
a game of size $2n - 1$. \end{prop}

  \begin{proof} As in the proof of Proposition \ref{prop29aa} we may assume that $\Pi$ is a tournament. Let $\Pi'$ be an order on $n$ vertices
   or, more generally, a tournament on $n$ vertices with a vertex of score $n - 1$.  By Theorem \ref{theopoint} there is a pointed game
   $\Gamma$ with $\Gamma_+$ isomorphic to $\Pi'$ and with $\Gamma_-$ isomorphic to $\Pi$. If $u$ is the vertex of $\Gamma^{-1}(0)$ with
   score $n - 1$ in $\Gamma_+$, then Proposition \ref{propmarriage2} implies that $\Gamma$ is reducible via $u \to 0$. The restriction of $\Gamma$
   to the vertices excluding $u$ and $0$ is a subgame of size $2n - 1$ which contains $\Gamma_-$.

       \end{proof}  \vspace{.5cm}

       Notice that if $\Pi$ itself contains a vertex with score $n - 1$ or $0$, then the smallest possible size for a game which contains
       $\Pi$ is $2(n - 1) + 1 = 2n - 1$.  Thus, if one does not restrict the score vector of the tournament $\Pi$, Proposition \ref{prop29aaimprove}
       is the best possible result.

  \vspace{1cm}

  \section{Interchange Graphs, Again}\label{secinterchangeagain}

  Fix $I = \{ 0 \} \cup I_+ \cup I_-$, disjoint sets with $|I_{\pm}| = n$. Define $I_0 = I \setminus \{ 0 \} = I_+\cup I_-$. Let
  $n[I_0]$ denote the set of subsets of $I_0$ of cardinality $n$ so that $|n[I_0]| = {2n \choose n}$. If $Games(I)$ is the set of
  games on $I$ then $\Gamma \mapsto \Gamma(0)$ is a mapping  $\pi: Games(I) \to n(I_0)$. If $J_- \in n(I_0)$, and $J_+ = I_0 \setminus J_-$, then
  a game $\Gamma$ has $\pi(\Gamma) = J_-$ exactly when $\Gamma$ is a pointed game on $(J_+,J_-)$. Furthermore, if $\rho$ is a bijection of $I$ which
  fixes $0$ and maps $I_-$ to $J_-$, then $\rho$ is an isomorphism of any pointed game on $(I_+,I_-)$ onto a pointed game on $(J_+,J_-)$. In
  particular, for every $J_- \in n[I_0]$ the cardinality of $\pi^{-1}(J_-)$ is that of $Games(I_+,I_-)$.  In particular, we have from (\ref{pointest})
  \begin{equation}\label{interag01}
  |Games(I)| \ = \ {2n \choose n} \cdot |Games(I_+,I_-)| \ \geq \ {2n \choose n} \cdot 2^{n(n-1)}.
  \end{equation}

  Observing that ${2n \choose n} \cdot 2^{n(n-1)} = {2n \choose n} \cdot \prod_{j=1}^{n-1} \ 2^{2j}$ we see that this is an improvement on
  the bound $\prod_{j=1}^{n} {2j \choose j}$ given in Theorem 4 of \cite{HQ}. Note that ${2j \choose j} < 2^{2j}$ since the number of subsets of
  size $j$ is less than the total number of subsets.

%  This is, for the special case of games, an improvement over the lower bound given in \cite{G} which is, in this case,
%  $|Games(I)| \geq 2^{2n - 1}$. Notice that when $n = 1$ the bounds both equal $2$, the two games which are the two orientations of a $3$-cycle.
%  When $n = 2$, Gibson's lower bound is $8$, while the bound given above is ${4 \choose 2} \cdot 2^2 = 24$. This is the number of games when $n = 2$
%  because for a pointed game with $n = 2$ the bipartite tournament $\Xi$ contains no cycles. This follows from a direct computation or
%  observing that such a $\Xi$ splits. Hence, $|Games(I_+,I_-)| = 2^{2 \cdot 1}$.

  Using this inequality we can obtain a lower bound for the number of isomorphism classes of games of a fixed size.

  \begin{prop}\label{isoest} Let $IS(n)$ denote the cardinality of the set of isomorphism classes of games on a set of size $2n + 1$.
  \begin{equation}\label{iso01}
  IS(n) \ \geq \ 2^{n(n-1)} \div [(2n + 1)\cdot (n!)^2].
  \end{equation}
  If $n \geq 7$, then
    \begin{equation}\label{iso02}
  \ln (IS(n)) \ \geq \ n \cdot [ n \ln 2 \ - \ 2 \ln n ]
  \end{equation}
  \end{prop}

    \begin{proof} We obtain (\ref{iso01}) by dividing (\ref{interag01}) by the order of the permutation group which is $(2n + 1)!$.
    Now we observe that
    \begin{align}\label{iso03}
    \begin{split}
    n(n - 1) \ln 2 \ = &\ n^2 \ln 2 \ - (n + 1) \ln 2 \ + \ \ln 2, \\
    \ln (2n + 1) \ &\leq \  \ln 2 \ + \ \ln (n + 1), \\
    \ln (n + 1) \ = \ \ln n &\ + \ \ln( 1 + \frac{1}{n}) \ \leq \ \ln n \ + \ \frac{1}{n}.
    \end{split}
    \end{align}
    Furthermore,
    \begin{equation}\label{iso04}
    \begin{split}
    \ln (n!) \ \leq \ \int_2^{n + 1} \ \ln t \ dt \ = \  (n + 1)\ln(n + 1) - (n + 1) - 2 \ln 2 + 2 \ \leq \\
    n \ln n + 1 + \ln(n + 1) - (n + 1) - 2 \ln 2 + 2. \hspace{2cm}
    \end{split}
    \end{equation}

    Putting these together we obtain
    \begin{equation}\label{iso05}
     \ln (IS(n)) \ \geq \ n \cdot [ n \ln 2 \ - \ 2 \ln n ] +  [ (2 - \ln 2)(n + 1) - 3 \ln (n + 1) + 4 \ln 2 - 6].
     \end{equation}
     The function $t \mapsto (2 - \ln 2)t - 3 \ln t + 4 \ln 2 - 6$ is increasing for $t \geq 3$ and it is positive for
     $t = 8$ and so for $t \geq 8$.

  \end{proof}\vspace{.5cm}

  The set $n[I_0]$ has a natural undirected graph structure, with $(J_1,J_2)$ an edge if $|J_1 \cap J_2| = n - 1$.  That is, $J_2$ is obtained from
  $J_1$ by exchanging the element of $J_1 \setminus J_1 \cap J_2$ with the element of  $J_2 \setminus J_1 \cap J_2$ which is in the complement of
  $J_1$.  Each element of $J_1$ can be paired up with an element of its complement and each of these $n^2$ choices yields a different set $J_2$. Thus,
  $n[I_0]$ is an $n^2$ regular graph. The distance from $J_1$ to $J_2$ is $k$ when $|J_1 \cap J_2| = n - k$. In that case, there are $(k!)^2$
  geodesics between $J_1$ and $J_2$.  These are obtained by choosing an ordering on $J_1 \setminus (J_1 \cap J_2)$ and on  $J_2 \setminus (J_1 \cap J_2)$
  for the $k$ exchanges.

  Now suppose $(\Gamma_1,\Gamma_2)$ is an edge in the interchange graph on $Games(I)$.  That is, $\Gamma_2$ is obtained from $\Gamma_1$ by reversing
  some $3$-cycle $\langle i, j, k \rangle$. Assume that  $\Gamma_1$ is a pointed graph on $(I_+,I_-)$. If $0$ is not a vertex of the cycle then
  the cycle intersects either $\Pi_+$ or $\Pi_-$ but not both.  It must meet one of them because the bipartite tournament $\Xi$ contains no $3$-cycle.
 It cannot meet both because a cycle which intersects both $\Pi_+$ and $\Pi_-$ contains at least four vertices. In this case, $\Gamma_2$ is also a pointed graph on $(I_+,I_-)$.
  Thus, $\pi(\Gamma_1) = \pi(\Gamma_2)$.

  If $0$ is a vertex of the cycle then the cycle is $\langle i, 0, k \rangle$ with $i \in I_+, k \in I_-$ and $(k,i)$ an edge in $\Xi$.
  In that case, $\Gamma_2(0) = I_- \cup \{ i \} \setminus \{ k \}$. That is, $\pi(\Gamma_2)$ is connected by the edge in $n[I_0]$ with
  $\pi(\Gamma_1)$ via the interchange of $k$ with $i$.  By Theorem \ref{theo08} the vertex $0$ is contained in $n(n+1)/2$ $3$-cycles. Each of these leads to
  a different element of $n[I_0]$. Thus, the edges from $\Gamma_1$ project to $n(n+1)/2$ of the $n^2$ edges from $I_-$ in $n[I_0]$. On the other hand,
  if $J = I_- \cup \{ i_1 \} \setminus \{ k_1 \}$ for arbitrary $i_1 \in I_-, k_1 \in I_+$ then the product of transpositions $(i_1,i)$ and $(k_1,k)$ is
  a permutation $\rho$ which induces an isomorphism from $\Gamma_1$ to $\hat \Gamma_1$ which has $\pi(\Gamma_1) = \pi(\hat \Gamma_1)$ and
  $\langle i_1, 0, k_1 \rangle$ is a cycle of $\hat \Gamma_1$.

  For an undirected graph $G$ we will call a set $T$ of vertices \emph{convex}\index{convex} when for all
  $t_1, t_2 \in T$ there exists a geodesic of $G$ between $t_1$ and $t_2$ and for every geodesic
  of $G$ between $t_1$ and $t_2$ the vertices are all contained in $T$.

\begin{theo}\label{theoconvex}  Let $\Pi$ be a game on a set $I$ of vertices and let $Q \subset I$. Define $Q^* \subset \Pi$ by
$(i,j) \in Q^*$ if and only if $i \in Q$ or $j \in Q$ (or both). Let $Games(Q^*)$ be the set of games $\Gamma$ on $I$ such that $Q^* \subset \Gamma$.
That is, every edge which connects to a vertex of $Q$ has the same orientation in $\Gamma$ as in $\Pi$.  The set $Games(Q^*)$ is a convex subset
of the interchange graph of all games on $I$.
 \end{theo}

 \begin{proof}: Clearly a game $\Gamma$ lies in $Games(Q^*)$ if and only if the Eulerian subgraph $\Delta(\Gamma,\Pi)$ is disjoint from $Q^*$.
 Recall that by Corollary \ref{cor29x} the distance between $\Gamma$ and $\Pi$ is $\b(\Delta(\Gamma,\Pi))$.

 Assume that $\Delta(\Gamma,\Pi)$ is disjoint from $Q^*$ and that we reverse a $3$-cycle which meets $Q^*$ to obtain $\Gamma'$ which is
 not in $Games(Q^*)$.  It suffices to show that $\b(\Delta(\Gamma',\Pi)) = \b(\Delta(\Gamma,\Pi)) + 1$ for then $\Gamma'$ cannot lie on
 a geodesic from $\Gamma$ to $\Pi$.  We consider the cases from  Theorem \ref{theo28x}.

 Let the reverse of the given $3$-cycle be $\langle i_1,i_2,i_3 \rangle$ so that $\langle i_3,i_2,i_1 \rangle$ is in $\Gamma$.
 By assumption at least one of the vertices, say $i_2$, is in $Q$. Hence, $(i_3,i_2), (i_2,i_1) \in Q^*$ and the vertex $i_2$ does not
 occur in $\Delta(\Gamma,\Pi)$. Hence, $(i_2,i_3),(i_1,i_2) \in \Delta(\Gamma',\Pi)$ and these are the only edges of $\Delta(\Gamma',\Pi)$
 which contain the vertex $i_2$. Hence, a cycle which contains $i_2$ from any decomposition for $\Delta(\Gamma',\Pi)$
 must contain both these edges.

 In the notation of the proof of Theorem \ref{theo28x} we first consider Case 1, with the cycle $\langle i_3,i_2,i_1 \rangle$ disjoint from
 $\Delta(\Gamma,\Pi)$ and with  $(i_2,i_3),(i_1,i_2)$  in a single cycle of the maximum decomposition for $\Delta(\Gamma',\Pi)$.  As shown there,
 $\b(\Delta(\Gamma',\Pi)) = \b(\Delta(\Gamma,\Pi)) + 1$.

 Alternatively, we could be in Case 3, with $(i_1,i_3) \in \Delta(\Gamma,\Pi)$ and so neither $i_1$ nor $i_3$ is in $Q$.  Again since $(i_2,i_3),(i_1,i_2)$
 are in a single cycle of the decomposition of $\Delta(\Gamma',\Pi)$ we again get $\b(\Delta(\Gamma',\Pi)) = \b(\Delta(\Gamma,\Pi)) + 1$.

 \end{proof}  \vspace{.5cm}

 \begin{cor}\label{corconvex} With $I = \{ 0 \} \cup I_+ \cup I_-$, the set $Games(I_+,I_-)$ is a convex subset of the interchange graph of games
 on $I$. For each $J \in n[I_0]$ the set $\pi^{-1}(J)$  is a convex subset of the interchange graph of games
 on $I$. \end{cor}

 \begin{proof} The above theorem applies with $Q = \{ 0 \}$.
 \end{proof}

\vspace{1cm}

\section{Coset Space Games and Actions on Games}\label{sechomogeneous}

For a subgroup $H$ of a group $G$ the \emph{double coset}\index{double coset} of $i \in G$ is the set $HiH$.
Clearly, $\{ (i,j) \in G \times G : i \in HjH \}$ is
an equivalence relation with equivalence classes the double cosets. Of course, $H$ itself is the double coset of the identity element $e$.

\begin{lem}\label{lem18} Let $G$ be a finite group with odd order and let $H$ be a subgroup of $G$. Let $i \in G$.

(a) If $i^2 \in H$ then $i \in H$.

(b) If $i \not\in H$ then $i^{-1} \not\in HiH$.
\end{lem}

\begin{proof} By Lagrange's Theorem $H$ and $i$ have odd order.

(a) Since $2$ is relatively prime to the order of $i$, $i^2$ is a generator of the cyclic group generated by $i$.  Hence, if $i^2 \in H$,
then $i$ is in the subgroup generated by $i^2$ which is contained in the subgroup $H$.

(b) Assume $i^{-1} \in HiH$ and so $i^{-1} = h_1ih_2$ for some $h_1, h_2 \in H$.  Then $(ih_1)(ih_1) = h_2^{-1}h_1 \in H$. So by (a),
$ih_1 \in H$ and $i = (ih_1)h_1^{-1} \in H$.

\end{proof} \vspace{.5cm}

\begin{df}\label{def19}  Let $G$ be a finite group with odd order and let $H$ be a subgroup of $G$. A subset $A$ of $G$ is a
game subset for the pair $(G,H)$\index{game subset for $(G,H)$} if it is a game subset of $G$
such that $i \in A \setminus H$ implies $HiH \subset A$. \end{df}
 \vspace{.5cm}

\begin{theo}\label{theo20} If $G$ is a finite group with odd order and  $H$ is a subgroup of $G$, then there exist
game subsets for $(G,H)$.\end{theo}

\begin{proof} By Lemma \ref{lem18} (b) the double cosets $HiH$ and $Hi^{-1}H$ are distinct for all $i \in G \setminus H$.
Let $T$ be the set of double
cosets other than $H$. We can partition $T$ by pairs $\{ \{HiH, Hi^{-1}H \}: i \in G \setminus H \}$. Choose $i_1, \dots, i_k \in G \setminus H$
so that $\{Hi_1H,\dots, Hi_kH \}$ includes exactly one double coset from each pair. Let $A_0 \subset H$ be a game subset for the odd order
group $H$.  Let $A = A_0 \ \cup \ \bigcup_{p = 1}^k \ Hi_pH $. Since $H = A_0 \cup A_0^{-1} \cup \{e \}$ and since $Hi^{-1}H$ consists of the
inverses of the elements of $HiH$, it follows that $A$ is a game subset. Clearly, $A \setminus H$ is a union of double cosets.

\end{proof} \vspace{.5cm}

{\bfseries Remark:} It is clear that this construction yields all the game subsets for $(G,H)$. Hence, if $2d$ is the number of double cosets
in $G \setminus H$ and $2k + 1$ is the order of $H$, then there are $2^{d + k}$ game subsets for $(G,H)$.
\vspace{.5cm}

\begin{theo}\label{theo21}  Let $G$ be a finite group with odd order, $H$ be a subgroup of $G$ and $A$ be a
game subset for $(G,H)$. Let $G/H$\index{$G/H$} be the coset space\index{coset space} of the left cosets, i.e. $G/H = \{ iH : i \in G \}$.
Define $A/H = \{ iH : iH \subset A \}$. The set $\Gamma[A/H] = \{ (iH, jH) : i^{-1}jH \in A/H \}$\index{$\Gamma[A/H]$} is a game on $G/H$.

For each $k \in G$ the bijection $\ell_k$ on $G/H$ given by $iH \mapsto kiH$ is an automorphism of $\Gamma[A/H]$ and so
there is a group homomorphism from $G$ to $Aut(\Gamma[A/H])$.

The surjection $\pi : G \tto G/H$ given by $i \mapsto iH$ is a morphism from $\Gamma[A]$ to $\Gamma[A/H]$.
\end{theo}

\begin{proof} Notice first that $iH \subset A$ requires $i \in G \setminus H$ since $A \cap H$ is a proper subset of $H$ (e.g. $e \not\in A$).
Hence if $i^{-1}jH \subset A$ then with $\hat i = ih_1, \hat j = jh_2$ then $\hat i^{-1}\hat j H \subset A$ because $A \setminus H$ is
a union double cosets.  Thus, $iH \to jH$ if and only if $i^{-1}j \in A \setminus H$.
Notice that $iH = \pi(i) = \pi(j) = jH$ if and only if $i^{-1}j \in H$.
Thus, we see that $\Gamma[A/H]$ is a tournament and that $\pi$ is a morphism from $\Gamma[A]$ to $\Gamma[A/H]$.

We see that $iH \to jH$ if and only if $j \in i(A \setminus H)$ if and only if $i \in j(A \setminus H)^{-1}$. Hence, the set of inputs and the set of
outputs of $iH$ with respect to $\Gamma[A/H]$ both have cardinality $|A \setminus H|/|H|$. Hence, $\Gamma[A/H]$ is Eulerian and so is
a game.

Finally, it is clear that $\ell_k$ is an automorphism of $\Gamma[A/H]$.

\end{proof} \vspace{.5cm}

{\bfseries Remark:} Clearly, $\ell_k$ acts as the identity on $G/H$ if and only if $iki^{-1} \in H$ for all $i \in G$. So the action of
$G$ on $G/H$ is \emph{effective}\index{action!effective}\index{effective action}, i.e. $\ell_k$ acts as the
identity only for $k = e$, exactly when $\hat H \ = \ \bigcap_{i \in G} iHi^{-1}$
is the trivial subgroup, or, equivalently, when $\{ e \}$ is the only subgroup of $H$ which is normal in $G$. In general, $G/\hat H$
acts effectively on $G/H$ and so injects into $Aut(\Gamma[A/H])$.
\vspace{.5cm}

We call the game $\Gamma[A/H]$ a \emph{coset space game}\index{coset space game}\index{game!coset space}. Of course, a group
game is a special case of a coset space game with $H$ the trivial subgroup. A game subset for $G$ is a game subset
for $(G,\{ e \})$.

\begin{prop}\label{prop21aa} Let $\pi: G \to T$ be a surjective group homomorphism with kernel $H  = \{a \in G : \pi(a) = e \}$.
If $A \subset G$ and $B \subset T$ are game subsets with
associated games $\Gamma[A]$ on $G$ and $\Gamma[B]$ on $T$, then the following are equivalent.
\begin{itemize}
\item[(i)] $\pi$ is a tournament morphism from $\Gamma[A]$ to $\Gamma[B]$.
\item[(ii)] $\pi(A) \subset B \cup \{e \}$.
\item[(iii)] $\pi^{-1}(B) \subset A$.
\item[(iv)] $A = \pi^{-1}(B) \cup (A \cap H)$.\end{itemize}
\end{prop}

\begin{proof} Because $\pi$ is a surjective group homomorphism, it is clear that $\pi$ is a tournament
morphism if and only if for all $a \in G$   $a \in A\setminus H \Leftrightarrow \pi(a) \in B$.

Notice that $H = \pi^{-1}(e)$ and $ \{ A \setminus H, A^{-1} \setminus H \}$ partition $G \setminus H$ and $\{ B, B^{-1} \}$
partition $T \setminus \{ e \}$.  Furthermore, $\pi(a^{-1}) = \pi(a)^{-1}$. Using these facts, it is easy to check that
each of (ii), (iii) and (iv) is equivalent to (i).

\end{proof} \vspace{.5cm}

\begin{theo}\label{theo21a}  Let $G$ be a finite group with odd order and $H$ be a normal subgroup of $G$ so that $\pi : G \tto G/H$ is
a group homomorphism onto the quotient group. A subset $A$ of $G$ is a game subset for $(G,H)$ if and only if there exist $B$ a game subset for
$G/H$ and $A_0$ a game subset of $H$ so that $A = A_0 \cup \pi^{-1}(B)$. In that case, the games $\Gamma[A/H]$ and $\Gamma[B]$ are equal.
\end{theo}

\begin{proof} When $H$ is normal, a double coset is just a coset. Thus, $A$ is a game subset for $(G,H)$ if and only if $A_0 \cap A$ is a game
subset for $H$ and $A \setminus H$ is a union of cosets. Normality of $H$ implies that $(iH)^{-1}jH = i^{-1}jH$. So $iH \to jH$
in $\Gamma[A/H]$ if and only if $i^{-1}jH \in \pi(A) \setminus \{ e \}$. Thus, $\pi(A)\setminus \{ e \} = B$ is a game subset with
$A \setminus H = \pi^{-1}(B)$.

\end{proof} \vspace{.5cm}

Recall that $A$ is a normal game subset when it is invariant with respect to the adjoint action of $G$ on itself via inner automorphisms.

\begin{add}\label{add21b} Assume that  $A = A_0 \cup \pi^{-1}(B)$ is a game subset for the pair $(G,H)$ with $H$ a normal subgroup.
The game subset $A$ is normal for $G$ if and only if $B$ is a normal game subset for $G/H$ and $A_0$ is invariant with respect
to the adjoint action of $G$. In particular, $A_0$ is a normal game subset for $H$.\end{add}

\begin{proof} Because the projection $\pi$ is a surjective group homomorphism, it is clear that $B$ is invariant with respect to
the adjoint action of $G/H$, i.e. $B$ is a normal game subset for $G/H$, if and only if $\pi^{-1}(B)$ is invariant with respect
to the adjoint action of $G$.

The normal subgroup $H $ of $G$  is invariant with respect to the adjoint action. Hence, $A$ is adjoint invariant if and only if
both $A_0 = A \cap H$ and $ \pi^{-1}(B) = A \cap (G \setminus H)$ are. Notice that invariance with respect to the $G$ adjoint action
is a stronger demand than invariance for the $H$ adjoint action. Nonetheless, since $H \setminus \{ e \}$ is partitioned by orbits of
the $G$ adjoint action, and the orbit of $h$ is disjoint from the orbit of $h^{-1}$ (see Lemma \ref{lem11norma}(b)) for $h \not= e$, it
follows that there always exist game subsets $A_0$ for $H$ which are invariant with respect to the $G$ adjoint action.

\end{proof} \vspace{.5cm}

Now let $\Pi$ be a game on $I$ and let $a \in I$. An action of a group $G$ on $\Pi$ is an action $\Phi: G \times I \tto I$ on $I$ such that each
$\Phi^g$ is an automorphism of $\Pi$. That is, $i \to j$ implies $gi \to gj$ for all $g \in G$.  Thus, an action of $G$ on
$\Pi$ is given by a group homomorphism $\Phi^{\#}:  G \tto Aut(\Pi)$.

For $i \in I$ the evaluation map $\Phi_i : G \tto I$ is defined by $\Phi_i(g) = g \cdot i$.
$Iso_i = \{ g : g \cdot i = i \} = \Phi_i^{-1}( \{ i \})$\index{$Iso_i$} is a subgroup of $G$
called the \emph{isotropy subgroup}\index{isotropy subgroup} of $i$.
We write $Gi = \Phi_i(G)$\index{$Gi$} for the $G$ orbit of $i$ so that
$\Pi|Gi = \Pi \cap (Gi \times Gi)$ is the restriction of $\Pi$ to the orbit of $i$. Of course,
$G$ acts transitively on $I$ exactly when $Gi = I$ in which case $\Pi|Gi = \Pi$.

Assume that $A$ is a game subset for $G$ with $\Gamma[A]$ the associated group game on $G$. $\Phi$ is an action of $\Gamma[A]$
\index{action!$\Gamma[A]$} on $\Pi$ when it is
an action of $G$ on $\Pi$ and, in addition, $\Phi_i : G \tto I$ is a morphism from $\Gamma[A]$ to $\Pi$ for every $i \in G$.   That is, for $i, j \in I$
and $g, h \in G$
\begin{equation}\label{eqact1} \begin{split}
i \to j \ \Rightarrow \ gi \to gj  \qquad [\Phi^g : \Pi \tto \Pi \ \ \text{is an automorphism}], \hspace{1cm}\\
  gi \not= hi,  \ g \to h \ \Rightarrow \ gi \to hi  \qquad [\Phi_i : \Gamma[A] \tto \Pi \ \ \text{is a morphism}].
\end{split}\end{equation}

Notice that if the action is free, then $g \to h$ implies  $gi \not= hi$ and so the above implies $gi \to hi$ for all $i \in I$.

At the other extreme, if $\Phi: G \times I \to I$ is the trivial action\index{trivial action} \index{action!trivial} given by $gi = i$ for all $(g,i)$,
then $\Phi$ is an action of $\Gamma[A]$ on $\Pi$ for any game subset $A$.

\begin{prop}\label{prop21d} If $A$ is a game subset for $G$, then the left translation action of $G$ on itself is an action $G$ on $\Gamma[A]$.
It is an action of $\Gamma[A]$ on itself
if and only if $A$ is a normal game subset. \end{prop}

\begin{proof} The first statement just recalls the structure of a group game. The second is immediate from Proposition \ref{prop11normb}.

\end{proof} \vspace{.5cm}

%The evaluation map $\i_a : Aut(\Pi) \to I$ is defined by $\i_a(\r) = \r(a)$.
%$Iso_a = \{ \r : \r(a) = a \} = \i_a^{-1}( \{ a \})$ is a subgroup of $Aut(\Pi)$ called the \emph{isotropy subgroup} of $a$.
%Let $I_a = \i_a(Aut(\Pi)) \subset I$ and let $\Pi_a = \Pi \cap (I_a \times I_a)$ be the restriction of $\Pi$ to $I_a$. Of course,
%$Aut(\Pi)$ acts transitively on $I$ exactly when $I_a = I$ in which case $\Pi_a = \Pi$.

\begin{theo}\label{theo22}  Let $\Pi$ be a game on $I$ and $a$ be a fixed element of $I$. Assume that $G$ is a finite group of odd order
and $\Phi: G \times I \tto I$ is an action of $G$  on $\Pi$.
For example, $G$ can be any subgroup of $Aut(\Pi)$ with $\Phi$ the evaluation map so that $\Phi^{\#}$ is the inclusion map.   Let
$H = Iso_a = \Phi_a^{-1}( \{ a \})$. Choose $A_0$ a game subset for $H$ and let $A = A_0 \cup \Phi_a^{-1}(\Pi(a))$.

The set $A \subset G$ is a game subset for $(G,H)$. Let $\pi : G \tto G/H$ be the canonical projection. The map $\Phi_a$ is a
morphism from $\Gamma[A]$ to $\Pi$ and it factors through $\pi$ to define $\theta_a : G/H \tto I$ which is an embedding
from $\Gamma[A/H]$ to $\Pi$.  The restriction $\Pi|Ga$ of $\Pi$ to $Ga$ is a subgame of $\Pi$ and the bijection $\theta_a : G/H \tto Ga$ is
an isomorphism from $\Gamma[A/H]$ to $\Pi|Ga$. \end{theo}

\begin{proof} Observe first that by Proposition \ref{prop04} every element of $Aut(\Pi)$ has odd order and so by the first Sylow Theorem,
$Aut(G)$ itself has odd order. By replacing $G$ by  $\Phi^{\#}(G) \to Aut(\Pi)$ we may assume that $G$ is a subgroup of $Aut(\Pi)$.
Technically, we use Theorem \ref{theo21} because, instead of $G$, we are using its quotient by the kernel of this map.

Clearly, for $b \in I$, if $g(a) = b$ then $\Phi_a^{-1}(\{ b \}) = gH$. Now suppose that $a \to b$ in $\Pi$. For $h \in H$,
$h^{-1}(a) = a \to b = g(a)$. Since $h$ acts as an automorphism of $\Pi$, $a = h h^{-1}(a) \to h(b) = h g(a)$. Hence,
$h g \in \Phi_a^{-1}(\Pi(a))$.  It follows that $\Phi_a^{-1}(\Pi(a))$ is a union of $H$ double cosets. Since $g$ is an automorphism,
$g^{-1}(a) \to g^{-1}g(a) = a$. Hence, $g^{-1} \not\in A$. Finally, if $g \not\in H$, i.e. $g(a) \not= a$ then either
$a \to g(a)$ and so $g \in A$ or $g(a) \to a$ and, as before, $a \to g^{-1}(a)$ which implies $g^{-1} \in A$. It follows that
$A$ is a  game subset for $(G,H)$. For $g, k \in G$, $k \to g$ in $\Gamma[A]$ if and only if $ k^{-1}g \in A$. So, when $k(a) \not= g(a)$,
\begin{equation}\label{eq3}
k \to g \  \Leftrightarrow \ k^{-1}g \in \Phi_a^{-1}(\Pi(a)) \ \Leftrightarrow \ a \to k^{-1}g(a) \ \Leftrightarrow \  k(a) \to g(a).
\end{equation}

This says that $\Phi_a$ is a morphism from $\Gamma[A]$ to $\Pi$.

Since  $g(a) = b$ implies $\Phi_a^{-1}(\{ b \}) = g H$, it follows that $\Phi_a$ factors through $\pi$ to define the injection $\theta_a$.

Since $\pi$ is a surjective morphism from $\Gamma[A]$ to $\Gamma[A/H]$ and $\Phi_a$  is a morphism from $\Gamma[A]$ to $\Pi$, it easily follows
that $\theta_a$ is a morphism from $\Gamma[A/H]$ to $\Pi$. Since $\Gamma[A/H]$ is Eulerian and $\bar \theta_a : \Gamma[A/H] \tto \Pi|Ga$
is a bijection, it follows that $\Pi_a$ is Eulerian and so is a subgame of $\Pi$ with $\theta_a : G/H \tto Ga$
defining an isomorphism from $\Gamma[A/H]$ to $\Pi|Ga$.

\end{proof} \vspace{.5cm}

We immediately obtain the following.

\begin{cor}\label{cor23} A game  is point-symmetric, if and only if it is isomorphic to a coset space game. \end{cor}

$\Box$ \vspace{.5cm}

Any finite partially ordered set has minimal elements and so if  $\Pi$ is point-symmetric, there is a subgroup $G \subset Aut(\Pi)$
which acts transitively the vertices of $\Pi$ and which is minimal among such subgroups.

In general, for a game $\Pi$ on $I$ the restriction of $\Pi$ to each orbit of the action of $Aut(\Pi)$ on $I$ is a subgame
isomorphic to a coset space game.

If $\Phi: G \times I \to I$ is a free action of $G$ on $\Pi$, for each $i \in I$, $\Phi_i : G \tto Gi$ is an isomorphism of
a group game on $G$ with $\Pi|Gi$. While this follows from  Theorem \ref{theo22},  it also follows directly from Sabidussi's Theorem
\ref{theo11}(a). If $\Phi$ is a free action of $\Gamma[A]$ on $\Pi$, then each $\Phi_i$ is an isomorphism from $\Gamma[A]$ onto $\Pi|Gi$.

\begin{cor}\label{cor23aaa} Assume that $\xi$ is an automorphism of a game $\Pi$ on $I$, so that $\xi$ is a permutation of $I$.
Assume that $(a_0,\dots,a_{2n})$ is a nontrivial permutation cycle in the permutation $\xi$, so that $n > 1$.
The restriction $\Pi|\{ a_0,\dots, a_{2n} \}$ is
a subgame of $\Pi$ which is isomorphic to a group game on $\Z_{2n+1}$. \end{cor}

\begin{proof} By Proposition \ref{prop04}, $\xi$ has odd order and so every
cycle contained in it has odd length. Let $G$ be the cyclic subgroup of $Aut(\Pi)$ generated by $\xi$ so that the orbit $Ga_0 = \{ a_0,\dots, a_{2n} \}$.

It follows from
Theorem \ref{theo22} that the restriction $\Pi|Ga_0$ is isomorphic to a coset space game on $G/Iso_{a_0}$. Since $G$ is cyclic, the coset space
is a quotient group isomorphic to  $\Z_{2n+1}$ and by Theorem \ref{theo21a} such a coset space game is a group game on the quotient group.

This can also be proved directly by mapping $\Z_{2n+1}$ onto $(a_0,\dots,a_{2n})$ by $i \mapsto a_i$ and applying Theorem \ref{theo11} (a)
to the game pulled back to $\Z_{2n+1}$.
%It follows from
%Theorem \ref{theo22} that the restriction is a subgame of $\Pi$. The map $i \mapsto a_i$ for $i = 0,\dots, 2n$  is a bijection
%$\r: \Z_{2n+1}  \to Ga_0$ which maps the translation $\ell_1$ to $\xi$.
%
%We define the game $\Gamma$  on $\Z_{2n+1}$ so that $\r$ is an isomorphism
%from $\Gamma$ to $\Pi|Ga$,
%i.e. $i \to j$ if and only if $a_i \to a_j$. Then $\ell_1$ is an automorphism of $\Gamma$ and so the translations of $\Z_{2n+1}$ are all automorphisms.
%It follows from Theorem \ref{theo11} (a) that $\Gamma$ is a group game on $\Z_{2n+1}$.

\end{proof} \vspace{.5cm}

\begin{theo}\label{theo23a}  Let $G$ be a finite group with odd order, $H$ be a subgroup of $G$ and $A$ be a
game subset for $(G,H)$ with $A_0 = A \cap H$ the game subset of $H$.
The game $\Gamma[A]$ is isomorphic to the lexicographic product $\Gamma[A/H] \ltimes \Gamma[A_0]$.
In particular, $Aut(\Gamma[A])$ is isomorphic to the semi-direct product $Aut(\Gamma[A/H]) \ltimes Aut(\Gamma[A_0])^{G/H}$.
\end{theo}

\begin{proof} Let $j : G/H \tto G$ be a map such that $\pi \circ j = 1_{G/H}$. So if $ x \in G/H$ then
$x$ is the coset $j(x)H$. We identify $G$ with the product $G/H \times H$ by the bijection $(x,h) \mapsto j(x)h$. This identifies
$\pi : G \tto G/H$ with the first coordinate projection. Notice that we are not assuming that $H$ is normal and so $G/H$ need not be a group.
Even if it is normal, a group homomorphism splitting $j$ need not exist.  However, we do not need any algebraic conditions on $j$.

As was observed in the proof of Theorem \ref{theo21}, if $x_1 \not= x_2$ then
$x_1 \to x_2$ in $\Gamma[A/H]$ if and only if $h_1^{-1}j(x_1)^{-1}j(x_2)h_2 \in A \setminus H$ for some $h_1, h_2 \in H$ and so for all $h_1, h_2 \in H$
since $A \setminus H$ is a union of double cosets. Hence,
\begin{equation}\label{h01}
(x_1, h_1) \to (x_2, h_2) \ \Leftrightarrow \ x_1 \to x_2 \qquad \text{when} \quad x_1 \not= x_2.
\end{equation}

On the other hand, if $x_1 = x_2$ then $(j(x_1)h_1)^{-1}(j(x_2)h_2) = h_1^{-1}h_2$. So $(j(x_1)h_1)^{-1}(j(x_2)h_2) \in A$ if and only if
$h_1^{-1}h_2 \in A_0$. Thus
\begin{equation}\label{h02}
(x_1, h_1) \to (x_2, h_2) \ \Leftrightarrow \ h_1 \to h_2 \qquad \text{when} \quad x_1 = x_2.
\end{equation}

From (\ref{eq18}) we see that this is exactly the lexicographic product $\Gamma[A/H] \ltimes \Gamma[A_0]$.
The automorphism result then follows from (\ref{eq21}).

\end{proof} \vspace{.5cm}

\begin{cor}\label{cor23bb} (a) The lexicographic product of two group games is isomorphic to a group game.

(b) The lexicographic product of two point-symmetric games is point-symmetric

(c) The lexicographic product of two coset space games is isomorphic to a coset space game. \end{cor}

\begin{proof} (a): Let $A_0$ be a game subset of a group $G$ and $B$ be a game subset of a group $T$. Let $R$ be any extension of $G$ by $T$. That is,
there is a short exact sequence
$\begin{tikzcd}
G \arrow[r,"i"] & R \arrow[r,"p"] & T.
\end{tikzcd}$ Let $A = i(A_0) \cup p^{-1}(B)$.
By Theorems \ref{theo21a}  and \ref{theo23a} $A$ is a game subset of $R$ and
$\Gamma[A]$ is isomorphic to $ \Gamma[B] \ltimes \Gamma[A_0]$.

For example, when $R$ is the product group $T \times G$, then $A = (\{ e_T \} \times A_0) \cup (B \times G)$. We have
\begin{equation}\label{h02a}
(t_1,g_1) \to (t_2,g_2) \ \Longleftrightarrow \ \begin{cases} \ \ t_1 \to t_2 \\ t_1 = t_2, g_1 \to g_2. \end{cases}
\end{equation}
and so $\Gamma[A] = \Gamma[B] \ltimes \Gamma[A_0]$.

(b): If $p, q \in \Gamma \ltimes \Pi$ and $\rho \in Aut(\Gamma)$, $\phi \in Aut(\Pi)$ satisfy $\rho(p_1) = q_1$ and $\phi(p_2) = q_2$ then
with $\gamma_i = \phi$ for all $i \in I$, $\rho \ltimes \gamma(p) = q$. Thus, $Aut(\Gamma \ltimes \Pi)$ acts transitively when
 $Aut(\Gamma)$, and $Aut(\Pi)$ do.

 (c): By Corollary \ref{cor23} a game  is isomorphic to a coset space game if and only if it is point-symmetric.

 Directly, if $A_0$ is a game subset for the pair $(G,H)$ and $B$ is a game subset for the pair $(T,K)$, then
 $A = (\{ e_T \} \times A_0) \cup (B \times G)$ is a game subset for the pair $(T \times G, K \times H)$ and $\Gamma[A/(K \times H)] =
 \Gamma[B/K] \ltimes \Gamma[A_0/H]$.

\end{proof} \vspace{.5cm}

{\bfseries Remark:} If $G, T$ are groups of odd order and $\begin{tikzcd}
G \arrow[r,"i_1"] & R_1 \arrow[r,"p_1"] & T
\end{tikzcd}$ and $\begin{tikzcd}
G \arrow[r,"i_2"] & R_2 \arrow[r,"p_2"] & T
\end{tikzcd}$ are possibly different group extensions, it follows
from the above proof that the game $\Gamma[A_1]$ is isomorphic to $\Gamma[A_2]$  when $A_{\ep} = i_{\ep}(A_0) \cup p_{\ep}^{-1}(B)$
for $\ep = 1,2$. \vspace{.5cm}

\begin{ex}\label{excomm} Commutative group examples. \end{ex}

With $G = \Z_{(2a+1)(2b+1)}$ we define the injection $\theta: \Z_{2b+1} \tto G$  by $\theta(j) = j(2a+1)$ for $j = 0, \dots, 2b$ and let
 $\pi : \Z_{(2a+1)(2b+1)} \tto \Z_{(2a+1)}$ be the surjection with $\pi(j(2a + 1) + i) = i$ for $i = 0,\dots,2a, \ j  = 0,\dots,2b$.

We  identify $\Z_{2b+1}$ with the subgroup $H = \theta(\Z_{2b+1})$ generated by $2a +1$ in $\Z_{(2a+1)(2b+1)}$ and we
 identify $\Z_{2a+1}$ with the quotient group $G/H$.

If $B \subset \Z_{2a+1}, A_0 \subset \Z_{2b+1}$ are game subsets then $A = A_0 \cup \pi^{-1}(B)$ satisfies
\begin{equation}\label{eq26}
j(2a + 1) + i \in A \quad \Longleftrightarrow \quad
\begin{cases} i \in B \quad \text{or}\\ i = 0 \ \text{and} \   j \ \in A_0. \end{cases}.
\end{equation}

By Theorems \ref{theo23a} and \ref{theo21a}  $\Gamma[A]$ is the lexicographic product $\Gamma[B] \ltimes \Gamma[A_0]$.

Furthermore, the translation map $\ell_1$ on $\Z_{(2a+1)(2b+1)}$ is given by $\r \triangleleft \gamma$ with
$\r = \ell_1$ on $\Z_{2a+1}$ and $\gamma_i$ equal to $ \ell_1$ on $\Z_{2b+1}$ for $i = 2n$ and equal to the identity on
$\Z_{2b+1}$ for the remaining $i$.

It follows that if $2n+1$ is  composite $= (2a+1)(2b+1)$, then there exist game subsets
such that the automorphism group of the associated game is non-abelian and so
contains $\Z_{2n+1}$ as proper subgroup. Notice that by considering
translations alone for $\r$ and the $\gamma_i$'s we see that the order of
$Aut(\Gamma[B] \ltimes  \Gamma[A_0])$
is at least $(2a + 1)(2b + 1)^{2a+1}$.

The affine group on $\Z_p$ is generated by translations and multiplications by units, i.e. $i \mapsto a \cdot i + k$ for
$k \in \Z_p, a \in \Z_p^*$.

The entire affine group for $\Z_{(2a+1)(2b+1)}$ has
order $(2a+1)(2b+1)\cdot \phi((2a+1)(2b+1)) \ < \ (2a+1)^2(2b+1)^2$. Note that if $b \geq 1$ and $a \geq 2$
then $2 a + 1 < 3^a \leq (2b + 1)^a < (2b + 1)^{2a - 1}$. Hence, in these examples
there are always automorphisms which are
not affine. Observe that if $2a+1$ and $2b+1$ are distinct Fermat primes, e.g. $3$ and $5$, then by  Theorem \ref{theo16}
the translations are the only affine automorphisms of any $\Gamma[B] \ltimes \Gamma[A_0]$.

If $2a + 1$ and $2b +1$ are relatively prime then the product group $\Z_{2a+1} \times \Z_{2b+1}$ is isomorphic as a group to
$\Z_{(2a+1)(2b+1)}$.  If $2a + 1$ and $2b +1$ are not relatively prime, e.g. if they are equal, then the product group $G = \Z_{2a+1} \times \Z_{2b+1}$
is not isomorphic to $\Z_{(2a+1)(2b+1)}$.  Nonetheless, it is an extension of  $\Z_{2b+1}$ by $\Z_{2a+1}$ and so has game subsets
isomorphic to the lexicographic product $\Gamma[B] \ltimes  \Gamma[A_0]$.

Finally, we note that, by induction, the Steiner game $\Gamma_k$ described at the
end of Section \ref{secdoublelex} is isomorphic to a group game on $\Z_{3^k}$. Define
$L_1$ to be the set of natural numbers such that the first nonzero digit in the base three expansion is $1$ (rather than $2$).
Using induction again, one can show that $A = \{ i \in L_1 : 0 < i < 3^k \} \subset \Z_{3^k}$ is an example with $\Gamma[A]$ isomorphic to $\Gamma_k$.

\begin{prop}\label{prop23bc}  For $\ep = 1, 2$ assume that $\Pi_{\ep}$ is a tournament on $I_{\ep}$
and that  $\Phi_{\ep}: G \times I_{\ep} \tto I_{\ep}$ are actions. The \emph{diagonal action}\index{diagonal action}\index{action!diagonal}
 $\Phi : G \times I_1 \times I_2 \tto I_1 \times I_2$ is given by $g (i,j) = (gi,gj)$. It is effective (or free) if either of the factor actions
 is effective (resp. free).

 If $\Phi_{\ep}$ is an action of $G$ on $\Pi_{\ep}$ for $\ep = 1, 2$, then $\Phi$ is an action of $G$ on $\Pi_1 \ltimes \Pi_2$.

 If $A$ is a game subset of $G$ with associated group game $\Gamma[A]$ and $\Phi_{\ep}$ is an action of $\Gamma[A]$ on $\Pi_{\ep}$ for $\ep = 1, 2$,
 then $\Phi$ is an action of $\Gamma[A]$ on $\Pi_1 \ltimes \Pi_2$. \end{prop}

 \begin{proof} The easy checks are left to the reader.

 \end{proof} \vspace{.5cm}

 For example, if $G$ acts trivially on a tournament $\Pi$ and acts by left translation on $\Gamma[A]$ with $A$ a game subset of $G$ then
 the diagonal action is a free action of $G$ on $\Gamma[A] \ltimes \Pi$.  If $A$ is a normal game subset then the diagonal action
 is an action of $\Gamma[A]$ on $\Gamma[A] \ltimes \Pi$.

 Now suppose that for a tournament $\Pi_0$ on $I_0$, the odd order group $G$ acts on $\Pi_0$. Assume that $J$ and $G \times J$ are sets disjoint from
 $I_0$ and that $\Pi_1$ is a tournament on $I_0 \cup J$ with $\Pi_1|I_0 = \Pi_0$. Let $A$ be a game subset of $G$ with associated group game $\Gamma[A]$.

 Define the tournament $\Pi$ on $I_0 \cup (G \times J)$ so that
 \begin{equation}\label{eq26a}
 \begin{split}
 \Pi|I_0 \ = \ \Pi_0, \quad \Pi|(G \times J) = \Gamma[A] \ltimes (\Pi_1|J), \hspace{1cm} \\
 \text{For} \ \ (g,j) \in G \times J, i \in I_0, \quad (g,j) \to i \ \Leftrightarrow j \to g^{-1}i.
 \end{split}
 \end{equation}
 Notice that identifying $J$ with $\{ e \} \times J$ by $j \mapsto (e,j)$ identifies $\Pi_1|J$ with $\Pi|(\{ e \} \times J)$.

 Let $G$ act trivially on $\Pi_1|J$, by left translation on $\Gamma[A]$ and by the diagonal action on $\Gamma[A] \ltimes (\Pi_1|J)$.

 \begin{prop}\label{prop23bd} The concatenated action of $G$ on $I_0 \cup (G \times J)$ is effective and is free if the
 action of $G$ on $I_0$ is free. The concatenated action is an action of $G$ on $\Pi$. If $\Gamma[A]$ acts on $\Pi_0$ and
 $A$ is a normal game subset, then $\Gamma[A]$ acts via the concatenated action on $\Pi$. \end{prop}

  \begin{proof} The action on $(G \times J)$ is free and so the action on $I_0 \cup (G \times J)$ is effective. It is clearly free if the
  action on $I_0$ is free.

  If $(g,j) \to i$ so that $j \to g^{-1}i$, then $g^{-1}i = (hg)^{-1}hi$ implies that $h(g,j) = (hg,j) \to hi$.  Thus, $G$ acts on $\Pi$.

  If $A$ is a normal game subset, then $\Gamma[A]$ acts on $\Gamma[A] \ltimes (\Pi_1|J)$. Since the sets $I_0$ and $G \times J$ are $G$ invariant
  sets, it follows that if $\Gamma[A]$ acts on $\Pi_0$, then it acts on $\Pi$.

 \end{proof}

\vspace{1cm}

\section{Games of Size Seven}\label{secseven}

Now we consider the case $7 = 2 \cdot 3 + 1$.
\vspace{.5cm}

{\bfseries TYPE I}- $\Gamma_I = \Gamma[[1,2,3]]$ has $Aut(\Gamma[[1,2,3]]) = \Z_7$
acting via translation and is reducible via each pair $i, i+3$.
The collection $\{ m_a( [1,2,3]) : a \in \Z_7^* \}$ are the $6 = \phi(7)$ Type I game
subsets of $\Z_7$ whose games are isomorphic to
$\Gamma[[1,2,3]]$. See Theorem \ref{theo28} and Corollary \ref{cor14}.

The group game $\Gamma_I$ is isomorphic to the
double $2 \Pi$ with $\Pi$ the 3-order on $[1,2,3]$, see Example \ref{ex29ab}.

{\bfseries Type II}- $\Gamma_{II} = \Gamma[[1,2,4]]$ can be described by the following diagram:
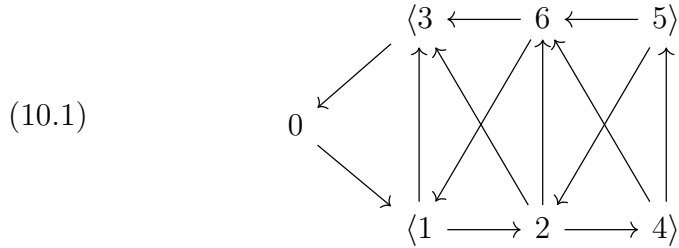
\begin{equation}\label{cd2}
\begin{tikzcd}
& |[alias=T]|\langle 3 \arrow[dl]   & |[alias=S]|6\arrow[dl,from=S,to=O] \arrow[l]  & |[alias=F]|5 \rangle \arrow[l]\arrow[dl,from=F,to=W] \\
0\arrow[dr]  \\
& |[alias=O]|\langle 1 \arrow[u,from=O,to=T] \arrow[r] & |[alias=W]|2 \arrow[u,from=W,to=S]
\arrow[ul,from=W,to=T] \arrow[r]  & |[alias=R]|4 \rangle \arrow[u,from=R,to=F]\arrow[ul,from=R,to=S]
\end{tikzcd}
%\[
%\begin{CD}
%    \quad \swarrow \quad@.     \langle 3 @<<< 6 @<<< 5\rangle  \\
%   0 \qquad @.       @AA\nwarrow A   @A\swarrow A\nwarrow A   @A\swarrow AA  \\
%   \quad\searrow \quad @.     \langle 1 @>>> 2 @>>> 4\rangle
%\end{CD} \]
\end{equation}
Clearly, $m_a$ is an automorphism of $\Gamma_{II}$ for
$a \in \{ 1, 2, 4 \} \subset \Z_7^*$. This is a Quadratic Residue Game of Example \ref{ex17aa} with $p = 7$, $k = 1$ and $H = \{ m_1, m_2, m_4 \}$.

Let $\r$ be an automorphism of $\Gamma_{II}$.
By composing with a translation we may assume that $\r(0) = 0$. Then
$\{ 1, 2, 4 \} =  \Gamma_{II}(0)$ is $\r$ invariant. By composing with
an element of $H$ we may assume $\r(1) = 1$. Then Proposition \ref{prop04} implies that
$\r$ fixes $2$ and $4$ as well. From the
diagram it then follows that $\r$ is the identity. Thus, every automorphism is affine, i.e.
a composition of a translation and a multiplication by an  element of $\{ 1, 2, 4 \} \subset \Z_7^*$.
%and so
%$Aut(\Gamma_{II}) = \a(\Gamma[[1,2,4]])$.

From Theorem \ref{theo28} it follows that $\Gamma_{II}$ is not reducible.

 The two Type II
game subsets are $[1,2,4]$ and $[6,5,3] = m_6([1,2,4])$. $\Gamma[[6,5,3]]$ is the reversed game of $\Gamma[[1,2,4]]$ and
is isomorphic to it via $m_6 = m_{-1}$.

With $\Pi$ the $3$-cycle $\langle 1, 2, 4 \rangle$ game, $\Gamma_{II}$ is isomorphic
to $2 \Pi/\Pi_+$. That is, the double with the cycle $\langle 3, 6, 5\rangle$ of the double reversed.

%\[
%\begin{CD}
%    \quad \swarrow \quad@.     (6 @<<< 5 @<<< 3)  \\
%   0 \qquad @.       @VVV   @VVV   @VVV  \\
%   \quad \searrow \quad @.     (1 @>>> 2 @>>> 4)
%\end{CD} \]

{\bfseries Type III}- $\Gamma_{III}$ can be described by the following diagram:
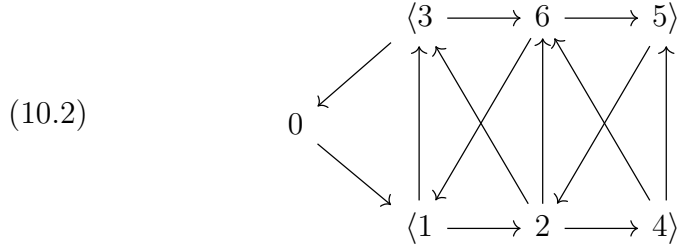
\begin{equation}\label{cd3}
\begin{tikzcd}
& |[alias=T]|\langle 3 \arrow[dl]\arrow[r]   & |[alias=S]|6\arrow[dl,from=S,to=O]\arrow[r]  & |[alias=F]|5 \rangle \arrow[dl,from=F,to=W] \\
0\arrow[dr]  \\
& |[alias=O]|\langle 1 \arrow[u,from=O,to=T] \arrow[r] & |[alias=W]|2 \arrow[u,from=W,to=S]
\arrow[ul,from=W,to=T] \arrow[r]  & |[alias=R]|4 \rangle \arrow[u,from=R,to=F]\arrow[ul,from=R,to=S]
\end{tikzcd}
\end{equation}
%\[
%\begin{CD}
%    \quad \swarrow \quad@.     \langle 3 @>>> 6 @>>> 5\rangle  \\
%   0 \qquad @.       @AA\nwarrow A   @A\swarrow A\nwarrow A   @A\swarrow AA  \\
%   \quad\searrow \quad @.     \langle 1 @>>> 2 @>>> 4\rangle
%\end{CD} \]

With $\Pi$ the $3$-cycle $\langle 1, 2, 4 \rangle$ game, $\Gamma_{III}$ is isomorphic
to $2 \Pi$. Proposition \ref{prop30}
implies that
$Aut(\Gamma_{III}) = \{ m_1, m_2, m_4 \}$ with
$0$ as a fixed point.  Thus, $Aut(\Gamma_{III})$ does not act transitively on
$\Z_7$.

Since $\Pi$ is isomorphic to its reversed game, it follows that $\Gamma_{III}$ is isomorphic to its reversed game as well.

$\Gamma_{III}$ is reducible but is not
reducible via any pair which includes $0$.

\begin{theo}\label{theo29} If $\Gamma$ is a game with 7 vertices then $\Gamma$ is
isomorphic to exactly one of $\Gamma_{I}, \Gamma_{II}$ or $ \Gamma_{III}$.
\end{theo}

\begin{proof} The three types are distinguished by their automorphism groups and so no two are isomorphic.

We use the labeling procedure as in  Theorem \ref{theo05}. Choose a vertex and label it $0$.
The three output vertices in $\Gamma(0)$ form either a
$3$-cycle or a 3-order. Similarly for the three input vertices of $\Gamma^{-1}(0)$.

{\bfseries Case 1} [The inputs and outputs both form 3-orders]: Label the output vertices
$1, 2, 3$ with $1 \to 2, 3$ and $2 \to 3$. Label the
input vertices so that $4 \to 5, 6$ with $5 \to 6$. The remaining arrows are now
determined.
\begin{itemize}
\item $0, 1, 2 \to 3 \quad \Rightarrow \quad 3 \to 4, 5, 6$.
\item $ 4 \to 0, 5, 6 \quad \Rightarrow \quad 1, 2, 3 \to 4$.
\item $1 \to 2, 3, 4 \quad \Rightarrow \quad 0, 5, 6 \to 1$.
\item $5 \to 6, 0, 1 \quad \Rightarrow \quad  2, 3, 4 \to 5$.
\item $3, 4, 5 \to 6 \quad \Rightarrow \quad 6 \to 0, 1, 2$.
\end{itemize} This is Type I.

{\bfseries Case 2} [The inputs and outputs both form $3$-cycles]: Label the output vertices
$1, 2, 4$ with $\langle 1 \to 2 \to 4 \rangle $. Each of these receives one
input from one of the vertices in $\Gamma^{-1}(0)$. Label by $3$ the vertex such
that $3 \to 4$. $3$ now has three outputs and so
$1, 2 \to 3$. Label $5$ so that $5 \to 2$ and then $6 \to 1$. Now there are
two possibilities. Either $3 \to 5$ which is Type II or
$5 \to 3$ which is Type III.

{\bfseries Case 3} [The inputs form a $3$-cycle and the outputs form a 3-order, or vice-versa]:
By replacing the game by its reverse if necessary we
may assume that the outputs form a 3-order.  Notice that for $\Gamma_{III}$ the
inputs $\Gamma_{III}^{-1}(1)$ form a $3$-cycle and the outputs
$\Gamma_{III}(1)$ form a 3-order. Relabel the vertex $0$ of $\Gamma$, calling it
$1$. Label the vertices of $\Gamma(1)$
as $2, 3, 5$ with $5 \to 2, 3$ and $2 \to 3$. Now $5 \to 2, 3$ and $1 \to 5$. Hence,
there is one output vertex from $5$ among the $\Gamma^{-1}(1)$.
Label it $0$, so that $ 5 \to 0$ and choose the remaining two labels so that
$\langle 6 \to 0 \to 4 \rangle$ is the input $3$-cycle for $1$. It suffices to show that
the remaining connections are determined by these choices. We began with $0, 4, 6 \to 1 \to 2, 3, 5$.
\begin{itemize}
\item $1, 5, 2 \to 3 \quad \Rightarrow \quad 3 \to 6, 0, 4$.
\item $3, 5, 6 \to 0 \quad \Rightarrow \quad  0 \to 1, 2, 4$.
\item $0, 1, 5 \to 2 \quad \Rightarrow \quad 2 \to 3, 4, 6$.
\item $5 \to 0, 2, 3 \quad \Rightarrow \quad 1, 4, 6 \to 5$.
\item $4 \to 5, 6, 1 \quad \Rightarrow \quad 0, 2, 3 \to 4$.
\end{itemize}
This is Type III.

Since $\Gamma_{III}$ is isomorphic to its reversed game,
it follows that if the inputs form a 3-order and the
outputs form a $3$-cycle then it is Type III as well.

\end{proof}\vspace{.5cm}

From the proof we obtain the following corollary.

\begin{cor}\label{cor29a} Let $\Gamma$ be a game on $I$ with $|I| = 7$.
\begin{itemize}
\item[(i)] If for some $i \in I$ both the input set $\Gamma^{-1}(i)$ and the output set $\Gamma(i)$ form 3-orders
then $\Gamma$ is of Type I, isomorphic to $\Gamma_I$. In that case, for every $j \in I$ the input set and the output set
are 3-orders.

 \item[(ii)] The game $\Gamma$ is of Type II, isomorphic to $\Gamma_{II}$, if and only if for every $j \in I$ the input set and the output set
are $3$-cycles.

\item[(iii)] If for some $i \in I$ either the input set or the output set forms a 3-order while the
other is a $3$-cycle, then $\Gamma$ is of Type III, isomorphic to $\Gamma_{III}$.
\end{itemize}
\end{cor}

$\Box$ \vspace{.5cm}

Observe that $\Gamma_{II}$ is obtained from $\Gamma_{III}$ by reversing the upper $3$-cycle.  It follows from Theorem
\ref{theo30aa} that the games of Type II are Steiner games.  In fact, the games of type III are Steiner games as well.
To obtain a decomposition by $3$-cycles for $\Gamma_{II}$ or $\Gamma_{III}$ we may use the upper $3$-cycle and, in addition,
\begin{equation}
\begin{split}
\langle 1, 2, 6 \rangle, \ \langle 2, 4, 5 \rangle, \ \langle 4, 1, 3 \rangle, \\
\langle 2, 3, 0 \rangle, \ \langle 4, 6, 0 \rangle, \ \langle 1, 5, 0 \rangle.
\end{split}
\end{equation}
\vspace{.5cm}

Using the results from Section \ref{secpointed} we can compute the number of games on a set $I$ of size seven. Using (\ref{interag01}) it
suffices to compute $|Games(I_+,I_-)|$ with $I$ decomposed as $I_+ \cup \{ 0 \} \cup I_-$. If either $\Pi_+$ or $\Pi_-$ is a 3-order
then it follows from the Splitting Lemma, and is easy to check directly, that the bipartite tournament $\Xi$ contains no cycles. If both
$\Pi_+$ and $\Pi_-$ are $3$-cycles, then the number of Eulerian subgraphs of $\Xi$ is the same as the number in the special case when
$\Pi$ is the double of a $3$-cycle. In Example \ref{expoint} it was shown that $\Xi$ then contains six distinct Eulerian subgraphs (including
the empty one). Thus, our lower bound $|Games(I_+,I_-)| \geq 2^{3 \cdot 2} = 64$ has to be corrected to account for each of the four
cases where $\Pi_+$ and $\Pi_-$ are $3$-cycles and so there are six pointed games $\Pi$ instead of one each. That is,
$|Games(I_+,I_-)|  = 64 \ + \ 6 \cdot 4 \ - \ 4 = 84$. Finally, from (\ref{interag01}) it follows that when $|I| = 7$
\begin{equation}
|Games(I)| \ = \ {6 \choose 3} \cdot 84 \ = \ 1680.
\end{equation}
 \vspace{1cm}

\section{Isomorphism Examples}\label{seciso}

If $\r : \Pi_1 \tto \Pi_2$ is an isomorphism between digraphs, and $[i_1, \dots, i_k]$ is a simple  path in
$\Pi_1$,  then $[\r(i_1), \dots, \r(i_k)]$ is a simple  path in $\Pi_2$.

If  $[j_1, \dots, j_{k}]$ is a simple path in $\Pi_2$, then (recall that $\bar \r = \r \times \r$)
\begin{equation}\label{eqreduce}
\bar \r([i_1, \dots, i_k]) \ \subset \ [j_1, \dots, j_{k}] \ \Longrightarrow \ \r(i_p) \ = j_p \ \text{for} \ p = 1, \dots, k.
\end{equation}
Recall that, regarded as a digraph, $$[i_1, \dots, i_{k}] = \{ (i_p,i_{p+1}): p = 1, \dots k-1 \}.$$

The implication (\ref{eqreduce}) is proved by induction on $k$. Observe that
$\r : \{i_1, \dots, i_k \} \tto \{j_1, \dots, j_k \}$ is injective and so is bijective.
The vertex $j_k$ is the unique vertex with no output in the digraph $[j_1, \dots, j_{k}]$ and similarly for $i_k$ in $[i_1, \dots, i_{k}]$. So
$\r(i_k) = j_k$.

Recall that the domination graph of a game $\Pi$ is given by $$dom(\Pi) = \{ (i,j) \in \Pi : \Pi \ \ \text{is reducible via} \ \ i \to j \}.$$

\begin{prop}\label{propreduce} (a) If $\r : \Pi_1 \tto \Pi_2$ is an isomorphism of games, then it restricts to an isomorphism between the
domination digraphs $\r : dom(\Pi_1) \tto dom(\Pi_2)$. Furthermore, it maps each maximal simple path in $dom(\Pi_1)$ to a maximal simple path
in $dom(\Pi_2)$.

(b) Assume $\r$ is an automorphism of a game $\Pi$.
\begin{itemize}
\item[(i)] If $[i_1, \dots, i_k]$ is a simple path in $\Pi$ and $\bar \r([i_1, \dots, i_k]) \ \subset \ [i_1, \dots, i_{k}]$ then
$\r$ fixes $i_p$ for $p = 1, \dots, k$. In particular, if $\r$ maps a maximal simple path of $dom(\Pi)$ to itself, then it fixes
every vertex of the path.  If $\r$ maps some vertex of a maximal simple path $[i_1, \dots, i_k]$ of $dom(\Pi)$ into $[i_1, \dots, i_k]$,
then it fixes every vertex of the path.
\item[(ii)] If $\Pi$ is reducible via $i \to j$ and either $i$ or $j$ is fixed by $\r$ then the other is as well.
\end{itemize}
\end{prop}

\begin{proof} (a): This is clear since $\Pi_1$ is reducible via $i \to j$ if and only if $\Pi_2$ is reducible via $\r(i) \to \r(j)$.

(b)(i): The first assertion is a special case of (\ref{eqreduce}).

Observe that if $[i_1, \dots, i_k]$ is a maximal simple path in $dom(\Pi)$ then Proposition \ref{prop07c}
implies that $[i_1, \dots, i_k] = dom(\Pi)|\{i_1, \dots, i_k \}$.
It follows that if $\r(i_p) \in \{i_1, \dots, i_k \}$, then $\bar \r([i_1, \dots, i_k]) \ \subset \ [i_1, \dots, i_{k}]$ and so $\r$ fixes
every vertex on the path.

(ii): If $\Pi$ is reducible via $i \to j$, and $\r(i) = i$, then $\Pi$ is reducible via $i = \r(i) \to \r(j)$ and so by
Proposition \ref{prop07}(f)  $j = \r(j)$.

\end{proof} \vspace{.5cm}

If two tournaments are isomorphic, then of course their doubles are isomorphic. If $\r : 2 \Pi \to 2 \Gamma$ is an isomorphism with $\r(0) = 0$, then
by Proposition \ref{prop30} $\r$ is itself the double of an isomorphism from $\Pi$ to $\Gamma$.
What happens when $\r$ is not an isomorphism of pointed games on $0$ ?

Let $\Pi$ be a tournament on $I$ with $|I| = n$. The double $2\Pi$ is a game on $\{ 0 \} \cup I_+ \cup I_-$ with $I_{\pm} = I \times \{ \pm 1 \}$.
Let $i,j \in I$ with $i \to j$. We recall the following reducibility results which follow from Proposition \ref{prop07}.
\vspace{.5cm}

(i) The game $2\Pi$ is reducible via $i- \to i+$ and via $j- \to j+$. By uniqueness in Proposition \ref{prop07}(f) it is
not reducible via $j- \to i+$.   Observe, for example, that $i- \to j-$ and $i- \to i+$. \vspace{.25cm}

(ii) Similarly, $2 \Pi$ is not
reducible via $i-\to j- $ because $0 \to i_-, j_- $, nor is it  reducible via $i+ \to j+$.\vspace{.25cm}

(iii) The game is reducible via
$j+ \to i-$ if and only if for every $k \in I \setminus \{ i,j \}$ either $i, j \to k$ or $k \to i,j$.\vspace{.25cm}

(iv) The game is reducible via $i+ \to 0$ if and only if $i \to k$ for all $k \in I \setminus \{ i \}$ and it reducible via
$0 \to i-$ if and only if $k \to i$ for all $k \in I \setminus \{ i \}$.
\vspace{.5cm}

Now let $\Gamma$ be a tournament on $K$ with $|K| = n$.  $2 \Gamma$ is a game on $\{ \bar 0 \} \cup K_+ \cup K_-$.

We now describe how $\r : 2 \Pi \tto 2 \Gamma$ can be an isomorphism which is not an isomorphism of pointed games, i.e. $\r(0) \not= \bar 0$.
Assume that $\r(0) = k_1+$. Since $\r^{-1}(k_1-) \to 0$ we have that $\r(i_1+) = k_1-$ for some $i_1 \in I$.
Now  $2 \Gamma$ is reducible via $\r(i_1-) \to \r(i_1+)$ and so $\r(i_1-) \in \{ \bar 0 \} \cup K_+$. We build parallel sequences of
distinct elements in the domain and range which are mapped across by $\r$.

 \begin{equation}\label{eqiso01} \begin{CD}
 0 @>>> k_1+ \\
@AAA      @AAA \\
i_1+ @>>> k_1- \\
@AAA     @AAA \\
i_1- @>>> k_2+ \\
@AAA      @AAA \\
i_2+ @>>> k_2- \\
@AAA      @AAA \\
i_2- @>>> k_3+ \\
\end{CD}  \end{equation}

The way the sequences terminate, as of course they must, is when for some $m \geq 1$, $\r(i_m-)$ is equal to $\bar 0$ instead of an element of $K_+$.
That is,

\begin{equation}\label{eqiso02}  \begin{CD}
i_m+ @>>> k_m- \\
@AAA      @AAA \\
i_m- @>>> \bar 0 \\
\end{CD}  \end{equation}

Notice that $\r^{-1}(\bar 0) \in I_-$. If we had begun with $\r(0) \in I_-$ we would have built the analogous sequence upward and would have obtained
$\r^{-1}(\bar 0) \in I_+$.

Since $2 \Pi$ is reducible via $i_1+ \to 0$, we have $i_1 \to j$ for
all $j \in I \setminus \{ i_i \}$. We now prove, inductively, that
for $p = 2, \dots, m$, $i_p \to j$ for all $j \in I \setminus \{ i_1, \dots, i_p \}$. This is because $2 \Pi$ is reducible via $i_p + \to i_{p-1}-$.
By induction hypothesis, $i_{p-1} \to j$ for all such $j$ and so by (iii) above, $i_p \to j$.

Similarly, $\bar j \to k_m$ for all
$\bar j \in K \setminus \{ k_m \}$ and so, inductively, for all $p = 1, \dots, m-1$, $\bar j \to k_p $ for all $\bar j \in K \setminus \{k_p, \dots, k_m \}$.

Let $J = I \setminus \{i_1,\dots,i_m \}$ and $\bar J  = K \setminus \{ k_1, \dots , k_m \}$. Notice that $J_- \subset (2 \Pi)^{-1}(i_m+)$ and
$J_+ \subset 2 \Pi(i_m+)$ while $\bar J_-  \subset (2 \Pi)^{-1}(k_m-)$ and $\bar J_+  \subset 2 \Pi(k_m-)$.

Thus, $\r$ restricts to an isomorphism from the tournament $2 \Pi|(J_+ \cup J_-)$ to the tournament $2 \Gamma|(\bar J_+ \cup \bar J_-)$
and it takes $J_+$ to $\bar J_+$ and $J_-$ to $\bar J_-$. If $\gamma : J \to \bar J$ is defined by $\r(j-) = \gamma(j)-$, then
since $2 \Gamma$ is reducible by $\r(j-) \to \r(j+)$ and by $\gamma(j)- \to \gamma(j)+ $ it follows
that $\r(j\pm) = \gamma(j)\pm$ for all $j \in J$. In particular, $\gamma : \Pi|J \to \Gamma|\bar J$ is an isomorphism.

Furthermore, $\Xi  = \{ i_1, \dots, i_m \} \times J \subset \Pi$ and $\bar \Xi = \bar J \times \{ k_1, \dots, k_m \} \subset \Gamma$.

If we define $\theta : I \to K$ by $\theta(i_p) = k_{p}$ for $p = 1, \dots, m$ and  $\theta(j) = \gamma(j)$ for $j \in J$, then, reversing $\Xi$ in
$\Pi$, we obtain an isomorphism
\begin{equation}\label{eqiso03}
\theta : \Pi/\Xi \ \to \Gamma.
\end{equation}

We can diagram this as follows:
\begin{align}\label{eqiso02a}
\begin{split}
  &\Pi : \quad  \begin{CD} i_1 @>>> i_2 @>>> \dots @>>> i_m @>>> J  \\
\end{CD}\\
  &\Gamma : \quad  \begin{CD} k_1 @>>> k_2 @>>> \dots @>>> k_m @<<< \bar J  \\
\end{CD}
\end{split}
\end{align}

Now we use from Proposition \ref{prop06order} the equivalent descriptions of an order, i.e. a transitive tournament.

\begin{lem}\label{lemiso01} Assume there exists an isomorphism $\r : 2 \Pi \tto 2 \Gamma$ with $\r(0) \not= \bar 0$.
The tournaments $\Pi$ and $\Gamma$ are isomorphic if and only if $\Pi$ is an order. \end{lem}

\begin{proof} If necessary we may reverse the roles of $\Pi$ and $\Gamma$ and use $\r^{-1}$ instead of $\r$. Thus, we may
assume that $\r(0) \in K_+$ as above.

If $m = |I|$, or, equivalently, $J$ is empty, then $\Pi$ is an order and $\theta$ is an isomorphism from $\Pi$ to $\Gamma$.

Now assume $n = |I| > m$.

%If $\Pi_0$ is the standard order on $[1,n]$ and $1 \leq j < k \leq n$, then $\Pi_0/\Pi_0|[j,k]$ is still an order.
%By  Proposition \ref{prop06order} an order of size $n$ is unique up to isomorphism and so it is isomorphic to $\Pi_0$.
If  $\Pi$ is an order, then by Proposition \ref{prop06order} (f) we can continue the numbering $i_1, \dots, i_m$ to $i_{m+1}, \dots, i_n$
so that $i_p \to i_q$ when
$p < q$. If we reverse $\Xi = \{i_1, \dots, i_m \} \times \{i_{m+1}, \dots, i_n \}$, then the result is again an order
with the vertices ordered as  $i_{m+1}, \dots, i_n, i_1, \dots, i_m$. By Proposition \ref{prop06order} again
 an order of size $n$ is unique up to isomorphism. Hence, $\Gamma$ is isomorphic to $\Pi$.

% In the above construction, the isomorphism $\theta$ induces
%an isomorphism from $\Pi/\Pi|J = (\Pi/(\Pi \setminus \Pi|J))^{-1}$ to $(\Gamma)^{-1}$. The vertices of $J$ are the vertices above
%$i_1, \dots, i_m$ in the order and so $\Pi/\Pi|J$ is an order. Thus, $(\Gamma)^{-1}$ is an order. The reverse of an order is
%an order and so $\Gamma$ is an order. Again the uniqueness proved in Proposition \ref{prop06order} implies that $\Gamma$ is isomorphic to
%$\Pi$.

Assume instead that $\Pi$ is not an order. At least one of the  $\O (1_I \cup \Pi) \cap \O (1_I \cup \Pi)^{-1}$ equivalence
classes is not a singleton. These are the \emph{fat equivalence classes}.
Recall that $\O \Pi$ induces an order on the set of $\O (1_I \cup \Pi) \cap \O (1_I \cup \Pi)^{-1}$ equivalence classes.
Obviously the equivalence class of each of the $i_1, \dots, i_m$ vertices is a singleton and each lies above all the other
classes. Count the classes as in Proposition \ref{prop06order} (f) and let $k^*(\Pi) > m$ be the label of
the first fat equivalence class. When we reverse $\Xi$  all of the $i_1, \dots, i_m$ classes are moved below
all the other classes in the ordering and the ordering
among the remaining classes is unchanged. Hence,
$k^*(\Gamma) = k^*(\Pi) - m$. Since this number is an isomorphism invariant, it follows that $\Pi$ is not isomorphic to $\Gamma$.

 \end{proof}

 \begin{ex}\label{exiso02}  There exist non-isomorphic tournaments $\Pi$ and $\Gamma$ such that $2 \Pi$ is isomorphic to $2 \Gamma$. \end{ex}

 \begin{proof} Let $\Theta$ be a tournament on $J$ which is not an order. Let $I = \{i_1,\dots, i_m \} \cup J$ and on it let
 $\Xi$ be the digraph $\{ i_1, \dots, i_m \} \times J$.
 \begin{align}\label{eqiso04}
 \begin{split}
 \Pi \ = \ \{ (i_p,i_{q}) : 1 \leq &p < q \leq m \} \ \cup \ \Xi \ \cup \ \Theta, \\
 \Gamma \ = \ &\Pi/\Xi.
 \end{split}
 \end{align}
 Define $\r : 2 \Pi \to 2 \Gamma$ according to the patterns of (\ref{eqiso01}) and (\ref{eqiso02}) with $k_p = i_p$, and with $\r(j\pm) = j\pm$ for
 $j \in J$. By Lemma \ref{lemiso01}, $\Pi$ is not isomorphic to $\Gamma$.

 \end{proof} \vspace{.5cm}

 We also obtain the following from Lemma \ref{lemiso01}.

 \begin{theo}\label{theoiso03} Let $\Pi$ be a tournament which is not an order. Any automorphism of $2 \Pi$ fixes $0$ and so the
 injection $2: Aut(\Pi) \tto Aut(2 \Pi)$ is an isomorphism. \end{theo}

 $\Box$ \vspace{.5cm}

 \begin{cor}\label{coriso03a} Let $\Pi$ be a tournament which is not an order. If $\Pi$ is a rigid tournament, then $2 \Pi$ is a rigid game. \end{cor}

  \begin{proof} By assumption $Aut(\Pi)$ is trivial and so $Aut(2 \Pi)$ is trivial by Theorem \ref{theoiso03}.

  \end{proof} \vspace{.5cm}

 Recall that if $\Pi$ is the standard order on $A = [1,n]$, then the double $2 \Pi$ is the group game $\Gamma[A]$ on $\Z_{2n + 1}$. Since an order is a rigid
 tournament by Lemma \ref{lemrigid}, the
 identity is the only automorphism which fixes $0$. On the other hand, the group $\Z_{2n + 1}$ acts transitively on $2 \Pi = \Gamma[A]$ by translation.
\vspace{.5cm}

Next we consider the possibility of non-isomorphic extensions of a game.

Observe that the games $\Gamma_{III}$ and $\Gamma_I$ of Section \ref{secseven} are non-isomorphic games of size $7$ and both are reducible.
Since both reduce to the unique
game of size $5$, we see that a game can admit non-isomorphic extensions. This phenomenon is quite general.

\begin{prop}\label{propiso04} Any game $\Pi$ with size greater than $3$ admits non-isomorphic extensions.\end{prop}

\begin{proof} Assume that $\Pi$ is a game on $I$ with $|I| = 2n + 1$. We have seen that the game of size $5$ has non-isomorphic extensions and
so we may assume that $2n + 1 \geq 7$ and so $n \geq 3$.

Choosing $K \subset I$ a subset of size $n + 1$ we extend via $u \to v$ to obtain
the game $\Gamma$. Recall that if $i, j \in I$, then  Proposition \ref{prop07} implies that
$\Pi^{-1}(i) = \Pi(j)$ if and only if $i \to j$ and $\Pi$ is reducible via $i \to j$.

Notice that each $\Pi(i)$ and $\Pi^{-1}(i)$ is a subset of $I$ of size $n$ and there are at most $2(2n + 1)$ of them. On the other hand, there
are a total of ${2n + 1 \choose n}$ subsets of size $n$ and for $n \geq 3$,  ${2n + 1 \choose n} > 4n + 2$.
Hence, for $n \geq 3$ there exists a subset $L$ of size $n$ which is not equal to $\Pi(i)$ or $\Pi^{-1}(i)$ for any $i \in I$. We let
$K = I \setminus L$ to obtain $\Gamma$. By Corollary \ref{cor07b}, $\Gamma$ is not reducible via  $i \to u$ or $v \to i$  for any $i \in I$ and
reducibility via $i \to j$ for $i, j \in I$ requires $\Pi$ be reducible via $i \to j$. Hence,

\begin{equation}\label{eqiso06a}
\begin{split}
dom(\Gamma) \ \subset \ dom(\Pi) \ \cup \ \{ (u,v) \} \\
|dom(\Gamma)| \ \leq \ |dom(\Pi)| + 1.
\end{split}
\end{equation}

In particular, if $\Pi$ is not reducible and so $dom(\Pi) = \emptyset$,

\begin{equation}\label{eqiso06b}
dom(\Gamma) \ = \ \{ (u,v) \}
\end{equation}

\vspace{.25cm}

{\bfseries Case 1:} [$\Pi$ is not reducible, i.e. $dom(\Pi) = \emptyset$] \ \ If for $i \in I$
we use $K = I \setminus \Pi(i)$, to obtain $\Gamma_1$, then by Corollary \ref{cor07b} (b)
$\Gamma_1$ is reducible via $v \to i$ as well as $u \to v$. Since $\Pi$ is not reducible,
$\Pi(i) \not= \Pi^{-1}(j)$ for any $j \in I$. So $\Gamma_1$ is not reducible via $j \to u$ by Corollary \ref{cor07b} (b) again.
Since $\Pi$ is not reducible,
$\Gamma_1$ is not reducible via any pair $j_1, j_2$ in $I$. Thus, with this choice of $K$,
\begin{equation}\label{eqiso05a}
dom(\Gamma_1) \ = \ [u, v, i].
\end{equation}

Similarly, if we use $K = I \setminus \Pi^{-1}(i)$, we obtain $\Gamma_1$ with
\begin{equation}\label{eqiso05b}
dom(\Gamma_1) \ = \ [i, u, v].
\end{equation}

\vspace{.25cm}

{\bfseries Case 2:} [$\Pi$ is  reducible, but $dom(\Pi)$ is not a Hamiltonian cycle] \ \
By Proposition \ref{prop07c} the domination graph
 is the union of separate maximal simple paths $[i_0, \dots, i_m]$ with $(i_m,i_0) \not\in dom(\Pi)$.
Call $i_0, i_2, \dots $ the \emph{even vertices} of the path and $i_1, i_3, \dots $
the \emph{odd vertices} of the path. By Proposition \ref{prop07c} again
$\Pi(i_0)$ contains only the odd vertices of the path and it intersects each of the other maximal paths either in the set of its odd vertices
or its even vertices.

If we use $K = I \setminus \Pi(i_0)$ to obtain $\Gamma_1$, then from Corollary \ref{cor07b} (d),
we see that $\Gamma_1$ is reducible via every edge of $dom(\Pi)$. In addition,
as above it is reducible via $v \to i_0$ as well as  via $u \to v$. By maximality for no $j \in I$ is $(j,i_0) \in dom(\Pi)$ and so
$\Pi(i_0) \not= \Pi^{-1}(j)$ for any $j \in I$. So, as in Case 1, $\Gamma_1$ is not reducible via $j \to u$ by Corollary \ref{cor07b} (b) again.
Thus, we have
\begin{equation}\label{eqiso05c}
dom(\Gamma_1) \ = \ dom(\Pi) \ \cup \ [u, v, i_0].
\end{equation}

Similarly, if we use $K = I \setminus \Pi^{-1}(i_m)$, then
\begin{equation}\label{eqiso05d}
dom(\Gamma_1) \ = \ dom(\Pi) \ \cup \ [i_m, u, v].
\end{equation}

\vspace{.25cm}

{\bfseries Case 3:} [ $dom(\Pi)$] is a Hamiltonian cycle] \ \
In the Hamiltonian case, $dom(\Pi) = \langle i_0, \dots, i_{2n} \rangle$ and $\Pi(i_0) = \Pi^{-1}(i_{2n})$ is the set of odd vertices.  So with
$K = I \setminus \Pi(i_0) =  I \setminus \Pi^{-1}(i_{2n})$ we obtain $\Gamma_1$ with
\begin{equation}\label{eqiso05e}
dom(\Gamma_1) \ = \ \langle i_0, \dots, i_{2n}, u, v \rangle.
\end{equation}
\vspace{.25cm}

In all of these cases $|dom(\Gamma_1)| = | dom(\Pi)| + 2$. Thus, in each case, $\Gamma_1$ is not isomorphic to $\Gamma$.

In the Hamiltonian cycle case, which includes the case $n = 2$, we can choose $K = \{ i_0, \dots, i_n \}$. With $n \geq 2$ the complement
$\{ i_{n + 1}, \dots i_{2n} \}$  is not equal to any $\Pi(i_p)$ or $\Pi^{-1}(i_p)$ since it contains both even and odd vertices. Hence, for this extension
$dom(\Gamma) = \{ (i_n, i_{n + 1}), (i_{2n}, i_0), (u,v) \}$ with $|dom(\Gamma)| = 3 < 2n + 1 = |dom(\Pi)|$.

\end{proof}

If $\Pi$ is not reducible, then from the proof of Proposition \ref{propiso04} we have the following possibilities for
the domination graph of $\Gamma$ the extension of $\Pi$ on $I$ via $K \subset I$ and $u \to v$.
\vspace{1cm}

\begin{equation}\label{eqiso06c}
\begin{array}{cr}
dom(\Gamma) \quad &  I \setminus K\\ [0.5ex]
(u,v),(v,i)  &  \Pi(i)\\
 (i,u),(u,v)   &  \Pi^{-1}(i)\\
 (u,v)  &  \text{otherwise}
 \end{array}
 \end{equation}
 \vspace{.5cm}

In Cases 1-3 of the above proof, we constructed examples which enlarge $dom(\Pi)$. We pause to consider the opposite extreme.

\begin{df}\label{defiso05} A game $\Pi$ on a set $I$ with $|I| = 2n + 1$ is called \emph{uniquely reducible}\index{uniquely reducible}
\index{reducible!uniquely} when there is a unique subset $J \subset I$ with $|I| = 2n - 1$ such that the restriction $\Pi|J$ is a subgame. \end{df}
\vspace{.5cm}

Thus, $\Pi$ is uniquely reducible when $dom(\Pi)$ consists of a single edge.  For example, a double of a tournament of size at least $2$ is never
uniquely reducible. In particular, the unique game of size $5$, which is the double of a single edge, is not uniquely reducible.
Both of the types of reducible games of size $7$ are isomorphic to doubles and so are not uniquely reducible.

\begin{prop}\label{propiso06} Let $\Pi$ be a  game on  $I$ with $|I| = 2n + 1$ and let $K $ be a subset of $I$ with $|K| = n + 1$. The extension
$\Gamma$ of $\Pi$ via $K$ and $u \to v$ is uniquely reducible if and only if the following
conditions hold.
\begin{itemize}
\item[(i)] For every path $[i_0, \dots, i_m]$ in $dom(\Pi)$ either
$\{ i_0, \dots, i_m \}$ is a subset of $K$ or is disjoint from $K$.
\item[(ii)] $I \setminus K$ is not equal to $\Pi(j)$ or $\Pi^{-1}(j)$ for any $j \in I$.
\end{itemize}

Assume $(i_0, i_1) \in dom(\Pi)$ and so $\Pi$ is reducible. If $\Pi$ is uniquely reducible and
$K = I \setminus \Pi(i_1) =  I \setminus \Pi^{-1}(i_0)$, then $K$ satisfies condition (i) but not
condition (ii). For all other cases with $\Pi$ reducible, condition (ii) follows from condition (i).

\end{prop}

\begin{proof} By Corollary \ref{cor07b}  condition (i)  is equivalent to non-reducibility via $i \to j$ for $(i,j) \in \Pi$.
Condition (ii) is equivalent to non-reducibility via $v \to j$ or $j \to u$ for
$j \in K$.

Now assume $(i_0,i_1) \in dom(\Pi)$. For any $j \in I \setminus \{ i_0, i_1 \}$ either
$i_0 \to j \to i_1$ or $i_1 \to j \to i_0$ by Proposition \ref{prop07}. Hence, both $\{ i_0, i_1 \} \cap \Pi(j)$ and
$\{ i_0, i_1 \} \cap \Pi^{-1}(j)$ are singletons as are
$\{ i_0, i_1 \} \cap \Pi(i_0)$ and $\{ i_0, i_1 \} \cap \Pi^{-1}(i_1)$. Hence, given condition (i)
the only possibility with $I \setminus K$ equal to $\Pi(j)$ or $\Pi^{-1}(j)$ is when
$K = I \setminus \Pi(i_1) =  I \setminus \Pi^{-1}(i_0)$.

Furthermore, if $dom(\Pi)$ contains another edge $(j_0, j_1)$, then $\{ j_0, j_1 \} \cap \Pi^{-1}(i_0)$
is a singleton (even if $j_1 = i_0$ and $j_0 = i_0$ can't happen). So if $\Pi$ is not
uniquely reducible, then $K = I \setminus \Pi(i_1) =  I \setminus \Pi^{-1}(i_0)$ violates condition (i).

If $\{ (i_0,i_1) \} =  dom(\Pi)$ and $K = I \setminus \Pi(i_1) =  I \setminus \Pi^{-1}(i_0)$, then condition (i) is satisfied, but
$dom(\Gamma) = [i_0, u, v, i_1]$.

\end{proof} \vspace{.5cm}

 \begin{prop}\label{propiso07} If $\Pi$ is a game on $I$, then the double $2\Pi$ has a uniquely reducible extension.
 In particular, a game
 of type $\Gamma_{III}$ from Section \ref{secseven},
 has a uniquely reducible extension.
 \end{prop}

 \begin{proof} By Proposition \ref{prop29bbbb} $dom(2\Pi) = \{ (i-,i+) : i \in I \}$. With $2k + 1 = |I|$ choose
 $\hat K \subset I$ with $|\hat K| = k + 1$. Use $K = \bigcup \{ \{ i-, i+ \} : i \in \hat K \},$ with $|K| = 2k + 2$,
 to define the extension $\Gamma$.

Condition (i) of Proposition \ref{propiso06} is obvious and since
$2 \Pi$ is not uniquely reducible, condition (ii) holds as well. The proposition implies
 that $\Gamma$ is uniquely reducible.

 The type III game of Section \ref{secseven} is the double of  a $3$-cycle.

\end{proof} \vspace{.5cm}

\begin{prop}\label{propiso08} If $\Pi$ is a game which is either not reducible or uniquely reducible,
then $\Pi$ has a uniquely reducible extension. \end{prop}

\begin{proof} Assume $\Pi$ is a game on $I$.

If $\Pi$ is not reducible, then as in the proof of Proposition \ref{propiso04}
we choose $K$  such that $I \setminus K$
is not equal to $\Pi(i)$ or $\Pi^{-1}(i)$ for any $i \in I$.

If $\Pi$ is uniquely reducible then its size is at least $9$.

If $dom(\Pi) = \{ (i_0,i_1) \}$ and $k $ is a vertex of $I \setminus \Pi(i_1) =  I \setminus \Pi^{-1}(i_0)$.
Then we choose $K$ so that it contains or is disjoint from $\{ i_0, i_1, k \}$. Condition (i) of
Proposition \ref{propiso06} is obvious and since $I \setminus K$ is not equal to
$\Pi(i_1) =   \Pi^{-1}(i_0)$, condition (ii) follows. Again the proposition implies that the extension via $K$ and $u \to v$
is uniquely reducible.

\end{proof} \vspace{.5cm}

Thus, beginning with the double of a three cycle we can build a totally reducible game $\Pi$ on $I$ with $I_1 \subset I_2 \dots \subset I_n = I$ with
$|I_k| = 2k + 1$ such that when $2k + 1 > 7$ the subgame $\Pi|I_k$ is uniquely reducible.
\vspace{.5cm}

Above we saw that games have non-isomorphic extensions.  Now we consider the reverse question.
Can non-isomorphic games have isomorphic extensions? Equivalently, can a game $\Gamma$ be reduced to two non-isomorphic games.
This, of course, requires that the game be reducible via different pairs and so the obvious places to look are at doubles $\Gamma = 2 \Pi$ with
$\Pi$ a tournament on $I$ with $|I| = n$. For any vertex $i \in I$, $2 \Pi$ is an extension of $2 (\Pi|I \setminus \{ i \})$.  So we want a tournament $\Pi$
such that for $i_1, i_2 \in I$, $2 (\Pi|I \setminus \{ i_1 \})$ is not isomorphic to $2 (\Pi|I \setminus \{ i_2 \})$.

Let us first consider when this fails.  If $\Pi$ is a point-symmetric game, then
$\Pi|I \setminus \{ i_1 \}$ and $\Pi|I \setminus \{ i_2 \} $ are isomorphic
for any pair $i_1$ and $i_2$ since there is an automorphism taking $i_1$ to $i_2$. If $\Pi$ is an order, then $\Pi|I \setminus \{ i \}$ is an order
for every $i \in I$ and so, despite the rigidity of $\Pi$, all of the $\Pi|I \setminus \{ i \}$'s are isomorphic. Isomorphic tournaments have
isomorphic doubles.

A more interesting example, is $\Pi$ on $I = \{ 1, 2, 3, 4, 5, 6 \}$ with
 \begin{equation}\label{eqiso07a}
 \Pi \ = \ \langle 1, 2, 3 \rangle \ \cup \  \langle 4, 5, 6 \rangle \ \cup \ \{ 1, 2, 3 \} \times \{ 4, 5, 6 \}.
 \end{equation}
 The tournaments $\Pi|I \setminus \{ 1 \}$ and $\Pi|I \setminus \{ 4 \}$ are not isomorphic. The former has score vector $(1, 1, 1, 3, 4)$ and the
 latter has score vector $(0, 1, 3, 3, 3)$. However, as Example \ref{exiso02} shows, they have isomorphic doubles.  Hence,
 all of the $2 (\Pi|I \setminus \{ i \})$'s are isomorphic.

 On the other hand, if $\Pi|I \setminus \{ i_1 \}$ and $\Pi|I \setminus \{ i_2 \}$ have different score vectors and neither $0$ nor $n - 2$ occur among
 the scores, then, by Proposition \ref{prop30}, $2(\Pi|I \setminus \{ i_1 \})$ is not isomorphic to $2(\Pi|I \setminus \{ i_2 \})$. It is not hard to
 construct such examples.

 More interesting is the case when $n = 2k + 1$ and $\Pi$ is itself a game.  In that case, each $\Pi|I \setminus \{ i \}$ has
 score vector $(k-1, \dots, k-1, k, \dots, k)$ with $k$ each of the scores $k-1$ and $k$. However, if the further restriction of
 $\Pi|I \setminus \{ i_1 \}$ and $\Pi|I \setminus \{ i_2 \}$ to the vertices
 with score $k$ are not isomorphic, then  $\Pi|I \setminus \{ i_1 \}$ and $\Pi|I \setminus \{ i_2 \}$ cannot be isomorphic.

 \begin{ex}\label{exiso05} Let $\Pi$ be the game $\Gamma_{III}$ on $\{ 0, 1, 2, 3, 4, 5, 6 \}$  of Section \ref{secseven} so that it is
 the double of the three cycle $\langle 1, 2, 3 \rangle$. In  $\Pi|I \setminus \{ 0 \}$
 the vertices $1, 2, 3$  have score $3$ and form a $3$-cycle.  In $\Pi|I \setminus \{ 4 \}$ the vertices $6, 5, 1$  have score $3$ and form a
 3-order. The non-isomorphic games $2(\Pi|I \setminus \{ 0 \})$ and
$2(\Pi|I \setminus \{ 4 \})$ extend via $(0,-) \to (0,+)$ and via $(4,-) \to (4,+)$, respectively, to $2 \Pi$. \end{ex}

$\Box$

\vspace{1cm}

\section{Games of Size Nine}\label{secnine}

For the case $9 = 2 \cdot 4 + 1$ we will  first describe the isomorphism classes of the group games. There are $2^4 = 16$ game subsets. These are
naturally pointed games with tournaments $\Pi_+, \Pi_-$ each of size $4$.

\begin{prop}\label{propfour} Each tournament of size $4$ is uniquely determined up to isomorphism by its score vector. \end{prop}

\begin{proof}  The possible score vectors are:
\begin{equation}\label{eqfour}
s_1 = (0,1,2,3), \quad s_2 = (1,1,1,3), \quad \bar s_2 = (0,2,2,2), \quad s_3 = (1,1,2,2).
\end{equation}

Let $\Theta_p$ be a tournament of size $4$ with score vector $s_p$ for $p = 1,2,3$. So $\Theta_2^{-1}$ has score vector $\bar s_2$.

By Proposition \ref{prop06order} a tournament with score vector $s_1$ is an order and the order of size $4$ is unique up to isomorphism.

If the score vector is $s_2$, then the output set of the vertex with score $3$ is a $3$-cycle. It is obvious that any two such are isomorphic.
By using the reverse tournaments we obtain the result for $\bar s_2$.

Next observe that if $\Theta$ is a tournament on $J$ with $|J| = 2n$ and score vector
$(n-1,\dots,n-1,n,\dots,n)$, then we there is a game $\Pi$ of size $2n + 1$
which contains $\Pi$.
If $u$ is the additional vertex, then $\Pi(u)$ is the set of vertices of $J$ with $\Theta$ score $n$. Conversely, if we remove a vertex from a game
of size $2n+1$, then we are left with a tournament of size $2n$ and with score vector $(n-1, \dots, n-1, n, \dots, n)$.

When $n = 4$ we apply uniqueness of the game of size $5$.
If $\Theta$ and $\bar \Theta$ are tournaments of size $4$ with score vectors $s_3$, we can adjoin vertices $u$ and $\bar u$ to obtain
games $\Pi$ and $\bar \Pi$ of size $5$.  By Theorem \ref{theo05} there exists a isomorphism between them.
Since the game is a group game, $\Z_5$ acts transitively
on the vertices and so we may assume that the isomorphism takes $u$ to $\bar u$. It then restricts to an isomorphism from $\Theta$ to $\bar \Theta$.

\end{proof}

{\bfseries Remark:} Notice that an isomorphism
between two  tournaments $\Theta_1$ and $\Theta_2$ with score vectors $(n-1,\dots,n-1,n,\dots,n)$
extends uniquely to an isomorphism between the games $\Pi_1$ and $\Pi_2$.
Thus, if we begin with non-isomorphic games of size $2n + 1$ and we remove an
arbitrary vertex from each we obtain non-isomorphic tournaments of size $2n$
each with score vector $(n-1, \dots, n-1, n, \dots, n)$. Thus, with $n > 2$, there are  always non-isomorphic tournaments of size $2n$
each with score vector $(n-1, \dots, n-1, n, \dots, n)$.
\vspace{.5cm}

In particular, we see that $\Theta_1$ and $\Theta_3$ are each isomorphic to its reverse tournament.

Let $\Gamma_3$ denote the game on $\Z_3$ with $\langle 0, 1, 2 \rangle$.
\vspace{.5cm}

\begin{theo}\label{theonine01} For $G = \Z_9$ there are three types of group games. \end{theo}
\vspace{.5cm}

{\bfseries TYPE I}( $A = \{ 1, 3, 4, 7 \}$ or $ = [1,4]$)- The game $\Gamma[A]$ is reducible, with
$Aut(\Gamma[A]) = \Z_9$.  The six Type I game subsets $B$
such that $\Gamma[B]$ is isomorphic to $\Gamma[A]$ are the elements of $\{ m_a(A) : a \in \Z_9^* \}$.

{\bfseries TYPE II}( $A = \{ 1, 5, 6, 7 \} $) - The game $\Gamma[A]$ is not
reducible, but has $Aut(\Gamma[A]) = \Z_9$. The six Type II game subsets $B$
such that $\Gamma[B]$ is isomorphic to $\Gamma[A]$ are the elements of $\{ m_a(A) : a \in \Z_9^* \}$.

{\bfseries TYPE III}( $A_1 = \{ 1, 3, 4, 7 \}, A_2 = \{ 1, 4, 6, 7 \}$) - There is a non-affine isomorphism
between $\Gamma[A_1]$ and $\Gamma[A_2]$. The Type III games are isomorphic to
the lexicographic product $\Gamma_3 \ltimes \Gamma_3$ with automorphism
group the semi-direct product  $\Z_3 \ltimes (\Z_3)^3$.
The  four Type III game subsets  are $A_1, -A_1, A_2, -A_2$.
\vspace{.25cm}

\begin{proof} Each group game is a pointed game on the pair $(-A,A)$. We let $s_A$ and $s_{-A}$ denote the score
vectors of the tournaments $\Gamma[A]|A$ and $\Gamma[A]|(-A)$. Since the group is commutative, the tournament
$\Gamma[A]|(-A)$ is the reverse of $\Gamma[A]|A$.

TYPE I- With $A = [1,4]$ or $Odd_4$ the score vectors are $s_A = s_{-A} = s_1 = (0,1,2,3)$. These are rigid tournaments and so
the only automorphism which fixes $0$ is the identity. Hence, as we saw in Theorem \ref{theo13}, $Aut(\Gamma[A]) = \Z_9$.

TYPE II- With $A = \{ 1, 5, 6, 7 \} $, $s_A = s_{-A} = s_3 = (1,1,2,2)$. By Lemma \ref{lemrigid} and Theorem \ref{theo12rigid} again have
$Aut(\Gamma[A]) = \Z_9$.

By Corollary \ref{cor14} $\{m_a(A) : a \in \Z_9^* \}$ are the $\phi(9) = 6$ distinct game subsets whose games are isomorphic to $\Gamma[A]$ for
each of these two types.

Thus, Type I and Type II games are non-isomorphic tournament regular representations of the group $\Z_9$.

%
%Type I - The description is that of the $Odd_n$ game or, equivalently, the $[1,n]$ game,   with $n = 4$.
%
%Type II - Observe that
%\begin{itemize}
%\item $ \ A \cap (A + 1) \ = \ \{ 6,7 \} \qquad $ and $\qquad   A \cap (A + 5) \ = \ \{ 6,1 \} $.
%\item $ \ A \cap (A + 6) \ = \ \{ 7 \} \qquad $ and $\qquad   A \cap (A + 7) \ = \ \{ 5 \} $.
%\end{itemize}
%From Theorem \ref{theo28} and its proof, we see that because there is no $i \in A$ with
%$A \cap (A + i)$ empty, $\Gamma[A]$ is not reducible.
%In particular, it is not isomorphic to $\Gamma[Odd_4]$.
%
%Let $\r$ be an automorphism which fixes $0$. As usual, it suffices to show that
%$\r$ is the identity. If $A \cap (A + i)$ contains
%two elements then $\r(A \cap (A + i)) = A \cap (A + \r(i))$ contains two elements.
%Hence, the pair $\{ 1, 5 \}$ is invariant as
%is the pair $\{ 6, 7 \}$.  From Proposition \ref{prop04} it follows that $\r$ fixes each element of $A$.
%
%Applying $m_{-1}$ to the above equations we see that $(-A) \cap (-A + (-1)) = \{ -6, -7 \}$.
%Since $\rho(-A) = -A$ we get that
%$\{ -1, -5 \} = \{ 8, 4 \}$ and $\{ -6, -7 \} = \{ 3, 2 \}$ are each invariant. So
%Proposition \ref{prop04} implies that $\r$ fixes every element of
%$-A$.
%
%It follows from Corollary \ref{cor14} that $\{ m_a(A) : a \in \Z_9^* \}$ are the
%$\phi(9) = 6$ game subsets $B$ with $\Gamma[B]$ isomorphic to $\Gamma[A]$.

TYPE III- With $A = \{ 1, 3, 4, 7 \}$ or $ \{ 1, 4, 6, 7 \}$ we have $s_A = s_2 = (1,1,1,3),$\\ $ s_{-A} = \bar s_2 = (0,2,2,2)$.

Type III is a special case of Example \ref{excomm}. We have \\ $\begin{tikzcd}
\Z_3 \arrow[r,"\theta"] & \Z_9 \arrow[r,"\pi"] & \Z_3
\end{tikzcd}$ with $\theta(j) = 3j$, and $\pi(3j+i) = i$ for $i,j = 0,1,2$. Let $B = \{ 1 \} \subset \Z_3$.
With $A_{0,1} = \{ 1 \}$, $A_1 = A_{0,1} \cup \pi^{-1}(B) $ and with $A_{0,2} = \{ 2 \}$, $A_2 = A_{0,2} \cup \pi^{-1}(B) $.
Thus, $A_1$ and $A_2$ are game subsets for the pair $(\Z_9,\theta(\Z_3))$ and so their inverses $-A_1, -A_2$ are also game subsets for the
pair. Thus, the associated games are all isomorphic to $\Gamma_3 \ltimes \Gamma_3$. By Theorem \ref{theo31} the automorphism groups
are isomorphic to $\Z_3 \ltimes [\Z_3]^3$.

We can describe $\Gamma[A_1]$  as a $3$-cycle of $3$-cycles.
$$ \langle \ \langle 2, 5, 8 \rangle \ \Rightarrow \ \langle 0, 3,  6 \rangle \ \Rightarrow \ \langle 1, 4, 7 \rangle \ \rangle.$$

On the other hand, $\Gamma[A_2]$ is  a   $3$-cycle of $3$-cycles.
$$ \langle \ \langle 2, 8, 5 \rangle \ \Rightarrow \ \langle 0, 6, 3 \rangle \ \Rightarrow \ \langle 1, 7, 4 \rangle \ \rangle.$$

Clearly, the product of transpositions $\r = (8, 5)(6, 3)(7, 4)$ is an isomorphism
between them which fixes $0$.  With $\phi = m_2 = m_{-1}$ on $\Z_3$ and $\gamma_i = \phi$ for $i = 0,1,2$,
$\rho = 1_{\Z_3} \ltimes \gamma$.

The group $\Z_9^* = \{ 1, 2, 4, 8 = -1, -2 = 7, -4 = 5, -8 = 1 \}$ is cyclic of order $6$ with generator $m_2$.
The subgroup  $[m_4]$ is contained in the automorphism groups of
$\Gamma[A_1]$ and $\Gamma[A_2]$ while $m_{2}$  maps
each game to its reverse. In particular, $\r$ is not affine.

We can also view Type III by using the construction of Theorem \ref{theo17}. The set $H = [m_4] = \{ 1, 4, 7 \}$ is the multiplicative subgroup of
$\Z_9^*$  of order $3$. The action of $H$ fixes $3$ and $-3 = 6$. The four game subsets are obtained by choosing one from each pair of $H$ orbits:
$\{H, -H \}, \{ \{3 \}, \{ 6 \} \}$, e.g. $A_1 = H \cup \{ 3 \}, A_2 = H \cup \{ 6 \}$.

\end{proof} \vspace{.5cm}

The Type III cases provide examples of game subsets $A_1, A_2$ of $\Z_9$ with isomorphic associated games but which are not related by
the action of an element of $Z_9^*$.  This is a special case of the following.

\begin{ex}\label{exisogames} Let $p$ be an odd prime. There are $p - 1$ game subsets $$A_1, \dots A_{p-1} \subset \Z_{p^2}$$
with isomorphic associated
games but with no two related by an element of $\Z_{p^2}^*$.\end{ex}

\begin{proof} Define $\xi : \Z_p \tto \Z_{p^2}$ by $\xi(j) = jp$ and $\pi : \Z_{p^2} \tto \Z_p$ by $\pi(j) = j$, or, equivalently, $\pi(j + kp) = j$.
Thus, $\pi$ is a surjective ring homomorphism with kernel $H = \xi(\Z_p)$ an ideal in the ring.
Let $K = \pi^{-1}(1) = \{ 1 + kp \}$, the unique subgroup of $ \Z_{p^2}^*$ of
order $p$, generated by $1 + p$. Each coset $i + H$ is $K$ invariant. If $i \not\in H$, then $i + H$ is a single $K$ orbit.
On the other hand, each element of
$H$ is fixed by $K$. It follows that a $\Z_{p^2}$ game subset $A$ is $K$ invariant if and only if it is a game subset for the pair $(\Z_{p^2},H)$, i.e. if
and only if there are game subsets $A_0, B$ of $\Z_{p}$ such that $A = \xi(A_0) \cup \pi^{-1}(B)$. See Theorem \ref{theo21a}.

Let $A_{0,1} = B = [1,(p-1)/2]$ and for $a = 1, \dots, p-1$ let $A_{0,k} = m_a(A_{0,1})$. Let $\Pi$ be the game on $\Z_{p}$ associated with
$[1,(p-1)/2]$. Thus, $\Pi$ is isomorphic to each of the games on $\Z_{p}$ associated to $B, A_{0,1}, \dots, A_{0,p-1}$.
Define $A_k = A_{0,k} \cup \pi^{-1}(B)$. By Theorem \ref{theo23a} the game $\Gamma[A_k]$ on $\Z_{p^2}$ is isomorphic to the
lexicographic product $\Gamma[B] \ltimes \Gamma[A_{0,k}]$ and so is isomorphic to $\Pi \ltimes \Pi$ for all $k$. If we use
$\r$ the identity on $\Gamma[B]$ and $\gamma_i = m_k$ for all $i \in \Z_{p}$ then $\r \ltimes \gamma$ is an isomorphism from
$\Gamma[A_1] = \Gamma[B] \ltimes \Gamma[A_{0,1}]$ to $\Gamma[A_k] = \Gamma[B] \ltimes \Gamma[A_{0,k}]$.

 By Theorem \ref{theo13} the only automorphisms of $\Pi$ are translations by elements of $\Z_p$. This implies that the game subsets
 $A_{0,1}, \dots, A_{0,p-1}$ are distinct. Furthermore,  $m_a(A_{0,k}) = A_{0,k}$ for $a \in Z_p^*$ only for $a = 1$.

 Now assume that $u \in \Z_{p^2}^*$ and that $m_u(A_{k_1}) = A_{k_2}$. The ideal $H$ is invariant with respect to multiplication by $u$
 and so $m_u(\xi(A_{0,k_1})) = \xi(A_{0,k_2})$.
Because $\Z_{p^2} \setminus H$ is invariant as well, $\pi^{-1}(B) = A_{k_1} \cap (\Z_{p^2} \setminus H) = A_{k_2} \cap (\Z_{p^2} \setminus H)$
 is $m_u$ invariant. This implies that $B$ is invariant with respect multiplication by $\pi(u) \in \Z_{p}^*$. It follows that $\pi(u) = 1 $ and so
 $u \in K$. Since $K$ fixes every element of $H$, it follows that $\xi(A_{0,k_1}) = \xi(A_{0,k_2})$ and so $k_1 = k_2$.

 It follows that distinct subsets $A_{k_1}$ and $A_{k_2}$ are not related by the action of $\Z_{p^2}^*$.

 Now consider the special case when $p$ is a Fermat prime so that $p - 1$ is a power of $2$ and so
 every element of $\Z_{p}^* \setminus \{ 1 \}$ has  even order. It follows that every element of
 $\Z_{p^2}^* \setminus K$ has even order. It then follows from  Theorem \ref{theo16} and its proof that
 $\Z_{p}^*$ acts freely on the game subsets of $\Z_{p}$ and $\Z_{p^2}^*$ acts freely on
 the game subsets of $\Z_{p^2}$ which are not $K$ invariant.

 If $A$ is $K$ invariant then $A = \xi(A_0) \cup \pi^{-1}(B)$ with $A_0$ and $B$ game subsets of $\Z_{p}$. We obtain $(p - 1)^2$ game subsets $A$
 by replacing $A_0, B$ by $m_a(A_0), m_b(B)$ for $a, b \in \Z_{p}^*$. All of the associated games on
 $\Z_{p^2}$ are isomorphic to $\Gamma[B] \ltimes \Gamma[A_0]$. On the other hand, only when $a = b$ is $\xi(A_0) \cup \pi^{-1}(B)$ related to
 $\xi(m_a(A_0)) \cup \pi^{-1}(m_b(B))$ by an element of $\Z_{p^2}^*$.

 Since there are $2^{(p-1)/2}$ game subsets of $\Z_{p}$, it follows that there are $2^{p-1}$ game subsets of $\Z_{p^2}$ which are $K$ invariant.
 These are partitioned into classes of size $(p - 1)^2$ all members of which have isomorphic games. Each of these is in turn partitioned into
 $p - 1$ sets of size $p - 1$ by the action of $\Z_{p^2}^*$.

\end{proof} \vspace{.5cm}

The following question remains open, as far as I know.

\begin{ques}\label{quesisogames} Does there exists a collection of more than $\phi(2n + 1)$ game subsets $A$ of $\Z_{2n + 1}$ all of whose
associated games are isomorphic? In particular, with $2n + 1$ a square-free product of Fermat primes, do there exist game subsets $A$ and $B$ of
$\Z_{2n + 1}$ which are not related by an element of $\Z_{2n + 1}^*$ but which have isomorphic associated games? \end{ques}

$\Box$ \vspace{.5cm}

Returning to the case with $2n + 1 = 9$ we observe the following.

\begin{theo}\label{theonine02} For $G =
\Z_3 \times \Z_3$, $\Gamma[A]$ is isomorphic to $\Gamma_3 \ltimes \Gamma_3$ for every game subset $A$ of $G$. \end{theo}

\begin{proof} The group $\Z_3 \times \Z_3$ is a two-dimensional vector space over the
field $\Z_3$. The group $G^*$ of $2 \times 2$ invertible matrices on $\Z_3$ acts transitively on the set of nonzero vectors.
If $H$ is a one-dimensional subspace, then there are four game subsets for the pair $(G,H)$.
Such a game subset contains a unique  affine subspace and it is parallel to $H$.  Thus, the game subsets of $(G,H)$ determine $H$.
  The subgroup of matrices which
fix $H$ acts transitively on these four game subsets. If a matrix maps $H_1$ to $H_2$ then it maps the game subsets for $(G,H_1)$ to the
game subsets for $(G,H_2)$. There are four subspaces $H$ and so there are 16 game subsets of
 $(G,H)$ for a suitably chosen one-dimensional
subspace $H$. As there are $2^4 = 16$ game subsets for $G$ it follows that all are isomorphic and the associated games are lexicographic products.

\end{proof}
{\bfseries Remark:} By Theorem \ref{theo31} the automorphism group
is isomorphic to $\Z_3 \ltimes [\Z_3]^3$ for each $\Gamma[A]$.
It follows that $\Z_3 \times \Z_3$ does not have a tournament regular representation.

\vspace{.25cm}

Finally, we consider a non-trivial automorphism $\r$ on a game $\Gamma$ of size $9$. If $\r$ has a fixed point, then by Proposition \ref{prop04} and
Proposition \ref{prop31af} every non-trivial cycle in the permutation $\r$ must have length an odd number at most $4$ and greater than $1$.
That is, it has length $3$ and so $\r$ has order $3$. If $\r$ consists of a single cycle of length $9$, then we can identify the set of
vertices with $\Z_9$ so that $\r$ is translation by a generator. In that case, Theorem \ref{theo11} implies that $\Gamma$ is isomorphic with one
of the group games on $\Z_9$ described above. If the permutation has no fixed points and does not consist of a single cycle, then by Proposition
\ref{prop10} (g), it contains an odd number of cycles whose lengths sum to $9$. So the remaining possibility has $\r$ with three $3$-cycles and
so again $\r$ has order $3$. For the games of Type III isomorphic to $\Gamma_3 \ltimes \Gamma_3$ with automorphism group $\Z_3 \ltimes [\Z_3]^3$, it
is easy to check that the automorphism group contains permutations of all four types: i.e. of order $9$ and of order $3$ with exactly one, two or
three $3$-cycles.

In \cite{CH} Chamberland and Herman compute the number of isomorphism classes  and associated automorphism groups for games of size $9, 11$ and $13$.
In addition, they provide a beautiful geometric description of the three games of size $7$. For size $9$, in addition to the three group games, they
find seven rigid games
and five with automorphism group $\Z_3$.

We have seen in Example \ref{ex31a} a rigid example of size $9$. We begin with the tournament $\Theta_3$ of size $4$ with score vector $(1,1,2,2)$, for example,
we may use $\Theta_3 = \Gamma[A]|A$ where $\Gamma[A]$ is a Type II group game.
The tournament $\Theta_3 $ is rigid and the scores $0$ and $3$ do not occur.  It follows that the double $2 \Theta_3$ is a rigid game of size $9$.

Let $\Theta$ be the tournament on $J = \{ 1, 2, 3, 4 \}$ with score vector $(1,1,1,3)$ such that $\langle 2, 3, 4 \rangle$ is the $3$-cycle in $\Theta$.
It is clear that $Aut(\Theta)$ is isomorphic to $\Z_3$. Let $\Gamma = 2 \Pi$ and  $\bar \Gamma = \Gamma/ \langle 2+, 3+, 4+ \rangle$.

In Example \ref{ex29ac} we observed that
\begin{align}\label{eqnine02}
\begin{split}
dom(\Gamma) \ = \ [1-, 1+, 0]  \ &\cup \  \{ (p-, p+) : p = 2, 3, 4 \}, \\
dom(\bar \Gamma) \ = &\ [1-, 1+, 0]. \\
\end{split}
\end{align}

\begin{ex}\label{exnine03} The games $\Gamma$ and $\bar \Gamma$ are non-isomorphic games of size $9$. Each has automorphism group isomorphic to
$\Z_3$. \end{ex}

\begin{proof} Since an isomorphism associates the domination graphs, it follows from (\ref{eqnine02}) that $\Gamma$ and $\bar \Gamma$ are not
isomorphic. Since  the domination graph is invariant with respect to an automorphism, it is clear that an automorphism of either must
map $[1-, 1+, 0]$ to itself and so is fixed on $[1-, 1+, 0]$ by Proposition \ref{propreduce}. In particular, it must
fix $0$. For the double, $\Gamma = 2 \Theta$, Proposition \ref{prop30} implies
that the inclusion $2: Aut(\Theta) \tto Aut(\Gamma)$ is an isomorphism and so
$\Gamma$ has automorphism group isomorphic to $\Z_3$.

Because the cycle  $ \langle 2+, 3+, 4+ \rangle$ is invariant with respect to the $\Z_3$ action,
it follows from Proposition \ref{prop31ah} that $\Z_3$ acts on
$\bar \Gamma$. Since an automorphism
 fixes $0$, it restricts to an automorphism of $\Gamma|J_-$ which is isomorphic to $\Theta$.  Hence, the $\Z_3$ action includes all of the
 automorphisms of $\bar \Gamma$.

\end{proof}

We note that by using the construction of Exercise \ref{exiso02} we obtain an isomorphism from $\Gamma$ to $\Gamma^{-1}$.  The same map
induces an isomorphism from $\bar \Gamma$ to $\bar \Gamma^{-1}$. Finally, we observe that $\Gamma|\Gamma(1+)$ has score $\bar s_2 = (0,2,2,2)$
and $\Gamma|\Gamma^{-1}(1+)$ has score $s_2 = (1,1,1,3)$. Thus, using $1+$ as a base point we obtain a different view of $\Gamma$ as a
pointed game.

 \vspace{1cm}

\section{Universal Tournaments}\label{unitour}
\vspace{.5cm}

In this section, we will consider infinite as well as finite tournaments. We will write $(S,\Pi)$ for a tournament $\Pi$ on a set $S$ or just use $\Pi$ when
$S$ is understood.
For a set $S$ we write $|S|$ for the cardinality of $S$.

Let $(S_1,\Pi_1)$ and $(S_2,\Pi_2)$ be tournaments. Generalizing from the finite case, a
\emph{tournament morphism}\index{tournament morphism} $\r : \Pi_1 \tto \Pi_2$ is
a mapping $\r : S_1 \tto S_2$ such that $\bar \r^{-1}(\Pi_2) \subset \Pi_1$. That is, $\r(i) \to \r(j)$ in $\Pi_2$ implies
$i \to j$ in $\Pi_1$. Because $\Pi_1$ and $\Pi_2$ are tournaments, $i \to j$ in $\Pi_1$ implies $\r(i) \to \r(j)$ in $\Pi_2$ unless $\r(i) = \r(j)$.
An injective morphism is called an \emph{embedding}\index{tournament embedding} in which case $i \to j$ in $\Pi_1$ if and only if $\r(i) \to \r(j)$ in $\Pi_2$.
A bijective morphism is called an \emph{isomorphism}\index{tournament isomorphism}, in which case, the inverse map $\r^{-1} : S_2 \tto S_i$ defines the
inverse isomorphism $\r^{-1} : \Pi_2 \tto \Pi_1$. An isomorphism from $\Pi$ to itself is called an \emph{automorphism}\index{tournament automorphism} of $\Pi$.

When a tournament morphism $\pi : \Pi_1 \tto \Pi_2$ is surjective we call it a \emph{tournament projection}.
\index{tournament projection} We may then choose for each $j \in S_2$ an
element $\r(j) \in \pi^{-1}(j) \subset S_1$. This defines a \emph{splitting}\index{splitting} of $\pi$ which is an embedding
$\r : \Pi_2 \tto \Pi_1$ such that $\pi \circ \r = 1_{S_2}$, the identity on $S_2$.

Recall that for a tournament  $(S,\Pi)$ and a subset $S_1$ of $S$,  the tournament $\Pi|S_1 = \Pi \cap (S_1 \times S_1)$
 on $S_1$ is the \emph{restriction} of $\Pi$ to $S_1$.
The inclusion map $inc : S_1 \tto S$ defines an embedding of $\Pi|S_1$ into $\Pi$.  On the other hand, if $\r : \Pi_1 \tto \Pi_2$ is
an embedding and $S_3 = \r(S_1) \subset S_2$, then $\r : S_1 \to S_3$ defines an isomorphism of $\Pi_1$ onto the restriction
$\Pi_2|S_3$.

Clearly, the composition of morphisms is a morphism and so we obtain the category of tournaments.  For a tournament $\Pi$ on $S$ we let
$Aut(\Pi)$ denote the group of automorphisms of $\Pi$ with identity $1_S$.

For tournaments $(S,\Pi)$ and $(T,\Gamma)$ if $S_0 \subset S$ and $\r : \Pi|S_0$ $ \tto \Gamma$ is an embedding, we say that
$\r$ \emph{extends} to $\Pi$ if there exists an embedding $\t : \Pi \tto \Gamma$ such that $\t = \r$ on $S_0$.

\begin{df}\label{unidef00} If $(T,\Gamma)$ is a tournament and $T_0 \subset T$, then we say that
$T_0$ satisfies the \emph{simple extension property }\index{simple extension property}\index{extension property!simple}
in $\Gamma$ if for every  subset $J \subset T_0$ which is either finite or cofinite there exists $v_J \in T$ such that in $\Gamma$
\begin{equation}\label{unieq00}
v_J \to j \quad \text{for all} \ j \in J, \qquad \text{and} \qquad j \to v_J  \quad \text{for all} \ j \in T_0 \setminus J.
\end{equation}
We will describe this by saying $v_J$ \emph{chooses} $J \subset T_0$ for $\Gamma$.\index{$v_J$ chooses $J$}
\end{df}
\vspace{.5cm}%

Since a tournament contains no diagonal pairs, it follows that the set $\{ v_J : J \subset T_0 \}$ consists of $|P_f(T_0)|$ vertices
disjoint from $T_0$ where $P_f(T_0)$ is the set  subsets of $T_0$ which are finite or cofinite. So if $T_0$ is finite $|P_f(T_0)| =  2^{|T_0|}$ and if
$T_0$ is countably infinite, then $P_f(T_0)$ is countably infinite.

\begin{lem}\label{unilem00a} (a) Assume $(S,\Pi)$ is a finite tournament, and $S_0 \subset S$ with $|S \setminus S_0| = 1$, i.e.
$S$ contains a single additional vertex. If $\r : \Pi|S_0 \tto \Gamma$ is an embedding and $\r(S_0)$ satisfies the simple
extension property in $\Gamma$, then $\r$ extends to $\Pi$.

(b) If $(S_0,\Pi_0), (S_1,\Pi_1)$ are tournaments with $S_0 \cap S_1 = \emptyset$ and with $|S_1| \geq |P_f(S_0)|$, then there exists a
tournament $\Pi$ on  set $S_0 \cup S_1$ with
$\Pi_0 = \Pi|S_0, \Pi_1 = \Pi|S_1$ and such that $S_0$ satisfies the simple extension property in $\Pi$.

%(c) If $(S_0,\Pi_0), (S_1,\Pi_1)$ are tournaments with $S_0 \cap S_1 = \emptyset$ and with $S_0, S_1$ countably
%infinite, then there exists a
%tournament $\Pi$ on  set $S_0 \cup S_1$ with
%$\Pi_0 = \Pi|S_0, \Pi_1 = \Pi|S_1$ such that $S$ satisfies the simple extension property in $\Pi$ for every finite subset $S \subset S_0$.
\end{lem}

\begin{proof} (a) If $S = \{ v \} \cup S_0$, then we let $J = \r(\Pi(v)) \subset \r(S_0)$ and we obtain the extension by mapping
$v$ to $v_J$.

(b) Let $J \mapsto v_J$ be an injective map from the set $P_f(S_0)$ to $S_1$. Define $\Pi$ so that $\Pi_0 \cup \Pi_1  \subset  \Pi$
and for all $J \in P_f(S_0)$
\begin{equation}\label{unieq01}
\{ (v_J,j) : \ j \in J, \} \ \cup \ \{(j,v_J) :   \ j \in S_0 \setminus J \}  \ \subset \ \Pi, \hspace{1cm}
\end{equation}
with the orientation of the edges between the vertices in $S_0$ and the remaining vertices in $S_1$ chosen arbitrarily.

Clearly, $v_J $ chooses $J \subset S_0$ for $\Pi$ and so  $S_0$ satisfies the simple extension property in $\Pi$.

\end{proof} \vspace{.5cm}

\begin{df}\label{unidef01} A tournament $(T,\Gamma)$  is called \emph{universal}\index{universal tournament}\index{tournament!universal}  when it satisfies

{\bfseries The Extension Property} If $\Pi$ is a  tournament on a countable set $S$, $S_0$ is a finite subset of $S$ and
$\r : \Pi|S_0 \tto \Gamma$ is an embedding, then $\r$ extends to $\Pi$.
\end{df}
\vspace{.5cm}

In \cite{C} Cherlin defines genericity with respect to a class of finite tournaments. What we are calling a universal tournament is a tournament which is
generic for the class of all finite tournaments.

\begin{prop}\label{uniprop01a} If $(T,\Gamma)$  is a universal tournament and $(S,\Pi)$ is an arbitrary countable tournament then there exists
an embedding $\r : \Pi \tto \Gamma$. \end{prop}

\begin{proof} For $u \in S$ and $v \in T$, let $S_0 = \{ u \}$. A tournament on a singleton set is the trivial, empty, tournament. Hence,
$\r(u) = v$ is an embedding of $\Pi|S_0$ into $\Gamma$. By the extension property $\r$ extends to an embedding of $\Pi$ into $\Gamma$.

\end{proof} \vspace{.5cm}

\begin{prop}\label{uniprop02} In order that a tournament $(T,\Gamma)$  be universal, it is necessary and sufficient that
every finite subset $T_0$ of $T$ satisfies the simple extension property in $\Gamma$. \end{prop}

\begin{proof} If $\Gamma$ is universal and $T_0$ is a finite subset of $T$, then by Lemma \ref{unilem00a} (b) there exists
a finite set $T_1 \supset T_0$  and a tournament $\Pi_1$ on $T_1$ with $\Pi_1|T_0 = \Gamma|T_0 $ and such that $T_0$
satisfies the simple extension property
in $\Pi_1$.  Let $\t : \Pi_1 \to \Gamma$ be an extension of the inclusion of $\Gamma|S_0$ into $\Gamma$. If $J \subset T_0$ and $u_J \in T_1$ chooses
$J \subset T_0$ for $\Pi_1$, then
 $v_J = \t(u_J)$ chooses $J \subset T_0$ for $\Gamma$.

Now assume that every finite subset of $T$ satisfies the simple extension property in $\Gamma$. Assume that
$\Pi$ is a  tournament on a countable set $S$, $S_0$ is a finite subset of $S$ and
$\r : \Pi|S_0 \tto \Gamma$ is an embedding.

Count the finite or countably infinite set of vertices $v_1, v_2, \dots $ of $S \setminus S_0$.
Let $S_k = S_0 \cup \{ v_1, \dots, v_k \}$. Inductively, with $\t_0 = \r$ we can apply Lemma  \ref{unilem00a} (a) to
define an embedding $\t_k :  \Pi|S_k \tto \Gamma$
which extends $\t_{k-1}$ for $k \geq 1$. If $S$ is finite with $|S \setminus S_0| = N$ then $\t = \t_N$ is the required
extension.  If $S$ is countably infinite then we use $\t = \bigcup_k \t_k$, with $\t(v_i) = \t_k(v_i) $ for all $k \geq i$.

\end{proof} \vspace{.5cm}

\begin{df}\label{unidef03a} For a tournament $(T,\Gamma)$ we call $R \subset T$  a \emph{generic subset}\index{generic subset} for $\Gamma$
when every finite subset $T_0 \subset T$ satisfies the simple extension property in $\Gamma|(R \cup T_0)$. If $T$ is the disjoint union of generic subsets
$R,R'$ then $(R, R')$ is called a \emph{generic partition}\index{generic partition} for $\Gamma$. \end{df}
\vspace{.25cm}

\begin{prop}\label{uniprop03b} Let $(T,\Gamma)$ be a tournament.
\begin{itemize}
\item[(a)]  If $R$ is a generic subset for $\Gamma$, then $(R,\Gamma|R)$ is a universal tournament.

\item[(b)]  A tournament $(T,\Gamma)$ is universal when $T$ is a generic subset for $\Gamma$.

\item[(c)] If $R_1 \subset R_2 \subset T$ and $R_1$ is a generic subset for $\Gamma$, then $R_2$ is a generic subset for $\Gamma$. In particular,
a tournament admits a generic subset if and only if it is a universal tournament.

\item[(d)] If $R$ is a generic subset for $\Gamma$ and $S \subset T$ is finite, then $R \setminus S$ is a generic subset for $\Gamma$.
\end{itemize}
\end{prop}

\begin{proof} (a), (b) and (c) are obvious.

(d) If $T_0 \cup S$ satisfies the simple extension property in $\Gamma|(R \cup T_0 \cup S)$, then $T_0$
satisfies the simple extension property in $\Gamma|((R \setminus S) \cup T_0)$. To see this, observe that if $J \subset T_0$ and
$v_J$ chooses $J \subset T_0 \cup S$ then $v_J \in R \setminus S$.

 \end{proof} \vspace{.5cm}

\begin{prop}\label{uniprop03c} Let $\xi_1, \xi_2$ be distinct automorphisms of a universal tournament $(T,\Gamma)$.
The set $ Eq = \{ i \in T : \xi_1(i) = \xi_2(i) \}$
is not a generic subset
for $\Gamma$, and its complement $Eq'$ is a generic subset for $\Gamma$.\end{prop}

\begin{proof} By composing with $\xi_2^{-1}$ we reduce to the case with $\xi_2$ the identity so that with $\xi = \xi_2^{-1}\xi_1$ the set
 $Eq = \{ i : \xi(i) = i \}$. If  $j \not= \xi(j)$, then for all $i \in T$,
$i \to j$ implies $i = \xi(i) \to \xi(j)$ and similarly, $j \to i$ implies $\xi(j) \to i$. Since there does not exist $i \in Eq$
which chooses $\{ j \} \subset \{ j, \xi(j) \}$ for $\Gamma|(Eq \cup \{ j, \xi(j) \})$ it follows that $Eq$ is not a generic subset for $\Gamma$.

Since $T$ is a generic subset for $\Gamma$, Proposition \ref{uniprop03b}(d) then implies that $Eq'$ is infinite.

Let $J \subset S$ with $S$ a finite subset of $T$. Choose $j \in Eq' \setminus (S \cup \xi^{-1}(S)$.  That is, the pair $\{j, \xi(j) \}$ is disjoint from
$S$. If $v_J$ chooses $J \cup \{ j \} \subset S \cup \{j, \xi(j) \}$ then, as above, $v_J \in Eq'$. Clearly, $v_J$ chooses $J \subset S$.
It follows that $Eq'$ is a generic subset for $\Gamma$.

 \end{proof} \vspace{.5cm}

Now let $\{ (S_k,\Pi_k) \}$ be a sequence of tournaments with $\{ S_k : k \in \N \}$ a pairwise disjoint sequence sets.
Let $T_k = \bigcup_{i=1}^k S_i$. Assume that
$|S_{k+1}| \geq |P_f(T_k)|$ for all $k \in \N$.

 Let $(T_1,\Gamma_1) = (S_1,\Pi_1)$. Proceed inductively. to define $(T_k,\Gamma_k)$.

Use Lemma \ref{unilem00a}(b) to
construct $(T_{k+1},\Gamma_{k+1})$ so that $\Gamma_{k+1}|T_k = \Gamma_k, \Gamma_{k+1}|S_{k+1} = \Pi_{k+1}$,
and so that
 $ T_k$ satisfies the simple extension property in $\Gamma_{k+1}$.

  Finally, let $(T,\Gamma) = \bigcup_k  (T_k,\Gamma_k)$.
 For $K$ any infinite subset of $\N$ let $R_K =  \bigcup_{k \in K} S_k$.

 \begin{lem}\label{unilem03}The tournament $(T,\Gamma)$ is universal and if $K \subset \N$ is infinite,
 then $R_K$ is a generic subset for $\Gamma$. \end{lem}

\begin{proof} Let $J \subset S \subset T$ with $S$ finite. There exists $k$ such that $S \subset T_k$ and $k+1 \in K$.
Because $T_k$ has the simple extension property in $T_{k+1}$ there exists $v_J \in T_{k+1}$ which chooses $J \subset T_k$ for
$\Gamma_{k+1}$. In particular, $v_J \in S_{k+1}$. Since $k+1 \in K$, $v_J \in R_K$.  Clearly, $v_J$ chooses $J \subset S$ for $\Gamma|(R_K \cup S)$.
Hence, $R_K$ is a generic subset and so $\Gamma$ is a universal tournament.

 \end{proof} \vspace{.5cm}

\begin{theo}\label{unitheo04} Assume that $(T_1,\Gamma_1)$ and $(T_2,\Gamma_2)$ are countable, universal tournaments.
\begin{enumerate}
\item[(a)] If $S$ is a finite subset of $T_1$ and
$\r : \Gamma_1|S \tto \Gamma_2$ is an embedding, then $\r$ extends to an isomorphism  $\xi : \Gamma_1 \tto \Gamma_2$.

\item[(b)] If $(R_1,R_1')$ and $(R_2,R_2')$ are generic partitions of $\Gamma_1$ and $\Gamma_2$ respectively, then there exists
an isomorphism  $\xi : \Gamma_1 \tto \Gamma_2$ with $\xi(R_1) = R_2$ and $\xi(R_1') = R_2'$.
\end{enumerate} \end{theo}

\begin{proof}  (a) This is a standard back and forth argument. First note that Proposition \ref{uniprop01a} implies that  a
generic tournament is infinite. Let $u_1,u_2, \dots $ be a counting of the vertices of $T_1 \setminus S_0$ and $v_1, v_2, \dots $ be
a counting of the vertices of $T_2 \setminus \bar S_0$ with $S_0 = S$ and  $\bar S_0 = \r(S)$. Let
$\xi_0: \Gamma_1|S_0 \tto \Gamma_2|\bar S_0$ be the isomorphism
obtained by restricting $\r$.

Inductively, we construct for $k \geq 1$   \begin{itemize}
\item $S_k \supset S_{k-1} \cup \{ u_k \},$
\item $ \bar S_k \supset \bar S_{k-1} \cup \{  v_k \}, $
\item  $\xi_k : \Gamma_1|S_k \tto \Gamma_2|\bar S_k$ an isomorphism which extends $\xi_{k-1}$.
\end{itemize}
  Define $S_{k+.5}$ to be $S_k$ together with the
first vertex of $T_1 \setminus S_0$ which is not in $S_k$. This vertex is $u_{k+1}$ unless $u_{k+1}$ is already an element of
$S_k$. Extend $\xi_k$ to define an embedding $\xi_{k + .5}$ on $S_{k + .5}$.
Let $\bar S_{k + .5} = \xi_{k + .5}(S_{k + .5})$ so that $\xi_{k + .5}^{-1} : \Gamma_2|\bar S_{k + .5} \tto \Gamma_1$ is an embedding.
Define $\bar S_{k+1} $ to be $\bar S_{k +.5}$ together with the first  vertex of $T_2 \setminus \bar S_0$
which is not in $\bar S_{k +.5}$. Extend to define the embedding
$\xi_{k+1}^{-1}: \Gamma_2|\bar S_{k+1} \tto \Gamma_1$ and let $S_{k+1} = \xi_{k+1}^{-1}(\bar S_{k+1})$.

The union $\xi = \bigcup_k \xi_k$ is the required isomorphism.

(b) The argument is similar to that of (a). Instead of going first from $T_1$ to $T_2$ and then back from $T_2$ to $T_1$ we go from $R_1$ to $R_2$
then from $R_1'$ to $R_2'$ and then back.

\end{proof} \vspace{.5cm}

\begin{cor}\label{unicor05} (a) There exist  countable universal tournaments, unique up to isomorphism.  In fact, if
$(T_1,\Gamma_1)$ and $(T_2,\Gamma_2)$ are countable universal tournaments with $i_1 \in T_1, i_2 \in T_2$ then there
exists an isomorphism $\xi : \Gamma_1 \tto \Gamma_2$ with $\xi(i_1) = i_2$.

(b) A countable universal tournament admits uncountably many distinct generic partitions.

(c) For a countable universal tournament the automorphism group is uncountable.
\end{cor}

\begin{proof} (a) From Lemma \ref{unilem03}, it follows that there exist countable universal tournaments.

As in the proof of Proposition \ref{uniprop01a} the
map $i_1 \mapsto i_2$ gives an embedding of trivial tournament $\Gamma_1|\{ i_1 \}$ into $\Gamma_2$. It extends to an isomorphism by
Theorem \ref{unitheo04}(a).

(b) From the uniqueness of part (a), we may assume that the universal tournament is $(T,\Gamma)$ constructed for Lemma \ref{unilem03}.

The set $P_f(\N)$ is countable.  Hence, the subsets  $K \subset \N$ with $K$ and its complement
$K'$ both infinite form an uncountable collection in the power set of $\N$. For such a set $K$, the pair $(R_K,R_{K'})$ is a generic partition
for $\Gamma$ by the lemma. Clearly, distinct sets $K$ yield distinct partitions.

(c) This is clear from (b) and Theorem \ref{unitheo04} (b).

\end{proof} \vspace{.5cm}

\begin{cor}\label{unicor06} Let $(T,\Gamma)$ be a universal tournament and $S$ be a finite subset of  $T$.

(a) If $(i_1,j_1), (i_2,j_2) \in \Gamma$, then there exists an automorphism $\xi$ of $\Gamma$ with $\xi(i_1) = i_2$ and $\xi(j_1) = j_2$.
That is, the tournament $(T,\Gamma)$ is symmetric.

(b) If $\r$ is an embedding  of $\Gamma|S$ into $\Gamma$, then there exists an automorphism $\xi$ of $\Gamma$ which
restricts to $\r$ on $S$.

(c) If $\r$ is an automorphism of $\Gamma|S$, then there exists an automorphism $\xi$ of $\Gamma$ which
restricts to $\r$ on $S$.
\end{cor}

\begin{proof} (a) Let $S_0 = \{ i_1, j_1 \}$ and define $\r : S_0 \to T$ by $\r(i_1) = i_2$ and $\r(j_1) = j_2$. Clearly,
$\r$ is an embedding of $\Gamma|S_0$ into $\Gamma$.
By Theorem \ref{unitheo04}(a) $\r$ extends to an automorphism $\xi$.

(b) follows from Theorem \ref{unitheo04}(a) and (c) is a special case of (b).

\end{proof} \vspace{.5cm}

Let $(T,\Gamma)$ be a tournament and $S$ be a nonempty, finite subset of $T$. For $J \subset S$, let
  \begin{align}\label{unieq02}
  \begin{split}
T_J = \{ i \in T : (i,j) \in \Gamma \ &\text{for all} \ j \in J  \\
\text{and} \ (j,i) \in \Gamma \ &\text{for all} \ j \in S \setminus J \}.
\end{split}
\end{align}
That is, $T_J$ is the set of $i \in T$ which choose $J \subset S$  for $\Gamma$. For example, for $j \in T$ and $S = \{ j \}$,
$T_{ \emptyset} = \Gamma(j), T_{\{ j \}} = \Gamma^{-1}(j)$.

Clearly, $\{ S \} \cup \{ T_J : J \subset S \}$ is a partition of $T$ into $1 + 2^{|S|}$ subsets.

\begin{prop}\label{uniprop07} Assume $(T,\Gamma)$ is a universal tournament, $S$ is a nonempty, finite subset of  $T$ and $J $ is a subset of $S$.
\begin{enumerate}
\item[(a)] $T \setminus S$ is a generic subset for $\Gamma$ and so the restriction $(T \setminus S, \Gamma|(T \setminus S))$ is a universal tournament.
\item[(b)]$T_J$ is a generic subset for $\Gamma|(T \setminus S)$ and so the restriction $(T_J,\Gamma|T_J)$ is a universal tournament.
\item[(c)] For any $j \in S$ the restriction $(T_J \cup \{ j \},\Gamma|(T_J \cup \{ j \}))$ is not a universal tournament.
\item[(d)] The subset $T_J$ is not a generic subset for $\Gamma$.
\end{enumerate} \end{prop}

\begin{proof} (a) $T$ is a generic subset for $\Gamma$ and $S$ is finite. By Proposition \ref{uniprop03b}(d)
 $T \setminus S$ is a generic subset for $\Gamma$. So $(T \setminus S,\Gamma|(T \setminus S))$ is universal by Proposition \ref{uniprop03b}(a).

(b) Assume that $S_1$ is a finite subset of $T \setminus S$ and $K \subset S_1$. Let $S_1^+ = S_1 \cup S$ and $K^+ = K \cup J$. There
exists  $v_{K^+} \in T$ such that $v_{K^+} $ chooses $K^+ \subset S_1^+$  for $\Gamma$.
It follows, first, that $v_{K^+} $ chooses $J \subset S$  for $\Gamma$ and so
$v_{K^+}\in T_J$. But also $v_{K^+} $ chooses $K \subset S_1$  for $\Gamma|(T \setminus S)$. Thus, $S_1$ satisfies
 the simple extension property in $\Gamma|(T_J \cup S_1)$. Thus, $T_J$ is a generic subset for  $\Gamma|(T \setminus S)$.
 So $(T_J, \Gamma|(T_J))$ is a universal tournament by Proposition \ref{uniprop03b}(a)again.

(c) If $j \in J$ there does not exist $i \in T_J \cup \{ j \}$ such that $j \to i$. $j \in S \setminus J$
there does not exist $i \in T_J \cup \{ j \}$ such that $i \to j$. In either case, we see that $\{ j \}$ does not satisfy the
simple extension property in $\Gamma|(T_J \cup \{ j \})$.

(d) If $T_J$ were a generic subset for $\Gamma$, then $T_J \cup \{ j \}$ would be a generic subset for any $j \in S$ and so the restriction
 $\Gamma|(T_J \cup \{ j \})$ would be
a universal tournament, contradicting (c).

\end{proof} \vspace{.5cm}

\begin{ex}\label{uniex08} (a) For $(T,\Gamma)$ a countable universal tournament, there exists $T_0$ a proper infinite subset of $T$
and an embedding of $\Gamma|T_0$ into $\Gamma$ which cannot be extended to an embedding of $\Gamma$ into itself.

(b) There exists a tournament which is not universal but into which every countable tournament can be embedded.
\end{ex}

\begin{proof} Let $i \in T$, $J = S = \{ i \}$, $T_0 = T_J$ and $T_1 = S \cup T_J$. Since $\Gamma|T_0$ is universal by Proposition \ref{uniprop07},
Corollary \ref{unicor05} implies that there exists an isomorphism $\xi : \Gamma|T_0 \tto \Gamma$. Since $\xi$ is surjective,
it cannot be extended to an embedding even of $\Gamma|T_1$ into $\Gamma$.

Since $(T_0,\Gamma|T_0)$ is universal and $T_0 \subset T_1$, it follows that every countable tournament can be embedded into $\Gamma|T_1$.
However, by Proposition \ref{uniprop07}(c) $\Gamma|T_1$ is not universal.

\end{proof} \vspace{.5cm}

Following \cite{L} and \cite{C} we call a tournament homogeneous when it satisfies condition (b) of Corollary \ref{unicor06}.

\begin{df}\label{unidef08a}  A tournament $(T,\Gamma)$ is called \emph{homogeneous}\index{homogeneous}\index{tournament!homogeneous}
if whenever
 $\r$ is an embedding  of $\Gamma|S$ into $\Gamma$ with $S$ a finite subset of $T$, there exists an automorphism $\xi$ of $\Gamma$ which
restricts to $\r$ on $S$. \end{df}\vspace{.5cm}

The trivial tournament on a singleton and the $3$-cycle are the only finite homogeneous tournaments. Among countably infinite tournaments,
Lachlan proved in \cite{L} that there are two isomorphism classes which are homogeneous in addition to the universal tournament.  That is, there
exist countably infinite homogeneous tournaments which are not universal.  For a clear picture see Cherlin's exposition in \cite{C}. We saw above that
there exist tournaments which contain all finite tournaments but which are nonetheless not universal. On the other hand,
these two conditions together
are equivalent to universality. This is Lachlan's observation that the homogeneous tournaments which contain
every finite tournament form a single isomorphism class.

\begin{prop}\label{uniprop08b} If $(T,\Gamma)$ is a homogeneous tournament into which every finite tournament can be embedded, then
$(T,\Gamma)$ is universal. \end{prop}

\begin{proof}  Let $(S,\Pi)$ be a tournament and $\r :\Pi|S_0 \tto \Gamma$ be an embedding with $S_0$ a finite subset of $S$.
If $S$ is countably infinite we can write it as an increasing union of finite subsets beginning with $S_0$ and extend inductively.
Thus, we reduce to the case where $S$ itself is finite. By assumption, there exists some embedding $\phi : \Pi \to \Gamma$. Let
$T_0 = \phi(S_0)$ so that $\r \circ (\phi|T_0)^{-1}$ is an embedding of $\Gamma|T_0$ into $\Gamma$. By homogeneity there exists
an automorphism $\xi$ of $\Gamma$ which extends $\r \circ (\phi|T_0)^{-1}$. On $\Pi|S_0$ the composition $\xi \circ \phi$ agrees with $\r$ and so
is the required extension.

\end{proof} \vspace{.5cm}

Call a group $G$, finite or infinite, an \emph{odd group}\index{odd group} when it contains no element of finite even order. A subset $A \subset G$ for
is a \emph{graph subset}\index{graph subset} when $A \cap A^{-1} = \emptyset$. It is a
  \emph{game subset}\index{game subset} when the three sets $\{ e \}, A, A^{-1}$ partition $G$. The existence of a game subset requires that
the group be odd. Conversely, when a group is odd we can obtain game subsets by using the Axiom of Choice to choose an
element from each pair $\{ a, a^{-1} \}$ with $a \not= e$. A game subset is normal \index{normal game subset}\index{game subset!normal}
when it is invariant with respect to all inner automorphisms.  From Lemma \ref{lem11norma} and its proof we see that $G$ admits a normal
game subset if and only if $a^2 = b^2$ implies $a = b$ in $G$ and that this condition holds if every element of $G$ has odd finite order.
On the other hand, the $1$-relator group  generated by $a$ and $b$ and with the single relation
$a^2b^{-2}$ is a countable group with no elements of finite order and so is an odd group in the above sense.
As my colleague Ben Steinberg pointed out to me, it
 is isomorphic to the fundamental group of the Klein Bottle, thought of as two M\"{o}bius Strips joined along their boundaries.
 So we will call it the \emph{Klein Bottle group}\index{Klein Bottle group}. It is an odd group which does not admit a normal game subset. In general, an
odd group fails to admit a normal game subset if and only if it contains a non-abelian factor of the Klein Bottle group.

If $A$ is a graph subset for $G$ then, as in the finite case, we define the digraph $\Theta[A] = \{ (i,j) \in G \times G : i^{-1}j \in A \}$. When $A$ is
a game subset, we write $\Gamma[A]$ for the associated tournament
which we
call the \emph{group game}\index{group game} associated with $A \subset G$.
In any case, left translation action of $G$ is an action of $G$ on $\Theta[A]$ and if $A$ is a normal game subset, then
the right, as well as the left,  translations on $G$ are isomorphisms of $\Gamma[A]$.

The two isomorphism classes of countably infinite, homogeneous tournaments other than the universal tournament contain group game representatives.
We pause to describe them, leaving to Lachlan \cite{L} and Cherlin \cite{C} the beautiful and intricate proof that there are no other isomorphism classes.

Let $\Q$ be the additive group of rational numbers and let $\Q_+$ be the set of positive rationals, i.e.
$\Q_+  = \{ r \in \Q : r > 0 \}$. Clearly, $\Q_+$ is a game subset with
$\Gamma[\Q_+] = \{ (r_1,r_2) : r_1 < r_2 \}$.
The finite subtournaments of $(\Q,\Gamma[\Q_+])$ are exactly the finite orders. 

\begin{prop}\label{prophomo1} The tournament $(\Q,\Gamma[\Q_+])$  is homogeneous. \end{prop}

\begin{proof} If $D_1, D_2$ are countable, order-dense
ordered sets, with $S \subset D_1$ finite and
$\r : S \to D_2$ an order-preserving injection, then the usual back-and-forth argument shows that $\r$ can be extended to an order isomorphism
$\xi : D_1 \to D_2$. 

This shows that the group game $(\Q,\Gamma[\Q_+])$ is homogeneous.

\end{proof} \vspace{.5cm}

Let $Z$ be the unit circle in the complex plane, a compact abelian topological group under multiplication. Let $Z_+$ be the set of elements of
the circle with positive imaginary part, i.e. $Z_+ = \{ z \in Z : Im(z) > 0 \}$.
If $z \in Z_+$ then its inverse, the conjugate $\bar z \not\in Z_+$. Furthermore, $-z \not\in Z_+$. The digraph $\Theta[Z_+]$
fails to be a tournament because of the omission of edges between antipodal pairs $(z,-z)$. Because the group is abelian, conjugation $z \mapsto \bar z$
is an isomorphism from $\Theta[Z_+]$ to its inverse. Because $-1 \in Z$, $z \mapsto -z$ is an automorphism of $\Theta[Z_+]$. Thus, for $z_1, z_2 \in Z$
\begin{equation}\label{eqcircle}
z_1 \to z_2 \quad \Longrightarrow \quad \bar z_2 \to \bar z_1,\ \ -z_2 \to z_1, \ \ -z_1 \to -z_2.
\end{equation}

If $S \subset Z$,  then the restriction $\Theta[Z_+]|S$ is a tournament if and only if $S$ is disjoint from $-S$. We call such a set
antipode free or an \emph{a-free subset}\index{a-free subset}.

If $\om_{2n+1}$ is the primitive $2n+1$ root of unity $\om_{2n+1}  = exp(\frac{2 \pi i}{2n+1})$,
then $k \to \om_{2n+1}^k$ is a group isomorphism from $\Z_{2n+1}$
onto the subgroup of $[\om_{2n+1}] \subset Z$ of all $2n+1$ roots of unity. It is an isomorphism from $\Gamma[[1,n]]$ onto $\Theta[Z_+]|[\om_{2n+1}]$.

If $S = \{ s_1, ..., s_N \}$ is a finite a-free subset of $Z$ then because $Z_+$ is open, there exists $\ep > 0$ such that for
$T = \{ t_1, ..., t_N \} \subset Z$, $|t_i - s_i| < \ep$ for all $i$ implies that $s_i \mapsto t_i$ is a bijection from
$S$ to $T$ inducing an isomorphism from $\Theta[Z_+]|S$ to $\Theta[Z_+]|T$. There exists $n$ large enough that $S$ is contained in the
$\ep$ neighborhood of $[\om_{2n+1}]$. It follows that a tournament is isomorphic to $\Theta[Z_+]|S$ for some finite a-free subset $S$ if and
only if it is isomorphic to a subtournament of some $(\Z_{2n+1},\Gamma[[1,n]])$. On the other hand, if $\Q^*$ is any countable, dense a-free
subset of $Z$, then  we can choose points of $\Q^*$ $\ep$ close to each point of $S $
and so there exists $T \subset \Q^*$ with $\Theta[Z_+]|T$ isomorphic to $\Theta[Z_+]|S$ .

The tournament $(\Q^*,\Theta[Z_+]|\Q^*)$ for $\Q^*$ is any countable, dense a-free
subset of $Z$, represents the
remaining homogeneous isomorphism class. For example, if $\a = exp(2 \pi i a)$ with $a$ irrational,
then $k \mapsto \a^k$ is a group isomorphism from $\Z$ onto a dense, a-free subgroup of
$Z$.  Hence, with $\Q^*$ equal to the subgroup  $[\a]$,$(\Q^*,\Theta[Z_+]|\Q^*)$ is isomorphic to a group game on $\Z$.

If $z_1, z_2$ are distinct, non-antipodal points of $Z$ then they are the end-points of a unique \emph{short arc}\index{short arc}
$[z_1,z_2]$ which subtends an angle smaller than $\pi$. Clearly, if $z_1, z_2$ lie in some common semi-circle $U$, then, because
$U$ is connected, $[z_1,z_2] \subset U$. Thus, if $z \to z_1$ and either $z \to z_2$ or $z = - z_2$, then $z \to z_3$ for all $z_3 \in (z_1,z_2)$.
Similarly, $ z_1 \to z$ and either $z_2 \to z$ or $z = - z_2$ implies $z_3 \to z$ for all $z_3 \in (z_1,z_2)$.
On the other hand, for every $z \in Z$, either
$z$ or $-z$ lies in the complementary long arc, $Z \setminus [z_1,z_2]$.

\begin{lem}\label{lemsep} For any pair of disjoint closed short arcs in $ Z$,  there exists $z \in Z$ such that
the two arcs lie in distinct semi-circle components of $Z \setminus \{ z, -z \}$. That is, the $\{ z, -z \}$ diameter separates the two arcs.
\end{lem}

\begin{proof} If each of the two arcs is a single point, and $z$ is an interior point of the shorter arc between them, then the
$\{ z, -z \}$ diameter separates the points.

Now assume that $[a_1,a_2]$ and $[b_1,b_2]$ are the two arcs, with the $a_1 \to a_2$ and
with the arclength of $[b_1,b_2]$ less than or equal to that of $[a_1,a_2]$. Because $[a_1,a_2]$ is short, it is entirely contained in one of the
closed semi-circles with end-points $a_1, -a_1$. If $[b_1,b_2]$ is contained in the other closed semi-circle, then with $z \to a_1$ close enough to
$a_1$ the $\{ z, -z \}$ diameter separates the two arcs. This fails only if the arc $[b_1,b_2] $ meets the open arc $(a_2,-a_1)$.
Similarly, we can use $z$ with $a_2 \to z$ and $z$
close to $a_2$ unless $[b_1,b_2] $  meets the open arc $(-a_2,a_1)$. If $[b_1,b_2] $ meets both of these open arcs, then because it is
connected, the open arc $(b_1,b_2)$ must contain the closed arc $[-a_1,-a_2]$. This contradicts the arclength assumption.

\end{proof} \vspace{.5cm}

\begin{prop}\label{propqstar} A tournament $\Pi$ on a finite set $I$ is isomorphic to a subtournament of $(Z,\Theta[Z_+])$ if and
only if $\Pi|\Pi(i)$ and $\Pi|\Pi^{-1}(i)$ are orders for all $i \in I$. \end{prop}

\begin{proof} Because $(\Z_{2n+1},\Gamma[[1,n]])$ and each of its subtournaments satisfies this order condition, it follows that the
condition is necessary. For sufficiency, we use induction on $|I|$.

Select a vertex $i \in I$. We may assume that $\Pi|(\Pi(i)\cup \Pi^{-1}(i))$ is a subtournament of $(Z,\Theta[Z_+])$ so that
$\Pi|\Pi^{-1}(i)$ is the order $\{ a_1 \to a_2 \to \dots a_p \}$ and $\Pi|\Pi(i)$ is the order $\{ b_1 \to b_2 \to \dots b_q \}$.

If $\Pi(i) = \emptyset$, then choose $j \in Z$ so that $a_p \to j$ and close enough that $a_1 \to j$ as well. Similarly,
if $\Pi^{-1}(i) = \emptyset$, then choose $j \in Z$ so that $j \to b_1$ and close enough that $j \to b_q$.

Now assume that neither $\Pi(i)$ nor $\Pi^{-1}(i)$ is empty.

Observe that the short arcs $[a_1,a_p]$ and $[b_1,b_q]$ must be disjoint. For if we had $a \to b \to a'$ with $a,a' \in  \Pi^{-1}(i)$
and $b \in \Pi(i)$, then $\langle i, b, a' \rangle$ is a $3$-cycle in $\Pi(a)$. Similarly, $b \to a \to b'$ would imply that
$\langle i, b, a\rangle$ is a $3$-cycle in $\Pi^{-1}(b')$.

By Lemma \ref{lemsep} there exists $z \in Z$ such that the diameter through $z, -z$ separates the arcs. This implies either for $j = z$ or else for
$j = -z$,  $a \to j \to b$
for all $a \in [a_1,a_p], b\in [b_1,b_q]$.

In each of these cases mapping $i $ to $j$ completes the embedding of $\Pi$ into $\Theta[Z_+]$.

\end{proof} \vspace{.5cm}

\begin{prop}\label{prophomo2} If $\Q^*$ is a countable, dense, a-free subset of $Z$, then the tournament $(\Q^*,\Theta[Z_+]|\Q^*)$  is homogeneous. \end{prop}

\begin{proof} For an a-free  $S \subset Z$
we let $S^{\pm }$ denote the disjoint union  $S \cup -S$. From (\ref{eqcircle}) it is clear that if $\r :S \to T$ is a bijection
of a-free subsets which is a tournament isomorphism from $\Theta[Z_+]|S$ to $\Theta[Z_+]|T$, then $-z \mapsto - \r(z)$ extends
$\r $ to $\r^{\pm } $ a digraph isomorphism from $\Theta[Z_+]|S^{\pm }$ to $\Theta[Z_+]|T^{\pm }$.

Now let $S \subset \Q^*$, $T \subset \tilde \Q^*$ be finite subsets of countable dense a-free subsets of $Z$ and assume that
$\r : S \to T$ is a bijection inducing an isomorphism on the restricted tournaments. Let $\r^{\pm } : S^{\pm } \to T^{\pm }$ be the
extended digraph automorphism. Assume that $|S| = |T| = n \geq 2$. Let $\{ s_1, s_2, \dots, s_{2n} \}$ be a listing of the
points of $S^{\pm }$ in counter-clockwise order beginning from an arbitrary choice for $s_1$. Let $t_i = \r^{\pm }(s_i)$ for $i = 1, \dots, 2n$.
Any open long arc contains some element of $S^{\pm }$ (recall $n \geq 2$). It follows that each arc from $s_i$ to $s_{i+1}$ is short.
The point on the open short arc between them are those points $z$ such that $s_i \to z \to s_{i+1}$. If $s_{i+1} \to z \to s_i $ then
$- z$ lies in the open short arc $(s_i,s_{i+1})$. It follows that for every $s_j \not= \pm s_i, \pm s_{i+1}$ either $s_j \to s_i, s_{i+1}$ or
$s_i, s_{i+1} \to s_j$. Furthermore, it follows that no $t_j$ lies on the arc from $t_i$ to $t_{i+1}$. Hence, $t_1, t_2, \dots, t_{2n}$ lists the points
of $T$ in counter-clockwise order. Finally, for each $i = 1, \dots, 2n$ extend $\r^{\pm }$ by choosing an
order isomorphism from $\Q^* \cap (s_i,s_{i+1})$ to
$\tilde \Q^* \cap (t_i,t_{i+1})$. Finally, extend continuously to get a continuous automorphism of $(Z,\Theta[Z_+])$ which takes
$\Q^*$ to $\tilde \Q^*$ and which extends $\r$. 

In particular, we see that each $\Q^*$ is homogeneous.

\end{proof} \vspace{.5cm}

If $\Phi : G \times I \tto I$ is an action of the group $G$ on $I$ and $\Pi$ is a tournament on $I$, then as before, $G$ acts on
$\Pi$ when  $\Phi^g$ is a tournament automorphism of $\Pi$ for each $g \in G$. For a game subset $A \subset G$, $\Gamma[A]$ acts on
$\Pi$ when, in addition, $\Phi_i$ is a tournament morphism from $\Gamma[A]$ to $\Pi$ for each $i \in G$. As before $\Gamma[A]$ acts on
itself by left translation if and only if $A$ is normal.

If $\{ \pi_k : G_{k+1} \tto G_k \}$ is a sequence of surjections, then the \emph{inverse limit}\index{inverse limit} $G_{\infty}$
is the subset of the product  $\prod_{k=1}^{\infty} G_k$ with $x \in G_{\infty} $ when $\pi_k(x_{k+1}) = x_k$ for all $k \in \N$.
The coordinate projection $p_k : G_{\infty} \tto G_k$ is given by $x \mapsto x_k$. It is a surjection with $p_k = \pi_k \circ p_{k+1}$.
If the $G_k$'s are groups and $\pi_k$'s are group homomorphisms, then $G_{\infty}$ is a subgroup of the product group and the projections
$p_k$ are group homomorphisms.

Even when the $G_k$'s are finite, the inverse limit is usually uncountable. For example, if $\{ n_k \}$ is a strictly increasing positive integers
with $n_k | n_{k+1}$ and $\pi_k : \Z_{n_{k+1}} \tto \Z_{n_k}$ is the canonical projection taking a mod $n_{k+1}$ congruence class to its
$n_k$ congruence class, then the inverse limit group is an uncountable adding machine group. As no element has finite order,
it is an odd group even if  the  $n_k$'s are even.

Assume that each $G_k$ is an odd group with $A_k \subset G_k$ a game subset. By Proposition
\ref{prop21aa} $\pi_k$ is a tournament morphism from $ \Gamma[A_{k+1}]$ to $ \Gamma[A_k]$ if and only if
$\pi_{k}(A_{k+1} \cup \{ e \}) = A_{k} \cup \{ e \}$. We define $A_{\infty}$ to be the inverse limit of the sequence
$\{ \pi_{k}: A_{k+1} \cup \{ e \} \tto  A_{k} \cup \{ e \} \}$ with the identity element removed. Thus, $x \in G_{\infty}$ lies in $A_{\infty}$
 if and only if for some $k$, $x_k \in A_k$ which then implies $x_j \in A_j$ for all $j > k$, by Proposition
\ref{prop21aa} again. It easily follows that $A_{\infty}$ is a game subset for $G_{\infty}$. It is normal if and only if all the $A_k$'s are normal.

  \begin{theo}\label{unitheo08bb} Let $(T,\Gamma)$ be a universal tournament.
  \begin{enumerate}
  \item[(a)] If $G$ is a countable odd group, then there exists a free action of $G$ on $\Gamma$. If $G$ admits a normal game subset $A$, then
  the action can be chosen to be an action of $\Gamma[A]$ on $\Gamma$.
  \item[(b)] If $G$ is an inverse limit of countable odd groups, then there exists an effective action of $G$ on $\Gamma$. In particular, there is
  a subgroup of $Aut(\Gamma)$ which is isomorphic to $G$.
  \end{enumerate}
  \end{theo}

\begin{proof} (a): We adjust to construction leading to Lemma \ref{unilem03}. First, choose a game subset $A \subset G$ with associated group game
$\Gamma[A]$. Choose $A$ to be normal if $G$ admits such.

Let $\{ (S_k,\Pi_k) \}$ be a sequence of tournaments with $\{ G \times S_k : k \in \N \}$ a pairwise disjoint sequence of countably infinite sets, all
disjoint from the $S_k$'s as well.
Let $T_k = \bigcup_{i=1}^k G \times S_i$ and $\widehat{T_{k+1}} = T_k \cup S_{k+1}$.  Let $G$ act trivially on each $\Pi_k$.

 Let $(T_1,\Gamma_1) = (G \times S_1, \Gamma[A] \ltimes \Pi_1)$. Proceed inductively. Assume $(T_k,\Gamma_k)$ is defined so that $G$ acts freely
 on $\Gamma_k$ and as a $\Gamma[A]$ action when $A$ is normal.

Use Lemma \ref{unilem00a}(b) to
construct $(\widehat{T_{k+1}},\widehat{\Gamma_{k+1}})$ so that $\widehat{\Gamma_{k+1}}|T_k = \Gamma_k$,
$\widehat{\Gamma_{k+1}} |S_{k+1} = \Pi_{k+1}$ and so that
 $T_k$ satisfies the simple extension property in $\widehat{\Gamma_{k+1}}$.

Now we use the construction of Proposition \ref{prop23bd} to obtain $(T_{k+1},\Gamma_{k+1})$ so that
$\Gamma_{k+1}|G \times S_{k+1} = \Gamma[A] \ltimes \Pi_{k+1}$ and with the identification of
$S_{k+1}$ with $ \{ e \} \times S_{k+1} \subset G \times S_{k+1}$ we have
$\Gamma_{k+1}|\widehat{T_{k+1}} = \widehat{\Gamma_{k+1}}$. By  Proposition \ref{prop23bd} the free action of $G$
on $\Gamma_k$ extends to a free action of $G$ on $\Gamma_{k+1}$. If $A$ is normal, then the action is a $\Gamma[A]$ action.

  Finally, let $(T,\Gamma) = \bigcup_k  (T_k,\Gamma_k)$.  It is clear that $R = \bigcup_k \{ e \} \times S_k$ is a generic subset for $\Gamma$ and
  so the constructed $(T,\Gamma)$ is universal. By uniqueness, we can use an isomorphism to transfer the action to any other universal tournament.

  (b):  Assume that $G$ is the inverse limit of the sequence of group homomorphisms $\{ \pi_k : G_{k+1} \to G_k \}$ with the $G_k$'s odd groups.
  Inductively we can choose game subsets $A_k \subset G_k$ so that each $\pi_k$ is a tournament morphism from $\Gamma[A_{k+1}]$ to $\Gamma[A_k]$.

  Let $T_k = \bigcup_{i=1}^k G_i \times S_i$ and $\widehat{T_{k+1}} = T_k \cup S_{k+1}$.

  In the above inductive construction assume that $G_k$ acts effectively on $(T_k,\Gamma_k)$. Using the homomorphism $\pi_k$ we obtain
  an action of $G_{k+1}$ on $\Gamma_k$ which is, of course, not effective.

  Construct $(\widehat{T_{k+1}},\widehat{\Gamma_{k+1}})$ as before. This time we construct $(T_{k+1},\Gamma_{k+1})$
  $\Gamma_{k+1}|G_{k+1} \times S_{k+1} = \Gamma[A_{k+1}] \ltimes \Pi_{k+1}$. Since $G_{k+1}$ acts freely on $G_{k+1} \times S_{k+1}$
  it acts effectively on $T_{k+1}$.

  Again,  let $(T,\Gamma) = \bigcup_k  (T_k,\Gamma_k)$. It is easy to check that $G$ acts effectively on $(T,\Gamma)$. On the invariant subset
  $T_k$ the action is obtained by pulling back the action of $G_k$ via $p_k : G \to G_k$.

  An effective action on $(T,\Gamma)$ induces an injection of the group into $Aut(\Gamma)$.

\end{proof} \vspace{.5cm}

The inverse limit of the sequence $\{ \Z_{2^{k+1}} \tto \Z_{2^{k}} \}$ is the group of two-adic integers. Every element has infinite order and
so it is an uncountable odd group.  I do not know whether it injects into the automorphism group of the universal tournament.

Now fix $\Gamma$ a universal tournament on $\N$ and let $Aut$ denote the automorphism group of $\Gamma$. Define the subset $B_{Aut} \subset Aut$ by
  \begin{align}\label{unieq07}
  \begin{split}
\xi \in B_{Aut} \quad \Longleftrightarrow \quad &\text{there exists} \ n \in \N \ \text{such that} \\
 \xi(i) = i \ \text{for all }& \ i < n, \ \text{and} \ n \to \xi(n) \ \text{in} \ \Gamma.
  \end{split}
  \end{align}

  Since $\xi$ is an automorphism, $n \to \xi(n)$ implies $ \xi^{-1}(n) \to n$. It follows that for every $\xi \in Aut \setminus \{ 1_{\N}\}$
  either $\xi \in B_{Aut}$ or $\xi^{-1} \in B_{Aut}$ and not both. That is, $B_{Aut}$ is a game subset of the group $Aut$. By Theorem
  \ref{unitheo08bb} $Aut$ contains copies of the Klein Bottle group and so it does not admit a normal game subset.

 % For a possibly infinite group $G$ which contains no elements of even order,  $B \subset G$ is a  \emph{game subset}\index{game subset} when
%  $\{ \{ e \}, B, B^{-1} \}$ is a partition of $G$ where $e$ is the identity element and $B^{-1} = \{ g^{-1} : g \in B \}$.
%  As in the finite case define  $\Gamma[B]$ by
%  \begin{equation}\label{unieq08}
%  \Gamma[B] \ = \ \{ (g_1,g_2) \in G \times G : g_1^{-1} \cdot g_2 \in B \}.
%  \end{equation}
%  Because $B$ is a game subset,  $\Gamma[B]$ is a tournament on $G$, with $\Gamma[B](e) = B$ and $\Gamma[B]^{-1}(e) = B^{-1}$.
%  Furthermore, $\Gamma[B]$ is invariant with respect to all left translations $\ell_{t}$
% So $\Gamma[B](g) = gB$ and $\Gamma[B]^{-1}(g) = gB^{-1}$.  As in the finite case, the tournament $\Gamma[B]$ is called the
%  \emph{group game}\index{group game} on $G$ associated with the game subset $B$.

  In particular, $\Gamma[B_{Aut}]$ is a group game on $Aut$.
  Define $\pi : Aut \tto \N$ by $\pi(\xi) = \xi(1)$. If $\xi_1(1) \to \xi_2(1)$ in $\Gamma$ then $1 \to \xi_1^{-1}(\xi_2(1))$
  and so $\xi_1 \to \xi_2$ in $\Gamma[B_{Aut}]$.
 It follows that $\pi : \Gamma[B_{Aut}] \tto \Gamma$ is a tournament morphism which is surjective by Corollary \ref{unicor06} (a).
  Let $\r : \N \to Aut$ be a splitting of this projection. That is, for each $i \in \N$ we choose $\xi_i$ such that
  $\xi_i(1) = i$. We may choose $\xi_1$ to be the identity map  on $\N$. Thus, $\r : \Gamma \tto \Gamma[B_{Aut}]$ is an embedding.

  Now let $G$ be the subgroup of $Aut$ which is generated by $\{ \xi_1, \xi_2, \dots \}$. The group $G$ is countable and
  $B = B_{Aut} \cap G$ is a game subset of $G$ with associated group game $\Gamma[B] = \Gamma[B_{Aut}]|G$.
  The restriction $\pi : \Gamma[B] \tto \Gamma$
  is a projection, with splitting $\r :  \Gamma \tto \Gamma[B]$. We thus have:

  \begin{theo}\label{unitheo08a} There exists a countable group $G$ with game subset $B$ such that the associated group game
  $\Gamma[B]$ projects to, and so contains, a universal tournament. \end{theo}

 \vspace{.5cm}

  If for every isomorphism $\r : \Gamma|S_1  \to \Gamma|S_2$ with $S_1, S_2$ finite subsets of $\N$ we choose an element $\xi \in Aut$
  which extends $\r$, then we obtain a countable subset of $Aut$ and so there exists $G_1$ a countable subgroup of $Aut$ which
  contains $G$ and all of the chosen extensions $\xi$.  That is, every $\r$ extends to an element of $G_1$. Again, with $B_1 = B_{Aut} \cap G_1$
  we get a game subset  such that the associated group game
  $(G_1, \Gamma[B_1])$ projects to, and so contains, a universal tournament.

In each of these cases, one can show, as in Theorem \ref{theo23a}, that the group game on the group $G$ is the
lexicographic product of the universal tournament $\Gamma$
with the restricted group game on the subgroup $H = \{ \xi \in G : \xi(1) = 1 \}$.
\vspace{1cm}

{\bfseries Addendum:}  Here we briefly consider uncountable sets. 

Write $|X|$ for the cardinality of a set $X$.  Recall that a cardinal number ${\mathbf a}$ is the minimum
ordinal $i$ with $|i| = |{\mathbf a}|$. In particular, an infinite cardinal is a limit ordinal. Furthermore, if $i < a$, then
$|i| < |a|$.

Write $\bar {\mathbf a}$ for the cardinal of the
power set of ${\mathbf a}$, i.e.
$\bar {\mathbf a} = 2^{{\mathbf a}} > {\mathbf a}$.

Now assume that ${\mathbf a}$ is an infinite cardinal  and write ${\mathbf b}$ for the smallest cardinal
such that $\bar {\mathbf b} > \bar {\mathbf a}$. So ${\mathbf a} < {\mathbf b} \leq \bar {\mathbf a}$.
Hence,
  \begin{equation}\label{unieq12}
  i < {\mathbf b} \qquad \Longrightarrow \qquad 2^{|i|} \leq \bar {\mathbf a}.
  \end{equation}
  Note that the Generalized Continuum Hypothesis
(= GCH) says that there is no cardinal strictly between ${\mathbf a}$ and $\bar {\mathbf a}$ and so if it is assumed to hold, then
${\mathbf b} = \bar {\mathbf a}$.

For a cardinal ${\mathbf d}$ with ${\mathbf a} \leq {\mathbf d} < {\mathbf b}$ we define
  \begin{equation}\label{unieq13}
P_{{\mathbf d}}(X)  \ = \ \{ J : J \subset X, \quad \text{with} \ |J| \leq {\mathbf d} \ \text{or} \
 |X \setminus J| \leq {\mathbf d} \}.
\end{equation}
Since $\bar {\mathbf a}^{{\mathbf d}} = (2^{{\mathbf d}})^{{\mathbf d}} = 2^{{\mathbf d} \cdot {\mathbf d}} = 2^{{\mathbf d}} = \bar {\mathbf a}$,
it follows that
 \begin{equation}\label{unieq14}
|X| \leq \bar {\mathbf a} \qquad \Longrightarrow \qquad |P_{{\mathbf d}}(X)| \leq \bar {\mathbf a}.
\end{equation}

\begin{df}\label{unidef00xx} If $(T,\Gamma)$ is a tournament and $T_0 \subset T$, then we say that
$T_0$ satisfies the \emph{${\mathbf d}$ simple extension property }
\index{${\mathbf d}$ simple extension property}\index{extension property!${\mathbf d}$ simple}
in $\Gamma$ if for every  subset $J \subset T_0$ with either $|J| \leq {\mathbf d} $ or
$ |X \setminus J| \leq {\mathbf d} $ there exists $v_J \in T$ which chooses $J \subset T_0$ for $\Gamma$. \end{df}

\begin{lem}\label{unilem11a} Let $(S_0,\Pi_0), (S_1,\Pi_1)$ be tournaments with $S_0 \cap S_1 = \emptyset$ and with $|S_0| = |S_1| = \bar {\mathbf a}$.
If ${\mathbf d}$ is a cardinal with ${\mathbf d} < {\mathbf b}$,
then there exists a
tournament $\Pi$ on  set $S_0 \cup S_1$ with
$\Pi_0 = \Pi|S_0, \Pi_1 = \Pi|S_1$ and such that $S_0$ satisfies the ${\mathbf d}$ simple extension property in $\Pi$.
\end{lem}

\begin{proof}  Proceed as in the proof of Lemma \ref{unilem00a} (b) using a bijection $J \mapsto v_J$ from $P_{{\mathbf d}}(S_0)$ to $S_1$.

\end{proof} \vspace{.5cm}

%\begin{lem}\label{unilem11a}  If $(S,\Pi)$ is a tournament with $|S| \leq \bar {\mathbf a}$ and if ${\mathbf d}$ is a cardinal
%with ${\mathbf d} < {\mathbf b}$, then there exists a tournament $\Pi_1$ on a  set $S_1$ with $S \subset S_1$ and
%$\Pi = \Pi_1|S$ such that $|S_1| \leq \bar {\mathbf a}$ and  $S_0$ satisfies the simple extension property in $\Pi_1$
%for all  $S_0 \subset S$ with $|S_0| \leq {\mathbf d}$.
%\end{lem}
%
%\begin{proof}  Proceed as in the proof of Lemma \ref{unilem00a} (b) with $P_{{\mathbf d}}$ replacing $P$ in the construction.
%
%\end{proof} \vspace{.5cm}

Let $\{ (S_i,\Pi_i) : i < {\mathbf b} \}$ be an indexed family of tournaments on pairwise disjoint sets with $|S_i| = \bar {\mathbf a}$.
Let  $U_i = \bigcup_{j < i} S_j$
Now we use transfinite induction to define tournaments $(U_i,\Gamma_i)$ for every ordinal $i \leq {\mathbf b}$.
\begin{itemize}
\item Let $(U_0, \Gamma_0) = (\emptyset,\emptyset)$.
\item For an ordinal $i < {\mathbf b}$ we use Lemma \ref{unilem11a} to construct
$(U_{i+1},\Gamma_{i+1})$ with  $\Gamma_{i+1}|U_i = \Gamma_i, \ \Gamma_{i+1}|S_i = \Pi_i$ and such that
$U_i$ satisfies the ${\mathbf d}$ simple extension property in $\Gamma_{i+1}$.

\item For $i\leq {\mathbf b}$  a limit ordinal
\begin{equation}\label{unieq16}
 \Gamma_i \ = \ \bigcup_{j<i} \Gamma_j.
\end{equation}
\end{itemize}

%For all $i \leq {\mathbf b}$ $|U_i| = \bar {\mathbf a}$.

The tournament $(U_{{\mathbf b}},\Gamma_{{\mathbf b}})$ with $|U_{{\mathbf b}}| =  \bar {\mathbf a}$ is ${\mathbf b}$-universal in  the following sense.

\begin{theo}\label{unitheo12} Assume $(T,\Pi)$ is a  tournament and $T_0 \subset T$ with
$|T_0| < |T| \leq {\mathbf b}$. If
$\r : \Pi|T_0 \tto \Gamma_{ {\mathbf b}}$ is an embedding, then $\r$ extends to an embedding of
$\Pi$ into $\Gamma_{ {\mathbf b}}$. \end{theo}

\begin{proof} Observe first that if $T_0 \subset U_{{\mathbf b}}$ with $|T_0| < {\mathbf b}$, then there
exists an ordinal $i < {\mathbf b}$ such that $|T_0| \leq |i|$ and $T_0 \subset U_i$. It follows that $T_0$ satisfies the
${\mathbf d}$ simple extension property in
$\Gamma_{i+1}$ and so in $\Gamma_{{\mathbf b}}$.

It follows from Lemma \ref{unilem00a} (a) that the extension exists when  $T\setminus T_0$ consists of a single vertex $v$.

In general, count the vertices of $T \setminus T_0$ putting them in a bijective correspondence
$i \to v_i$ with $1 \leq i < {\mathbf b}' = |T \setminus T_0|$
if $|T \setminus T_0|$ is infinite and ${\mathbf b}' = |T \setminus T_0| + 1$ when $|T \setminus T_0|$ is finite.
Thus, ${\mathbf b}' \leq {\mathbf b}$. For $i < {\mathbf b}'$ let $T_i = T_0 \cup \{ v_j : j \leq i \}$.
Since ${\mathbf b}'$ is a cardinal, $|T_i| < {\mathbf b}$ for all $i < {\mathbf b}'$.
Thus, beginning with $\t_0 = \r$ we can extend $\t_i$ to $\t_{i+1}$ and take the union at limit ordinals.

\end{proof} \vspace{.5cm}

If ${\mathbf b} = \bar {\mathbf a}$, e.g. if the GCH holds, then we can, as in Theorem \ref{unitheo04}, use a back and forth argument to
show that the tournament of cardinality $\bar {\mathbf a}$ which is $\bar {\mathbf a}$-universal in the sense of
Theorem \ref{unitheo12} is unique up to isomorphism.

\vspace{1cm}

\bibliographystyle{amsplain}

\printindex

\end{document}